\documentclass[12pt]{amsart}

\usepackage{vmargin}
\usepackage[dvips]{graphicx}
\usepackage{color}
\input{amssym.def}
\input{amssym}
\usepackage{a4wide}
\usepackage{supertabular}
\usepackage{array}
\usepackage{multicol}

\setmarginsrb{3cm}{2cm}{3cm}{3cm}{75pt}{20pt}{20pt}{30mm}
\def\derpar#1#2{\frac{\partial#1}{\partial#2}}
\def\R{\mathbb R}
\def\N{\mathbb N}
\def\T{\mathbb T}
\def\C{\mathbb C}

\def\Z{\mathbb Z}

\def\var{\varepsilon}
\def\eps{\varepsilon}

\def\pa{\partial}
\def\om{\omega}
\def\Om{\Omega}
\def\om{\omega}
\def\OmX{\Omega X} 
\def\OmV{\Omega V}
\def\ov{\overline}
\def\cal{\mathcal}

\def\hat{\widehat}

\def\eps{\varepsilon}

\def\tilde{\widetilde}

\def\Id{{\rm Id}\,}
\def\dps{\displaystyle}

\def\Lip{{\rm Lip}}
\def\<{\langle}
\def\>{\rangle}

\def\dag{\dagger}
\def\ddag{\ddagger}

\def\und{\underline}
\def\sign{{\rm sign\,}}

\def\nohat#1{#1}
\def\RR{\mathcal{R}}
\def\K{\mathcal{K}}
\def\II{{\rm I\!I}}
\def\III{{\rm I\!I\!I}}
\def\Tm{T_{\rm max}}

\def\trn{\|\!|} 
\def\btrn{\Big\|\!\Big|}
\def\qn{\|\!\|} 
\def\Bqn{\Big\|\!\Big\|}
\def\bqn{\big\|\!\big\|}
\def\bbqn{\bigg\|\!\bigg\|}
\def\qqquad{\qquad\qquad\qquad\qquad\qquad}
\def\Cnk#1#2{\small \bigl(\!\begin{array}{c} #1 
                \\ #2 \end{array} \!\bigr)} 

\def\cF{{\cal F}} 
\def\cC{{\cal C}}
\def\cZ{{\cal Z}}
\def\cY{{\cal Y}}
\def\cX{{\cal X}}
\def\cW{{\cal W}}

\def\med{\medskip}
\def\sm{\smallskip}
\def\bul{$\bullet$\ }

\def\begeq{\begin{equation}}
\def\endeq{\end{equation}}
\def\begar{\begin{eqnarray}}
\def\endar{\end{eqnarray}}
\def\begar*{\begin{eqnarray*}}
\def\endar*{\end{eqnarray*}}
\def\begal{\begin{align}}
\def\endal{\end{align}}
\def\begal*{\begin{align*}}
\def\endal*{\end{align*}}

\newtheorem{Thm}{Theorem}
\newtheorem{Lem}[Thm]{Lemma}

\newtheorem{Cor}[Thm]{Corollary}
\newtheorem{Prop}[Thm]{Proposition}

\numberwithin{equation}{section}
\numberwithin{Thm}{section}

\theoremstyle{definition}
\newtheorem{Def}[Thm]{Definition}
\newtheorem{Rk}[Thm]{Remark}

\theoremstyle{remark}

\newtheorem{Ex}[Thm]{Example}

\newtheorem*{Thm*}{Theorem}
\newtheorem*{Lem*}{Lemma}
\newtheorem*{Conj*}{Conjecture}
\newtheorem*{Cor*}{Corollary}
\newtheorem*{Def*}{Definition}
\newtheorem*{Prop*}{Proposition}
\newtheorem*{Exo*}{Exercise}
\newtheorem*{Exs*}{Examples}
\newtheorem*{Ex*}{Example}
\newtheorem*{Rk*}{Remark}
\newtheorem*{Rks*}{Remarks}

\def\signcv{\bigskip \begin{center} {\sc C\'edric Villani\par\vspace{3mm}
ENS Lyon \& Institut Universitaire de France\par
UMPA, UMR CNRS 5669\par
46 all\'ee d'Italie\par
69364 Lyon Cedex 07\par
FRANCE\par\vspace{3mm}
e-mail:} \tt{cvillani@umpa.ens-lyon.fr} \end{center}}

\def\signcm{\bigskip \begin{center} {\sc Cl\'ement
      Mouhot\par\vspace{3mm}
University of Cambridge\par
DAMTP, Centre for Mathematical Sciences\par
Wilberforce Road\par
Cambridge CB3 0WA\par
ENGLAND\par
{\it On leave from:}\par
\'ENS Paris \& CNRS\par
DMA, UMR CNRS 8553\par
45 rue d'Ulm\par
F 75320 Paris cedex 05\par
FRANCE\par\vspace{3mm}
e-mail:} \tt{Clement.Mouhot@ens.fr} \end{center}}

\begin{document}

\title[On Landau damping]{On Landau damping \\[1mm]
  {(Second series, \today)}}

\vspace*{-8mm}

\author{C. Mouhot}
\author{C. Villani}

\begin{abstract} Going beyond the linearized study has been a
  longstanding problem in the theory of Landau damping.  In this paper
  we establish exponential Landau damping in analytic regularity.  The
  damping phenomenon is reinterpreted in terms of transfer of
  regularity between kinetic and spatial variables, rather than
  exchanges of energy; phase mixing is the driving mechanism.  The analysis involves new families of analytic
  norms, measuring regularity by comparison with solutions of the free
  transport equation; new functional inequalities; a control of
  nonlinear echoes; sharp scattering estimates; and a Newton
  approximation scheme. Our results hold for any potential no more
  singular than Coulomb or Newton interaction; the limit cases are
  included with specific technical effort. As a side result, the
  stability of homogeneous equilibria of the nonlinear Vlasov equation
  is established under sharp assumptions. We point out the strong
  analogy with the KAM theory, and discuss physical implications.
\end{abstract}

\maketitle

\vspace*{-13mm}

\tableofcontents

\vspace*{-11mm}

{\bf Keywords.} Landau damping; plasma physics; galactic dynamics;
Vlasov-Poisson equation.

{\bf AMS Subject Classification.} 35B35 (35J05, 62E20, 70F15,
82C99, 85A05, 82D10, 35Q60, 76X05).

\bigskip

Landau damping may be the single most famous mystery of classical
plasma physics. For the past sixty years it has been treated in the
linear setting at various degrees of rigor; but its nonlinear version
has remained elusive, since the only available results
\cite{caglimaff:VP:98,HV:landau} prove the existence of {\em some}
damped solutions, without telling anything about their genericity.

In the present work we close this gap by treating the nonlinear
version of Landau damping in arbitrarily large times, under
assumptions which cover both attractive and repulsive interactions, of
any regularity down to Coulomb/Newton.

This will lead us to discover a distinctive mathematical theory of
Landau damping, complete with its own functional spaces and functional
inequalities.  Let us make it clear that this study is not just for
the sake of mathematical rigor: indeed, we shall get new insights in
the physics of the problem, and identify new mathematical phenomena.

The plan of the paper is as follows.

In Section \ref{sec:intro} we provide an introduction to Landau
damping, including historical comments and a review of the existing
literature. Then in Section \ref{sec:main}, we state and comment on
our main result about ``nonlinear Landau damping'' (Theorem
\ref{thmmain}).

In Section \ref{sec:linear} we provide a rather complete treatment of
linear Landau damping, slightly improving on the existing results both
in generality and simplicity.  This section can be read independently
of the rest.

In Section \ref{sec:analytic} we define the spaces of analytic
functions which are used in the remainder of the paper. The careful
choice of norms is one of the keys of our analysis; the complexity of
the problem will naturally lead us to work with norms having up to 5
parameters. As a first application, we shall revisit linear Landau
damping within this framework.

In Sections \ref{sec:scattering} to \ref{sec:response} we establish
four types of new estimates (scattering estimates, short-term and
long-term regularity extortion, echo control); these are the key
sections containing in particular the physically relevant new
material.

In Section \ref{sec:approx} we adapt the Newton algorithm to the
setting of the nonlinear Vlasov equation. Then in Sections
\ref{sec:local} to \ref{sec:coulomb} we establish some iterative
estimates along this scheme. (Section \ref{sec:coulomb} is devoted
specifically to a technical refinement allowing to handle
Coulomb/Newton interaction.)

From these estimates our main theorem is easily deduced in Section
\ref{sec:convergence}.

An extension to non-analytic perturbations is presented in Section
\ref{sec:NA}.

Some counterexamples and asymptotic expansions are studied in Section
\ref{sec:cex}.

Final comments about the scope and range of applicability of these
results are provided in Section \ref{sec:beyond}.  \sm

Even though it basically proves one main result, this paper is very
long. This is due partly to the intrinsic complexity and richness of
the problem, partly to the need to develop an adequate functional
theory from scratch, and partly to the inclusion of remarks,
explanations and comments intended to help the reader understand the
proof and the scope of the results. The whole process culminates in
the extremely technical iteration performed in Sections \ref{sec:iter}
and \ref{sec:coulomb}. A short summary of our results and methods of
proofs can be found in the expository paper \cite{MV:JMP}.

This project started from an unlikely conjunction of discussions of
the authors with various people, most notably Yan Guo, Dong Li, Freddy
Bouchet and \'Etienne Ghys. We also got crucial inspiration from the
books \cite{BT1,BT2} by James Binney and Scott Tremaine; and
\cite{AlGe} by Serge Alinhac and Patrick G\'erard. Warm thanks to
Julien Barr\'e, Jean Dolbeault, Thierry Gallay, Stephen Gustafson,
Gregory Hammett, Donald Lynden-Bell, Michael Sigal, \'Eric S\'er\'e
and especially Michael Kiessling for useful exchanges and references;
and to Francis Filbet and Irene Gamba for providing numerical
simulations. We are also grateful to Patrick Bernard, Freddy Bouchet,
Emanuele Caglioti, Yves Elskens, Yan Guo, Zhiwu Lin, Michael Loss, Peter
Markowich, Govind Menon, Yann Ollivier, Mario Pulvirenti, Jeff Rauch,
Igor Rodnianski, Peter Smereka, Yoshio Sone, Tom Spencer, and the team
of the Princeton Plasma Physics Laboratory for further constructive
discussions about our results. Finally, we acknowledge the generous
hospitality of several institutions: Brown University, where the first
author was introduced to Landau damping by Yan Guo in early 2005; the
Institute for Advanced Study in Princeton, who offered the second
author a serene atmosphere of work and concentration during the best
part of the preparation of this work; Cambridge University, who
provided repeated hospitality to the first author thanks to the Award
No. KUK-I1-007-43, funded by the King Abdullah University of Science
and Technology (KAUST); and the University of Michigan, where
conversations with Jeff Rauch and others triggered a significant
improvement of our results.

\section{Introduction to Landau damping}
\label{sec:intro}

\subsection{Discovery}

Under adequate assumptions (collisionless regime, nonrelativistic
motion, heavy ions, no magnetic field), a dilute plasma is well
described by the nonlinear Vlasov--Poisson equation 
\begeq\label{VP}
\derpar{f}{t} + v\cdot\nabla_x f + \frac{F}{m} \cdot\nabla_v f =0,
\endeq
where $f=f(t,x,v)\geq 0$ is the density of electrons in phase space
($x$=\, position, $v$=\, velocity), $m$ is the mass of an electron, and
$F=F(t,x)$ is the mean-field (self-consistent) electrostatic force: 
\begeq\label{FE} F =
-e \, E,\qquad E = \nabla \Delta^{-1} (4\pi\rho).
\endeq
Here $e>0$ is the absolute electron charge, $E=E(t,x)$ is the electric field,
and $\rho=\rho(t,x)$ is the density of charges 
\begeq\label{rhoe} \rho
= \rho_i - e \, \int f\,dv,
\endeq
$\rho_i$ being the density of charges due to ions.  This model and
its many variants are of tantamount importance in plasma physics
\cite{akhiezer,balescu:charged:63,KT:plasma,LL:kin:81}.

In contrast to models incorporating collisions
\cite{vill:handbook:02}, the Vlasov--Poisson equation is {\em
  time-reversible}. However, in 1946 Landau \cite{landau} stunned the
physical community by predicting an irreversible behavior on the basis
of this equation. This ``astonishing result'' (as it was called in \cite{stix})
relied on the solution of the Cauchy
problem for the linearized Vlasov--Poisson equation around a
spatially homogeneous Maxwellian (Gaussian) equilibrium.  Landau formally solved the
equation by means of Fourier and Laplace transforms, and after a study
of singularities in the complex plane, concluded that the electric
field decays exponentially fast; he further studied the rate of decay as a
function of the wave vector $k$. Landau's computations are reproduced
in \cite[Section~34]{LL:kin:81} or \cite[Section~4.2]{akhiezer}.

An alternative argument appears in \cite[Section 30]{LL:kin:81}: there
the thermodynamical formalism is used to compute the amount of heat
$Q$ which is dissipated when a (small) oscillating electric field
$E(t,x) = E\, e^{i (k\cdot x - \omega t)}$ ($k$ a wave vector,
$\omega>0$ a frequency) is applied to a plasma whose distribution
$f^0$ is homogeneous in space and isotropic in velocity space; the result is 
\begeq\label{Q} Q =
- |E|^2\, \,\frac{\pi m e^2 \omega}{|k|^2}\,
\phi'\left(\frac{\omega}{|k|}\right),
\endeq
where $\phi(v_1) = \int f^0(v_1,v_2,v_3)\,dv_2\,dv_3$. In particular,
\eqref{Q} is always positive (see the last remark in \cite[Section
30]{LL:kin:81}), which means that the system reacts against the perturbation, and thus possesses some
``active'' stabilization mechanism.

A third argument \cite[Section 32]{LL:kin:81} consists in studying the
dispersion relation, or equivalently searching for the (generalized)
eigenmodes of the linearized Vlasov--Poisson equation, now with
complex frequency $\omega$. After appropriate selection, these
eigenmodes are all {\em decaying} ($\Im \omega<0$) as $t\to\infty$.  This again suggests stability, although
in a somewhat weaker sense than the computation of heat release.

The first and third arguments also apply to the {\em gravitational}
Vlasov--Poisson equation, which is the main model for nonrelativistic
galactic dynamics. This equation is similar to \eqref{VP}, but now $m$
is the mass of a typical star (!), and $f$ is the density of stars in phase
space; moreover the first equation of \eqref{FE} and the relation
\eqref{rhoe} should be replaced by 
\begeq\label{FGrho} F = - {\cal G}  
m E, \qquad \rho = m \, \int f\,dv;
\endeq
where ${\cal G}$ is the gravitational constant, $E$ the gravitational
field, and $\rho$ the density of mass. The books by Binney and
Tremaine \cite{BT1,BT2} constitute excellent references about the use
of the Vlasov--Poisson equation in stellar dynamics --- where it is
often called the ``collisionless Boltzmann equation'', see footnote on
\cite[p.~276]{BT2}.  On ``intermediate'' time scales, the Vlasov--Poisson equation is 
thought to be an accurate description of very large star systems \cite{FP:book}, which are now
accessible to numerical simulations.

Since the work of Lynden-Bell \cite{LB:violent} it has been
recognized that Landau damping, and wilder collisionless relaxation
processes generically dubbed ``violent relaxation'', constitute a
fundamental stabilizing ingredient of galactic dynamics.  Without
these still poorly understood mechanisms, the surprisingly short time
scales for relaxation of the galaxies would remain unexplained.

One main difference between the electrostatic and the gravitational interactions is that in the latter case 
Landau damping should occur only at wavelengths smaller than the {\bf Jeans length} \cite[Section 5.2]{BT2}; 
beyond this scale, even for Maxwellian velocity profiles, the Jeans instability takes over and
governs planet and galaxy aggregation.\footnote{or at least would do, if galactic matter was smoothly distributed; 
in presence of ``microscopic'' heterogeneities, a phase transition for aggregation can occur far below this scale
\cite{kiessling:perso}. In the language of statistical mechanics, the Jeans length corresponds to a ``spinodal point'' 
rather than a phase transition~\cite{SKS}.}

On the contrary, in (classical) plasma physics, Landau damping should hold at all scales under suitable assumptions on 
the velocity profile; and in fact one is in general not interested in scales smaller than the {\bf Debye length},
which is roughly defined in the same way as the Jeans length.

Nowadays, not only has Landau damping become a cornerstone of plasma
physics\footnote{Ryutov \cite{ryutov} estimated in 1998 that
  ``approximately every third paper on plasma physics and its
  applications contains a direct reference to Landau damping''.}, but
it has also made its way in other areas of physics (astrophysics, but
also wind waves, fluids, superfluids,\ldots) and even biophysics.
One may
consult the concise survey papers \cite{ONC,ryutov,vekstein} for a
discussion of its influence and some applications.

\subsection{Interpretation}

True to his legend, Landau deduced the damping effect from a mathematical-style 
study\footnote{not completely rigorous from the mathematical point of view, but formally correct, in contrast to
the previous studies by Landau's fellow physicists --- as Landau himself pointed out without mercy \cite{landau}.},
without bothering to give a physical explanation of the underlying mechanism. His arguments anyway yield exact formulas,
which in principle can be checked experimentally, and indeed provide good qualitative agreement with
observations \cite{malmbergwharton}.

A first set of problems in the interpretation is related to the arrow
of time.  In the thermodynamic argument, the exterior field is
awkwardly imposed from time $-\infty$ on; moreover, reconciling a
positive energy dissipation with the reversibility of the equation is
not obvious. In the dispersion argument, one has to arbitrarily impose
the location of the singularities taking into account the arrow of
time; at mathematical level this is equivalent to a choice of
principal value:
\[ \frac1{z-i\, 0} = {\rm p.v.} \left(\frac1{z}\right) + i\pi \,
\delta_0.\] 
This is not so serious, but then the spectral study
requires some thinking. All in all, the most convincing argument remains
Landau's original one, since it is based only on the study of the
Cauchy problem, which makes more physical sense than the study of the
dispersion relation (see the remark in \cite[p.~682]{BT1}).

A more fundamental issue resides in the use of analytic function
theory, with contour integration, singularities and residue computation, which has
played a major role in the theory of the Vlasov--Poisson equation ever
since Landau \cite[Chapter 32]{LL:kin:81} \cite[Subsection
5.2.4]{BT2} and helps little, if at all, to understand the underlying
physical mechanism.\footnote{Van Kampen \cite{VK:plasma} summarizes the conceptual
problems posed to his contemporaries by Landau's treatment, and comments on more
or less clumsy attempts to resolve the apparent paradox caused by the singularity
in the complex plane.}

The most popular interpretation of Landau damping considers the
phenomenon from an energetic point of view, as the result of the
interaction of a plasma wave with particles of nearby velocity
\cite[p.~18]{landau:terhaar:book} \cite[p.~412]{BT2}
\cite[Section~4.2.3]{akhiezer} \cite[p.~127]{LL:kin:81}.  In a
nutshell, the argument says that dominant exchanges occur with those
particles which are ``trapped'' by the wave because their velocity is
close to the wave velocity. If the distribution function is a
decreasing function of $|v|$, among trapped particles more are
accelerated than are decelerated, so the wave loses energy to the
plasma --- or the plasma surfs on the wave --- and the wave is
damped by the interaction.

Appealing as this image may seem, to a mathematically-oriented mind it
will probably make little sense at first hearing.\footnote{Escande 
\cite[Chapter 4, Footnote 6]{escande:houches} 
points out some misconceptions associated with the surfer image.}
A more down-to-Earth 
interpretation emerged in the fifties from the ``wave packet'' analysis of 
Van Kampen \cite{VK:plasma} and Case \cite{Case:plasma}: Landau damping would result 
from {\bf phase mixing}.  This phenomenon, well-known in galactic dynamics,
describes the damping of oscillations occurring when a continuum is
transported in phase space along an anharmonic Hamiltonian flow
\cite[pp. 379--380]{BT2}.  The mixing results from the simple fact
that particles following different orbits travel at different 
angular\footnote{``Angular'' here refers to action-angle variables, and applies
even for straight trajectories in a torus.}
speeds, so perturbations start ``spiralling'' (see Figure 4.27 on
\cite[p. 379]{BT2}) and homogenize by fast spatial oscillation.  From
the mathematical point of view, phase mixing results in {\em weak
  convergence}; from the physical point of view, this is just the
convergence of observables, defined as averages over the velocity space
(this is sometimes called ``convergence in the mean'').

At first sight, both points of view seem hardly compatible: Landau's
scenario suggests a very smooth process, while phase mixing
involves tremendous oscillations.  The coexistence of these two
interpretations did generate some speculation on the nature of the damping, 
and on its relation to phase mixing, see e.g. \cite{kandrup:violent:98} or
\cite[p.~413]{BT2}.  There is actually no contradiction between the
two points of view: many physicists have rightly pointed out that that
Landau damping should come with filamentation and oscillations of the
distribution function \cite[p.~962]{VK:plasma} \cite[p.~141]{LL:kin:81}
\cite[Vol.~1, pp.~223--224]{akhiezer} \cite[pp.~294--295]{LB:damping}.
Nowadays these oscillations can be visualized spectacularly thanks to
deterministic numerical schemes, see e.g. \cite{ZGS:landau}
\cite[Fig.~3]{HGMM} \cite{filbet:web}. We reproduce below some examples provided by Filbet.

\begin{figure}[htbp]
\begin{minipage}[t]{.4\linewidth}
  \includegraphics[height=5cm]{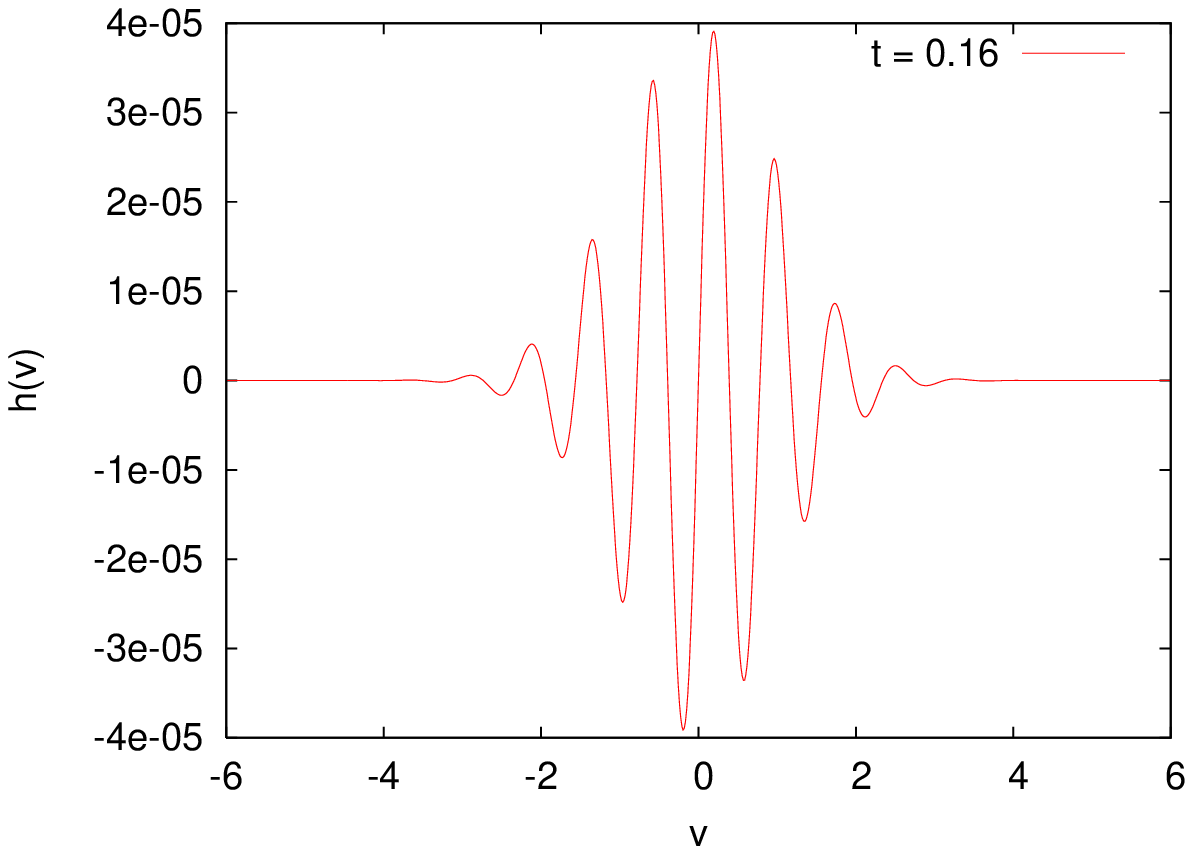}
\end{minipage}\hfill
\begin{minipage}[t]{.4\linewidth}
  \includegraphics[height=5cm]{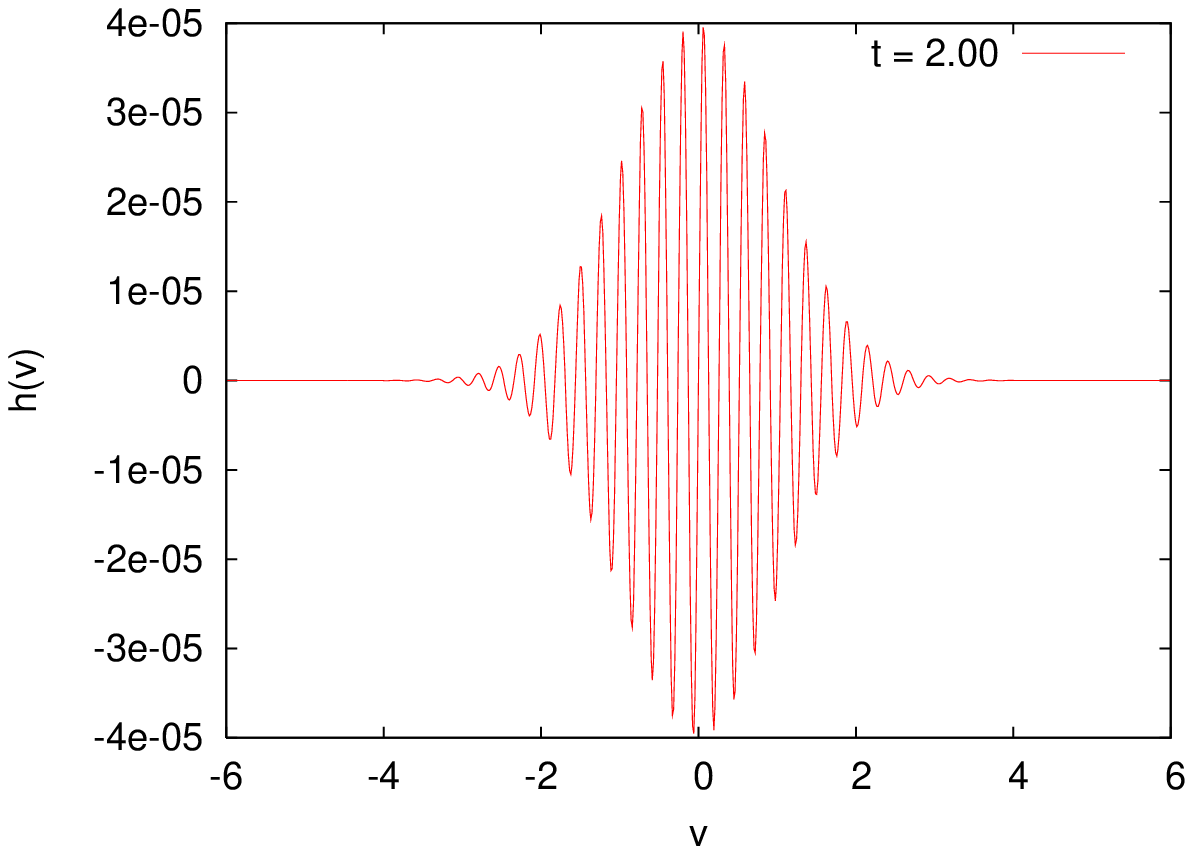}
\end{minipage}\hfill \hspace*{1mm}
\caption{A slice of the distribution function (relative to a homogeneous equilibrium)
for gravitational Landau damping, at two different times.}
\label{fig1}
\end{figure}

\begin{figure}[htbp]
\begin{minipage}[t]{.4\linewidth}
  \includegraphics[height=5cm]{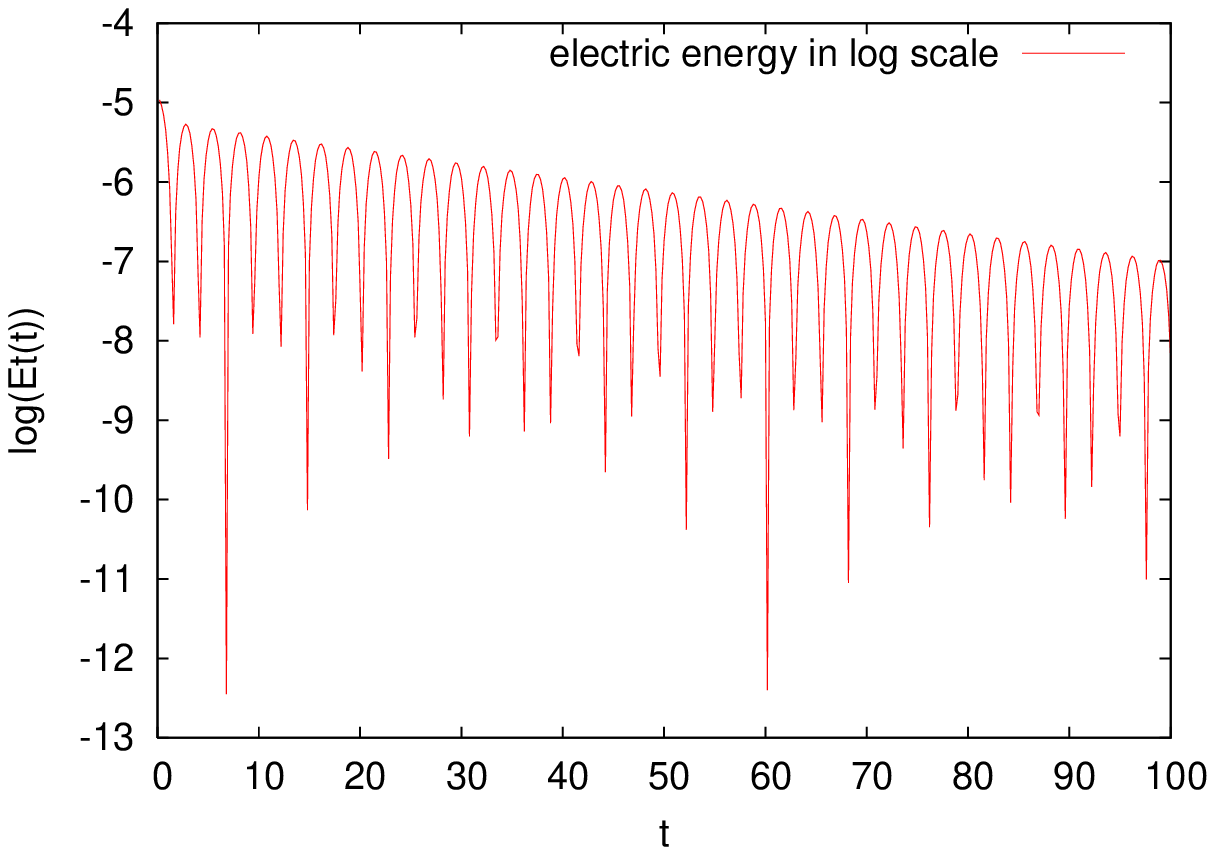}
\end{minipage}\hfill
\begin{minipage}[t]{.4\linewidth}
  \includegraphics[height=5cm]{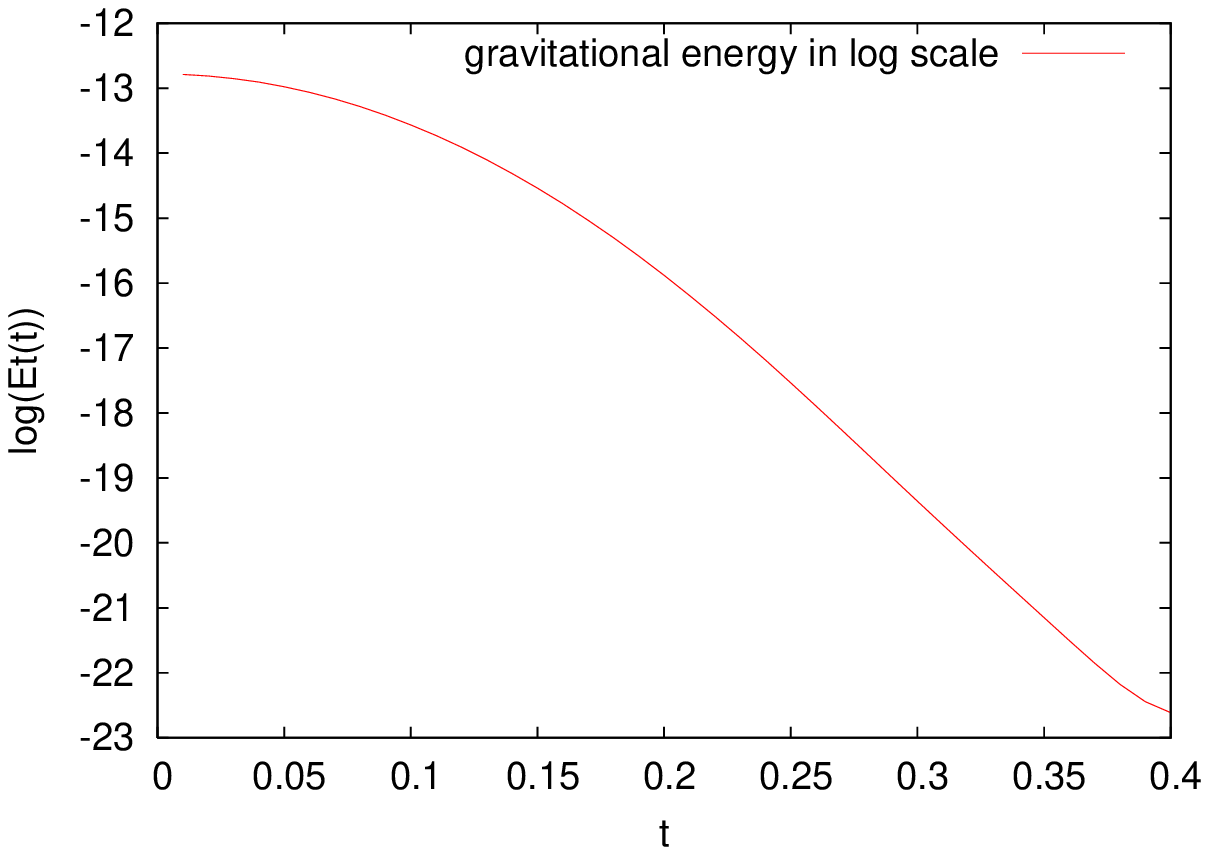}
\end{minipage}\hfill \hspace*{1mm}
\caption{Time-evolution of the norm of the field, for electrostatic
  (on the left) and gravitational (on the right) interactions. Notice
  the fast Langmuir oscillations in the electrostatic case.}
\label{fig2}
\end{figure}

In any case, there is still no definite interpretation of Landau damping: as noted by Ryutov
\cite[Section 9]{ryutov}, papers devoted to the interpretation and teaching of Landau damping
were still appearing regularly fifty years after its discovery; to quote just a couple
of more recent examples let us mention works by Elskens and 
Escande \cite{elskens:fields,EE:book,escande:houches}.
The present paper will also contribute a new point of view.

\subsection{Range of validity}

The following issues are addressed in the literature
\cite{isichenko:nonlinearlandau,kandrup:violent:98,manfredi:landau:97,ZGS:landau}
and slightly controversial: \sm

\bul Does Landau damping really hold for gravitational interaction?
The case seems thinner in this situation than for plasma interaction,
all the more that there are many instability results in the
gravitational context; up to now there has been no
consensus among mathematical physicists \cite{rein:perso}.
(Numerical evidence is not conclusive because of the difficulty of accurate simulations
in very large time --- even in one dimension of space.)
\sm

\bul Does the damping hold for unbounded systems? Counterexamples from
\cite{glasseyschaeffer:decayvlasov:94,glasseyschaeffer:decayvlasov:95}
show that some kind of confinement is necessary, even in the
electrostatic case. More precisely, Glassey and Schaeffer show that a
solution of the linearized Vlasov--Poisson equation in the whole space
(linearized around a homogeneous equilibrium $f^0$ of infinite mass)
decays at best like $O(t^{-1})$, modulo logarithmic corrections, for
$f^0(v) = c/(1+|v|^2)$; and like $O((\log t)^{-\alpha})$ if $f^0$ is a
Gaussian. In fact, Landau's original calculations already indicated
that the damping is {\em extremely} weak at large wavenumbers; see the
discussion in \cite[Section 32]{LL:kin:81}. Of course, in the
gravitational case, this is even more dramatic because of the Jeans
instability.  \sm

\bul Does convergence hold in infinite time for the solution of the
``full'' nonlinear equation?  This is not clear at all since there is
no mechanism that would keep the distribution close to the original
equilibrium {\em for all
  times}.  
Some authors do not believe that there is convergence as $t\to\infty$;
others believe that there is convergence but argue that it should be
very slow \cite{isichenko:nonlinearlandau}, say $O(1/t)$. In the first
mathematically rigorous study of the subject, Backus \cite{backus} notes 
that in general the linear and nonlinear evolution break apart after some 
(not very large) time, and questions the validity of the
linearization.\footnote{From the abstract: ``The linear theory predicts that in stable plasmas the neglected
term will grow linearly with time at a rate proportional to the initial disturbance amplitude,
destroying the validity of the linear theory, and vitiating positive conclusions about stability
based on it.''} O'Neil \cite{oneil} argues that relaxation holds in
the ``quasilinear regime'' on larger time scales, when the ``trapping
time'' (roughly proportional the inverse square root of the size of
the perturbation) is much smaller than the damping time.  Other
speculations and arguments related to trapping appear in many sources,
e.g. \cite{manfredi:landau:97,MDRS}. Kaganovich \cite{kaganovich} argues that nonlinear
effects may quantitatively affect Landau damping related phenomena by
several orders of magnitude.

The so-called ``quasilinear relaxation theory''
\cite[Section~49]{LL:kin:81} \cite[Section~9.1.2]{akhiezer}
\cite[Chapter~10]{KT:plasma} uses second-order approximation of the
Vlasov equation to predict the convergence of the spatial average of
the distribution function. The procedure is most esoteric, involving
averaging over statistical ensembles, and diffusion equations with
discontinuous coefficients, acting only near the resonance velocity
for particle-wave exchanges. Because of these discontinuities, the
predicted asymptotic state is discontinuous, and collisions are
invoked to restore smoothness.  Linear Fokker--Planck
equations\footnote{These equations act on some ensemble average of the
  distribution; they are different from the Vlasov--Landau equation.}
in velocity space have also been used in astrophysics
\cite[p.~111]{LB:violent}, but only on phenomenological grounds (the
{\it ad hoc} addition of a friction term leading to a Gaussian
stationary state); and this procedure has been exported to the study
of two-dimensional incompressible fluids \cite{chavanis:euler,CSR}.

Even if it were more rigorous, quasilinear theory only aims at
second-order corrections. But the effect of higher order perturbations
might be even worse: think of something like $e^{-t} \sum_n (\var^n
t^n)/\sqrt{n!}$: truncation at any order in $\var$ converges
exponentially fast as $t\to\infty$, but the whole sum diverges to
infinity.

Careful numerical simulation \cite{ZGS:landau} seems to show that the
solution of the nonlinear Vlasov--Poisson equation does converge to a
spatially homogeneous distribution, but only as long as the size of
the perturbation is small enough. We shall call this phenomenon
{\bf nonlinear Landau damping}. This terminology summarizes well the
problem, still it is subject to criticism since (a) Landau himself
sticked to the linear case and did not discuss the large-time
convergence of the data; (b) damping is expected to hold when the
regime is close to linear, but not necessarily when the nonlinear term
dominates\footnote{although phase mixing might still play a crucial
  role in violent relaxation or other unclassified nonlinear
  phenomena.}; and (c) this expression is also used to
designate a related but different phenomenon \cite[Section~10.1.3]{akhiezer}.\sm

\bul Is Landau damping related to the more classical notion of
stability in orbital sense?  Orbital stability means that the system,
slightly perturbed at initial time from an equilibrium distribution,
will always remain close to this equilibrium. Even in the favorable
electrostatic case, stability is not granted; the most prominent
phenomenon being the Penrose instability \cite{penrose:instability}
according to which a distribution with two deep bumps may be
unstable. In the more subtle gravitational case, various stability and
instability criteria are associated with the names of Chandrasekhar,
Antonov, Goodman, Doremus, Feix, Baumann, \ldots \cite[Section
7.4]{BT2}.  There is a widespread agreement (see e.g. the comments in
\cite{ZGS:landau}) that Landau damping and stability are related, and
that Landau damping cannot be hoped for if there is no orbital
stability.

\subsection{Conceptual problems}

Summarizing, we can identify three main conceptual obstacles which
make Landau damping mysterious, even sixty years after its
discovery: \sm

{\bf (i)} The equation is time-reversible, yet we are looking for an
irreversible behavior as $t\to +\infty$ (or $t\to -\infty$). The value
of the entropy does not change in time, which physically speaking
means that there is no loss of information in the distribution
function.  The spectacular experiment of the ``plasma echo''
illustrates this conservation of microscopic information
\cite{echo,echo:expe}: a plasma which is apparently back to
equilibrium after an initial disturbance, will react to a second
disturbance in a way that shows that it has not forgotten the first
one.\footnote{Interestingly enough, this experiment was suggested as a
  way to evaluate the strength of irreversible phenomena going on
  inside a plasma, e.g. the collision frequency, by measuring attenuations
  with respect to the predicted echo. See \cite{SOC} for an
  interesting application and striking pictures.}  And at the linear
level, if there are decaying modes, there also has to be growing
modes!  \sm

{\bf (ii)} When one perturbs an equilibrium, there is no mechanism
forcing the system to go back to this equilibrium in large time; so
there is no justification in the use of linearization to predict the
large-time behavior.\sm

{\bf (iii)} At the technical level, Landau damping (in Landau's own treatment) rests on
analyticity, and its most attractive interpretation is in terms of
phase mixing. But both phenomena are {\em incompatible in the
  large-time limit}: phase mixing implies an irreversible
deterioration of analyticity. For instance, it is easily checked that
free transport induces an exponential growth of analytic norms as
$t\to \infty$ --- except if the initial datum is spatially homogeneous.
In particular, the Vlasov--Poisson equation is {\bf unstable} (in
large time) in any norm incorporating velocity regularity.
(Space-averaging is one of the ingredients used in the quasilinear
theory to formally get rid of this instability.)  \med

How can we respond to these issues?

One way to solve the first problem (time-reversibility) is to appeal
to Van Kampen modes as in \cite[p.~415]{BT2}; however these are not so
physical, as noticed in \cite[p.~682]{BT1}. A simpler conceptual
solution is to invoke the notion of weak convergence: reversibility
manifests itself in the conservation of the information contained in
the density function; but information may be lost irreversibly in the
limit when we consider weak convergence. Weak convergence only
describes the long-time behavior of arbitrary observables, each of
which does not contain as much information as the density
function.\footnote{In Lynden-Bell's appealing words
  \cite[p.295]{LB:damping}, ``a system whose density has achieved a
  steady state will have information about its birth still stored in
  the peculiar velocities of its stars.''}  As a very simple
illustration, consider the time-reversible evolution defined by
$u(t,x) = e^{itx} u_i(x)$, and notice that it does converge weakly to
0 as $t\to \pm\infty$; this convergence is even exponentially fast if
the initial datum $u_i$ is analytic. (Our example is not chosen at
random: although it is extremely simple, it may be a good illustration
of what happens in phase mixing.)  In a way, microsocopic reversibility is
compatible with macroscopic irreversibility, provided that the
  ``microscopic regularity'' is destroyed asymptotically.

Still in respect to this reversibility, it should be noted that the
``dual'' mechanism of {\bf radiation}, according to which an
infinite-dimensional system may lose energy towards very large scales,
is relatively well understood and recognized as a crucial stability
mechanism \cite{BFS:radiation,SW:radiation}.

The second problem (lack of justification of the linearization)
only indicates that there is a wide gap between
the understanding of linear Landau damping, and that of the
nonlinear phenomenon. Even if unbounded corrections appear in the linearization
procedure, the effect of the large terms might be averaged over time
or other variables.

The third problem, maybe the most troubling from an analyst's perspective, does not dismiss the
phase mixing explanation, but suggests that we shall have to keep
track of the initial time, in the sense that a rigorous proof cannot
be based on the propagation of some phenomenon. 
This situation is of course in sharp contrast with the study of
dissipative systems possessing a Lyapunov functional, as do many
collisional kinetic equations \cite{vill:handbook:02,vill:hypoco}; it
will require completely different mathematical techniques.

\subsection{Previous mathematical results}

At the linear level, the first rigorous treatments of Landau damping were performed
in the sixties; see Saenz \cite{saenz} for rather complete results and a review
of earlier works. The theory was rediscovered and renewed at the beginning of the eighties 
by Degond \cite{degond:landau}, and Maslov and Fedoryuk \cite{MF:landaudamping}.  
In all these works, analytic arguments play a crucial role (for instance for the analytic extension of resolvent operators),
and asymptotic expansions for the electric field associated to the linearized Vlasov--Poisson equation are obtained.

Also at the linearized level, there are counterexamples by Glassey and Schaeffer
\cite{glasseyschaeffer:decayvlasov:94,glasseyschaeffer:decayvlasov:95}
showing that there is in general no exponential decay for the
linearized Vlasov--Poisson equation without analyticity, or without
confining.

In a {\em nonlinear} setting, the only rigorous treatments so far are
those by Caglioti--Maffei \cite{caglimaff:VP:98}, and later
Hwang--V\'elazquez \cite{HV:landau}. Both sets of authors work in the
one-dimensional torus and use fixed-point theorems and perturbative
arguments to prove the {\em existence} of a class of analytic
solutions behaving, asymptotically as $t\to +\infty$, and in a strong
sense, like perturbed solutions of free transport. Since solutions of
free transport weakly converge to spatially homogeneous distributions,
the solutions constructed by this ``scattering'' approach are indeed
damped.
The weakness of these results is that they say nothing about the
initial perturbations leading to such solutions, which could be very
special. In other words: damped solutions do exist, but do we ever
reach them?

Sparse as it may seem, this list is kind of exhaustive. On the other
hand, there is a rather large mathematical literature on the orbital
stability problem, due to Guo, Rein, Strauss, Wolansky and Lemou--M\'ehats--Rapha\"el.
In this respect see for instance \cite{GS:instability:95} for the plasma
case, and \cite{GR:stability:07} for the gravitational case; both
sources contain many references on the subject. This body of works has
confirmed the intuition of physicists, although with quite different
methods. The gap between a formal, linear treatment and a rigorous,
nonlinear one is striking: Compare the Appendix of
\cite{GR:stability:07} to the rest of the paper.  In the gravitational
case, these works do not consider homogeneous equilibria, but only
localized solutions.

Our treatment of Landau damping will be performed from scratch, and
will not rely on any of these results.

\section{Main result}
\label{sec:main}

\subsection{Modelling} We shall work in adimensional units throughout
the paper, in $d$ dimensions of space and $d$ dimensions of velocity
($d\in\N$).

As should be clear from our presentation in Section \ref{sec:intro},
to observe Landau damping, we need to put a restriction on the length
scale (anyway plasmas in experiments are usually confined). To achieve
this we shall take the position space to be the $d$-dimensional torus
of sidelength $L$, namely $\T^d_L = \R^d/( L\Z)^d$.  This is
admittedly a bit unrealistic, but it is commonly done in plasma
physics (see e.g. \cite{balescu:charged:63}).

In a periodic setting the Poisson equation has to be reinterpreted,
since $\Delta^{-1}\rho$ is not well-defined unless $\int_{\T^d_L} \rho
=0$. The natural solution consists in removing the mean value of
$\rho$, independently of any ``neutrality'' assumption; in galactic
dynamics this is known as the {\bf Jeans swindle}, a trick considered
as efficient but logically absurd. However, in 2003 Kiessling
\cite{kiessling:swindle} re-opened the case and acquitted Jeans, on
the basis that his ``swindle'' can be justified by a simple limit
procedure. In the present case, one may adapt Kiessling's argument and
approximate the Coulomb potential $V$ by some potential
$V_\kappa$ exhibiting a ``cutoff'' at large distances,
e.g. of Debye type (invoking screening for a plasma, or a
cosmological constant for stellar systems; anyway the particular
choice of approximation has no influence on the result). If $\nabla V_\kappa
\in L^1(\R^d)$, then 
$\nabla V_\kappa\ast\rho$ makes sense for a periodic $\rho$, and moreover
\[ (\nabla V_\kappa \ast \rho)  (x) = 
\int_{\R^d} \nabla V_\kappa(x-y)\,\rho(y)\,dy
= \int_{[0,L]^d} \nabla V_\kappa^{(L)}(x-y)\,\rho(y)\,dy,\]
where $V_\kappa ^{(L)} (z) = \sum_{\ell\in\Z^d} V_\kappa(z+\ell L)$.
Passing to the limit as $\kappa\to 0$ yields
\[ \int_{[0,L]^d} \nabla V^{(L)}(x-y)\,\rho(y)\,dy = \int_{[0,L]^d}
\nabla V^{(L)}(x-y)\,\bigl(\rho-\<\rho\>\bigr)(y)\,dy = - \nabla
\Delta_L^{-1}\bigl(\rho-\<\rho\>\bigr),\] where $\Delta_L^{-1}$ is the
inverse Laplace operator on $\T^d_L$. We refer to
\cite{kiessling:swindle} for a discussion of the physics underlying
this limit $\kappa\to 0$.  

More generally, we may consider any interaction potential $W$ on
$\T^d_L$, satisfying certain regularity assumptions.  Then the
self-consistent field will be given by
\[ F = -\nabla W \ast\rho,\qquad \rho (x) = \int f(x,v)\,dv,\]
where now $\ast$ denotes the convolution on $\T^d_L$.

In accordance with our conventions from Appendix \ref{appFourier}, we
shall write $\hat W^{(L)}(k) = \int_{\T^d_L} e^{-2i\pi k\cdot \frac{x}{L}}
\, W(x)\,dx$. In particular, if $W$ is the periodization of a potential
$\R^d\to\R$ (still denoted $W$ by abuse of notation), {\it i.e.},
\[ W(x) = W^{(L)}(x) = \sum_{\ell\in\Z^d} W(x+\ell L),\]
then
\begeq\label{WL} \hat{W}^{(L)}(k) = \hat{W}\left(\frac{k}{L}\right),
\endeq
where $\hat{W}(\xi) = \int_{\R^d} e^{-2i\pi \xi\cdot x} \, W(x)\,dx$
is the original Fourier transform in the whole space.

\subsection{Linear damping} \label{sub:lineardamping}

It is well-known that Landau damping requires some stability
assumptions on the unperturbed homogeneous distribution function, say
$f^0(v)$. In this paper we shall use a very general assumption,
expressed in terms of the Fourier transform 
\begeq\label{f0trans}
\tilde{f}^0(\eta) = \int_{\R^d} e^{-2i\pi \eta\cdot v} \, f^0(v)\,dv,
\endeq
the length $L$, and the interaction potential $W$. To state it, we
define, for $t\geq 0$ and $k\in\Z^d$, 
\begeq\label{K0tk} K^0(t,k) = -
4\pi^2\,\hat{W}^{(L)}(k)\,\tilde{f}^0\left(\frac{kt}{L}\right)\,\frac{|k|^2}{L^2}\,t;
\endeq
and, for any $\xi\in\C$, we define a function ${\cal L}$ {\it via} the
following Fourier--Laplace transform of $K^0$ in the time variable:
\begeq\label{K0hatxik} 
{\cal L}(\xi,k) = \int_0^{+\infty}
e^{2\pi\xi^*\frac{|k|}{L} t}\,K^0(t,k)\,dt,
\endeq
where $\xi^*$ is the complex conjugate to $\xi$.  
Our linear damping condition is expressed as follows:
\bigskip
\med

\hspace*{-5mm}\begin{tabular}{>{\bfseries}p{1cm} >{\itshape}p{13cm}}
  (L) & There are constants $C_0,\lambda,\kappa>0$ such that for any $\eta\in\R^d$,
$|\tilde{f}^0(\eta)| \leq C_0\,e^{-2\pi\lambda|\eta|}$; and for any $\xi\in\C$ with 
$0\leq \Re \, \xi < \lambda$,
\med
\[ \hspace*{-35mm} \inf_{k\in\Z^d}\ \bigl|{\cal L}(\xi,k)-1\bigr| \geq \kappa.\]
\end{tabular}
\medskip

We shall prove in Section \ref{sec:linear} that {\bf (L)} implies Landau damping.
For the moment, let us give a few sufficient conditions for {\bf (L)} to be
satisfied. The first one can be thought of as a smallness assumption
on either the length, or the potential, or the velocity
distribution. The other conditions involve the
marginals of $f^0$ along arbitrary wave vectors $k$:
\begeq\label{marginals} 
\varphi_k(v) = \int_{\frac{k}{|k|}v + k^\bot}
f^0(w)\,dw, \quad v \in \R.
\endeq
All studies known to us are based on one of these assumptions, so {\bf (L)} appears as a unifying condition
for linear Landau damping around a homogeneous equilibrium.

\begin{Prop} \label{suffL} Let $f^0=f^0(v)$ be a velocity distribution
  such that $\tilde{f}^0$ decays exponentially fast at infinity, let
  $L>0$ and let $W$ be an interaction potential on $\T^d_L$, $W\in
  L^1(\T^d)$.  If any one of the following conditions is
  satisfied: \sm

  (a) \underline{smallness:} \begeq\label{condstab} 4\pi^2 \,\left(
    \max_{k\in\Z^d _*}\, \bigl|\hat{W}^{(L)}(k)\bigr|\right)\
  \left(\sup_{|\sigma|=1}\, \int_0^\infty \bigl|
    \tilde{f}^0(r\sigma)\bigr|\,r\,dr\right) < 1;
\endeq

(b) \underline{repulsive interaction and decreasing marginals:} for
all $k\in\Z^d$ and $v\in\R$, \begeq\label{condstab1}
\hat{W}^{(L)}(k)\geq 0;\qquad
\begin{cases} v<0 \Longrightarrow \varphi'_k(v) > 0\\[2mm]
v>0 \Longrightarrow \varphi'_k(v) < 0;
\end{cases}
\endeq

(c) \underline{generalized Penrose condition on marginals:}\ for all
$k\in\Z^d$, 
\begeq\label{penrose} 
\forall \, w\in\R,\qquad
\varphi'_k(w)=0 \ \Longrightarrow \quad \hat{W}^{(L)}(k)\, \left({\rm
    p.v.} \int_\R \frac{\varphi'_k(v)}{v-w}\,dv \right) <1;
\endeq
\sm

\noindent then {\bf (L)} holds true for some $C_0,\lambda,\kappa>0$.
\end{Prop}

\begin{Rk} \cite[Problem, Section 30]{LL:kin:81} If $f^0$ is radially
  symmetric and positive, and $d\geq 3$, then all marginals of $f^0$
  are decreasing functions of $|v|$. Indeed, if $\varphi(v) =
  \int_{\R^{d-1}} f(\sqrt{v^2+|w|^2})\,dw$, then after
  differentiation and integration by parts we find
  \[ \begin{cases}
    \dps \varphi'(v) = 
    - (d-3)\, v \int_{\R^{d-1}} 
    f\bigl(\sqrt{v^2+|w|^2}\bigr)\,\frac{dw}{|w|^2}\qquad (d\geq 4)\\[3mm]
    \varphi'(v) = - 2\pi\,v\,f(|v|)\qquad (d=3).
  \end{cases}\]
\end{Rk}

\begin{Ex} Take a gravitational interaction and Mawellian background:
  \[ \hat{W}(k) = - \frac{\cal G}{\pi\,|k|^2},\qquad f^0(v) = 
  \rho^0\,\frac{e^{-\frac{|v|^2}{2T}}}{(2\pi T)^{d/2}}.\] 
  Recalling \eqref{WL}, we see that \eqref{condstab} becomes 
  \begeq\label{LJ} L
  < \sqrt{\frac{\pi\,T}{{\cal G}\,\rho^0}} =: L_J(T,\rho^0).
  \endeq
  The length $L_J$ is the celebrated Jeans length
  \cite{BT2,kiessling:swindle}, so criterion (a) can be applied, all the
  way up to the onset of the Jeans instability.
\end{Ex}

\begin{Ex} If we replace the gravitational interaction by the
  electrostatic interaction, the same computation yields
  \begeq\label{LD}
  L < \sqrt{\frac{\pi\,T}{e^2\,\rho^0}} =: L_D(T,\rho^0),
  \endeq
  and now $L_D$ is essentially the Debye length. Then criterion (a)
  becomes quite restrictive, but because the interaction is repulsive we
  can use criterion (b) as soon as $f^0$ is a strictly monotone function of $|v|$;
  this covers in particular Maxwellian distributions, {\em independently of the
  size of the box}. Criterion (b) also applies if
  $d\geq 3$ and $f^0$ has radial symmetry. For given $L>0$, the condition
{\bf (L)} being open, it will also be satisfied if $f^0$ is a small (analytic)
perturbation of a profile satisfying (b); this includes the so-called
``small bump on tail'' stability. Then if the distribution presents two large bumps,
the Penrose instability will take over.
\end{Ex}

\begin{Ex} For the electrostatic interaction in dimension~1, \eqref{penrose} becomes
\begeq\label{genpenrose} 
(f^0)'(w)=0 \ \Longrightarrow \ \int \frac{(f^0)'(v)}{v-w}\,dv 
< \frac{\pi}{e^2\,L^2}.
\endeq
This is a variant of the Penrose stability condition \cite{penrose:instability}.
This criterion is in general sharp for linear stability
(up to the replacement of the strict inequality
by the nonstrict one, and assuming that the critical points of $f^0$ are nondegenerate);
see \cite[Appendix]{linzeng:preprint} for precise statements.
\end{Ex}

We shall show in Section \ref{sec:linear} that {\bf (L)} implies
linear Landau damping (Theorem \ref{thmlineardamping}); then we shall prove
Proposition \ref{suffL} at the end of that section.
The general ideas are close to those appearing in previous works,
including Landau himself; the only novelties lie in the more
general assumptions, the elementary nature of the arguments, and the
slightly more precise quantitative results.

\subsection{Nonlinear damping} 

As others have done before in the study of Vlasov--Poisson
\cite{caglimaff:VP:98}, we shall quantify the analyticity by means of
natural norms involving Fourier transform in both variables (also denoted with
a tilde in the sequel). So we define
\begeq \label{flambdamu}
 \|f\|_{\lambda,\mu} = \sup_{k,\eta}\, \Bigl( e^{2\pi \lambda
  |\eta|}\, e^{2\pi \mu \frac{|k|}{L}}
\bigl|\tilde{f}^{(L)}(k,\eta)\bigr|\Bigr),
\endeq
 where $k$ varies in
$\Z^d$, $\eta\in\R^d$, $\lambda,\mu$ are positive parameters, and we
recall the dependence of the Fourier transform on $L$ (see Appendix
\ref{appFourier} for conventions). Now we can state our main result as
follows:

\begin{Thm}[Nonlinear Landau damping] \label{thmmain}
Let $f^0:\R^d\to\R_+$ be an analytic velocity profile.
Let $L>0$ and $W:\T^d_L\to\R$ be an interaction potential satisfying
\begeq\label{nablaW}
\forall \, k\in\Z^d,\qquad |\hat{W}^{(L)}(k)| \leq \frac{C_W}{|k|^{1+\gamma}}
\endeq
for some constants $C_W>0$, $\gamma\ge 1$.
Assume that $f^0$ and $W$ satisfy the stability
condition {\bf (L)} from Subsection \ref{sub:lineardamping}, with some
constants $\lambda,\kappa>0$; further assume that, for the same
parameter $\lambda$,
\begin{equation}\label{condanalf0} 
\sup_{\eta\in\R^d}
\Bigl(|\tilde{f}^0(\eta)|\,e^{2\pi\lambda|\eta|}\Bigr)\leq C_0,\qquad
\sum_{n\in\N_0^d} \frac{\lambda^n}{n!} \, \|\nabla_v^n
f^0\|_{L^1(\R^d)} \leq C_0 < +\infty.
\end{equation}

Then for any $0 < \lambda'<\lambda$, $\beta>0$, $0 < \mu' < \mu$, there is 
$\var = \var (d,L,C_W,C_0,\kappa,\lambda,\lambda',\mu,\mu',\beta,\gamma)$
with the following property: if $f_i = f_i(x,v)$ is an initial datum satisfying
\begeq\label{fif0} 
\delta := \|f_i-f^0\|_{\lambda,\mu} + 
\iint_{\T^d_L\times\R^d} |f_i-f^0|\,e^{\beta |v|}\,dv\,dx \leq \var,
\endeq
then\\

\noindent\bul the unique classical solution $f$ to the nonlinear
Vlasov equation 
\begeq\label{nlV} 
\derpar{f}{t} + v\cdot\nabla_x f 
- (\nabla W \ast \rho)\cdot\nabla_v f = 0,\qquad \rho = \int_{\R^d} f\,dv,
\endeq
with initial datum $f(0,\, \cdot\,)=f_i$, converges in the weak topology
as $t\to \pm\infty$, with rate $O(e^{-2\pi\lambda'|t|})$, to a
spatially homogeneous equilibrium $f_{\pm\infty}$;
\sm

\noindent \bul the density $\rho(t,x)=\int f(t,x,v)\,dv$ converges in the
strong topology as $t\to\pm\infty$, with rate
$O(e^{-2\pi\lambda'|t|})$, to the constant density
\[ \rho_\infty = \frac1{L^d} \int_{\R^d} \int_{\T_L ^d} f_i(x,v)\,dx\,dv;\]
in particular the force $F=-\nabla W\ast \rho$ converges exponentially fast to~0.
\sm

\noindent \bul the space average $\<f\>(t,v) = \int f(t,x,v)\,dx$ converges in
the strong topology as $t\to\pm\infty$, with rate
$O(e^{-2\pi\lambda'|t|})$, to $f_{\pm\infty}$.  \sm

\noindent More precisely, there are $C>0$, and spatially homogeneous
distributions $f_{+\infty}(v)$ and $f_{-\infty}(v)$, depending continuously on
$f_i$ and $W$, such that 
\begeq\label{f0orb} \sup_{t\in\R}\:
\Bigl\|f(t,x+vt,v)-f^0(v)\Bigr\|_{\lambda',\mu'} \leq C\,\delta;
\endeq
\[ \forall \, \eta\in\R^d,\qquad
| \tilde{f}_{\pm\infty}(\eta) - \tilde{f}^0(\eta) | \leq 
C\, \delta\,e^{-2\pi \lambda' |\eta|};\]
and
\[ \forall  \, (k,\eta)\in\Z^d\times\R^d,\qquad 
\Bigl|L^{-d}\,\tilde{f}^{(L)}(t,k,\eta) - \tilde{f}_{+\infty}(\eta) 1_{k=0}\Bigr| 
= O(e^{-2\pi \frac{\lambda'}{L} t})\quad \text{as $t\to +\infty$};\]
\[ \forall  \, (k,\eta)\in\Z^d\times\R^d,\qquad \Bigl|L^{-d}\,\tilde{f}^{(L)}(t,k,\eta) - \tilde{f}_{-\infty}(\eta) 1_{k=0}\Bigr| 
= O(e^{-2\pi \frac{\lambda'}{L} |t|})\quad \text{as $t\to -\infty$};\]
\begeq\label{convrho} 
\forall \, r\in\N,\qquad \bigl\|\rho(t,\cdot)
-\rho_\infty\bigr\|_{C^r(\T^d)} 
= O \bigl( e^{-2\pi \frac{\lambda'}{L} |t|}\bigr)
\qquad \text{as $|t|\to\infty$};
\endeq
\begeq\label{convF} 
\forall \, r\in\N,\qquad \bigl\|F(t,\,\cdot\,)\|_{C^r(\T^d)}
= O \bigl( e^{-2\pi \frac{\lambda'}{L} |t|}\bigr)
\qquad \text{as $|t|\to\infty$};
\endeq
\begeq\label{convav}
\forall\, r\in\N,\ \forall \sigma>0, \qquad 
\Bigl\| \bigl\< f(t,\cdot,v)\bigr\> - f_{\pm\infty} \Bigr\|_{C^r_\sigma(\R^d_v)}
= O \bigl( e^{-2\pi \frac{\lambda'}{L} |t|}\bigr) \quad 
\text{as $t\to\pm\infty$}.
\endeq
\end{Thm}

In this statement $C^r$ stands for the usual norm on $r$ times
continuously differentiable functions, and $C^r_\sigma$ involves in
addition moments of order $\sigma$, namely $\|f\|_{C^r_\sigma} =
\sup_{r' \le r, v \in \R^d} |f^{(r')}(v)\,(1+|v|^\sigma)|$.
These results could be reformulated in a number of alternative norms,
both for the strong and for the weak topology.

\subsection{Comments}

Let us start with a list of remarks about Theorem \ref{thmmain}.  \sm

\bul The decay of the force field, statement \eqref{convF}, is the
experimentally measurable phenomenon which may be called Landau
damping.\sm

\bul Since the energy
\[ E = \frac12 \iint \rho(x)\,\rho(y)\,W(x-y)\,dx\,dy + \int
f(x,v)\,\frac{|v|^2}{2}\,dv\,dx\] (=\ potential + kinetic energy) is
conserved by the nonlinear Vlasov evolution, there is a conversion of
potential energy into kinetic energy as $t\to\infty$ (kinetic energy
goes up for Coulomb interaction, goes down for Newton interaction).
Similarly, the entropy
\[ S = - \iint f \log f = - \left(\int \rho\log\rho + \int
  f\log\frac{f}{\rho}\right)\] (=\ spatial + kinetic entropy) is
preserved, and there is a transfer of information from spatial to
kinetic variables in large time.\sm

\bul Our result covers both attractive and repulsive interactions, as
long as the linear damping condition is satisfied; it covers
Newton/Coulomb potential as a limit case ($\gamma=1$ in
\eqref{nablaW}). The proof breaks down for $\gamma<1$; this is a
nonlinear effect, as any $\gamma> 0$ would work for the linearized
equation.  The singularity of the interaction at short scales will be
the source of important technical problems.\footnote{In a related
  subject, this singularity is also the reason why the Vlasov--Poisson
  equation is still far from being established as a mean-field limit
  of particle dynamics (see \cite{JH} for partial results covering
  much less singular interactions).}\sm

\bul Condition \eqref{condanalf0} could be replaced by
\begin{equation}\label{altcondanalf0} 
  |\tilde{f}^0(\eta)| \leq
C_0\,e^{-2\pi\lambda|\eta|},\qquad \int f^0(v)\, e^{\beta |v|}\,dv
\leq C_0.
\end{equation}
But condition \eqref{condanalf0} is more general, in view of Theorem
\ref{thminj} below.  For instance, $f^0(v) = 1/(1+v^2)$ in dimension
$d=1$ satisfies \eqref{condanalf0} but not \eqref{altcondanalf0}; this
distribution is commonly used in theoretical and numerical studies,
see e.g. \cite{HGMM}.  We shall also establish slightly more precise
estimates under slightly more stringent conditions on $f^0$, see
\eqref{condanalf0'}.  \sm

\bul Our conditions are expressed in terms of the initial datum, which
is a considerable improvement over \cite{caglimaff:VP:98,HV:landau}.
Still it is of interest to pursue the ``scattering'' program started
in \cite{caglimaff:VP:98}, e.g. in a hope of better understanding of
the nonperturbative regime.  \sm

\bul The smallness assumption on $f_i-f^0$ is expected, for instance
in view of the work of O'Neil \cite{oneil},
or the numerical results of \cite{ZGS:landau}. We also make
the standard assumption that $f_i-f^0$ is well localized.  \sm

\bul No convergence can be hoped for if the initial datum is
only close to $f^0$ in the weak topology: indeed there is {\em
  instability} in the weak topology, even around a Maxwellian
\cite{caglimaff:VP:98}.  \sm

\bul Strictly speaking, known existence and uniqueness results for
solutions of the nonlinear Vlasov--Poisson equation
\cite{BattRein,lionsperth:VP:91} do not apply to the present setting
of close-to-homogeneous analytic solutions. (The problem with
\cite{BattRein} is that velocities are assumed to be uniformly
bounded, and the problem with \cite{lionsperth:VP:91} is that the
position space is the whole of $\R^d$; in both papers these
assumptions are not superficial.) However, this really is not a big
deal: our proof will provide an existence theorem, together with
regularity estimates which are considerably stronger than what is
needed to prove the uniqueness. We shall not come back to these issues
which are rather irrelevant for our study: uniqueness only needs local
in time regularity estimates, while all the difficulty in the study of
Landau damping consists in handling (very) large time.  \sm

\bul $f(t,\cdot)$ is {\em not} close to $f^0$ in analytic norm as
$t\to\infty$, and does not converge to anything in the strong topology, so
the conclusion cannot be improved much. Still we shall establish more
precise quantitative results, and the limit profiles $f_{\pm\infty}$
are obtained by a constructive argument.
\sm

\bul Estimate \eqref{f0orb} expresses the orbital ``travelling
stability'' around $f^0$; it is much stronger than the usual orbital
stability in Lebesgue norms \cite{GS:instability:95,GR:stability:07}.
An equivalent formulation is that if $(T_t)_{t\in\R}$ stands for the
nonlinear Vlasov evolution operator, and $(T_t^0)_{t\in\R}$ for the
free transport operator, then in a neighborhood of a homogeneous
equilibrium satisfying the stability criterion {\bf (L)},
$T^0_{-t}\circ T_t$ remains uniformly close to $\Id$ for all $t$. Note
the important difference: unlike in the usual orbital stability
theory, our conclusions are expressed in functional spaces involving
smoothness, {\em which are not invariant under the free transport
  semigroup}. This a source of difficulty (our functional spaces are
sensitive to the filamentation phenomenon), but it is also the reason
for which this ``analytic'' orbital stability contains much more
information, and in particular the damping of the density. \sm

\bul Compared with known nonlinear stability results, and even
forgetting about the smoothness, estimate \eqref{f0orb} is new in
several respects. In the context of plasma physics, it is the first
one to prove stability for a distribution which is not necessarily a
decreasing function of $|v|$ (``small bump on tail''); while in the
context of astrophysics, it is the first one to establish stability of
a homogeneous equilibrium against periodic perturbations with
wavelength smaller than the Jeans length. \sm

\bul While analyticity is the usual setting for Landau damping, both
in mathematical and physical studies, it is natural to ask whether
this restriction can be dispended with.  (This can be done only at the
price of losing the exponential decay.) In the linear case, this is
easy, as we shall recall later in Remark \ref{rklowreg}; but in the
nonlinear setting, leaving the analytic world is much more tricky. In
Section \ref{sec:NA}, we shall present the first results in this
direction. 

\med

With respect to the questions raised above, our analysis brings the
following answers:\sm

(a) Convergence of the distribution $f$ does hold for $t\to +\infty$;
it is indeed based on phase mixing, and therefore involves very fast
oscillations.  In this sense it is right to consider Landau damping as
a ``wild'' process.  But on the other hand, the spatial density (and
therefore the force field) converges strongly and smoothly.  \sm

(b) The space average $\<f\>$ does converge in large time. However the
conclusions are quite different from those of quasilinear
relaxation theory, since there is no need for extra randomness, and
the limiting distribution is smooth, even without collisions.  \sm

(c) Landau damping is a linear phenomenon, which survives nonlinear
perturbation thanks to the structure of the Vlasov--Poisson
equation. The nonlinearity manifests itself by the presence of {\em
  echoes}.  Echoes were well-known to specialists of plasma physics
\cite[Section~35]{LL:kin:81} \cite[Section~12.7]{akhiezer}, but were
not identified as a possible source of unstability. Controlling the
echoes will be a main technical difficulty; but the fact that the
response appears in this form, with an associated {\em time-delay}
and localized in time, will in the end explain the
stability of Landau damping. These features can be expected in other
equations exhibiting oscillatory behavior.  \sm

(d) The large-time limit is in general different from the limit
predicted by the linearized equation, and depends on the interaction and initial datum
(more precise statements will be given in Section \ref{sec:cex});
still the linearized equation, or higher-order expansions, do provide a good approximation. 
We shall also set up a systematic recipe for approximating the large-time
limit with arbitrarily high precision as the strength of the
perturbation becomes small. This justifies {\it a posteriori} many
known computations.  \sm

(e) From the point of view of dynamical systems, the nonlinear Vlasov
equation exhibits a truly remarkable behavior. It is not uncommon for
a Hamiltonian system to have many, or even countably many heteroclinic
orbits (there are various theories for this, a popular one being the
Melnikov method); but in the present case we see that
heteroclinic/homoclinic orbits\footnote{Here we use these words just
  to designate solutions connecting two distinct/equal equilibria,
  without any mention of stable or unstable manifolds.}  are so
numerous as to fill up {\em a whole neighborhood} of the equilibrium.
This is possible only because of the infinite-dimensional nature of
the system, and the possibility to work with nonequivalent norms; such
a behavior has already been reported for other systems
\cite{kuksin:LNM:93,kuksin:oxford:00}, in relation with
infinite-dimensional KAM theory. 
\sm

(f) As a matter of fact, the nonlinear Landau damping has strong similarities with
the KAM theory. It has been known since the early days of the theory that the linearized Vlasov equation
can be reduced to an infinite system of uncoupled Volterra equations, which makes this equation
completely integrable in some sense. (Morrison \cite{morrison} gave a more precise meaning to
this property.) To see a parallel with classical KAM, one step of our result
is to prove the preservation of the phase-mixing property under
nonlinear perturbation of the interaction. (Although there is no
ergodicity in phase space, the mixing will imply an ergodic behavior
for the spatial density.) The analogy goes further since the proof of Theorem
\ref{thmmain} shares many features with the proof of the KAM theorem
(closest to Kolmogorov's original version, see \cite{chierchia} for a
complete exposition). 

Thus we see that three of the most famous paradoxical phenomena from
twentieth century classical physics: Landau damping, echoes, and KAM theorem, are
intimately related (only in the nonlinear variant of Landau's linear argument!). 
This relation, which we did not expect, is one of the main discoveries of the present paper.

\subsection{Interpretation}

A successful point of view adopted in this paper is that Landau
damping is a {\bf relaxation by smoothness} and by {\bf mixing}. 
In a way, phase mixing converts the smoothness into decay. Thus Landau damping emerges
as a rare example of a physical phenomenon in which regularity is not only crucial
from the mathematical point of view, but also can be ``measured'' by a physical experiment.

\subsection{Main ingredients}

Some of our ingredients are similar to those in
\cite{caglimaff:VP:98}: in particular, the use of Fourier transform to
quantify analytic regularity and to implement phase mixing.  New
ingredients used in our work include \sm

\bul the introduction of a time-shift parameter to keep memory of the
initial time (Sections \ref{sec:analytic} and \ref{sec:scattering}),
thus getting uniform estimates in spite of the loss of regularity in
large time.  We call this the {\bf gliding regularity}: it shifts in
phase space from low to high modes.  Gliding regularity automatically
comes with an improvement of the regularity in $x$, and a
deterioration of the regularity in $v$, as time passes by.  \sm

\bul the use of carefully designed flexible analytic norms behaving
well with respect to composition (Section \ref{sec:analytic}). This
requires care, because analytic norms are very sensitive to
composition, contrary to, say, Sobolev norms.  \sm

\bul ``finite-time scattering'' {\em at the level of trajectories} to
reduce the problem to homogenization of free flow (Section
\ref{sec:scattering}) {\it via} composition. The physical meaning is
the following: when a background with gliding regularity acts on (say)
a plasma, the trajectories of plasma particles are asymptotic to free
transport trajectories.  \sm

\bul new functional inequalities of bilinear type, involving analytic
functional spaces, integration in time and velocity variables, and
evolution by free transport (Section \ref{sec:regdecay}). These
inequalities morally mean the following: when a plasma acts (by
forcing) on a smooth background of particles, the background reacts by
lending a bit of its (gliding) regularity to the plasma, uniformly in
time. Eventually the plasma will exhaust itself (the force will decay).
This most subtle effect, which is at the heart of Landau's
damping, will be mathematically expressed in the
formalism of analytic norms with gliding regularity. 
\sm

\bul a new analysis of the time response associated to the
Vlasov--Poisson equation (Section \ref{sec:response}), aimed
ultimately at controlling the self-induced echoes of the plasma.  For
any interaction less singular than Coulomb/Newton, this will be done
by analyzing time-integral equations involving a norm of the spatial
density.  To treat Coulomb/Newton potential we shall refine the analysis,
considering individual modes of the spatial density. \sm

\bul a Newton iteration scheme, solving the nonlinear evolution
problem as a succession of linear ones (Section
\ref{sec:iter}). Picard iteration schemes still play a role, since
they are run at each step of the iteration process, to estimate the
scattering operators.\med

It is only in the linear study of Section \ref{sec:linear} that the
length scale $L$ will play a crucial role, {\it via} the stability
condition {\bf (L)}. In all the rest of the paper we shall normalize
$L$ to~1 for simplicity.

\subsection{About phase mixing}

A physical mechanism transferring energy from large scales to very
fine scales, asymptotically in time, is sometimes called {\bf weak
  turbulence}. Phase mixing provides such a mechanism, and in a
way our study shows that the Vlasov--Poisson equation is subject to
weak turbulence. But the phase mixing interpretation provides a more
precise picture. While one often sees weak turbulence as a ``cascade''
from low to high Fourier modes, the relevant picture would rather be a
two-dimensional figure with an interplay between spatial Fourier modes
and velocity Fourier modes. More precisely, phase mixing transfers the
energy from each nonzero spatial frequency $k$, to large velocity
frequences $\eta$, and this transfer occurs at a speed proportional to
$k$. This picture is clear from the solution of free transport in
Fourier space, and is illustrated in Fig. \ref{transfer}. (Note the
resemblance with a shear flow.) So there is transfer of energy from
one variable (here $x$) to another (here $v$); homogenization in the
first variable going together with filamentation in the second one.
The same mechanism may also underlie other cases of weak turbulence.

\begin{figure}[htbp]
\begin{center}
\input{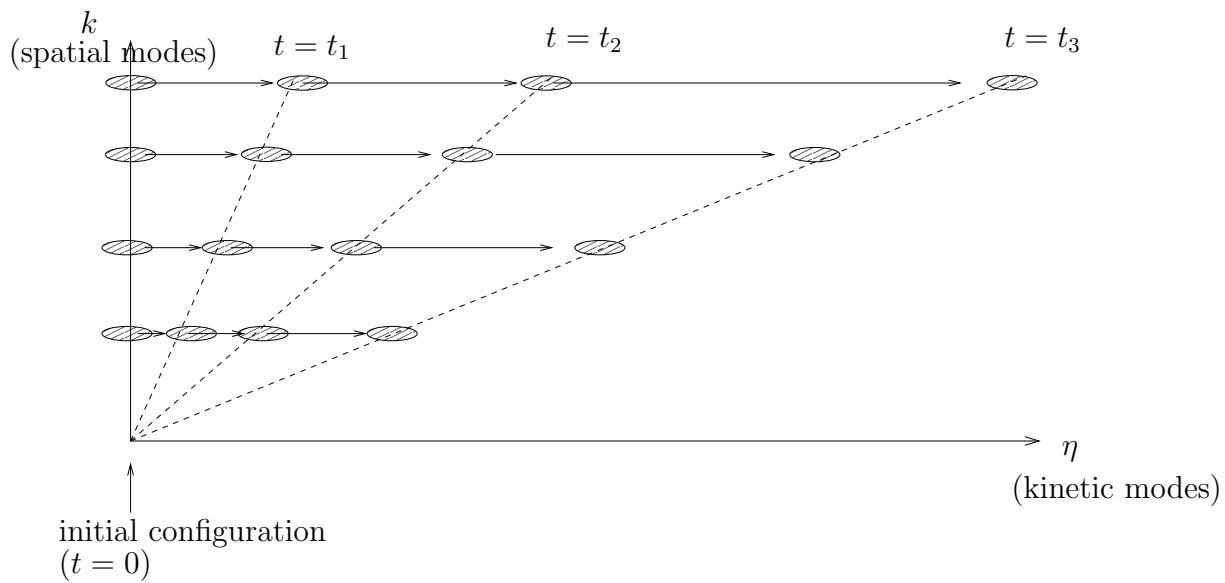}
\caption{Schematic picture of the evolution of energy by free
  transport, or perturbation thereof; marks indicate localization of
  energy in phase space.} \label{transfer}
\end{center}
\end{figure}

\begin{figure}[htbp]
  \includegraphics[height=5cm]{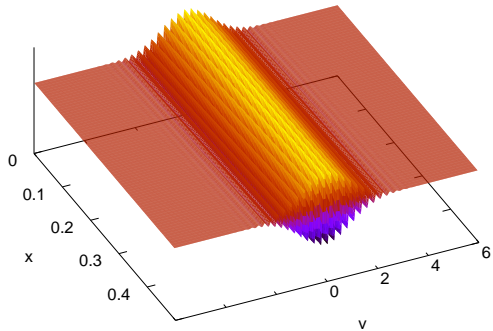}
  \caption{The distribution function in phase space (position,
    velocity) at a given time; notice how the fast oscillations in $v$
    contrast with the slower variations in $x$.}
\label{figxv}
\end{figure}

Whether ultimately the high modes are damped by some ``random''
microscopic process (collisions, diffusion, \ldots) not described by
the Vlasov--Poisson equation is certainly undisputed in plasma physics
\cite[Section~41]{LL:kin:81}\footnote{See
  \cite[Problem~41]{LL:kin:81}: thanks to Landau damping, collisions
  are expected to smooth the distribution quite efficiently; this is a hypoelliptic problematic.}, 
and is the object
of debate in galactic dynamics; anyway this is a different story. Some
mathematical statistical theories of Euler and Vlasov--Poisson
equations do postulate the existence of some small-scale coarse
graining mechanism, but resulting in mixing rather than dissipation
\cite{robert,turkington}.

\section{Linear damping}
\label{sec:linear}

In this section we establish Landau damping for the linearized Vlasov equation.
Beforehand, let us recall that the free transport equation
\begeq\label{FTE}
\derpar{f}{t} + v\cdot\nabla_x f =0
\endeq
has a strong {\em mixing} property: any solution of \eqref{FTE} converges weakly in large time
to a spatially homogeneous distribution equal to the space-averaging of the initial datum.
Let us sketch the proof. 

If $f$ solves \eqref{FTE} in $\T^d\times\R^d$, with initial datum
$f_i= f(0,\,\cdot\,)$, then $f(t,x,v) = f_i(x-vt,v)$, so the
space-velocity Fourier transform of $f$ is given by the formula
\begeq\label{fitilde} \tilde{f}(t,k,\eta) = \tilde{f}_i(k,\eta+kt).
\endeq  
On the other hand, if $f_\infty$ is defined by
\[ f_\infty(v) = \< f_i(\,\cdot\,,v)\> = \int_{\T^d} f_i(x,v)\,dx,\]
then $\tilde{f}_\infty(k,\eta) = \tilde{f}_i(0,\eta)\,1_{k=0}$. 
So, by the Riemann--Lebesgue lemma, for any fixed $(k,\eta)$ we have
\[ \Bigl|\tilde{f}(t,k,\eta) - \tilde{f}_\infty(k,\eta) \Bigr|
\xrightarrow[|t|\to \infty]{} 0,\] which shows that $f$ converges
weakly to $f_\infty$. The convergence holds as soon as $f_i$ is merely
integrable; and by \eqref{fitilde}, the rate of convergence is
determined by the decay of $\tilde{f}_i(k,\eta)$ as $|\eta|\to\infty$,
or equivalently the smoothness in the velocity variable.  In
particular, the convergence is exponentially fast if (and only if)
$f_i(x,v)$ is analytic in $v$.  \sm

This argument obviously works independently of the size of the box.
But when we turn to the Vlasov equation, length scales will matter, so
we shall introduce a length $L>0$, and work in $\T^d_L =
\R^d/(L\Z^d)$. Then the length scale will appear in the Fourier
transform: see Appendix \ref{appFourier}. (This is the only section in
this paper where the scale will play a nontrivial role, so in all the
rest of the paper we shall take $L=1$.)

Any velocity distribution $f^0=f^0(v)$ defines a stationary state for
the nonlinear Vlasov equation with interaction potential $W$. Then the
linearization of that equation around $f^0$ yields 
\begeq\label{Vllin} \begin{cases} 
\dps \derpar{f}{t} + v\cdot\nabla_x f 
- (\nabla W\ast \rho)\cdot\nabla_v f^0 =0\\[2mm]
\qquad \dps \rho = \int f\,dv.
\end{cases}
\endeq
Note that there is no force term in \eqref{Vllin}, due to the fact that $f^0$ does not depend on $x$. 
This equation describes what happens to a plasma density $f$ which tries to force a stationary homogeneous 
background $f^0$; equivalently, it describes the {\em reaction} exerted by the background which is acted upon.

\begin{Thm}[Linear Landau damping] \label{thmlineardamping} Let
  $f^0=f^0(v)$, $L>0$, $W:\T^d_L\to\R$ such that $\|\nabla
  W\|_{L^1}\leq C_W <+\infty$, and $f_i(x,v)$ such that \sm

  (i) Condition {\bf (L)} from Subsection \ref{sub:lineardamping}
  holds for some constants $\lambda,\kappa>0$; \sm

  (ii) $\forall \, \eta\in\R^d,\quad |\tilde{f}^0(\eta)| \leq
  C_0\,e^{-2\pi\lambda |\eta|}$ for some constant $C_0>0$; \sm

  (iii) $\forall \, k\in\Z^d,\ \forall \, \eta\in\R^d, \quad
  |\tilde{f}^{(L)}_i(k,\eta)|\leq C_i\,e^{-2\pi\alpha |\eta|}$ for
  some constants $\alpha > 0$, $C_i>0$.  \sm \sm

  \noindent Then as $t\to +\infty$ the solution $f(t,\cdot)$ to the
  linearized Vlasov equation \eqref{Vllin} with initial datum $f_i$
  converges weakly to $f_\infty = \<f_i\>$ defined by
  \[ f_\infty(v) = \frac1{L^d} \int_{\T^d_L} f_i(x,v)\,dx;\] and
  $\rho(x) = \dps \int f(x,v)\,dv$ converges strongly to the constant
  \[ \rho_\infty = \frac1{L^d} \iint_{\T^d_L\times\R^d}
  f_i(x,v)\,dx\,dv.\] More precisely, for any
  $\lambda'<\min\{\lambda \,; \, \alpha\}$,
  \[\begin{cases} \dps
    \forall \, r\in\N,\quad \bigl\|\rho(t,\cdot) -
    \rho_\infty\bigr\|_{C^r} 
    = O\bigl(e^{-\frac{2\pi\lambda'}{L} |t|}\bigr)\\[3mm]
    \dps \forall \, (k,\eta)\in\Z^d\times\Z^d,\qquad
    \Bigl|\tilde{f}^{(L)}(t,k,\eta) -
    \tilde{f}^{(L)}_\infty(k,\eta)\Bigr| =
    O\bigl(e^{-\frac{2\pi\lambda'}{L}|kt|}\bigr).\end{cases}\]
\end{Thm}
\sm

\begin{Rk} Even if the initial datum is more regular than analytic,
  the convergence will in general not be better than exponential
  (except in some exceptional cases \cite{hayes:cimento}). See
  \cite[pp.~414--416]{BT2} for an illustration. Conversely, if the analyticity
  width $\alpha$ for the initial datum is smaller than the ``Landau rate'' $\lambda$,
then the rate of decay will not be better than $O(e^{-\alpha t})$. See \cite{belmont1,belmont2}
for a discussion of this fact, often overlooked in the physical literature.
\end{Rk}

\begin{Rk} The fact that the convergence is to the average of the
  initial datum will not survive nonlinear perturbation, as shown by
  the counterexamples in Subsection~\ref{sec:cex}.
\end{Rk}

\begin{Rk} Dimension does not play any role in the linear
  analysis.  This can be attributed to the fact that only longitudinal
  waves occur, so everything happens ``in the direction of the wave
  vector''. Transversal waves arise in plasma physics only when
  magnetic effects are taken into account \cite[Chapter~5]{akhiezer}.
\end{Rk}

\begin{Rk} \label{rklowreg} The proof can be adapted to the case when
  $f^0$ and $f_i$ are only $C^\infty$; then the convergence is not
  exponential, but still $O(t^{-\infty})$. The regularity can also be
  further decreased, down to $W^{s,1}$, at least for any $s>2$; more
  precisely, if $f^0\in W^{s_0,1}$ and $f_i\in W^{s_i,1}$ there will
  be damping with a rate $O(t^{-\kappa})$ for any $\kappa <
  \max\{s_0-2\, ;\, s_i\}$.  (Compare with \cite[Vol.~1,
  p.~189]{akhiezer}.)  This is independent of the regularity of the
  interaction.
\end{Rk}

The proof of Theorem \ref{thmlineardamping} relies on the following
elementary estimate for Volterra equations. We use the notation of
Subsection \ref{sub:lineardamping}.

\begin{Lem} \label{lemvolterra} 
  Assume that {\bf (L)} holds true for some constants $C_0,\kappa,\lambda>0$; let 
  $C_W=\|W\|_{L^1(\T^d_L)}$ and let $K^0$ be defined by \eqref{K0tk}.
  Then any solution $\varphi(t,k)$ of
  \begeq\label{volterra} 
  \varphi(t,k) = a(t,k) + \int_0^t
  K^0(t-\tau,k)\,\varphi(\tau,k)\,d\tau
  \endeq 
  satisfies, for any $k\in\Z^d$ and any $\lambda'<\lambda$,
  \[ \sup_{t\geq 0} 
  \Bigl( |\varphi(t,k)|\,e^{2\pi\lambda'\frac{|k|}{L}t}\Bigr)
  \leq \Bigl[ 1+ C_0\, C_W\, C(\lambda,\lambda',\kappa)\Bigr]\
  \sup_{t\geq 0}\, \Bigl( |a(t,k)|\,e^{2\pi\lambda \frac{|k|}{L}t}\Bigr).\]
Here $C(\lambda,\lambda',\kappa) = C\,(1+\kappa^{-1} (1+(\lambda-\lambda')^{-1/2}))$ for some universal constant $C$.
\end{Lem}

\begin{Rk} 
  It is standard to solve these Volterra equations by Laplace
  transform; but, with a view to the nonlinear setting, we shall
  prefer a more flexible and quantitative approach.
\end{Rk}

\begin{proof}[Proof of Lemma \ref{lemvolterra}]
If $k=0$ this is obvious since $K^0(t,0)=0$; so we assume $k\neq 0$.
Consider $\lambda'<\lambda$, multiply \eqref{volterra} 
by $e^{2\pi\lambda'\frac{|k|}{L}t}$, and write
\[ \Phi(t,k) = \varphi(t,k)\,e^{2\pi\lambda'\frac{|k|}{L}t},\qquad
A(t,k) = a(t,k)\,e^{2\pi\lambda'\frac{|k|}{L}t};\]
then \eqref{volterra} becomes
\begeq\label{PhiA} \Phi(t,k) = A(t,k) + \int_0^t K^0(t-\tau,k)\,
e^{2\pi\lambda'\frac{|k|}{L}(t-\tau)}\,\Phi(\tau,k)\,d\tau.
\endeq
\med

\noindent {\bf A particular case:}
The proof is extremely simple if we make the stronger assumption
\[ \int_0^{+\infty}
|K^0(\tau,k)|\,e^{2\pi\lambda'\,\frac{|k|}{L}\tau}\,d\tau \leq
1-\kappa, \qquad \kappa \in (0,1).\] 
Then from \eqref{PhiA},
\begin{multline*} \sup_{0\leq t\leq T}\ |\Phi(t,k)|
\leq \sup_{0\leq t\leq T}\ |A(t,k)|\\
+ \sup_{0\leq t\leq T}\ \left(\int_0^t \bigl|K^0(t-\tau,k)\bigr|\,
     e^{2\pi\lambda' \frac{|k|}{L}(t-\tau)}\,d\tau\right)\,\
\sup_{0\leq \tau\leq T} \ |\Phi(\tau,k)|,
\end{multline*}
whence
\[ \sup_{0\leq \tau\leq t} \dps
|\Phi(\tau,k)| \leq \frac{\dps \sup_{0\leq \tau\leq t} |A(\tau,k)|}
{\dps 1-\int_0^{+\infty} |K^0(\tau,k)|\,e^{2\pi\lambda' \frac{|k|}{L}\tau}\,d\tau}
\leq \frac{\dps \sup_{0\leq\tau\leq t} |A(\tau,k)|}{\kappa},\]
and therefore
\[ \sup_{t\geq 0} \left(e^{2\pi\lambda' \frac{|k|}{L}t} |\varphi(t,k)|\right)
\leq \left(\frac1{\kappa}\right)\, \sup_{t\geq 0} \Bigl( |a(t,k)|\,
e^{2\pi\lambda' \frac{|k|}{L}t}\Bigr).\]
\med

\noindent{\bf The general case:}
To treat the general case we take the Fourier transform in the time
variable, after extending $K$, $A$ and $\Phi$ by~$0$
at negative times. (This presentation was suggested to us by Sigal,
and appears to be technically simpler than the use of the Laplace
transform.) Denoting the Fourier transform with a hat and recalling \eqref{K0hatxik}, we have, for
$\xi=\lambda'+i\omega L/|k|$, 
\[ \hat{\Phi}(\omega,k) = \hat{A}(\omega,k)
+ {\cal L}(\xi,k)\,\hat{\Phi}(\omega,k).\]
By assumption ${\cal L}(\xi,k)\neq 1$, so
\[ \hat{\Phi}(\omega,k) = \frac{\hat A(\omega,k)}{1-{\cal L}(\xi,k)}.\]

From there, it is traditional to apply the Fourier (or Laplace)
inversion transform.  Instead, we apply Plancherel's identity to find
(for each $k$)
\[ \|\Phi\|_{L^2(dt)} \leq \frac{\|A\|_{L^2(dt)}}{\kappa};\]
and then we plug this in the equation \eqref{PhiA} to get
\begin{align} \label{Phiinfty}
\|\Phi\|_{L^\infty(dt)} & \leq \|A\|_{L^\infty(dt)} + \|K^0\,e^{2\pi\lambda'\frac{|k|}{L}t}\|_{L^2(dt)}\, \|\Phi\|_{L^2(dt)}\\
& \leq \|A\|_{L^\infty(dt)} + \frac{\bigl\|K^0\,e^{2\pi\lambda'\frac{|k|}{L}t}\bigr\|_{L^2(dt)}\, \|A\|_{L^2(dt)}}{\kappa}.
\nonumber \end{align}

It remains to bound the second term. On the one hand,
\begin{align} \label{onehandA}
\|A\|_{L^2(dt)} & = \left( \int_0^\infty |a(t,k)|^2\,e^{4\pi\lambda'\frac{|k|}{L}t}\,dt\right)^{1/2}\\
& \leq \left(\int_0^\infty e^{-4\pi(\lambda-\lambda')\frac{|k|}{L}t}\right)^{1/2}\
\sup_{t\geq 0} \Bigl( |a(t,k)|\,e^{2\pi\lambda\frac{|k|}{L}t}\Bigr) \nonumber \\
& = \left(\frac{L}{4\pi|k|\,(\lambda-\lambda')}\right)^{\frac12}\ 
\sup_{t\geq 0} \Bigl( |a(t,k)|\,e^{2\pi\lambda\frac{|k|}{L}t}\Bigr). \nonumber
\end{align}
On the other hand,
\begin{align} \label{otherhandB} \bigl\|K^0
  \,e^{2\pi\lambda'\frac{|k|}{L}t}\bigr\|_{L^2(dt)} & = 4\pi^2\,
  |\hat{W}^{(L)}(k)|\,\frac{|k|^2}{L^2} \left(\int_0^\infty
    e^{4\pi\lambda'\frac{|k|}{L}t}\,
    \left|\tilde{f}^0\left(\frac{kt}{L}\right)\right|^2\,
    t^2\,dt \right)^{1/2}\\
  & = 4\pi^2\, |\hat{W}^{(L)}(k)|\,\frac{|k|^{1/2}}{L^{1/2}}
  \,\left(\int_0^\infty e^{4\pi\lambda'u}\, |\tilde{f}^0(\sigma\,
    u)|^2\,u^2\,du\right)^{1/2}, \nonumber
\end{align}
where $\sigma = k/|k|$.  The estimate follows immediately. (Note that the
factor $|k|^{-1/2}$ in \eqref{onehandA} cancels with $|k|^{1/2}$ in
\eqref{otherhandB}.)  \sm

It seems that we only used properties of the function ${\cal L}$ in a
strip $\Re\xi \simeq \lambda$; but this is an illusion. Indeed, we
have taken the Fourier transform of $\Phi$ without checking that it
belongs to $(L^1+L^2)(dt)$, so what we have established is only an
{\it a priori} estimate. To convert it into a rigorous result, one can
use a continuity argument after replacing $\lambda'$ by a parameter
$\alpha$ which varies from $-\epsilon$ to $\lambda'$. (By the
integrability of $K^0$ and Gronwall's lemma, $\varphi$ is obviously
bounded as a function of $t$; so $\varphi(k,t)\,e^{-\epsilon |k|t/L}$
is integrable for any $\epsilon>0$, and continuous as $\epsilon\to
0$.) Then assumption {\bf (L)} guarantees that our bounds
are uniform in the strip $0\leq\Re\xi\leq\lambda'$, and the proof goes
through.
\end{proof}

\begin{proof}[Proof of Theorem \ref{thmlineardamping}]
  Without loss of generality we assume $t\geq 0$.  Considering
  \eqref{Vllin} as a perturbation of free transport, we apply
  Duhamel's formula to get 
  \begeq\label{duhamel} f(t,x,v) =
  f_i(x-vt,v) + \int_0^t \bigl[ (\nabla W\ast \rho)\cdot\nabla_v
  f^0\bigr] \bigl(\tau,x-v(t-\tau),v\bigr)\,d\tau.
  \endeq
Integration in $v$ yields \begeq\label{intvrho} \rho(t,x) =
\int_{\R^d} f_i(x-vt,v)\,dv + \int_0^t \int_{\R^d} \bigl[(\nabla
W\ast\rho)\cdot\nabla_v f^0\bigr]
\bigl(\tau,x-v(t-\tau),v\bigr)\,dv\,d\tau.
\endeq
Of course, $\int \rho(t,x)\,dx = \iint f_i(x,v)\,dx\,dv$.

For $k\neq 0$, taking the Fourier transform of \eqref{intvrho}, we obtain
\begin{align*}
  \hat{\rho}^{(L)}(t,k) & =
  \int_{\T^d_L} \int_{\R^d} f_i(x-vt,v)\,e^{-2i \pi \frac{k}{L}\cdot x}\,dv\,dx \\
  & \qquad + \int_0^t \int_{\T^d_L} \int_{\R^d} \bigl[ (\nabla W\ast
  \rho)\cdot \nabla_v f^0\bigr]
  \bigl(\tau, x-v(t-\tau),v\bigr)\,e^{-2i\pi \frac{k}{L}\cdot x}\,dv\,dx\,d\tau\\[2mm]
  & = \int_{\T^d_L} \int_{\R^d} f_i(x,v)\, e^{-2i\pi \frac{k}{L}\cdot
    x}\, e^{-2i\pi \frac{k}{L}\cdot vt}\,dv\,dx \\
  & \qquad + \int_0^t \int_{\T^d_L} \int_{\R^d} \bigl[(\nabla W\ast \rho)\cdot \nabla_v
  f^0\bigr](\tau,x,v)\,
  e^{-2i\pi \frac{k}{L}\cdot x}\, e^{-2i\pi \frac{k}{L}\cdot v(t-\tau)}\,dv\,dx\,d\tau\\[2mm]
  & = \tilde{f}^{(L)}_i\left(k, \frac{kt}{L}\right) + \int_0^t (\nabla
  W\ast\rho)^{\hat{ }(L)}(\tau,k)\cdot\tilde{\nabla_v f^0}
  \left(\frac{k(t-\tau)}{L}\right)\,d\tau\\[2mm]
  & = \tilde{f}_i^{(L)} \left(k,\frac{kt}{L}\right) + \int_0^t \left(
    2i\pi \frac{k}{L}
    \hat{W}^{(L)}(k)\,\hat{\rho}^{(L)}(\tau,k)\right)\cdot \left(
    2i\pi
    \frac{k(t-\tau)}{L}\,\tilde{f}^0\left(\frac{k(t-\tau)}{L}\right)\right)\,d\tau.
\end{align*}
In conclusion, we have established the {\em closed equation} on $\hat{\rho}^{(L)}$:
\begin{multline}\label{closedrho}
\hat{\rho}^{(L)}(t,k)
= \tilde{f}_i^{(L)}\left(k, \frac{kt}{L}\right)\\
- 4 \pi^2 \, \hat{W}^{(L)}(k)
\int_0^t \hat{\rho}^{(L)}(\tau,k)\,\tilde{f}^0\left(\frac{k(t-\tau)}{L}\right)\,
\frac{|k|^2}{L^2}\,(t-\tau)\,d\tau.
\end{multline}

Recalling \eqref{K0tk}, this is the same as
\[ \hat{\rho}^{(L)}(t,k) = \tilde{f}_i^{(L)}\left(k,\frac{kt}{L}\right)
+ \int_0^t K^0(t-\tau,k)\,\hat{\rho}^{(L)}(\tau,k)\,d\tau.\]
Without loss of generality, $\lambda \leq\alpha$. By Assumption {(\bf L)} and Lemma \ref{lemvolterra},
\[ \bigl| \hat{\rho}^{(L)}(t,k)\bigr| \leq 
C_0\,C_W\,C(\lambda,\lambda',\kappa)\, C_i\, e^{-2\pi \frac{\lambda'\,|k|}{L}\,t}.\]
In particular, for $k\neq 0$ we have
\[ \forall \, t \ge 1, \quad |\hat{\rho}^{(L)}(t,k)| = O \Bigl(
e^{-\frac{2\pi\lambda''}{L}t}\,
e^{-\frac{2\pi(\lambda'-\lambda'')}{L}|k|}\Bigr);\] so any Sobolev
norm of $\rho-\rho_\infty$ converges to zero like $O(e^{-\frac{2 \pi
    \lambda''}{L} t})$, where $\lambda''$ is arbitrarily close to
$\lambda'$ and therefore also to $\lambda$.  By Sobolev embedding, the
same is true for any $C^r$ norm.  \sm

Next, we go back to \eqref{duhamel} and take the Fourier transform in
both variables $x$ and $v$, to find
\begin{align*}
  \tilde{f}^{(L)}(t,k,\eta)
 & = \int_{\T^d} \int_{\R^d} f_i(x-vt,v)\, e^{-2i\pi\frac{k}{L}\cdot x}\, e^{-2i\pi \eta\cdot v}\,dx\,dv\\
 &\qquad + \int_0^t \int_{\T^d} \int_{\R^d} (\nabla W\ast \rho) \bigl(\tau,x-v(t-\tau)\bigr)\cdot\nabla_v f^0(v)\,
e^{-2i\pi \frac{k}{L}\cdot x}\,e^{-2i\pi\eta\cdot v}\,dx\,dv\,d\tau\\[2mm]
& = \int_{\T^d} \int_{\R^d} f_i(x,v)\,e^{-2i\pi \frac{k}{L}\cdot x}\,e^{-2i\pi \frac{k}{L}\cdot vt}\,
e^{-2i\pi \eta\cdot v}\,dx\,dv\\
& \qquad + \int_0^t \int_{\T^d} \int_{\R^d} 
(\nabla W \ast\rho)(\tau,x)\cdot\nabla_v f^0(v)\,
e^{-2i\pi \frac{k}{L}\cdot x}\,e^{-2i\pi \frac{k}{L}\cdot v(t-\tau)}\,
e^{-2i\pi \eta\cdot v}\,dx\,dv\,d\tau\\[2mm]
& = \tilde{f}_i^{(L)} \left(k, \eta+\frac{kt}{L}\right)
+ \int_0^t \hat{\nabla W}^{(L)}(k)\,\hat{\rho}^{(L)}(\tau,k)\cdot
\tilde{\nabla_v f^0} \left(\eta + \frac{k}{L} (t-\tau)\right)\,d\tau.
\end{align*}
So
\begeq\label{sofL}
\tilde{f}^{(L)} \left( t,k,\eta-\frac{kt}{L}\right)
= \tilde{f}_i^{(L)}(k,\eta)
+ \int_0^t \hat{\nabla W}^{(L)}(k)\, \hat{\rho}^{(L)}(\tau,k)\cdot
\tilde{\nabla_v f^0}\left(\eta - \frac{k\tau}{L}\right)\,d\tau.
\endeq

In particular, for any $\eta\in\R^d$,
\begeq\label{fL=fi}
\tilde{f}^{(L)}(t,0,\eta) = \tilde{f}_i^{(L)}(0,\eta);
\endeq
in other words, $\<f\> = \int f\,dx$ remains equal to $\<f_i\>$ 
for all times.

On the other hand, if $k\neq 0$,
\begin{align} \label{boundfL} \left| \tilde{f}^{(L)}\left( t, k,
      \eta-\frac{kt}{L}\right) \right| & \leq
  \bigl|\tilde{f}_i^{(L)}(k,\eta)\bigr|\\ \nonumber & \qquad +
  \int_0^t \bigl|\hat{\nabla W}^{(L)}(k)\bigr|\,
  \bigl|\hat{\rho}^{(L)}(\tau,k)\bigr|\, \left| \tilde{\nabla_v
      f^0}\left(\eta - \frac{k\tau}{L}\right)\right|\,d\tau\\[1mm]
  \nonumber & \leq C_i\,e^{-2\pi\alpha |\eta|}\\ \nonumber & \qquad +
  \int_0^t C_W\, C(\lambda,\lambda',\kappa)\, C_i\,e^{-2\pi\lambda'
    \frac{|k|}{L}\tau}\, \left( 2\pi C_0\,
    \left|\eta-\frac{k\tau}{L}\right|\, e^{-2\pi\lambda
      \left|\eta-\frac{k\tau}{L}\right|}\right)\,d\tau\\[1mm]
  \nonumber & \leq C \left( e^{-2\pi\alpha |\eta|} + \int_0^t
    e^{-2\pi\lambda'\frac{|k|}{L}\tau}\,e^{-2\pi\frac{(\lambda'+
        \lambda)}2 \, \left|\eta -
        \frac{k\tau}{L}\right|}\,d\tau\right),
\end{align}
where we have used $\lambda'< (\lambda'+\lambda)/2 < \lambda$, 
and $C$ only depends on $C_W,C_i,\lambda,\lambda',\kappa$.

In the end, 
\begin{align*}
\int_0^t e^{-2\pi \lambda' \frac{|k|}{L}\tau}\,e^{-2\pi \frac{(\lambda'+
    \lambda)}2 \, \left|\eta-\frac{k\tau}{L}\right|}\,d\tau
& \leq \int_0^t e^{-2\pi\lambda' |\eta|}\, 
e^{- 2\pi\frac{(\lambda-\lambda')}2 \, \left|\eta - \frac{k \tau}{L} 
  \right|}\,d\tau\\
& \leq \frac{L}{\pi(\lambda-\lambda')}\, e^{-2\pi \left( \lambda' -
    \frac{(\lambda - \lambda')}{2}\right) |\eta|}.
\end{align*}

Plugging this back in \eqref{boundfL}, we obtain, with $\lambda''=
\lambda' - (\lambda - \lambda')/2$,
 \begeq\label{boundflin} \left|
  \tilde{f}^{(L)}\left( t, k, \eta-\frac{kt}{L}\right) \right| \leq
C\, e^{-2\pi \lambda'' |\eta|}.
\endeq
In particular, for any fixed $\eta$ and $k\neq 0$,
\[ \bigl| \tilde{f}^{(L)}(t,k,\eta)\bigr|\leq C\, e^{-2\pi\lambda''
  \left|\eta+\frac{kt}{L}\right|} = O \bigl( e^{-2\pi
  \frac{\lambda''}L |t|}\bigr).\] We conclude that $\tilde{f}^{(L)}$
converges pointwise, exponentially fast, to the Fourier transform of
$\<f_i\>$. Since $\lambda'$ and then $\lambda''$ can be taken as close
to $\lambda$ as wanted, this ends the proof.
\end{proof}

We close this section by proving Proposition \ref{suffL}.

\begin{proof}[Proof of Proposition \ref{suffL}]
  First assume (a). Since $\tilde{f}^0$ decreases exponentially fast,
  we can find $\lambda,\kappa>0$ such that
  \[ 4\pi^2\,\max \bigl|\hat{W}^{(L)}(k)\bigr|\, \sup_{|\sigma|=1}
  \int_0^\infty \bigl|\tilde{f}^0(r\sigma)\bigr|\,r\,e^{2\pi\lambda
    r}\,dr \leq 1-\kappa.\] Performing the change of variables
  $kt/L=r\sigma$ inside the integral, we deduce
\[ \int_0^\infty 4\pi^2\,|\hat{W}^{(L)}(k)|\,
\left|\tilde{f}^0\left(\frac{kt}{L}\right)\right|\,
\frac{|k|^2\,t}{L^2}\,e^{2\pi\lambda \frac{|k|}{L}t}\,dt \leq
1-\kappa,\] and this obviously implies {\bf (L)}.  \sm

The choice $w=0$ in \eqref{penrose} shows that Condition (b) is a particular case of (c),
so we only treat the latter assumption.
The reasoning is more subtle than for case (a). First we note that
\begin{align*}
K^0(t,k) & = -4\pi^2\,\hat{W}(k)\int_{\R^d} f^0(v)\,e^{-2i\pi \frac{kt}{L}\cdot v}\,\frac{|k|^2}{L^2}\,t\,dv\\
& = -4\pi^2\,\hat{W}(k)\int_{\R} \varphi_k(v)\,e^{-2i \pi \frac{|k|}{L}t v}\,\frac{|k|^2}{L^2}\,t\,dv\\
& = -4\pi^2\,\frac{|k|^2\,\hat{W}(k)\,t}{L^2}
\int_{\R} \left(\frac{2i\pi\,|k|t}{L}\right)^{-1}\,\varphi'_k(v)\,e^{-2i \pi \frac{|k|}{L}t v}\,dv\\
& = 2i\pi \frac{|k|\,\hat{W}(k)}{L}
\int_\R \varphi'_k(v)\,e^{-2i\pi \frac{|k|}{L} t v} \, dv.
\end{align*}

Then, for $\xi = \gamma + i\omega$, using the formula 
\[ \int_0^\infty e^{-st}\,e^{i\omega t}\,dt = \frac{s+i\omega}{s^2+\omega^2},\]
we get from \eqref{K0hatxik}
\[ {\cal L}(\xi,k) = \hat{W}(k) \, \int_\R \varphi'_k(v)
\left[\frac{(v+\omega) - i\gamma}{(v+\omega)^2 +
    \gamma^2}\right]\,dv.\] (To be rigorous, one may first establish
this formula for $\gamma<0$, and then use analyticity to derive it for
$\gamma\in [0,\lambda)$.)

As $\gamma\to 0$, this expression approaches, uniformly in $k$ and $\omega$,
\begin{align}\label{Lio}
  {\cal L}(i\omega, k) & = \hat{W}(k) \int \frac{\varphi'_k}{v+\omega+i\,0}\,dv \\
& = \hat{W}(k)\ {\rm p.v.} \left(\int_\R \frac{\varphi'_k(v)}{v+\omega}\,dv\right)
- i\pi\, \hat{W}(k)\,\varphi'_k(-\omega) \nonumber
\end{align}
(Plemelj formula for the Cauchy transform). The problem is to show
that ${\cal L}(i\omega,k)$ stays away from~1 as $\omega$ varies in
$\R$; then the same will be true for $\xi = \gamma+i\omega$ with
$\gamma$ small enough. Equation \eqref{Lio} shows that the imaginary
part of ${\cal L}(i\omega,k)$ vanishes only in the limit
$\hat{W}(k)\to 0$ (but then also the real part approaches~0), or in
the limit $|\omega|\to\infty$ (but then also the real part
approaches~0), or if $\varphi'_k(-\om)=0$; but then by \eqref{penrose}
\[ {\cal L}(i\omega,k)=
 \hat{W}(k)\int_\R\frac{\varphi'_k(v)}{v+\omega}\,dv < 1,\]
so even in this case ${\cal L}$ cannot approach~1. Case (c) of
Proposition \ref{suffL} follows by a compactness argument.
\end{proof}

\section{Analytic norms}
\label{sec:analytic}

In this section we introduce some functional spaces of analytic functions on $\R^d$, $\T^d=\R^d/\Z^d$, 
and most importantly $\T^d\times\R^d$. (Changing the sidelength of the torus only results in some
changes in the constants.) Then we establish a number of functional inequalities which will be
crucial in the subsequent analysis. At the end of this section we shall reformulate the linear study in this
new setting.

Throughout the whole section $d$ is a positive integer. Working with analytic functions will force us to
be careful with combinatorial issues, and proofs will at times involve summation over many indices.

\subsection{Single-variable analytic norms}

Here ``single-variable'' means that the variable lives in either $\R^d$ or $\T^d$, but $d$ may be greater than~1.

Among many possible families of norms for analytic functions,
two will be of particular interest for us; they will be denoted by
$\cC^{\lambda;p}$ and $\cF^{\lambda;p}$. The $\cC^{\lambda;p}$ norms are defined for functions
on $\R^d$ or $\T^d$, while the $\cF^{\lambda;p}$ norms are defined only for $\T^d$
(although we could easily cook up a variant in $\R^d$).
We shall write $\N_0^d$ for the set of $d$-tuples of integers (the subscript being here to insist 
that 0 is allowed). If $n\in\N_0^d$ and $\lambda\geq 0$ we shall write
$\lambda^n = \lambda^{|n|}$. Conventions about Fourier transform and multidimensional differential 
calculus are gathered in the Appendix.

\begin{Def}[One-variable analytic norms]
For any $p\in [1,\infty]$ and $\lambda\geq 0$, we define
\begeq\label{CFlambda} 
\|f\|_{\cC^{\lambda;p}} := \sum_{n\in\N_0^d}
\frac{\lambda^n}{n!}\, \|f^{(n)}\|_{L^p};\qquad
\|f\|_{\cF^{\lambda;p}} := \left( 
\sum_{k\in\Z^d} e^{2\pi\lambda p |k|}\, |\hat{f}(k)|^p \right)^{1/p};
\endeq
the latter expression standing for $\sup_k (e^{2\pi\lambda|k|}\,|\hat{f}(k)|)$ if $p=\infty$.
We further write
\begeq
\cC^{\lambda,\infty} = \cC^\lambda,\qquad \cF^{\lambda,1} = \cF^\lambda.
\endeq
\end{Def}

\begin{Rk} The parameter $\lambda$ can be interpreted as a radius of convergence.
\end{Rk}

\begin{Rk} The norms $\cC^\lambda$ and $\cF^\lambda$ are of particular interest because they are algebra norms.
\end{Rk}

We shall sometimes abbreviate $\|\cdot\|_{\cC^{\lambda;p}}$ or
$\|\cdot\|_{\cF^{\lambda;p}}$ into $\|\cdot\|_{\lambda;p}$ when no
confusion is possible, or when the statement works for either. 

The norms in \eqref{CFlambda} extend to vector-valued functions in a natural way: 
if $f$ is valued in $\R^d$ or $\T^d$ or $\Z^d$, define $f^{(n)}=(f_1^{(n)},\ldots,f_d^{(n)})$,
$\hat{f}(k)=(\hat{f}_1(k),\ldots,\hat{f}_d(k))$; then the formulas in \eqref{CFlambda}
make sense provided that we choose a norm on $\R^d$ or $\T^d$ or $\Z^d$. Which norm we choose
will depend on the context; the choice will always be done in such a way to get the duality right
in the inequality $|a\cdot b|\leq \|a\|\,\|b\|_*$. For instance if $f$ is valued in $\Z^d$ and
$g$ in $\T^d$, and we have to estimate $f\cdot g$, we may norm $\Z^d$ by $|k|=\sum |k_i|$
and $\T^d$ by $|x|=\sup |x_i|$.\footnote{Of course all norms are equivalent, still the choice is not innocent
when the estimates are iterated infinitely many times; an advantage of the supremum norm on $\R^d$ is that it has
the algebra property.} This will not pose any problem, and the reader can forget about
this issue; we shall just make remarks about it whenever needed. For the rest of this section,
we shall focus on scalar-valued functions for simplicity of exposition.
\med

Next, we define ``homogeneous'' analytic seminorms by removing the zero-order term.
We write $\N_*^d=\N_0^d\setminus \{0\}$, $\Z_*^d=\Z^d\setminus\{0\}$. 

\begin{Def}[One-variable homogeneous analytic seminorms]
For $p\in [1,\infty]$ and $\lambda\geq 0$ we write
\[ \|f\|_{\dot{\cC}^{\lambda;p}} =
\sum_{{n\in\N_*^d}}\frac{\lambda^n}{n!}\, \|f^{(n)}\|_{L^p};\qquad
 \|f\|_{\dot{\cF}^{\lambda;p}} = \left( \sum_{k\in\Z_*^d} e^{2\pi\lambda p |k|}\,
|\hat{f}(k)|^p \right)^{1/p}.\] 
\end{Def}

It is interesting to note that affine functions $x \mapsto a \cdot x
+b$ can be included in $\dot{\cC}^\lambda=\dot{\cC}^{\lambda;\infty}$,
even though they are unbounded; in particular
$\| a \cdot x + b \|_{\dot{\cC}^\lambda} = \lambda \,|a|$. On the other hand,
linear forms $x\longmapsto a\cdot x$ do not naturally belong to
$\dot{\cF}^\lambda$, because their Fourier expansion is not even summable 
(it decays like $1/k$).

The spaces $\cC^{\lambda;p}$ and $\cF^{\lambda;p}$ enjoy remarkable properties, 
summarized in Propositions \ref{propalg}, \ref{propcompos} and \ref{propgrad} below.
Some of these properties are well-known, other not so.

\begin{Prop}[Algebra property] \label{propalg} 
  
  (i) For any $\lambda \ge 0$, and $p,q,r \in [1,+\infty]$ such that
  $1/p+1/q=1/r$, we have
  \[ \| f \, g \|_{\cC^{\lambda;r}} \le \| f \|_{\cC^{\lambda;p}} \, 
                                     \| g \|_{\cC^{\lambda;q}}. \]
  
  (ii) For any $\lambda \ge 0$, and $p,q,r \in [1,+\infty]$ such that
  $1/p+1/q=1/r+1$, we have
  \[ \| f \, g \|_{\cF^{\lambda;r}} \le \| f \|_{\cF^{\lambda;p}} \, 
                                     \| g \|_{\cF^{\lambda;q}}. \]
  
  (iii) As a consequence, for any $\lambda \ge 0$, 
  $\cC^\lambda = \cC^{\lambda;\infty}$ and 
  $\cF^\lambda=\cF^{\lambda;1}$ are normed algebras: for either space,
  \[ \|fg\|_\lambda \leq \|f\|_\lambda\, \|g\|_\lambda.\] 
  In particular, $\|f^n\|_\lambda \leq \|f\|_\lambda^n$ for any
  $n\in\N_0$, and $\|e^f\|_\lambda \leq e^{\|f\|_\lambda}$.
\end{Prop}

\begin{Rk} Ultimately, property (iii) relies on the fact that
  $L^\infty$ and $L^1$ are normed algebras for the multiplication and
  convolution, respectively.
\end{Rk}

\begin{Rk} \label{rkcontrol} It follows from the Fourier inversion
  formula and Proposition \ref{propalg} that $\|f\|_{\cC^\lambda} \le
  \|f\|_{\cF^\lambda}$ (and $\|f\|_{\dot{\cC}^\lambda} \le
  \|f\|_{\dot{\cF}^\lambda}$); this is a special case of Proposition
  \ref{propcompos} (iv) below. The
    reverse inequality does not hold, because $\|f\|_\infty$ does not
    control $\|\hat{f}\|_{L^1}$.
\end{Rk}

Analytic norms are very sensitive to composition; think that if $a>0$
then $\|f\circ (a\,\Id)\|_{\cC^{\lambda;p}} = a^{-d/p}
\,\|f\|_{\cC^{a\lambda;p}}$; so we typically lose on the functional
space.  This is a major difference with more traditional norms used in
partial differential equations theory, such as H\"older or Sobolev
norms, for which composition may affect constants but not regularity
indices. The next proposition controls the loss of regularity implied
by composition.

\begin{Prop}[Composition inequality] \label{propcompos} 
(i) For any $\lambda >0$ and any $p \in [1,+\infty]$,
\[ \|f\circ H\|_{\cC^{\lambda;p}} \leq \bigl\| (\det\nabla
H)^{-1}\bigr\|^{1/p}_\infty \, \|f\|_{\cC^{\nu;p}},\qquad \nu =
\|H\|_{\dot{\cC}^\lambda},\] where $H$ is possibly unbounded.

(ii) For any $\lambda>0$, any $p \in [1,\infty]$ and any $a>0$,
\[ \Bigl\| f\circ (a\,\Id + G) \Bigr\|_{\cC^{\lambda;p}} \leq 
a^{-d/p} \, \|f\|_{\cC^{a \lambda+\nu \, ;p}},\qquad
\nu = \|G\|_{\cC^{\lambda}}.\]

(iii) For any $\lambda>0$,
\[ \Bigl\| f\circ (\Id + G) \Bigr\|_{\cF^{\lambda}} \leq
\|f\|_{\cF^{\lambda+\nu}},\qquad \nu =
\|G\|_{\dot{\cF}^{\lambda}}.\] 

(iv) For any $\lambda>0$ and any $a> 0$,
\[ \Bigl\| f\circ (a\,\Id + G) \Bigr\|_{\cC^\lambda} \leq 
\|f\|_{\cF^{a\lambda+\nu}},\qquad \nu = \|G\|_{\dot{\cC}^\lambda}.\]
\end{Prop}

\begin{Rk} Inequality (iv), with $\cC$ on the left and $\cF$ on the
  right, will be most useful.  The reverse inequality is not likely to
  hold, in view of Remark \ref{rkcontrol}.
\end{Rk}

The last property of interest for us is the control of the loss of
regularity involved by differentiation.

\begin{Prop}[Control of gradients] \label{propgrad} For any
  $\ov{\lambda}>\lambda$, any $p \in [1,+\infty]$, we have
  \begeq\label{nablafC} \|\nabla
  f\|_{\cC^{\lambda;p}} \leq \left( \frac1{\lambda e\, \log
      (\ov{\lambda}/\lambda)}\right)\, \|f\|_{\dot{\cC}^{\ov{\lambda};p}};
  \endeq
  \begeq\label{nablafF} \|\nabla f\|_{\cF^{\lambda;p}} \leq
  \left(\frac1{2\pi e\, (\ov{\lambda}-\lambda)}\right)\,
  \|f\|_{\dot{\cF}^{\ov{\lambda};p}}.
  \endeq
\end{Prop}

The proofs of Propositions \ref{propalg} to \ref{propgrad} will be
preparations for the more complicated situations considered in the
sequel.

\begin{proof}[Proof of Proposition \ref{propalg}]
(i) Denoting by $\|\cdot\|_{\lambda;p}$ the norm of $\cC^{\lambda;p}$, using
the multidimensional Leibniz formula from Appendix \ref{appdc}, we have
\begin{align*}
\|fg\|_{\lambda;r} = \sum_{\ell\in\N_0^d} \|(fg)^{(\ell)}\|_{L^r} \frac{\lambda^\ell}{\ell!}
& \leq \sum_{\ell\in\N_0^d} \sum_{m\leq\ell} \Cnk{\ell}{m}\,
\bigl\| f^{(m)} g^{(\ell-m)} \bigr\|_{L^r}\, \frac{\lambda^\ell}{\ell!}\\
& \leq \sum_{\ell\in\N_0^d} \sum_{m\leq \ell}
\Cnk{\ell}{m} \,\|f^{(m)}\|_{L^p} \, \|g^{(\ell-m)}\|_{L^q} \, \frac{\lambda^\ell}{\ell!} \\
& = \sum_\ell \sum_m \frac{\|f^{(m)}\|_{L^p} \,\lambda^m}{m!}\, 
\frac{\|g^{(\ell-m)}\|_{L^q} \,\lambda^{\ell-m}}{(\ell-m)!} \\
& = \|f\|_{\lambda;p}\, \|g\|_{\lambda;q}.
\end{align*}
\sm

(ii) Denoting now by $\|\cdot\|_{\lambda;p}$ the norm of
$\cF^{\lambda;p}$, and applying Young's convolution inequality, we get
\begin{multline*} \|fg\|_{\lambda;r} = \left( \sum |\hat{fg}(k)|^r
    \,e^{2\pi\lambda r |k|} \right)^{1/r} 
  \leq \left( \sum_k \Bigl( \sum_\ell |\hat{f}(\ell)|\,
    |\hat{g}(k-\ell)| \, e^{2\pi\lambda |k-\ell|} e^{2\pi \lambda
     |\ell|} \Bigr)^r \right)^{1/r}\\
\leq \left(\sum_k |\hat{f}(k)|^p\,e^{2\pi\lambda p|k-\ell|}\right)^{\frac1{p}}\
\left(\sum_\ell |\hat{g}(\ell)|^q\,e^{2\pi\lambda q|\ell|}\right)^{\frac1{q}}.
\end{multline*}
\end{proof}

\begin{proof}[Proof of Proposition \ref{propcompos}]

  {\bf Case (i).} 
  We use the (multi-dimensional) Fa\`a di Bruno formula:
  \[ (f\circ H)^{(n)} = \sum_{\sum_{j=1} ^n j \, m_j = n}
  \frac{n!}{m_1!\ldots m_n!}\, \bigl( f^{(m_1+\ldots + m_n)}\circ H\bigr)\,
  \prod_{j=1}^n \left(\frac{H^{(j)}}{j!}\right)^{m_j};\]
  so
  \[ \bigl\|(f\circ H)^{(n)}\bigr\|_{L^p} \leq \sum_{\sum_{j=1} ^n j
    \, m_j = n}
  \frac{n!}{m_1!\ldots m_n!}\, \Bigl\| f^{(m_1+\ldots + m_n)} \circ H\Bigr\|_{L^p}\,
  \prod_{j=1}^n \left\|\frac{H^{(j)}}{j!}\right\|^{m_j} _\infty;\]
  thus
\begin{align*}
  \sum_{n\ge 1} \frac{\lambda^n}{n!}\, \bigl\|(f\circ H)^{(n)}\bigr\|_{L^p}
& \leq \bigl\| (\det\nabla H)^{-1}\bigr\|^{1/p} _\infty \, 
\Bigg( \sum_{k=1} ^{+\infty} \|f^{(k)}\|_{L^p} \\
& \qquad \qquad \qquad \sum_{\sum_{j=1} ^n j \, m_j = n,\ \sum_{j=1} ^n m_j = k}\,
\frac{\lambda^n}{m_1!\ldots m_n!}\prod_{j=1}^n 
\left\|\frac{H^{(j)}}{j!}\right\|_\infty^{m_j} \Bigg) \\
& = \bigl\| (\det\nabla H)^{-1}\bigr\|^{1/p} _\infty \, \left( 
\sum_{k \ge 1} \|f^{(k)}\|_{L^p}\,\frac1{k!} \, 
\left( \sum_{|\ell|\geq 1} \frac{\lambda^\ell}{\ell!}\,
  \|H^{(\ell)}\|_\infty \right)^k \right),
\end{align*}
where the last step follows from the multidimensional binomial formula.

\med

{\bf Case (ii).} We decompose $h(x) := f(ax + G(x))$ as 
$$
h(x) = \sum_{n \in\N^d_0} \frac{(f^{(n)})(ax)}{n!} \, G(x)^n
$$
and we apply $\nabla^k$:
$$
\nabla^k h(x) = \sum_{k_1+k_2 = k, \, \in \N^d _0} \ \sum_{n \in \N^d_0}
\frac{k! \, a^{k_1}}{k_1 ! \, k_2 ! \, n!} \, (\nabla^{k_1+n} f)(ax) \, 
(\nabla^{k_2}(G^n))(x).
$$
Then we take the $L^p$ norm, multiply by $\lambda^k/k!$ and sum over $k$:
\begin{align*}
  \| h \|_{\cC^{\lambda;p}} & \le |a|^{-d/p} \, \sum_{k_1,k_2, n \ge 0}
  \frac{\lambda^{k_1+k_2} \, |a|^{k_1}}{k_1 ! \, k_2 ! \, n!} \, \|
  \nabla^{k_1+n} f \|_{L^p} \,
  \bigl\| \nabla^{k_2}(G^n) \bigr\|_\infty\\
  & = |a|^{-d/p} \, \sum_{k_1, n \ge 0} \frac{\lambda^{k_1} \,
    |a|^{k_1}}{k_1 ! \, n!} \, \| \nabla^{k_1+n} f \|_{L^p} \,
  \| G^n \|_{\cC^\lambda}\\
  & \le |a|^{-d/p} \, \sum_{k_1, n \ge 0} \frac{\lambda^{k_1} \,
    |a|^{k_1}}{k_1 ! \, n!} \, \| \nabla^{k_1+n} f \|_{L^p} \,
  \| G \|_{\cC^\lambda} ^n \\
  & =|a|^{-d/p} \,\sum_{m \ge 0} \frac{\left(a \, \lambda+\|
    G\|_{\cC^\lambda}\right)^m}{m!} \, \| \nabla^{m} f \|_{L^p},
\end{align*}
where Proposition \ref{propalg} (iii) was used in the but-to-last step.
\med

{\bf Case (iii).} In this case we write, with $G_0=\hat{G}(0)$,
\[ h(x) = f(x+G(x)) = \sum_k \hat{f}(k)\,e^{2i\pi k\cdot x} \, e^{2i\pi k\cdot G_0}
\, e^{2i\pi k\cdot (G(x)-G_0)};\]
so
\[ \hat{h}(\ell) = \sum_k \hat{f}(k) \, e^{2i\pi k\cdot G_0}\,
\bigl[ e^{2i\pi k\cdot (G-G_0)}\bigr]^{\hat{\ }}\,(\ell-k).\]
Then (using again Proposition \ref{propalg})
\begin{align*}
\sum_\ell |\hat{h}(\ell)|\,e^{2\pi \lambda |\ell|}
& \leq \sum_k \sum_\ell |\hat{f}(k)|\,e^{2\pi\lambda |k|}\,
e^{2\pi\lambda |\ell-k|}\,
\Bigl| \bigl[ e^{2i\pi k\cdot (G-G_0)}\bigr]^{\hat{\ }}\,(\ell-k)\Bigr|\\
& = \sum_k |\hat{f}(k)|\,e^{2\pi\lambda |k|}\, \Bigl\| e^{2i\pi k\cdot (G-G_0)}\Bigr\|_\lambda\\
& \leq \sum_k |\hat{f}(k)|\,e^{2\pi\lambda |k|}\, e^{\|2\pi k\cdot (G-G_0)\|_\lambda}\\
& \leq \sum_k |\hat{f}(k)|\, e^{2\pi\lambda |k|}\,e^{2\pi|k|\,\|G-G_0\|_\lambda}\\
& = \|f\|_{\lambda + \|G-G_0\|_\lambda} = \|f\|_{\lambda + \nu},\qquad \nu =\|G\|_{\dot{\cF}^\lambda}.
\end{align*}
\med

{\bf Case (iv).} We actually have the more precise result
\begeq\label{CFcompos}
\|f\circ H\|_{\cC^\lambda} \leq \sum |\hat{f}(k)|\, e^{2\pi|k|\, \|H\|_{\dot{\cC}^\lambda}}.
\endeq
Writing $f\circ H = \sum \hat{f}(k)\,e^{2i\pi k\cdot H}$, we see that \eqref{CFcompos} 
follows from
\begeq\label{normeih} \|e^{ih}\|_{\cC^\lambda} \leq e^{\|h\|_{\dot{\cC}^\lambda}}.
\endeq
To prove \eqref{normeih}, let $P_n$ be the polynomial in the variables
$X_m$ ($m\leq n$) defined by the identity $(e^f)^{(n)}=
P_n((f^{(m)})_{m\leq n})\,e^f$; this polynomial (which can be made
more explicit from the Fa\`a di Bruno formula) has nonnegative
coefficients, so $\|(e^{if})^{(n)}\|_\infty \leq P_n
((\|f^{(m)}\|)_{m\leq n})$. The conclusion will follow from the identity
(between formal series!)  
\begeq\label{idtyseries} 1 +
\sum_{n\in\N_*^d} \frac{\lambda^n}{n!}  \, P_n((X_m)_{m\leq n}) = \exp
\left( \sum_{k\in\N_*^d}\frac{\lambda^k}{k!}\,X_k\right).
\endeq
To prove \eqref{idtyseries}, it is sufficient to note that the
left-hand side is the expansion of $e^g$ in powers of $\lambda$ at 0,
where $g(\lambda)= \sum_{k\in\N_*^d} \frac{\lambda^k}{k!}\,X_k$.
\end{proof}

\begin{proof}[Proof of Proposition \ref{propgrad}]
(a) Writing $\|\cdot\|_{\lambda;p} = \|\cdot\|_{\cC^{\lambda;p}}$, we have
\[ \|\pa_i f\|_{\lambda;p} = \sum_n \frac{\lambda^n}{n!}\, \|\pa_x^n \pa_i f\|_{L^p},\]
where $\pa_i = \pa/\pa x_i$. If $1_i$ is the $d$-uple of
integers with 1 in position $i$, then $(n+1_i)!\leq (|n|+1) n!$, so
\[ \|\pa_i f\|_{\lambda;p} \leq \sup_n
\left( \frac{(|n|+1)\, \lambda^n}{{\ov{\lambda}}^{n+1}}\right)\,
\sum_{|m|\geq 1} \frac{{\ov{\lambda}}^m}{m!}\, \|\nabla^m f\|_{L^p},\]
and the proof of \eqref{nablafC} follows easily.
\sm

(b) Writing $\|\cdot\|_{\lambda;p} = \|\cdot\|_{\cF^{\lambda;p}}$, we have
\begin{align*}
  \|\pa_i f\|_{\lambda;p} & = \left( \sum_k |k_i|^p \, |\hat{f}(k)|^p
    \,e^{2\pi\lambda p \, |k|} \right)^{1/p}\\
  & \leq \left[ \sup_{k\in\Z} \left( |k|\,e^{2\pi(\lambda
        -\ov{\lambda}) \, |k|}\right) \right]\, \left(\sum_{k\in\Z^d}
    |\hat{f}(k)|^p\,e^{2\pi\ov{\lambda} p \, |k|}\right)^{1/p},
\end{align*}
and \eqref{nablafF} follows.
\end{proof}

\subsection{Analytic norms in two variables} \label{suban2}

To estimate solutions and trajectories of kinetic equations we will
work on the phase space $\T^d_x\times\R^d_v$, and use three
parameters: $\lambda$ ({\em gliding} analytic regularity); $\mu$
(analytic regularity in $x$); and $\tau$ (time-shift along the free
transport semigroup). The regularity quantified by $\lambda$ is said
to be gliding because for $\tau=0$ this is an analytic regularity in
$v$, but as $\tau$ grows the regularity is progressively transferred
from velocity to spatial modes, according to the evolution by free
transport. This catch is crucial to our analysis: indeed, the solution
of a transport equation like free transport or Vlasov {\em cannot} be
uniformly analytic\footnote{By this we mean of course that some {\em norm} or seminorm quantifying 
the degree of analytic smoothness in $v$ will remain uniformly bounded.}
in $v$ as time goes by --- except of course if it
is spatially homogeneous.  Instead, the best we can do is compare the
solution at time $\tau$ to the solution of free transport at the same
time --- a kind of scattering point of view.

The parameters $\lambda,\mu$ will be nonnegative; $\tau$ will vary in $\R$, but often be restricted to
$\R_+$, just because we shall work in positive time. When $\tau$ is
not specified, this means $\tau=0$. Sometimes we shall abuse notation
by writing $\|f(x,v)\|$ instead of $\|f\|$, to stress the dependence
of $f$ on the two variables.

Putting aside the time-shift for a moment, we may generalize the norms $\cC^\lambda$ and $\cF^\lambda$ 
in an obvious way:

\begin{Def}[Two-variables analytic norms]
For any $\lambda,\mu\geq 0$, we define
\begeq\label{Clambdamu}
\|f\|_{\cC^{\lambda,\mu}} =
\sum_{m\in\N_0^d} \sum_{n\in\N_0^d} \frac{\lambda^n}{n!}\,\frac{\mu^m}{m!}\,
\Bigl\| \nabla_x^m \nabla_v ^nf\Bigr\|_{L^\infty(\T^d_x\times\R^d_v)};
\endeq
\begeq\label{Flambdamu}
\|f\|_{\cF^{\lambda,\mu}} = 
\sum_{k\in\Z^d} \int_{\eta\in\R^d}
|\tilde{f}(k,\eta)|\, e^{2\pi \lambda |\eta|}\, e^{2\pi\mu|k|}\,d\eta.
\endeq
\end{Def}

Of course one might also introduce variants based on $L^p$ or $\ell^p$
norms (with two additional parameters $p,q$, since one can make
different choices for the space and velocity variables).

The norm \eqref{Flambdamu} is better adapted to the periodic nature of
the problem, and is very well suited to estimate solutions of kinetic
equations (with fast decay as $|v|\to\infty$); but in the sequel we
shall also have to estimate characteristics (trajectories) which are
unbounded functions of $v$. We could hope to play with two
different families of norms, but this would entail considerable
technical difficulties. Instead, we shall mix the two recipes to get
the following {\em hybrid norms}:

\begin{Def}[Hybrid analytic norms] For any $\lambda,\mu\geq 0$, let
  \begeq\label{Zlambdamu} \|f\|_{\cZ^{\lambda,\mu}} =
  \sum_{\ell\in\Z^d} \sum_{n\in\N_0^d} \frac{\lambda^n}{n!}\,
  e^{2\pi\mu|\ell|}\, \Bigl\| \hat{\nabla_v^n
    f}(\ell,v)\Bigr\|_{L^\infty(\R^d_v)}.
  \endeq
  More generally, for any $p\in [1,\infty]$ we define
  \begeq\label{Zlambdamu:p} \|f\|_{\cZ^{\lambda,\mu;p}} =
  \sum_{\ell\in\Z^d} \sum_{n\in\N_0^d} \frac{\lambda^n}{n!}\,
  e^{2\pi\mu|\ell|}\, \Bigl\| \hat{\nabla_v^n
    f}(\ell,v)\Bigr\|_{L^p(\R^d_v)}.
  \endeq
\end{Def}

Now let us introduce the time-shift $\tau$. We denote by
$(S^0_\tau)_{\tau\geq 0}$ the geodesic semigroup: $(S^0_\tau)(x,v) =
(x+v\tau,v)$.  Recall that the backward free transport semigroup is
defined by $(f\circ S^0_\tau)_{\tau\geq 0}$, and the forward semigroup
by $(f\circ S^0_{-\tau})_{\tau\geq 0}$. 

\begin{Def}[Time-shift pure and hybrid analytic norms]
\begeq\label{Clambdamu:shift}
\|f\|_{\cC^{\lambda,\mu}_\tau} = \| f\circ S^0_\tau\|_{\cC^{\lambda,\mu}} = 
\sum_{m\in\N_0^d} \sum_{n\in\N_0^d} \frac{\lambda^n}{n!}\,\frac{\mu^m}{m!}\,
\Bigl\| \nabla_x^m (\nabla_v + \tau\nabla_x)^nf\Bigr\|_{L^\infty(\T^d_x\times\R^d_v)};
\endeq
\begeq\label{Flambdamu:shift}
\|f\|_{\cF^{\lambda,\mu}_\tau} = \| f\circ S^0_\tau\|_{\cF^{\lambda,\mu}} = 
\sum_{k\in\Z^d} \int_{\eta\in\R^d}
|\tilde{f}(k,\eta)|\, e^{2\pi \lambda |k\tau + \eta|}\, e^{2\pi\mu|k|}\,d\eta;
\endeq
\begeq\label{Zlambdamu:shift}
\|f\|_{\cZ^{\lambda,\mu}_\tau} = \| f\circ S^0_\tau\|_{\cZ^{\lambda,\mu}} = 
\sum_{\ell\in\Z^d} \sum_{n\in\N_0^d} \frac{\lambda^n}{n!}\, e^{2\pi\mu|\ell|}\,
\Bigl\| (\nabla_v + 2i\pi\tau \ell)^n \hat{f}(\ell,v)\Bigr\|_{L^\infty(\R^d_v)};
\endeq
\begeq\label{Zlambdamup:shift:p}
\|f\|_{\cZ^{\lambda,\mu;p}_\tau} =
\sum_{\ell\in\Z^d} \sum_{n\in\N_0^d} \frac{\lambda^n}{n!}\, e^{2\pi\mu|\ell|}\,
\Bigl\| (\nabla_v + 2i\pi\tau \ell)^n \hat{f}(\ell,v)\Bigr\|_{L^p(\R^d_v)}.
\endeq
\end{Def}

This choice of norms is one of the cornerstones of our
analysis: first, because of their hybrid nature, they will connect
well to both periodic (in $x$) estimates on the force field, and
uniform (in $v$) estimates on the ``scattering transforms'' studied in
Section \ref{sec:scattering}.  Secondly, they are well-behaved with
respect to the properties of free transport, allowing to keep track of
the initial time without needing ridiculous (and inaccessible) amounts of regularity in
$x$ as time goes by.  Thirdly, they will satisfy the algebra property
(for $p=\infty$), the composition inequality and the gradient
inequality (for any $p\in [1,\infty]$).  Before going on with the
proof of these properties, we note the following alternative
representations.

\begin{Prop} \label{propalterZ}
The norm $\cZ^{\lambda,\mu;p}_\tau$ admits the alternative representations:
\begeq\label{alterZ1}
\|f\|_{\cZ^{\lambda,\mu;p}_\tau}
= \sum_{\ell\in\Z^d} \sum_{n\in\N_0^d}
\frac{\lambda^n}{n!}\,e^{2\pi\mu|\ell|}\,
\Bigl\| \nabla_v^n \bigl( \hat{f}(\ell,v)\, e^{2i\pi\tau \ell\cdot v}\bigr)\Bigr\|_{L^p(\R^d_v)};
\endeq
\begeq\label{alterZ2}
\|f\|_{\cZ^{\lambda,\mu;p}_\tau} = \sum_{n\in\N_0^d} \frac{\lambda^n}{n!}\,
\btrn (\nabla_v + \tau\nabla_x)^n f\btrn_{\mu;p},
\endeq
where
\begeq\label{trn} \trn g\trn_{\mu;p} = \sum_{\ell\in\Z^d} e^{2\pi\mu |\ell|}\,
\bigl\|\hat{g}(\ell,v)\bigr\|_{L^p(\R^d_v)}.
\endeq
\end{Prop}

\subsection{Relations between functional spaces}

The next propositions are easily checked. 

\begin{Prop} \label{propZCF} With the notation from Subsection
  \ref{suban2}, for any $\tau \in \R$,

(i) if $f$ is a function only of $x$ then
\[ \|f\|_{\cC^{\lambda,\mu}_\tau} = \|f\|_{\cC^{\lambda
    |\tau|+\mu}},\qquad \|f\|_{\cF^{\lambda,\mu}_\tau} =
\|f\|_{\cZ^{\lambda,\mu}_\tau} = \|f\|_{\cF^{\lambda |\tau|+\mu}} ;\]

(ii) if $f$ is a function only of $v$ then
\[ \|f\|_{\cC^{\lambda,\mu;p}_\tau} = \|f\|_{\cZ^{\lambda,\mu;p}_\tau}
= \|f\|_{\cC^{\lambda;p}},\qquad \|f\|_{\cF^{\lambda,\mu}_\tau} =
\|f\|_{\cF^\lambda};\]

(iii) for any function $f=f(x,v)$, if $\<\,\cdot\,\>$ stands for
spatial average then
\[ \|\<f\>\|_{\cC^{\lambda;p}} \leq \|f\|_{\cZ^{\lambda,\mu;p}_\tau};\]

(iv) for any function $f=f(x,v)$,
\[ \left\| \int_{\R^d} f\,dv \right\|_{\cF^{\lambda|\tau|+\mu}} \leq
\|f\|_{\cZ^{\lambda,\mu;1}_\tau}.\]
\end{Prop}

\begin{Rk} Note, in Proposition \ref{propZCF} (i) and (iv),
how the regularity in $x$ is improved by the time-shift.
\end{Rk}

\begin{proof}[Proof of Proposition \ref{propZCF}]
  Only (iv) requires some explanations. Let $\rho(x) = \int
  f(x,v)\,dv$. Then for any $k\in\Z^d$,
\[ \hat{\rho}(k) = \int_{\R^d} \hat{f}(k,v)\,dv;\]
so for any $n\in\N^d_0$,
\begin{align*}
(2i\pi tk)^n\,\hat{\rho}(k) & = \int (2i\pi tk)^n\hat{f}(k,v)\,dv\\
& = \int (\nabla_v + 2i\pi t k)^n\,\hat{f}(k,v)\,dv.
\end{align*}
Recalling the conventions from Appendix \ref{app:exp} we deduce
\begin{align*}
\sum_{k,n} e^{2\pi \mu|k|}\, \frac{|2\pi \lambda tk|^n}{n!}\,|\hat{\rho}(k)|
& \leq \sum_{k,n} e^{2\pi \mu|k|}\, \frac{\lambda^n}{n!}
\int \Bigl|(\nabla_v + 2i\pi tk)^n\,\hat{f}(k,v)\Bigr|\,dv\\
& = \|f\|_{\cZ^{\lambda,\mu;1}_t}.
\end{align*}
\end{proof}

\begin{Prop} \label{propincluZ}
With the notation from Subsection \ref{suban2},
\[ \lambda\leq \lambda', \ \mu\leq \mu' \ \Longrightarrow
\|f\|_{\cZ^{\lambda,\mu}_\tau} \leq
\|f\|_{\cZ^{\lambda',\mu'}_\tau}.\]
Moreover, for $\tau,\bar \tau \in \R$, and any $p\in [1,\infty]$,
\begeq\label{moreoverff}
\|f\|_{\cZ^{\lambda,\mu;p}_{\ov{\tau}}} \leq \|f\|_{\cZ^{\lambda,\mu+\lambda |\ov{\tau}-\tau|;p}_\tau}.
\endeq
\end{Prop}

\begin{Rk} Note carefully that the spaces $\cZ^{\lambda,\mu}_\tau$ are
  {\bf not} ordered with respect to the parameter $\tau$, which cannot
  be thought of as a regularity index. We could dispend with this
  parameter if we were working in time $O(1)$; but \eqref{moreoverff}
  is of course of absolutely no use. This means that errors on the exponent $\tau$ should
  remain somehow small, in order to be controllable by small losses on
  the exponent $\mu$.
\end{Rk}

Finally we state an easy proposition which follows from the
time-invariance of the free transport equation:

\begin{Prop} \label{propTL} 
  For any $X\in \{\cC, \cF, \cZ\}$, and any
  $t,\tau\in\R$,
  \[ \|f\circ S_t^0\|_{X^{\lambda,\mu}_\tau} =
  \|f\|_{X^{\lambda,\mu}_{t+\tau}}.\]
\end{Prop}

Now we shall see that the hybrid norms, and certain variants thereof,
enjoy properties rather similar to those of the single-variable
analytic norms studied before. This will be sometimes technical, and the
reader who would like to reconnect to physical problems is advised to go directly
to Subsection \ref{submeasur}.

\subsection{Injections}

In this section we relate $\cZ^{\lambda,\mu;p}_{\tau}$ norms to more
standard norms entirely based on Fourier space.  In the next theorem
we write 
\begeq\label{Y} \|f\|_{\cY^{\lambda,\mu}_\tau} =
\|f\|_{\cF^{\lambda,\mu;\infty}_\tau} = \sup_{k\in\Z^d}\,
\sup_{\eta\in\R^d}\, e^{2\pi\mu|k|}\,e^{2\pi\lambda |\eta+k\tau|}\,
|\tilde{f}(k,\eta)|.
\endeq

\begin{Thm}[Injections between analytic spaces] \label{thminj}
(i) If $\lambda,\mu\geq 0$ and $\tau\in\R$ then
\begeq\label{fcYZ}\|f\|_{\cY^{\lambda,\mu}_\tau} \leq \|f\|_{\cZ^{\lambda,\mu;1}_\tau}.
\endeq
\sm

(ii) If $0<\lambda<\ov{\lambda}$, $0<\mu<\ov{\mu}\leq M$, $\tau\in\R$, then
\begeq\label{fcZY} \|f\|_{\cZ^{\lambda,\mu}_\tau} \leq \frac{C(d,\ov{\mu})}{(\ov{\lambda}-\lambda)^d\, (\ov{\mu}-\mu)^d}\, 
\|f\|_{\cY^{\ov{\lambda},\ov{\mu}}_\tau}.
\endeq
\sm

(iii) If $0< \lambda <\ov{\lambda} \leq \Lambda$, 
$0< \mu< \ov{\mu}\leq M$, $b\leq \beta\leq B$, then
there is $C=C(\Lambda,M,b,B,d)$ such that
\begin{multline*}
\label{fcZZZZ} 
\|f\|_{\cZ^{\lambda,\mu;1}_\tau}
\leq C^{\frac1{\min \{ \ov{\lambda}-\lambda \, ; \, \ov{\mu}-\mu \}}}\,
\Bigg( \|f\|_{\cY^{\ov{\lambda},\ov{\mu}}_\tau} + \\
\max \left\{ \left( \iint | f(x,v)|\, e^{\beta |v|}\,dv\,dx \right) \, ; \
\Bigl( \iint |f(x,v)|\,e^{\beta |v|}\,dv\,dx \Bigr)^2\right\}\Bigg).
\end{multline*}
\end{Thm}

\begin{Rk} The combination of (ii) and (iii), plus elementary Lebesgue interpolation, enables to control
all norms $\cZ^{\lambda,\mu;p}_\tau$, $1\leq p\leq\infty$.
\end{Rk}

\begin{proof}[Proof of Theorem \ref{thminj}]
By the invariance under the action of free transport, it is sufficient to do the proof for $\tau=0$.

By integration by parts in the Fourier transform formula, we have
\[
\tilde{f}(k,\eta) = \int\hat{f}(k,v)\,e^{-2i\pi \eta\cdot v}\,dv
= \int \nabla_v^m \hat{f}(k,v)\,\frac{e^{-2i\pi \eta\cdot v}}{(2i\pi \eta)^m}\,dv.
\]
So
\[ |\tilde{f}(k,\eta)| \leq
\frac1{(2\pi|\eta|)^m} \int |\nabla_v^m \hat{f}(k,v)|\,dv;\]
and therefore
\begin{align*}
e^{2\pi \mu |k|}\, e^{2\pi \lambda |\eta|}\, |\tilde{f}(k,\eta)|
& \leq e^{2\pi \mu|k|}\, \sum_n \frac{(2\pi \lambda)^n}{n!} |\eta|^n \,|\tilde{f}(k,\eta)|\\
& \leq e^{2\pi \mu|k|}\, \sum_n \frac{\lambda^n}{n!}\int |\nabla_v^n\tilde{f}(k,v)|\,dv.
\end{align*}
This establishes (i).
\sm

Next, by differentiating the identity
\[ \hat{f}(k,v) = \int \tilde{f}(k,\eta)\,e^{2i\pi \eta\cdot
  v}\,d\eta,\]
we get
\begeq\label{Dmhatf}
\nabla_v^m \hat{f}(k,v)= \int \tilde{f}(k,\eta)\,(2i\pi\eta)^m\,e^{2i\pi \eta\cdot v}\,d\eta.
\endeq
Then we deduce (ii) by writing
\begin{align*}
\sum_{k,m}
e^{2\pi\mu|k|}\, & \frac{\lambda^m}{m!}\, \|\nabla_v^m \hat{f}(k,v)\|_{L^\infty(dv)}\\
& \leq \sum_k e^{2\pi\mu|k|} \int e^{2\pi\lambda |\eta|} |\tilde{f}(k,\eta)|\,d\eta\\
& \leq \left(\sum_k e^{-2\pi (\ov{\mu}-\mu)|k|}\right)
\left(\int e^{-2\pi (\ov{\lambda}-\lambda)|\eta|}\,d\eta\right)
\left(\sup_{k,\eta} e^{2\pi\ov{\lambda}|\eta|}\, e^{2\pi\ov{\mu}|k|}\, |\tilde{f}(k,\eta)|\right).
\end{align*}

The proof of (iii) is the most tricky. We start again from \eqref{Dmhatf}, but now we integrate by
parts in the $\eta$ variable:
\begeq\label{Dmveta}
\nabla_v^m \hat{f}(k,v)
= (-1)^q \int \nabla_\eta^q \bigl[ \tilde{f}(k,\eta)\, (2i\pi\eta)^m\bigr]\,
\frac{e^{2i\pi\eta\cdot v}}{(2i\pi v)^q}\,dv,
\endeq
where $q=q(v)$ is a multi-index to be chosen.

We split $\R^d_v$ into $2^d$ disjoint regions $\Delta(i_1,\ldots,i_n)$, where the $i_j$
are distinct indices in $\{1,\ldots,d\}$:
\[ \Delta(I) = \Bigl\{ v\in\R^d; \ |v_i|\geq 1 \ \forall \, i\in
I,\quad |v_i|<1\ \forall \, i \notin I\Bigr\}.\] 
If $v\in \Delta(i_1,\ldots, i_n)$ we apply \eqref{Dmveta} with the multi-index
$q$ defined by $q_j=2$ if $j\in\{i_1,\ldots,i_n\}$, $q_j=0$
otherwise. This gives
\[ \int_{\Delta(i_1,\ldots,i_n)}
|\nabla_v^m \hat{f}(k,v)|\,dv
\leq \left(\frac{1}{(2\pi)^{2n}}\int_{\Delta(i_1,\ldots,i_n)}
\frac{dv_{i_1}\ldots dv_{i_n}}{|v_{i_1}|^2\ldots |v_{i_n}|^2}\right)\,
\sup_{k,\eta} \Bigl| \nabla_\eta^q \bigl[ \tilde{f}(k,\eta)\, (2i\pi \eta)^m\bigr]\Bigr|.\]

Summing up all pieces and using the Leibniz formula, we get
\[ \int |\nabla_v^m \hat{f}(k,v)|\,dv
\leq C(d) (1+m^{2d})\ \sup_{k,\eta}\
\sup_{|q|\leq 2d} |\nabla_\eta^q\tilde{f}(k,\eta)|\,|2\pi\eta|^{m-q}.\]

At this point we apply Lemma \ref{leminterp} below with 
\[ \var = \frac14\, \min \left\{ \frac{\ov{\lambda}-\lambda}{\ov{\lambda}} \, ;
    \,  \frac{\ov{\mu}-\mu}{\ov{\mu}}\right\},\]
and we get, for $q \le 2d$, 
\begin{multline*}
|\nabla_\eta^q \tilde{f}(k,\eta)|\leq
C(d)^{\max \left\{\frac{\ov{\lambda}}{\ov{\lambda}-\lambda} \, ; \, \frac{\ov{\mu}}{\ov{\mu}-\mu}\right\}}
\, K(b,B)\, e^{-2\pi \frac{\lambda + \bar \lambda}2 |\eta|}\\
\Bigl(\sup_\eta e^{2\pi \ov{\lambda} |\eta|} |\tilde{f}(k,\eta)|\Bigr)^{1-\var}\
\max \left\{ \left( \sup_{\ell,\eta} \frac{\beta^\ell\,
      \|\nabla_\eta^\ell \tilde{f}\|_\infty}{\ell!}\right)^\var \, ;
  \, \left( \sup_{\ell,\eta} 
   \frac{\beta^\ell\, \|\nabla_\eta^\ell \tilde{f}\|_\infty}{\ell!}\right)^{2\var}\right\}.
\end{multline*}

Of course,
\begin{multline*}
\frac{\beta^\ell \, |\nabla_\eta^\ell \tilde{f}(k,\eta)|}{\ell!} \leq
(2\pi\beta)^\ell \int_{\R^d} |\hat{f}(k,v)|\,\frac{|v|^\ell}{\ell!}\,dv \\
\leq \int_{\R^d} |f(x,v)|\, (2\pi\beta)^\ell\, \frac{|v|^\ell}{\ell!}\,dv
\leq \int_{\R^d} |f(x,v)|\,e^{2\pi\beta |v|}\,dv.
\end{multline*}
So, all in all,
\begin{multline*}
\sum_{k,m} e^{2\pi \mu|k|}\,\frac{\lambda^m}{m!} 
\int |\nabla_v^m \hat{f}(k,v)|\,dv \\
\leq \sum_{|q|\leq 2d} 
C(d,\Lambda,M,b,B)^{\frac1{\min\{\ov{\lambda}-\lambda \, ; \, \ov{\mu}-\mu\}}}\\
\sup_{\eta \in \R^d} \left( e^{-2\pi \frac{\lambda + \bar \lambda}2  |\eta|}
\, \sum_m \frac{\lambda^m (1+m)^{2d}\,|2\pi\eta|^{m-q}}{m!}\right)\,
\left(\sum_k e^{-2\pi (\ov{\mu}(1-\var)-\mu)|k|}\right)\\
\left(\sup_{k,\eta} e^{2\pi\ov{\mu}|k|}\,e^{2\pi\ov{\lambda}|\eta|}\, 
|\tilde{f}(k,\eta)|\right)^{1-\var}\
\max \left\{ \left(\int_{\R^d} |f(x,v)|\,e^{\beta |v|}\,dv \right)^\var \, ; \
\left(\int_{\R^d} |f(x,v)|\,e^{\beta |v|}\,dv \right)^{2\var}\right\}.
\end{multline*}
Since 
$$
\sum_m \frac{\lambda^m (1+m)^{2d}\,|2\pi\eta|^{m-q}}{m!}
\leq C(q,\Lambda)\,e^{2\pi\frac{\lambda+\ov{\lambda}}2 |\eta|}
$$
and
\[ \sum_k e^{-2\pi (\ov{\mu}(1-\var)-\mu)|k|}\leq \sum_k e^{-\pi
  (\ov{\mu}-\mu)|k|} \leq C/(\ov{\mu}-\mu)^{d},\] we easily end up
with the desired result.
\end{proof}

\begin{Lem} \label{leminterp}
Let $f:\R^d\to\C$, and let $\alpha>0$, $A\geq 1$, $q\in\N_0^d$. Let $\beta$ such that
$0<b\leq \beta\leq B$. If $|f(x)|\leq A\, e^{-\alpha |x|}$ for all $x$, then
for any $\var\in (0,1/4)$ one has
\begin{multline} \label{nablaqf}
 |\nabla^q f(x)| \leq C(q,d)^{\frac1{\var}}\, K(b,B)\, A^{1-\var} \, e^{- (1-2\var) \alpha|x|}\\
\sup_{r\in\N_0^d} \max \left\{ \left( \beta^r \, \frac{\|\nabla^r
      f\|_\infty}{r!}\right)^{\var}\, ; \, 
\left(\beta^r \, \frac{\|\nabla^r f\|_\infty}{r!}\right)^{2\var}\right\}.  
\end{multline}
\end{Lem}

\begin{Rk} \label{rkMiro} One may conjecture that the optimal constant in the right-hand side of \eqref{nablaqf} is in fact
  polynomial in $1/\var$; if this conjecture holds true, then the constants in Theorem
  \ref{thminj} (iii) can be improved accordingly.
  Mironescu communicated to us a derivation of polynomial bounds for
  the optimal constant in the related inequality
  $$\|f^{(k)}\|_{L^\infty(\R)} \leq
  C(k)\,\|f\|^{1/(k+2)}_{L^1(\R)}\|f^{(k+1)}\|^{(k+1)/(k+2)}_{L^\infty(\R)},$$
  based on a real interpolation method.
\end{Rk}

\begin{proof}[Proof of Lemma \ref{leminterp}]
Let us first see $f$ as a function of $x_1$, and treat
$x'=(x_2,\ldots,x_d)$ as a parameter.
Thus the assumption is $|f(x_1,x')|\leq (A\,e^{-\alpha |x'|})\,e^{-\alpha |x_1|}$.
By a more or less standard interpolation inequality \cite[Lemma A.1]{DV:FP:01},
\begeq\label{interp2}
|\pa_1 f(x_1,x')|\leq 2\sqrt{A\,e^{-\alpha |x'|}}\,\sqrt{e^{-\alpha
    |x_1|}}\, \|\pa_1^2 f(x_1,x')\|^{\frac12} _\infty
= 2 \sqrt{A\,e^{-\alpha |x|}}\sqrt{\|\pa_1^2 f\|_\infty}.
\endeq
Let $C_{q_1,r_1}$ be the optimal constant (not smaller than 1) such that
\begeq\label{Cq1r1}
|\pa_1^{q_1} f(x_1,x')| \leq C_{q_1,r_1}\, (A\,e^{-\alpha |x|})^{1-\frac{q_1}{r_1}}\,
\|\pa_r^{r_1} f(x_1,x')\|_\infty^{\frac{q_1}{r_1}}.
\endeq
By iterating \eqref{interp2}, we find $C_{q_1,r_1} \leq 2\sqrt{C_{q_1-1,r_1}\, C_{q_1+1,r_1}}$. It follows
by induction that
\[ C_{q,r} \leq 2^{q(r-q)}.\]

Next, using \eqref{Cq1r1} and interpolating according to the second variable $x_2$ as in \eqref{interp2},
we get
\begin{align*}
|\pa_2^{q_2} \pa_1^{q_1} f(x)|
& \leq C_{q_2,r_2}\, \Bigl( C_{q_1,r_1}\,(A\,e^{-\alpha |x|})^{1-\frac{q_1}{r_1}}\,
\|\pa_1^{r_1} f\|_\infty^{\frac{q_1}{r_1}}\Bigr)^{1-\frac{q_2}{r_2}}\,
\|\pa_2^{r_2} \pa_1^{q_1} f\|_\infty^{\frac{q_2}{r_2}}\\
& \leq C_{q_1,r_1}\,C_{q_2,r_2}\,
(A\,e^{-\alpha |x|})^{(1-\frac{q_1}{r_1})(1-\frac{q_2}{r_2})}\,
\|\pa_1^{r_1} f\|_\infty^{\frac{q_1}{r_1}\,(1-\frac{q_2}{r_2})}\,
\|\pa_2^{r_2}\pa_1^{q_1}f\|_\infty^{\frac{q_2}{r_2}}.
\end{align*}

We repeat this until we get
\begin{multline} \label{bignabla}
|\nabla^q f(x)| \leq
(C_{q_1,r_1}\, \ldots\, C_{q_d,r_d})\, (A\,e^{-\alpha |x|})^{(1-\frac{q_1}{r_1})\ldots (1-\frac{q_d}{r_d})}\\
\|\pa_1^{r_1} f\|_\infty^{\frac{q_1}{r_1}(1-\frac{q_2}{r_2})\ldots (1-\frac{q_d}{r_d})}\,
\|\pa_1^{q_1}\pa_2^{r_2}f\|^{\frac{q_2}{r_2}(1-\frac{q_3}{r_3})\ldots (1-\frac{q_d}{r_d})}\,\ldots
\|\pa_1^{q_1}\pa_2^{q_2}\ldots \pa_{d-1}^{q_{d-1}}\pa_d^{r_d} f\|^{\frac{q_d}{r_d}}.
\end{multline}
Choose $r_i$ ($1\leq i\leq d$) in such a way that
\[ \frac{\var}{d} \leq \frac{q_i}{r_i}\leq \frac{2\var}{d};\] this is
always possible for $\var<d/4$. Then $C_{q_i,r_i}\leq
(2^{dq_i^2})^{1/\var}$, and \eqref{bignabla} implies
\[ |\nabla^q f(x)| \leq (2^{d|q|^2})^{1/\var}\,(A\,e^{-\alpha
  |x|})^{1-\var}\, \max_{s\leq r+q} \left\{ \|\nabla^s
f\|_\infty^{\var} \, ; \, \|\nabla^s f\|_\infty^{2\var}\right\}.\] 
Then,
since $2(r+q)\var\leq 3 d q$ we have, by a crude application of
Stirling's formula (in quantitative form), for $s \le r+q$, 
\begin{align*}
\|\nabla^s f\|^{\var}_\infty & \leq \left(\frac{\beta^s\,\|\nabla^s f\|_\infty}{s!}\right)^{\var}\,
\left(\frac{s!}{\beta^s}\right)^\var \\
& \leq \left( \sup_n \frac{\beta^n\,\|\nabla^n f\|_\infty}{n!}\right)^\var\,
C(\beta,q,d)\,\var^{-3dq},
\end{align*}
and the result follows easily.
\end{proof}

\subsection{Algebra property in two variables}

In this section we only consider the norms $\cZ^{\lambda,\mu;p}_\tau$;
but similar results would hold true for the two-variables $\cC$ and
$\cF$ spaces, and could be proven with the same method as those used
for the one-variable spaces $\cF^\lambda$ and $\cC^\lambda$
respectively (note that the Leibniz formula still applies because
$\nabla_x$ and $(\nabla_v + \tau \nabla_x)$ commute).

\begin{Prop} \label{propalgZ} (i) For any $\lambda, \mu\geq 0$, $\tau
  \in\R$ and $p,q,r \in [1,+\infty]$ such that $1/p+1/q=1/r$, we have
\[ \| f \, g \|_{\cZ^{\lambda,\mu;r}_\tau} \le  \| f \|_{\cZ^{\lambda,\mu;p}_\tau} \, 
\| g \|_{\cZ^{\lambda,\mu;q}_\tau}. \]

(ii) As a consequence, $\cZ^{\lambda,\mu}_\tau = \cZ^{\lambda,\mu;\infty} _\tau$ 
is a normed algebra: 
\[ \|fg\|_{\cZ^{\lambda,\mu}_\tau} \leq \|f\|_{\cZ^{\lambda,\mu}_\tau}\,  \|g\|_{\cZ^{\lambda,\mu} _\tau}.\] 
In particular, $\|f^n\|_{\cZ^{\lambda,\mu}_\tau} \leq  \|f\|_{\cZ^{\lambda,\mu}_\tau} ^n$ for any
$n\in\N_0$, and $\|e^f\|_{\cZ^{\lambda,\mu}_\tau} \leq e^{\|f\|_{\cZ^{\lambda,\mu}_\tau}}$.
\end{Prop}

\begin{proof}[Proof of Proposition \ref{propalgZ}]
  First we note that (with the notation \eqref{trn})
  $\trn\cdot\trn_{\mu;r}$ satisfies the ``$(p,q,r)$ property'':
  whenever $p,q,r \in [1,+\infty]$ satisfy $1/p + 1/q = 1/r$, we have
  \begin{align*}
    \trn fg\trn_{\mu;r} & = \sum_{\ell\in\Z^d} e^{2\pi\mu|\ell|}\,
    \|\hat{fg}(\ell,\,\cdot\,)\|_{L^r(\R^d_v)}\\
    & = \sum_{\ell\in\Z^d} e^{2\pi\mu|\ell|}\,
    \left\| \sum_k \hat{f}(k,\,\cdot\,)\,\hat{g}(\ell-k,\,\cdot\,) 
    \right\|_{L^r(\R^d_v)}\\
    & \leq \sum_{\ell\in\Z^d} \sum_{k\in\Z^d}
    e^{2\pi\mu|k|}\,e^{2\pi\mu|\ell-k|}\,
    \|\hat{f}(k,\,\cdot\,)\|_{L^p(\R^d_v)}\,
    \|\hat{g}(\ell-k,\,\cdot\,)\|_{L^q(\R^d_v)}\\
    & = \trn f\trn_{\mu;p}\, \trn g\trn_{\mu;q}.
\end{align*}

Next, we write
\begin{align*}
\trn & fg \trn_{\cZ^{\lambda,\mu;r}_\tau}
= \sum_{n\in\N_0^d} \frac{\lambda^n}{n!}\, 
\trn (\nabla_v + \tau\nabla_x)^n (fg)\trn_{\mu;r}\\
& = \sum_n \frac{\lambda^n}{n!}\, \btrn \sum_{m\leq n} \Cnk{n}{m} 
(\nabla_v+\tau\nabla_x)^m f\, (\nabla_v + \tau\nabla_x)^{n-m} g\btrn_{\mu;r} \\
& \leq \sum_n \frac{\lambda^n}{n!} \sum_{m\leq n} \Cnk{n}{m}\,
\trn (\nabla_v + \tau\nabla_x)^m f \trn_{\mu;p}\,
\trn (\nabla_v + \tau\nabla_x)^{n-m} g \trn_{\mu;q} \\
& = \left( \sum_m \frac{\lambda^m}{m!}\, 
\trn (\nabla_v +\tau\nabla_x)^mf\trn_{\mu;p}\right)\,
\left( \sum_\ell \frac{\lambda^\ell}{\ell!}\, 
\trn (\nabla_v + \tau\nabla_x)^\ell f\trn_{\mu;q}\right)\\
& = \trn  f \trn_{\cZ^{\lambda,\mu;p}_\tau} \, \trn  g
\trn_{\cZ^{\lambda,\mu;q}_\tau}.
\end{align*}
(We could also reduce to $\tau=0$ by means of Proposition \ref{propTL}.)
\end{proof}

\subsection{Composition inequality}

\begin{Prop}[Composition inequality in two variables] 
\label{propcompos2}
For any $\lambda,\mu\geq 0$ and any $p\in [1,\infty]$, $\tau\in\R$,
$\sigma \in \R$, $a \in\R\setminus\{0\}$, $b \in \R$,
\begeq\label{compos2} \Bigl\| f\bigl(x+ bv + X(x,v), av +
V(x,v)\bigr)\Bigr\|_{\cZ^{\lambda,\mu;p}_\tau} \leq |a|^{-d/p} \,
\|f\|_{\cZ^{\alpha,\beta;p}_\sigma},
\endeq
where
\begeq\label{alphaV} 
\alpha = \lambda|a| + \|V\|_{\cZ^{\lambda,\mu}_\tau},\qquad
\beta = \mu + \lambda\,|b+\tau-a \sigma| + 
\|X-\sigma V\|_{\cZ^{\lambda,\mu}_\tau}.
\endeq
\end{Prop}

\begin{Rk} The norms in \eqref{alphaV} for $X$ and $V$ have to be
  based on $L^\infty$, not just any $L^p$.  Also note: the fact that
  the second argument of $f$ has the form $av + V$ (and not $av + cx +
  V$) is related to Remark \ref{rkcontrol}.
\end{Rk}

\begin{proof}[Proof of Proposition \ref{propcompos2}]
  The proof is a combination of the arguments in Proposition
  \ref{propcompos}. In a first step, we do it for the case $\tau
  =\sigma=0$, and we write $\|\cdot\|_{\lambda,\mu;p} =
  \|\cdot\|_{\cZ^{\lambda,\mu;p}_0}$.

  From the expansion $f(x,v) = \sum \hat{f}(k,v)\,e^{2i\pi k\cdot x}$
  we deduce
\begin{align*}
h(x,v)& := f\Bigl( x+bv + X(x,v),\, av + V(x,v)\Bigr)\\
& = \sum_k \hat{f}(k,av + V)\, e^{2i\pi k\cdot (x+bv + X)}\\
& = \sum_k \sum_m \nabla_v^m \hat{f}(k,av)\cdot \frac{V^m}{m!} \,
e^{2i\pi k\cdot x}\, e^{2i\pi k\cdot bv}\, e^{2i\pi k\cdot X}.
\end{align*}
Taking the Fourier transform in $x$, we see that for any $\ell\in\Z^d$,
\[ \hat{h}(\ell,v) = \sum_k \sum_m \nabla_v^m \hat{f}(k,av)\, e^{2i\pi
  k\cdot bv} \, \sum_j \frac{(V^m)^{\hat{}}(j)}{m!}\, (e^{2i\pi k\cdot
  X})^{\hat{}}(\ell-k-j).\] 
Differentiating $n$ times {\it via} the Leibniz
formula (here applied to a product of four functions), we get
\begin{multline*}
\nabla_v^n \hat{h}(\ell,v) = 
\sum_{k,m,j}\ \sum_{n_1+n_2+n_3+ n_4=n} \frac{n!\, a^{n_1}}{n_1!\,n_2!\,
  n_3!\,n_4!} 
\nabla_v^{m+n_1}\hat{f}(k,av)\\
\frac{\nabla_v^{n_2} (V^m)^{\hat{}}(j)}{m!}\,
\nabla_v^{n_3} \bigl(e^{2i\pi k\cdot X}\bigr)^{\hat{}}(\ell-k-j,v)\,
(2i\pi b k)^{n_4}\,e^{2i\pi k\cdot b v}.
\end{multline*}

Multiplying by $\lambda^n\,e^{2\pi\mu|\ell|}/n!$ and summing over $n$
and $\ell$, taking $L^{p}$ norms and using $\|fg\|_{L^{p}} \leq
\|f\|_{L^{p}} \|g\|_{L^\infty}$, we finally obtain
\begin{align*}
& \|h\|_{\lambda,\mu} \leq
|a|^{-d/p} \, \sum_{k,j,\ell \in \Z^d _0;\ m,n, \, n_1+n_2+n_3+n_4=n \ge
0} \frac{\lambda^n\, e^{2\pi\mu|\ell|} |a|^{n_1}}
{n_1!\, n_2!\, n_3!\, n_4!}\, \bigl\|\nabla_v^{m+n_1} \hat{f}(k,\cdot)\bigr\|_{L^{p}}\,
\left\| \frac{\nabla_v^{n_2}(V^m)^{\hat{}}(j)}{m!}\right\|_\infty\,\\
& \qqquad\qqquad\qquad
\Bigl\| \nabla_v^{n_3} (e^{2i\pi k\cdot X})^{\hat{}}(\ell-k-j)\Bigr\|_\infty\,
(2\pi |b|\,|k|)^{n_4}\\[2mm]
& = |a|^{-d/p} \, \sum_{k,j,\ell \in \Z^d _0, \ m,n_1,n_2,n_3,n_4 \ge 0}
\frac{\lambda^{n_1+n_2+n_3+n_4}\,e^{2\pi\mu|k|}\,e^{2\pi\mu |j|}\, 
e^{2\pi\mu|\ell-k-j|}\,|a|^{n_1}}{n_1!\, n_2!\, n_3!\, n_4!}\,
\|\nabla_v^{m+n_1}\hat{f}(k,\cdot)\|_{L^{p}}\\
&\qqquad\qquad\qquad
 \left\|
  \frac{\nabla_v^{n_2} (V^m)^{\hat{}}(j)}{m!}\right\|_\infty\,
\Bigl\| \nabla_v^{n_3} (e^{2i\pi k\cdot X})^{\hat{}}(\ell-k-j)\Bigr\|_\infty\,
(2\pi |b|\,|k|)^{n_4}\\[2mm]
& \leq |a|^{-d/p} \, \sum_{k,n_1,m}
\frac{\lambda^{n_1} |a|^{n_1}}{n_1!}\, \bigl\|\nabla_v^{n_1+m} 
\hat{f}(k,\cdot)\bigr\|_{L^{p}}\,e^{2\pi\mu|k|}\,
\left(\frac1{m!} \sum_{n_2,j} \frac{\lambda^{n_2}}{n_2!}\, e^{2\pi\mu|j|}\,
\|\nabla_v^{n_2} (V^m)^{\hat{}}(j)\|_\infty\right)\,\\
&\qqquad\qquad\qquad
\left( \sum_{n_3,h}
  \frac{\lambda^{n_3}}{n_3!}\,e^{2\pi\mu|h|}\,\bigl\|\nabla_v^{n_3} 
(e^{2i\pi k\cdot X})^{\hat{}}(h)
\bigr\|_\infty \right)\,
\left( \sum_{n_4} \frac{(2\pi\lambda |b||k|)^{n_4}}{n_4!}\right)\\[2mm]
& = |a|^{-d/p} \, \sum_{k,p,m} \frac{(\lambda |a|)^{n_1}}{n_1!}\,e^{2\pi \mu|k|}\,
\bigl\|\nabla_v^{n_1+m}\hat{f}(k,\cdot)\|_{L^{p}}\,
\frac{\|V^m\|_{\lambda,\mu}}{m!}\, \bigl\|e^{2i\pi k\cdot X}\|_{\lambda,\mu}\,e^{2\pi\lambda |b||k|}\\
& \leq |a|^{-d/p} \,\sum_{k,n_1,m} \frac{(\lambda |a|)^{n_1}}{n_1!}\,
e^{2\pi (\mu+\lambda |b|)|k|}\, \bigl\| \nabla_v^{n_1+m}\hat{f}(k,\cdot)\bigr\|_{L^{p}}\,
\frac{\|V\|^m_{\lambda,\mu}}{m!}\,e^{2\pi|k|\,\|X\|_{\lambda,\mu}}\\
& = |a|^{-d/p} \, \sum_{k,n}\frac1{n!} \bigl(\lambda |a| + \|V\|_{\lambda,\mu}\bigr)^n\,
\bigl\|\nabla_v^n \hat{f}(k,\cdot)\bigr\|_{L^{\ov{p}}}\, 
e^{2\pi |k|(\mu + \lambda |b| + \|X\|_{\lambda,\mu})}\\
& = |a|^{-d/p} \,\|f\|_{\lambda |a| + \|V\|_{\lambda,\mu},\, \mu + \lambda |b| 
+ \|X\|_{\lambda,\mu}}.
\end{align*}
\med

Now we generalize this to arbitrary values of $\sigma$ and $\tau$: by
Proposition \ref{propTL},
\begin{align*}
\Bigl\| & f\Bigl( x+ bv + X(x,v),\, a v + V(x,v) 
         \Bigr)\Bigr\|_{\cZ^{\lambda,\mu;p}_\tau}\\
& = \Bigl\|  f\Bigl( x+ v (b+\tau) + X(x+v\tau,v),
\, a v + V(x+v\tau,v) \Bigr)\Bigr\|_{\cZ^{\lambda,\mu;p}}\\
& = \Bigl\| f\circ S^0_\sigma \circ S^0_{-\sigma} 
\Bigl( x+ v (b+\tau) + X(x+v\tau,v),
\, a v + V(x+v\tau,v) \Bigr)\Bigr\|_{\cZ^{\lambda,\mu;p}}\\
& = \Bigl\| (f\circ S^0_\sigma) \Bigl( x + v (b+\tau-a \sigma)
+ (X-\sigma V)(x+v\tau,v),
\, a v + V(x+v\tau, v)\Bigr)\Bigr\|_{\cZ^{\lambda,\mu;p}}\\
& = \Bigl\| (f\circ S^0_\sigma)\Bigl( x + v (b+\tau-a \sigma) + Y(x,v),\,
av + W(x,v)\Bigr)\Bigr\|_{\cZ^{\lambda,\mu;p}},
\end{align*}
where
\[ W(x,v) = V\circ S^0_\tau(x,v),\qquad
Y(x,v) = (X-\sigma V)\circ S^0_\tau(x,v).\]
Applying the result for $\tau=0$, we deduce that the norm of $h(x,v) = 
f( x+ bv + X(x,v),\, a v + V(x,v) )$ in $\cZ^{\lambda,\mu}_\tau$ is bounded by
\[ \|f\circ S^0_\sigma\|_{\cZ^{\alpha,\beta;p}} = \|f\|_{\cZ^{\alpha,\beta;p}_\sigma},\]
where
\[ \alpha = \lambda|a| + \|V\circ S^0_\tau\|_{\cZ^{\lambda,\mu}}
= |a|\lambda + \|V\|_{\cZ^{\lambda,\mu}_\tau},\]
and
\[ \beta = \mu + \lambda |b+\tau-a \sigma|
+ \|(X-\sigma V)\circ S^0_\tau\|_{\cZ^{\lambda,\mu}}
= \mu + \lambda |b+\tau-a\sigma| + \|X-\sigma V\|_{\cZ^{\lambda,\mu}_\tau}.\]
This establishes the desired bound.
\end{proof}

\subsection{Gradient inequality}

In the next proposition we shall write
\begeq\label{Z*}
\|f\|_{\dot{\cZ}^{\lambda,\mu}_\tau} =
\sum_{\ell\in\Z^d\setminus\{0\}} \sum_{n\in\N_0^d} \frac{\lambda^n}{n!}\, e^{2\pi\mu|\ell|}\,
\Bigl\| (\nabla_v + 2i\pi\tau \ell)^n \hat{f}(\ell,v)\Bigr\|_{L^\infty(\R^d_v)}.
\endeq
This is again a homogeneous (in the $x$ variable) seminorm.

\begin{Prop} \label{propgrad2} 
For $\ov{\lambda}>\lambda \ge 0$,
  $\ov{\mu}>\mu \ge 0$, we have the functional inequalities
\[ \|\nabla_x f\|_{\cZ^{\lambda,\mu;p}_\tau}\leq
\frac{C(d)}{(\ov{\mu}-\mu)}\,\|f\|_{\dot{\cZ}^{\lambda,\ov{\mu};p}_\tau};\]
\[ \bigl\|(\nabla_v +\tau\nabla_x)f \bigr\|_{\cZ^{\lambda,\mu;p}_\tau}
\leq \frac{C(d)}{\lambda\,\log(\ov{\lambda}/\lambda)}\,\|f\|_{\cZ^{\ov{\lambda},\mu;p}_\tau}.\]
In particular, for $\tau\geq 0$ we have
\[ \|\nabla_v f\|_{\cZ^{\lambda,\mu;p}_\tau}
\leq C(d) \left[ \left(\frac1{\lambda\,\log(\ov{\lambda}/\lambda)}\right)
\|f\|_{\cZ^{\ov{\lambda},\ov{\mu};p}_\tau}
+\left(\frac{\tau}{(\ov{\mu}-\mu)}\right)\, \|f\|_{\dot{\cZ}^{\ov{\lambda},\ov{\mu};p}_\tau}\right].\]
\end{Prop}

The proof is similar to the proof of Proposition \ref{propgrad}; the
constant $C(d)$ arises in the choice of norm on $\R^d$. 
As a consequence, if $1<\ov{\lambda}/\lambda \leq 2$, we have e.g. the bound
\[ 
\|\nabla  f\|_{\cZ^{\lambda,\mu;p}_\tau} \leq 
C(d) \left(\frac1{\ov{\lambda}-\lambda}
+ \frac{1+\tau}{\ov{\mu}-\mu}\right)\, 
\|f\|_{\cZ^{\ov{\lambda},\ov{\mu};p}_\tau}.
\]

\subsection{Inversion} \label{sub:inversion}

From the composition inequality follows an inversion estimate.

\begin{Prop}[Inversion inequality] \label{propinv}
(i) Let $\lambda,\mu\ge0$, $\tau\in \R$, and $F:\T^d\times\R^d\to\T^d\times\R^d$.
Then there is $\var=\var(d)$ such that if $F$ satisfies
\[ \|\nabla (F-\Id)\|_{\cZ^{\lambda',\mu'}_\tau}\leq \var(d),\]
where
\[ \lambda' = \lambda + 2 \|F-\Id\|_{\cZ^{\lambda,\mu}_\tau},\qquad
\mu' = \mu + 2 (1+|\tau|)\, \|F-\Id\|_{\cZ^{\lambda,\mu}_\tau},\]
then $F$ is invertible and
\begeq\label{F-1}
\|F^{-1}-\Id\|_{\cZ^{\lambda,\mu}_\tau} \leq 2\, \|F-\Id\|_{\cZ^{\lambda,\mu}_\tau}.
\endeq
\sm

(ii) More generally, if $F$ and $G$ are functions $\T^d\times\R^d\to\T^d\times\R^d$ such that
\begeq\label{nFI} \|\nabla(F-\Id)\|_{\cZ^{\lambda',\mu'}_\tau}\leq \var(d),
\endeq
where
\[ \lambda' = \lambda + 2 \|F-G\|_{\cZ^{\lambda,\mu}_\tau},\qquad
\mu' = \mu + 2 (1+|\tau|)\, \|F-G\|_{\cZ^{\lambda,\mu}_\tau},\]
then $F$ is invertible and
\begeq\label{FG-1}
\|F^{-1}\circ G-\Id\|_{\cZ^{\lambda,\mu}_\tau}
\leq 2\, \|F-G\|_{\cZ^{\lambda,\mu}_\tau}.
\endeq
\end{Prop}

\begin{Rk} The conditions become very stringent
  as $\tau$ becomes large: basically, $F-\Id$ (or $F-G$ in case
  (ii)) should be of order $o(1/\tau)$ for Proposition \ref{propinv}
  to be applicable.
\end{Rk}

\begin{Rk} By Proposition \ref{propgrad2}, a sufficient condition for
  \eqref{nFI} to hold is that there be $\lambda'',\mu''$ such that
  $\lambda\leq \lambda''\leq 2\lambda$, $\mu\leq \mu''$, and
  \[ \|F-\Id\|_{\cZ^{\lambda'',\mu''}_\tau} \leq
  \frac{\var'(d)}{1+\tau}\, \min\{ \lambda''-\lambda' \, ; \,  \mu''-\mu'\}.\]
  However, this condition is in practice hard to fulfill.
\end{Rk}

\begin{proof}[Proof of Proposition \ref{propinv}]
  We prove only (ii), of which (i) is a particular case. Let
  $f=F-\Id$, $h=F^{-1} \circ G -\Id$, $g=G-\Id$, so that $\Id+g = (\Id+f)\circ
  (\Id+h)$, or equivalently
  \[ h = g - f \circ (\Id+h).\] So $h$ is a fixed point of
  \[ \Phi: Z\longmapsto g - f\circ (\Id+Z).\] Note that
  $\Phi(0)=g-f$. If $\Phi$ is $(1/2)$-Lipschitz on the ball
  $B(0,2\|f-g\|)$ in ${\cZ^{\lambda,\mu}_\tau}$, then \eqref{FG-1}
  will follow by fixed point iteration as in Theorem \ref{thmfpt}.

So let $Z,\tilde{Z}$ be given with
\[ \|Z\|_{\cZ^{\lambda,\mu}_\tau}, \|\tilde{Z}\|_{\cZ^{\lambda,\mu}_\tau}\leq 2 \|f-g\|_{\cZ^{\lambda,\mu}_\tau}.\]
We have
\begin{align*}
\Phi(Z)-\Phi(\tilde{Z})
& = f(\Id + \tilde{Z}) - f(\Id+Z)\\
& = \left(\int_0^1 \nabla f \Bigl( \Id + (1-\theta)Z + \theta\tilde{Z}\Bigr)\,d\theta\right)\cdot (\tilde{Z}-Z).
\end{align*}
By Proposition \ref{propalgZ},
\[ \|\Phi(Z)-\Phi(\tilde{Z})\|_{\cZ^{\lambda,\mu}_\tau}\leq
\left(\int_0^1 \Bigl\| \nabla f \bigl(\Id + (1-\theta) Z + \theta\tilde{Z}\bigr)\Bigr\|_{\cZ^{\lambda,\mu}_\tau}\,d\theta\right)\
\|\tilde{Z}-Z\|_{\cZ^{\lambda,\mu}_\tau}.\]
For any $\theta\in [0,1]$, by Proposition \ref{propcompos2},
\[ \Bigl\| \nabla f \Bigl( \Id+(1-\theta) Z + \theta\tilde{Z}\Bigr)\Bigr\|_{\cZ^{\lambda,\mu}_\tau}
\leq \|\nabla f\|_{\cZ^{\hat{\lambda},\hat{\mu}}_\tau},\]
where
\[ \hat{\lambda} = \lambda + \max \{\|Z\| \, ; \, \|\tilde{Z}\|\} 
\leq \lambda + 2 \|f-g\|_{\cZ^{\lambda,\mu}_\tau}\]
and (writing $Z=(Z_x,Z_v)$, $\tilde{Z}=(\tilde{Z}_x,\tilde{Z}_v)$)
\[ \hat{\mu} = \mu + \max \bigl\{ \|Z_x - \tau Z_v\| \, ; \,  
\|\tilde{Z}_x-\tau\tilde{Z}_v\|\bigr\}
\leq \mu + 2 (1+|\tau|)\, \|f-g\|_{\cZ^{\lambda,\mu}_\tau}.\]
If $F$ and $G$ satisfy the assumptions of Proposition \ref{propinv}, we deduce that
\[ \|\Phi\|_{\Lip(B(0,2))} \leq C(d)\,\var(d),\]
and this is bounded above by $1/2$ if $\var(d)$ is small enough.
\end{proof}

\subsection{Sobolev corrections}

We shall need to quantify Sobolev regularity corrections to the analytic regularity, in the $x$ variable.

\begin{Def}[Hybrid analytic norms with Sobolev corrections]
For $\lambda,\mu,\gamma\geq 0$, $\tau\in\R$, $p\in [1,\infty]$, we define
\[ \|f\|_{\cZ^{\lambda, (\mu,\gamma);p}_\tau}
= \sum_{\ell\in\Z^d} \sum_{n\in\N^d_0}\,
\frac{\lambda^n}{n!} e^{2\pi \mu |\ell|}\, (1+|\ell|)^\gamma\,
\bigl\|(\nabla_v + 2i\pi\tau\ell)^n\, \hat{f}(\ell,v)\bigr\|_{L^p(\R^d_v)};\]
\[ \|f\|_{\cF^{\lambda,\gamma}} = \sum_{k\in\Z^d} e^{2\pi\lambda |k|}\,(1+|k|)^\gamma\, |\hat{f}(k)|.\]
\end{Def}

\begin{Prop} \label{prophybrsob}
Let $\lambda,\mu, \gamma \ge 0$, $\tau \in \R$ and $p \in [1,+\infty]$. 
We have the following functional inequalities:
\sm

(i) $\dps \|f\|_{\cZ^{\lambda, (\mu,\gamma);p}_{t+\tau}}
= \|f\circ S^0_t\|_{\cZ^{\lambda, (\mu,\gamma);p}_\tau}$;
\sm

(ii) $\dps 1/p+1/q = 1/r \ \Longrightarrow \
\|fg\|_{\cZ^{\lambda, (\mu,\gamma);r}_\tau}
\leq \|f\|_{\cZ^{\lambda, (\mu,\gamma);p}_\tau}\, \|g\|_{\cZ^{\lambda,
    (\mu,\gamma);q}_\tau}$ and therefore
in particular $\cZ^{\lambda, (\mu,\gamma)}_\tau = \cZ^{\lambda, (\mu,\gamma);\infty}_\tau$ 
is a normed algebra;
\sm

(iii) If $f$ depends only on $x$ then $\dps \|f\|_{\cZ^{\lambda, (\mu,\gamma)}_\tau} =
\|f\|_{\cF^{\lambda |\tau| + \mu,\gamma}}$;
\sm

(iv) $\dps \|f\|_{\cZ^{\lambda, (\mu,\gamma);p}_{\ov{\tau}}} \leq
\|f\|_{\cZ^{\lambda, (\mu+\lambda |\tau-\ov{\tau}|,\gamma);p}_\tau}$;
\sm

(v) for any $\sigma\in\R$, $a \in \R \setminus \{0\}$, $b\in\R$, $p\in
[1,\infty]$,
\[ \Bigl\| f\Bigl( x + bv + X(x,v),\ av + V(x,v)
\Bigr)\Bigr\|_{\cZ^{\lambda, (\mu,\gamma);p}_\tau} \leq |a|^{-d/p}\,
\|f\|_{\cZ^{\alpha,(\beta,\gamma);p}_\sigma},\] where $\alpha =
\lambda |a| + \|V\|_{\cZ^{\lambda, (\mu,\gamma)}_\tau}$ and $\beta =
\mu + \lambda |b + \tau - a\sigma| +\|X -\sigma V\|_{\cZ^{\lambda,
    (\mu,\gamma)}_\tau}$.  
\sm

(vi) Gradient inequality:
\[ \|\nabla_x f\|_{\cZ^{\lambda,(\mu,\gamma);p}_\tau}
\leq \frac{C(d)}{\ov{\mu}-\mu}\, \|f\|_{\cZ^{\lambda,(\ov{\mu},\gamma);p}_\tau},\]
\[ \|\nabla f\|_{\cZ^{\lambda,(\mu,\gamma);p}_\tau}
\leq C(d)\,\left(\frac1{\ov{\lambda}-\lambda} +
  \frac{1+\tau}{\ov{\mu}-\mu}\right)\,
\|f\|_{\cZ^{\ov{\lambda},(\ov{\mu},\gamma);p}_\tau}.\]
\sm

(vii) Inversion: If $F$ and $G$ are functions $\T^d\times\R^d\to\T^d\times\R^d$ such that
\[ \|\nabla (F-\Id)\|_{\cZ^{\lambda',(\mu',\gamma)}_\tau}\leq \var(d),\]
where
\[ \lambda' = \lambda + 2 \|F-G\|_{\cZ^{\lambda,(\mu,\gamma)}_\tau},\qquad
\mu' = \mu + 2 (1+|\tau|)\, \|F-G\|_{\cZ^{\lambda,(\mu,\gamma)}_\tau},\]
then
\begeq\label{FG-12}
\|F^{-1}\circ G-\Id\|_{\cZ^{\lambda,(\mu,\gamma)}_\tau}
\leq 2\, \|F-G\|_{\cZ^{\lambda,(\mu,\gamma)}_\tau}.
\endeq
\end{Prop}

\begin{proof}[Proof of Proposition \ref{prophybrsob}]
  The proofs are the same as for the ``plain'' hybrid norms; the only
  notable point is that for the proof of (ii) we use, in addition
  to $e^{2\pi\lambda |k|} \leq e^{2\pi\lambda |k-\ell|}\,
  e^{2\pi\lambda |\ell|}$, the inequality
  \[ (1+|k|)^\gamma \leq (1+|k-\ell|)^\gamma\, (1+|\ell|)^\gamma.\]
\end{proof}

\begin{Rk} Of course, some of the estimates in Proposition
  \ref{prophybrsob} can be ``improved'' by taking advantage of
  $\gamma$; e.g. for $\gamma\geq 1$ we have
  \[ \|\nabla_x f\|_{\cZ^{\lambda,\mu;p}_\tau}\leq C(d)\,
  \|f\|_{\cZ^{\lambda,(\mu,\gamma);p}_\tau}.\]
\end{Rk}

\subsection{Individual mode estimates}

To handle very singular cases, we shall at times need to estimate Fourier modes individually,
rather than full norms. If $f=f(x,v)$, we write
\begeq\label{Pkf}
(P_kf)(x,v) = \hat{f}(k,v)\,e^{2i\pi k\cdot x}.
\endeq
In particular the following estimates will be useful.

\begin{Prop}\label{prop424} 
  For any $\lambda,\mu\geq 0$, $\tau\in\R$, Lebesgue exponents $1/r =
  1/p + 1/q$ and $k\in\Z^d$, we have the estimate
\[ \|P_k(fg)\|_{\cZ^{\lambda,\mu;r}_\tau} \leq \sum_{\ell\in\Z^d}
\|P_\ell f\|_{\cZ^{\lambda,\mu;p}_\tau}\,
\|P_{k-\ell}g\|_{\cZ^{\lambda,\mu;q}_\tau}.\]
\end{Prop}

\begin{Prop} \label{prop425} For any $\lambda>0$, $\ov{\mu}\ge\mu\geq
  0$, $\tau\in\R$, $p \in [1,\infty]$ and $k\in\Z^d$, we have the
  estimate
  \[ \Bigl\|P_k\Bigl[
  f\bigl(x+X(x,v),v\bigr)\Bigr]\Bigr\|_{\cZ^{\lambda,\mu;p}_\tau} \leq
  \sum_{\ell\in\Z^d} e^{-2\pi (\ov{\mu}-\mu)|k-\ell|}\,\|P_\ell
  f\|_{\cZ^{\lambda,\nu;p}_\tau}, \qquad \nu = \mu +
  \|X\|_{\cZ^{\lambda,\ov{\mu}}_\tau}. \]
\end{Prop}

These estimates also have variants with Sobolev corrections.  Note
that when $\mu=\ov{\mu}$, Proposition \ref{prop425} is a direct
consequence of Proposition~\ref{propcompos2} with $V=0$, $b=0$ and
$a=1$:
\[
\Bigl\|P_k\Bigl[
f\bigl(x+X(x,v),v\bigr)\Bigr]\Bigr\|_{\cZ^{\lambda,\mu;p}_\tau} \leq 
\Bigl\| f\bigl(x+X(x,v),v\bigr)\Bigr\|_{\cZ^{\lambda,\mu;p}_\tau} \le
\| f\|_{\cZ^{\lambda,\nu;p}_\tau}, \qquad \nu = \mu +
\|X\|_{\cZ^{\lambda,\mu}_\tau}.
\]

\begin{proof}[Proof of Propositions~\ref{prop424} and \ref{prop425}]
  The proof of Proposition \ref{prop424} is quite similar to the proof
  of Proposition \ref{propalgZ} (It is no restriction to choose $\tau=0$
because $P_k$ commutes with the free transport semigroup.)
  Proposition \ref{prop425} needs a few words of explanation. As in
  the proof of Proposition \ref{propcompos2} we let $h(x,v) =
  f(x+X(x,v),v)$, and readily obtain
  \begin{align*} \|P_kh\|_{\cZ^{\lambda,\mu;p}_\tau} & =
    \sum_{n\in\N_0^d} \frac{\lambda^n\,e^{2\pi\mu|k|}}{n!}\,
    \|\nabla_v^n \hat{h}(k,v)\|_{L^p(dv)}\\
    & \leq \sum_{n \geq 0} \sum_{\ell\in\Z^d}
    \frac{\lambda^{n}\,e^{2\pi\mu|k|}}{n!}  \|\nabla_v^{n}
    \hat{f}(\ell,v)\|_{L^p(dv)}\, \left( \sum_{m \ge 0}
      \frac{\lambda^{m}}{m!} \, \left\|\nabla_v^{m} \left( e^{2i \pi
          \ell \, X} \right)^{\hat{}} (k-\ell,v)\right\|_{L^\infty(dv)}
    \right).
\end{align*}
At this stage we write
\[ e^{2\pi\mu|k|}\leq
e^{2\pi\mu|\ell|}\,e^{-2\pi(\ov{\mu}-\mu)|k-\ell|}\,e^{2\pi\ov{\mu}|k-\ell|},\]
and use the crude bound
\[\forall \, \ell\in\Z^d,\quad e^{2\pi \ov{\mu}|k-\ell|}
\left\|\nabla_v^{m} \left( e^{2i \pi
          \ell \, X} \right)^{\hat{}}
      (k-\ell,v)\right\|_{L^\infty(dv)} \leq \sum_{j\in\Z^d}
e^{2\pi\ov{\mu}|j|} \, 
\left\|\nabla_v^{m} \left( e^{2i \pi
          \ell \, X} \right)^{\hat{}} (j,v)\right\|_{L^\infty(dv)}.\] 
The rest of the proof is as in Proposition \ref{propcompos2}.
\end{proof}

\subsection{Measuring solutions of kinetic equations in large
  time} \label{submeasur}

As we already discussed, even for the simplest kinetic equation,
namely free transport, we cannot hope to have uniform in time
regularity estimates in the velocity variable: rather, because of
filamentation, we may have $\|\nabla_v f(t,\cdot)\| = O(t)$,
$\|\nabla^2_v f(t,\cdot)\|=O(t^2)$, etc.  For analytic norms we may at
best hope for an exponential growth.

But the invariance of the ``gliding'' norms ${\cZ^{\lambda,\mu}_\tau}$
under free transport (Proposition \ref{propTL}) makes it possible to
look for uniform estimates such as \begeq\label{uniff}
\|f(\tau,\cdot)\|_{\cZ^{\lambda,\mu}_\tau} = O(1) \quad \text{as
  $\tau\to+\infty$}. \endeq

Of course, by Proposition \ref{propgrad2}, \eqref{uniff} implies
\begeq\label{nablavftau} \|\nabla_v
f(\tau,\cdot)\|_{\cZ^{\lambda',\mu'}_\tau} = O(\tau), \qquad 
\lambda' <\lambda, \ \mu' < \mu,
\endeq 
and nothing better as far as the asymptotic behavior of $\nabla_vf$ is
concerned; but \eqref{uniff} is much more precise than
\eqref{nablavftau}. For instance it implies $\|(\nabla_v +
\tau\nabla_x) f(\tau,\cdot)\|_{\cZ^{\lambda',\mu'}_\tau} = O(1)$ for
$\lambda' < \lambda$, $\mu' < \mu$.

Another way to get rid of filamentation is to {\em average} over the
spatial variable $x$, a common sense procedure which has already been
used in physics \cite[Section~49]{LL:kin:81}.  Think that, if $f$
evolves according to free transport, or even according to the
linearized Vlasov equation \eqref{Vllin}, then its space-average
\begeq\label{fbar} \<f\> (\tau,v) := \int_{\T^d} f(\tau,x,v)\,dx
\endeq
is time-invariant. (We used this infinite number of conservation laws
to determine the long-time behavior in Theorem
\ref{thmlineardamping}.)

The bound \eqref{uniff} easily implies a bound on the space average: indeed,
\begeq\label{unifavf} \|\<f\>(\tau,\cdot)\|_{\cC^{\lambda}} = 
\|\<f\>(\tau,\cdot)\|_{\cZ^{\lambda,\mu}_\tau} \leq \|f(\tau,\cdot)\|_{\cZ^{\lambda,\mu}_\tau} =
O(1)\qquad \text{as $\tau\to\infty$}; \endeq
and in particular, for $\lambda'<\lambda$,
\begeq\label{unifavfg} \|\<\nabla_v f\>(\tau,\cdot)\|_{\cC^{\lambda'}} = O(1)\qquad \text{as $\tau\to\infty$}. \endeq
Again, \eqref{uniff} contains a lot more information than \eqref{unifavfg}.

\begin{Rk} \label{rkbourgain} The idea to estimate solutions of a
  nonlinear equation by comparison to some unperturbed (reversible)
  linear dynamics is already present in the definition of Bourgain
  spaces $X^{s,b}$ \cite{bourgain}. The analogy stops here, since time
  is a dummy variable in $X^{s,b}$ spaces, while in
  $\cZ^{\lambda,\mu}_t$ spaces it is frozen and appears as a
  parameter, on which we shall play later.
\end{Rk}

\subsection{Linear damping revisited} \label{sub:revisited}

As a simple illustration of the functional analysis introduced in this section, let us recast
the linear damping (Theorem \ref{thmlineardamping}) in this language. This will be the first step
for the study of the nonlinear damping. For simplicity we set $L=1$.

\begin{Thm}[Linear Landau damping again] \label{thmlindampingagain}
Let $f^0=f^0(v)$, $W:\T^d\to\R$ such that $\|\nabla W\|_{L^1}\leq C_W$, and $f_i(x,v)$ such that
\sm

(i) Condition {\bf (L)} from Subsection \ref{sub:lineardamping} holds for some constants $C_0,\lambda,\kappa>0$;
\sm

(ii) $\|f^0\|_{\cC^{\lambda;1}}\leq C_0$;
\sm

(iii) $\|f_i\|_{\cZ^{\lambda,\mu;1}} \leq\delta$ for some $\mu>0$, $\delta>0$;
\sm

\noindent Then for any $\lambda'<\lambda$ and $\mu'< \mu$, the solution of the linearized Vlasov equation \eqref{Vllin} satisfies
\begeq\label{Vllinsat}
\sup_{t\in\R}\ \bigl\|f(t,\,\cdot\,)\bigr\|_{\cZ^{\lambda',\mu';1}_t} \leq C\,\delta,
\endeq
for some constant $C=C(d,C_W,C_0,\lambda,\lambda',\mu,\mu',\kappa)$.
In particular, $\rho=\int f\,dv$ satisfies
\begeq\label{rllinsat}
\sup_{t\in\R}\ \bigl\|\rho(t,\,\cdot\,)\bigr\|_{\cF^{\lambda'|t|+\mu'}} \leq C\,\delta.
\endeq
As a consequence, as $|t|\to\infty$, $\rho$ converges strongly to
$\rho_\infty = \iint f_i(x,v)\,dx\,dv$, and $f$ converges weakly to
$\<f_i\>=\int f_i\,dx$, at rate $O(e^{-\lambda''|t|})$ for any
$\lambda''<\lambda'$.  \sm

If moreover $\|f^0\|_{\cC^{\lambda;p}}\leq C_0$ and
$\|f_i\|_{\cZ^{\lambda,\mu;p}} \leq \delta$ for all $p$ in some
interval $[1,\ov{p}]$, then \eqref{Vllinsat} can be reinforced into
\begeq\label{Vllinsat'} \sup_{t\in\R}\
\bigl\|f(t,\,\cdot\,)\bigr\|_{\cZ^{\lambda',\mu';p}_t} \leq
C\,\delta,\qquad 1\leq p\leq \ov{p}.
\endeq
\end{Thm}

\begin{Rk} The notions of weak and strong convergence are the same as those
  in Theorem \ref{thmlineardamping}. With respect to that statement,
  we have added an extra analyticity assumption in the $x$ variable;
  in this linear context this is an overkill (as the proof will show),
  but later in the nonlinear context this will be important.
\end{Rk}

\begin{proof}[Proof of Theorem \ref{thmlindampingagain}]
  Without loss of generality we restrict our attention to $t\geq
  0$. Although \eqref{rllinsat} follows from \eqref{Vllinsat} by
  Proposition \ref{propZCF}, we shall establish \eqref{rllinsat}
  first, and deduce \eqref{Vllinsat} thanks to the equation. We shall
  write $C$ for various constants depending only on the parameters in
  the statement of the theorem.

As in the proof of Theorem \ref{thmlineardamping}, we have
\[ \hat{\rho}(t,k) = \tilde{f}_i(k,kt) + \int_0^t K^0(t-\tau,k)\,\hat{\rho}(\tau,k)\,d\tau\]
for any $t\geq 0$, $k\in\Z^d$. By Lemma \ref{lemvolterra}, for any $\lambda'<\lambda$, $\mu'<\mu$,
\begin{align*}
& \sup_{t\geq 0} \left( \sum_k |\hat{\rho}(t,k)|\,e^{2\pi (\lambda't+\mu')|k|}\right)
\\ 
& \qquad \qquad 
\leq C(\lambda,\lambda',\kappa)\, \left(\sum_k e^{-2\pi (\mu-\mu')|k|}\right)\,
\sup_{t\geq 0} \ \sup_{k\in\Z^d} \ \bigl|\tilde{f}_i(k,kt)\bigr|\,e^{2\pi (\lambda' t + \mu)|k|}\\
& \qquad \qquad \leq \frac{C(\lambda,\lambda',\kappa)}{(\mu-\mu')^d}\, \sup_{t\geq 0} 
\left(\sum_{k\in\Z^d} \bigl|\tilde{f}_i(k,kt)\bigr|\,e^{2\pi(\lambda t + \mu)|k|}\right).
\end{align*}
Equivalently,
\begeq\label{rhofiS}
\sup_{t\geq 0}\ \bigl\|\rho(t,\,\cdot\,)\bigr\|_{\cF^{\lambda't+\mu'}}
\leq C\ \sup_{t\geq 0} \ \left\| \int f_i\circ S^0_{-t}\,dv\right\|_{\cF^{\lambda t+\mu}}.
\endeq
By Propositions \ref{propZCF} and \ref{propTL},
\[ \left\| \int f_i\circ S^0_{-t}\,dv\right\|_{\cF^{\lambda t+\mu}}
\leq \bigl\|f_i\circ S^0_{-t}\bigr\|_{\cZ^{\lambda,\mu;1}_t} = \|f_i\|_{\cZ^{\lambda,\mu;1}_0} \leq \delta.\]
This and \eqref{rhofiS} imply \eqref{rllinsat}. 

To deduce \eqref{Vllinsat}, we first write
\[ f(t,\,\cdot\,) = f_i\circ S^0_{-t}
+ \int_0^t \Bigl( (\nabla W\ast\rho_\tau)\circ S^0_{-(t-\tau)}\Bigr)\cdot\nabla_v f^0\,d\tau,\]
where $\rho_\tau=\rho(\tau,\,\cdot\,)$.
Then for any $\lambda''<\lambda'$ we have, by Propositions \ref{propalgZ} and \ref{propZCF},
for all $t\geq 0$,
\begin{align}\label{ffrf}
\|f\|_{\cZ^{\lambda'',\mu';1}_t}
& \leq \bigl\|f_i\circ S^0_{-t}\bigr\|_{\cZ^{\lambda'',\mu;1}_t}
+ \int_0^t \Bigl\| (\nabla W\ast\rho_\tau)\circ S^0_{-(t-\tau)}\Bigr\|_{\cZ^{\lambda'',\mu';\infty}_t}\,
\|\nabla_v f^0\|_{\cZ^{\lambda'',\mu;1}_t}\, d\tau\\
& = \|f_i\|_{\cZ^{\lambda'',\mu;1}} + 
\left(\int_0^t \|\nabla W\ast\rho_\tau\|_{\cF^{\lambda''\tau +\mu'}}\,d\tau\right)\
\|\nabla_v f^0\|_{\cC^{\lambda'';1}}. \nonumber
\end{align}

Since $\hat{\nabla W}(0)=0$, we have, for any $\tau\geq 0$,
\begin{align*}
\bigl\|\nabla W\ast\rho_\tau\bigr\|_{\cF^{\lambda''\tau+\mu}}
& \leq e^{-2\pi (\lambda''-\lambda')\tau}\, 
\bigl\|\nabla W\ast\rho_\tau\bigr\|_{\cF^{\lambda'\tau+\mu'}}\\
& \leq \|\nabla W\|_{L^1}\, e^{-2\pi (\lambda''-\lambda')\tau}\, 
\|\rho_\tau\|_{\cF^{\lambda'\tau+\mu'}}\\
& \leq C_W\,C\,\delta\, e^{-2\pi (\lambda''-\lambda')\tau};
\end{align*}
in particular
\begeq\label{int0tnablaW}
\int_0^t \bigl\|\nabla W\ast\rho_\tau\bigr\|_{\cF^{\lambda''+\mu'}}
\leq \frac{C\,\delta}{\lambda''-\lambda'}.
\endeq

Also, by Proposition \ref{propgrad}, for $1<\lambda'/\lambda''\leq 2$ we have
\begeq\label{alsonablav}
\|\nabla_v f^0\|_{\cC^{\lambda'';1}}
\leq \frac{C}{\lambda-\lambda''}\, \|f^0\|_{\cC^{\lambda;1}} \leq \frac{C\,C_0}{\lambda-\lambda''}.
\endeq
Plugging \eqref{int0tnablaW} and \eqref{alsonablav} in \eqref{ffrf},
we deduce \eqref{Vllinsat}.  The end of the proof is an easy exercise
if one recalls that $\<f(t,\,\cdot\,)\> = \<f_i\>$ for all $t$.
\end{proof}

\section{Scattering estimates}
\label{sec:scattering}

Let be given a small time-dependent force field, denoted by $\var\,
F(t,x)$, on $\T^d\times\R^d$, whose analytic regularity {\em improves}
linearly in time. (Think of $\var F$ as the force created by a damped density.) This force field
perturbs the trajectories $S^0_{\tau,t}$ of the free transport ($\tau$
the initial time, $t$ the current time) into trajectories
$S_{\tau,t}$.  The goal of this section is to get an estimate on the
maps $\Om_{t,\tau}= S_{t,\tau}\circ S^0_{\tau,t}$ (so that $S_{t,\tau}
= \Om_{t,\tau} \circ S^0 _{t,\tau}$). These bounds should be in an
analytic class about as good as $F$, with a loss of analyticity
depending on $\var$; they should also be (for $0 \le \tau \le t$)\sm

\bul uniform in $t\geq \tau$;
\sm

\bul small as $\tau\to\infty$;
\sm

\bul small as $\tau\to t$.
\sm

We shall informally say that $\Om_{t,\tau}$ is a {\bf scattering}
transform, even though this terminology is usually reserved for the
asymptotic regime $t\to\pm\infty$. 

\begin{Rk} The order of composition of the free semigroup and
  perturbed semigroup is dictated by the need to get uniformity as
  $t\to\infty$. If we had defined, say,
  $\Lambda_{t,\tau}=S^0_{\tau,t}\circ S_{t,\tau}$, so that $S_{t,\tau}
  = S^0_{t,\tau}\circ \Lambda_{t,\tau}$, and if the force was, say,
  supported in $0\leq t\leq 1$, we would get (denoting $S_{t,\tau}
  =(X_{t,\tau},V_{t,\tau})$) 
  \[ \Lambda_{t,0}(x,v) = \Bigl(X_{1,0}(x-v(t-1),v) +t
  V_{1,0}(x-v(t-1),v),\, V_{1,0}(x-v(t-1),v)\Bigr),\] which does not
  converge to anything as $t\to\infty$.
\end{Rk}

\subsection{Formal expansion}

Before stating the main result, we sketch a heuristic perturbation study.
Let us write a formal expansion of $V_{0,t}(x,v)$ as a perturbation series:
\[ V_{0,t}(x,v) = v+\var\,v^{(1)}(t,x,v) + \var^2\,v^{(2)}(t,x,v) +
\ldots\]
Then we deduce
\[ X_{0,t}(x,v) = x + vt + \var \int_0^t v^{(1)}(s,x,v)\,ds
+ \var^2 \int_0^t v^{(2)}(s,x,v)\,ds + \ldots,\]
with $v^{(i)}(t=0) = 0$.

So
\[ \frac{\partial^2 X_{0,t}}{\partial t^2}=
\var \,\derpar{v^{(1)}}{t} + \var^2 \, \derpar{v^{(2)}}{t} + \ldots.\]

On the other hand,
\begin{align*}
\var \, F(t,X_{0,t}) 
& = \var \, \sum_k \hat{F}(t,k)\, e^{2i\pi k\cdot x} e^{2i\pi k\cdot vt}
e^{2i\pi k\cdot \left[ \var \int_0^t v^{(1)}\,ds + \var^2 \int_0^t v^{(2)}\,ds + \ldots\right]} \\
& = \var \, \sum_k \hat{F}(t,k) \, e^{2i\pi k\cdot x} e^{2i\pi k\cdot vt}
\, \Bigl[ 1 + 2 i \pi \var k\cdot \int_0^t v^{(1)}\,ds
+ 2i \pi \var^2 k\cdot \int_0^t v^{(2)}\,ds \\
& \qquad\qquad\qquad\qquad
- (2\pi)^2 \var^2 \left( k\cdot \int_0^t v^{(1)}\,ds \right)^2 + \ldots \Bigr].
\end{align*}

By successive identification,
\[ \derpar{v^{(1)}}{t} = \sum_k \hat{F}(t,k)\, e^{2i\pi k\cdot x} e^{2i\pi k\cdot vt};\]
\[ \derpar{v^{(2)}}{t} = \sum_k \hat{F}(t,k)\, e^{2i\pi k\cdot x} e^{2i\pi k\cdot vt}\,
2i\pi k\cdot \int_0^t v^{(1)}\,ds;\]
\[ \derpar{v^{(3)}}{t} = \sum_k \hat{F}(t,k)\, e^{2i\pi k\cdot x} e^{2i\pi k\cdot vt}
\left[ 2i\pi k\cdot \int_0^t v^{(2)}\,ds - (2\pi)^2 \var^2
\left( k\cdot\int_0^t v^{(1)}\,ds\right)^2 \right],\]
etc.

In particular notice that $\left| \derpar{v^{(1)}}{t} \right| \leq \sum_k |\hat{F}(t,k)|$, so
\begin{multline*} \int_0^\infty \left|\derpar{v^{(1)}}{t} \right|\,dt 
\leq \int_0^\infty \sum_k |\hat{F}(t,k)|\,dt \leq 
\int_0^\infty \sum |\hat{F}(t,k)|\, e^{2\pi \mu t} e^{-2\pi \mu t}\,dt\\
\leq C_F \int_0^\infty e^{-2\pi \mu t} = \frac{C_F}{2\pi \mu}. 
\end{multline*}
So, under our uniform analyticity assumptions we expect $V_{0,t}(x,v)$ to be a {\em uniformly bounded}
analytic perturbation of $v$.

\subsection{Main result} \label{maincharac}

On $\T^d_x$ we consider the dynamical system
\[ \frac{d^2X}{dt^2} = \var\,F(t,X);\]
its phase space is $\T^d\times\R^d$.
Although this system is reversible, we shall only consider $t\geq 0$.
The parameter $\var$ is here only to recall the perturbative nature of the estimate.

For any $(x,v)\in\T^d\times\R^d$ and any two times $\tau,t\in\R_+$, 
let $S_{\tau,t}$ be the transform mapping the state of the system at 
time $\tau$, to the state of the system at time $t$.
In more precise terms, $S_{\tau,t}$ is described by the equations
\[ S_{\tau,t}(x,v) = \bigl(X_{\tau,t}(x,v), V_{\tau,t}(x,v)\bigr);\]
\[ X_{\tau,\tau}(x,v)=x,\qquad V_{\tau,\tau}(x,v) = v;\]
\begeq\label{eqS}
\frac{d}{dt} X_{\tau,t}(x,v) = V_{\tau,t}(x,v),\qquad
\frac{d}{dt} V_{\tau,t}(x,v) = \var\,F(t,X_{\tau,t}(x,v)).
\endeq
From the definition we have the composition identity 
\begeq\label{SSS} S_{t_2,t_3}\circ S_{t_1,t_2} = S_{t_1,t_3};
\endeq
in particular $S_{t,\tau}$ is the inverse of $S_{\tau,t}$.

We also write $S^0_{\tau,t}$ for the same transform in the case 
of the free dynamics ($\var=0$);
in this case there is an explicit expression:
\begeq\label{S0} S^0_{\tau,t}(x,v) = (x+v(t-\tau),v), \endeq
where $x+v(t-\tau)$ is evaluated modulo $\Z^d$.
Finally, we define the ``scattering transforms associated with $\var F$'':
\begeq\label{Omttau}
\Om_{t,\tau} = S_{t,\tau}\circ S^0_{\tau,t}.
\endeq
(There is no simple semigroup property for the transforms $\Om_{t,\tau}$.)

In this section we establish the following estimates:

\begin{Thm}[Analytic estimates on scattering transforms in hybrid norms] \label{thmscat}
Let $\var>0$ and let $F=F(t,x)$ on $\R_+\times\T^d$ satisfy
\begeq\label{CF}
\hat{F}(t,0) =0,\qquad
\sup_{t\geq 0} \Bigl( \|F(t,\cdot)\|_{{\cal F}^{\lambda t+\mu}} + 
\|\nabla_x F(t,\cdot)\|_{{\cal F}^{\lambda t+\mu}}\Bigr) \leq C_F
\endeq
for some parameters $\lambda,\mu>0$ and $C_F>0$. 
Let $t\geq \tau\geq 0$, and let
\[ \Om_{t,\tau} = \bigl(\OmX_{t,\tau},\OmV_{t,\tau}\bigr)\]
be the scattering transforms associated with $\var \, F$. 
Let $0 \le \lambda' < \lambda$, $0 \le \mu'<\mu$ and $\tau'\geq 0$ be such that
\begin{equation}\label{hypscat0}
\lambda' \, (\tau'-\tau) \le \frac{(\mu-\mu')}2.
\end{equation}
Let
\[ 
\begin{cases}
R_1(\tau,t) = C_F \, e^{-2 \pi \, (\lambda - \lambda') \, \tau} \,
\min \left\{ (t-\tau) \, ; \ (2 \pi (\lambda-\lambda'))^{-1} \right\}; \\[3mm]
R_2(\tau,t) = C_F \, e^{-2 \pi \, (\lambda - \lambda') \, \tau} \, 
\min \left\{ (t-\tau)^2/2 \, ; \ (2 \pi (\lambda-\lambda'))^{-2}
\right\}. 
\end{cases}
\]
Assume that
\begeq\label{asscat1}
\forall \, 0 \le \tau \le t, \quad 
\var\, R_2(\tau,t) \leq \frac{(\mu-\mu')}{4},
\endeq
and
\begeq\label{asscat2}
\var\, C_F \,\leq \frac{4 \pi^2 \, (\lambda-\lambda')^2}{2}.
\endeq
Then
\begeq\label{estOmX}
\forall \, 0 \le \tau \le t, \quad 
\|\OmX_{t,\tau}-\Id\|_{\cZ^{\lambda',\mu'}_{\tau'}} \leq 2 \, \eps \, R_2(\tau,t)
\endeq
and
\begeq\label{estOmV} 
\forall \, 0 \le \tau \le t, \quad 
\|\OmV_{t,\tau}-\Id\|_{\cZ^{\lambda',\mu'}_{\tau'}} \leq \eps \, R_1(\tau,t).
\endeq
\end{Thm}

\begin{Rk} The proof of Theorem \ref{thmscat} is easily adapted to include
Sobolev corrections. It is important to note that the scattering
transforms are smooth, uniformly in time, not just in gliding
regularity ($\tau'=0$ is admissible in \eqref{hypscat0}).
\end{Rk}

\begin{proof}[Proof of Theorem~\ref{thmscat}]
For a start, let us make the ansatz
\[ S_{t,\tau}(x,v) = \Bigl( x-v(t-\tau) + \var\, Z_{t,\tau}(x,v),\, v+ \var\,\pa_\tau Z_{t,\tau}(x,v)\Bigr),\]
with
\[ Z_{t,t}(x,v)=0,\qquad \pa_\tau Z_{t,\tau}\Bigr|_{\tau=t}(x,v) = 0.\]
Then it is easily checked that
\[ \Om_{t,\tau} - \Id = \var\, (Z,\pa_\tau Z)\circ S^0_{t-\tau};\]
in particular
\[ \|\Om_{t,\tau}-\Id\|_{\cZ^{\lambda',\mu'}_{\tau'}} 
= \var\, \bigl\|(Z,\pa_\tau Z)\bigr\|_{\cZ^{\lambda',\mu'}_{t+\tau'-\tau}}.\]

To estimate this we shall use a fixed point argument based on the equation for $S_{t,\tau}$, namely
\[ \frac{d^2 X_{t,\tau}}{d\tau^2}  = \var\, F(\tau,X_{t,\tau}),\]
or equivalently
\[ \frac{d^2 Z_{t,\tau}}{d\tau^2}  = F\Bigl(\tau,x-v(t-\tau) + \var\,Z_{t,\tau}\Bigr).\]
\sm

So let us fix $t$ and define
$$
\Psi : \left(W_{t,\tau} \right)_{0 \le \tau \le t} \longmapsto \left(Z_{t,\tau} \right)_{0 \le \tau \le t}
$$
such that $\left(Z_{t,\tau} \right)_{0 \le \tau \le t}$ is the solution of 
\begeq\label{edo}\begin{cases}
\dps \frac{\partial^2 Z_{t,\tau}}{\partial \tau^2} = F\Bigl(\tau,x - v (t-\tau) + \eps \, W_{t,\tau}\Bigr)\\[2mm]
Z_{t,t} = 0, \quad (\partial_\tau Z_{t,\tau})\Big|_{\tau=t} = 0.
\end{cases}
\endeq
What we are after is an estimate of the fixed point of $\Psi$. We do this in two steps.
\sm

\noindent
{\bf Step 1. Estimate of $\Psi(0)$.} 
Let $Z^0=\Psi(0)$. By integration of \eqref{edo} (for $W=0$) we have
$$ Z^0_{t,\tau} = \int_\tau ^t (s-\tau) \, F\left(s,x - v (t-s)\right) \, ds. $$

Let $\sigma$ such that $\lambda'\sigma \leq (\mu-\mu')/2$. 
We apply the $\cZ_{t+\sigma}^{\lambda',\mu'}$ norm and use Proposition~\ref{propTL}:
$$
\| Z^0_{t,\tau} \|_{\cZ_{t+\sigma} ^{\lambda',\mu'}} \le 
\int_\tau ^t (s-\tau) \, \| F\left(s,\,\cdot\, \right) \|_{\cZ_{s+\sigma}^{\lambda',\mu'}} \, ds = 
\int_\tau ^t (s-\tau) \, \| F\left(s,\,\cdot\, \right) \|_{\cF^{\lambda's + \lambda' \sigma + \mu'}} \, ds.
$$
Of course $\lambda'\sigma+\mu'\leq \mu$, so in particular
\[ 
\lambda' \, s + \lambda' \, \sigma + \mu' \le - (\lambda - \lambda') \, s + \lambda \, s + \mu.
\]
Combining this with the assumption $\hat{F}(s,0)=0$ yields
\begin{align*}
\| F\left(s,\cdot \right) \|_{\cF^{\lambda' s + \lambda' \, \sigma + \mu'}} 
& \le \|F\left(s,\cdot \right) \|_{\cF^{\lambda s + \mu}} \, e^{-2\pi (\lambda - \lambda')\, s}\\
& \le C_F\, e^{-2\pi(\lambda-\lambda')s}.
\end{align*}
So
\begin{align*}
\|Z^0_{t,\tau}\|_{\cZ^{\lambda',\mu'}_{t+\sigma}}
& \le C_F \, \int_\tau^t (s-\tau) \, e^{-2\pi (\lambda - \lambda')\, s} \, ds\\
& \le C_F \, e^{-2\pi (\lambda - \lambda')\, \tau} \,
\min \left\{ \frac{(t-\tau)^2}{2} \, ;
\ \frac1{(2\pi(\lambda-\lambda'))^2}\right\} \le R_2(\tau,t).
\end{align*}

With $t$ still fixed, we define the norm
\begeq\label{bbqn}
\bbqn \left(Z_{t,\tau} \right)_{0 \le \tau \le t} \bbqn := 
\sup\left\{ \frac{\| Z_{t,\tau} \|_{\cZ_{t+\sigma}^{\lambda',\mu'}}}{R_2(\tau,t)};\quad
0\leq\tau\leq t;\ \sigma+t\geq 0;\ \lambda'\sigma \leq \frac{\mu-\mu'}{2}\right\}.
\endeq
The above estimates show that $\qn\Psi(0)\qn \leq 1$.
(We can assume $t+\sigma\geq 0$ since $t+(\tau'-\tau) \ge t-\tau \ge 0$,
and we aim at finally choosing $\sigma=\tau'-\tau$.)
\sm

\noindent
{\bf Step 2. Lipschitz constant of $\Psi$.} 
We shall prove that under our assumptions, $\Psi$ is $1/2$-Lipschitz on the ball $B(0,2)$ in the norm $\qn \cdot \qn$. 
Let $W,\tilde{W}\in B(0,2)$, and $Z=\Psi(W)$, $\tilde{Z}=\Psi(\tilde{W})$. By solving the differential inequality
for $Z-\tilde{Z}$ we get
\begin{multline*}
Z_{t,\tau} - \tilde Z_{t,\tau} = 
\eps \, \Bigg[ \int_0 ^1 \int_\tau ^t (s-\tau) \, 
\nabla_x F\Big(s,x - v (t-s) \\+ \eps \, \Big( \theta W_{t,s}  + (1-\theta) \tilde W_{t,s} \Big) \Big) \, ds \,
d\theta \Bigg]
\cdot \left(W_{t,s} - \tilde W_{t,s}\right).
\end{multline*}

We divide by $R_2(\tau,t)$, take the $\cZ$ norm, and note that $R_2(s,t) \le R_2(\tau,t)$; we get
\[ \bbqn \left( Z_{t,\tau} - \tilde Z_{t,\tau} \right)_{0\le \tau \le t} \bbqn
\le \eps \, \bbqn \left( W_{t,s} - \tilde W_{t,s} \right)_{0\le s \le t} \bbqn \, A(t)
\]
with
\begin{multline*}
A(t)= \sup_{\sigma,\tau} \int_0 ^1 \int_\tau ^t (s-\tau) \,
\Bigg\| \nabla_x F\Big(s,x - v (t-s) + \eps \, \Big( \theta W_{t,s} + (1-\theta) \tilde W_{t,s} \Big) \Big) \Bigg\|
_{\cZ_{t+\sigma} ^{\lambda',\mu'}}\, ds \,
d\theta.
\end{multline*}

By Proposition \ref{propcompos2} (composition inequality),
\begin{multline*}
A(t) \le \int_\tau ^t (s-\tau) \, 
\left\| \nabla_x F (s, \,\cdot\,) \right\|_{\cZ_{s+\sigma}^{\lambda',\mu'+e(t,s,\sigma)}} \, ds \\ 
= \int_\tau ^t (s-\tau) \, 
\left\| \nabla_x F (s, \,\cdot\,) \right\|_{\cF^{\lambda' s + \lambda'
    \sigma+ \mu'+e(t,s,\sigma)}} \, ds,
\end{multline*}
with
\[ e(t,s,\sigma) := \eps \, 
\Bigl\| \theta \, W_{t,s} + (1-\theta) \, \tilde W_{t,s} \Bigr\|_{\cZ_{t+\sigma} ^{\lambda',\mu'}}
\leq 2\, \eps\, R_2(s,t)\leq 2\, \eps\, R_2(\tau,t).
\]

Using \eqref{asscat1}, we get
\begin{align*}
\lambda' s + \lambda' \sigma + \mu'+ e(s,t,\sigma) & \le 
\lambda' s + \lambda' \sigma + \mu' + 2 \, \eps \, R_2(\tau,t)\\
& \leq \lambda' s + \mu = (\lambda s+\mu) - (\lambda -\lambda')s.
\end{align*}
Using again the bound on $\nabla_x F$ and the assumption
$\hat{F}(s,0)=0$, we deduce
\[ 
A(t) \le \sup_\tau \int_\tau ^t (s-\tau) \, C_F \, e^{-2\pi
  (\lambda-\lambda') \, s} \, ds \le R_2(0,t) \leq
\frac{C_F}{4 \pi^2 \, (\lambda-\lambda')^2}.\] 
Using \eqref{asscat2}, we conclude that
\[
\bbqn \left( Z_{t,\tau} - \tilde Z_{t,\tau} \right)_{0\le \tau \le t} \bbqn
 \le \frac12
\bbqn \left( W_{t,s} - \tilde W_{t,s} \right)_{0\le s \le t} \bbqn.
\]
So $\Psi$ is $1/2$-Lipschitz on $B(0,2)$, and we can conclude the proof of \eqref{estOmX} by
applying Theorem \ref{thmfpt} and choosing $\sigma=\tau'-\tau$.
\med

It remains to control the velocity component of $\Om$, {\it i.e.},
establish \eqref{estOmV}; this will follow from the control of the
position component. Indeed, if we write $Q_{t,\tau} =
\var^{-1}(\OmV_{t,\tau}-\Id)(x,v)$, we have
\[ Q_{t,\tau} = \int_\tau ^t 
F\Big(s,x - v (t-s) + \eps \, W_{t,s} \Big) \, ds \]
so we can estimate as before
\[
\| Q_{t,\tau} \|_{\cZ_{t+(\tau'-\tau)}^{\lambda',\mu'}} 
\le \int_\tau ^t  \left\|  F (s, \cdot) \right\|_{\cF^{\lambda' s + \lambda'
    (\tau'-\tau) + \mu'+ e(t,s,\tau'-\tau)}} \, ds
\]
to get
\[ 
\| Q_{t,\tau} \|_{\cZ_{t+(\tau'-\tau)}^{\lambda',\mu'}}  
\le \int_\tau ^t C_F \, e^{-2\pi (\lambda-\lambda') \, s} \, ds \le R_1(\tau,t).
\]
Thus the proof is complete.
\end{proof}

\begin{Rk} Loss and Bernard independently suggested to compare 
  the estimates in the present section with the Nekhoroshev theorem
  in dynamical systems theory \cite{Nekhoroshev:partI,Nekhoroshev:partII}.
  The latter theorem roughly states that
  for a perturbation of a completely integrable system, trajectories
  remain close to those of the unperturbed system for a time growing
  exponentially in the inverse of the size of the perturbation 
  (unlike KAM theory, this result is not global in time; but it is more general 
  in the sense that it also applies outside invariant tori).
  In the present setting the situation is better since the perturbation decays.
\end{Rk}

\section{Bilinear regularity and decay estimates}
\label{sec:regdecay}

To introduce this crucial section, let us reproduce and improve a key
computation from Section \ref{sec:linear}. Let $G$ be a function of
$v$, and $R$ a time-dependent function of $x$ with $\hat{R}(0)=0$;
both $G$ and $R$ may be vector-valued.  (Think of $G(v)$ as $\nabla_v
f(v)$ and of $R(\tau,x)$ as $\nabla W\ast\rho(\tau,x)$.) Let further
\[
\sigma(t,x) = 
\int_0^t \int_{\R^d} G(v)\cdot R\bigl(\tau, x-v(t-\tau)\bigr)\,dv\,d\tau.
\]
Then
\begin{align*}
\hat{\sigma}(t,k) = 
& \int_0^t \int_{\T^d} \int_{\R^d} G(v)\cdot R\bigl(\tau, x-v(t-\tau)\bigr)\,e^{-2i\pi k\cdot x}\,dv\,dx\,d\tau\\
& = \int_0^t \int_{\T^d} \int_{\R^d} G(v)\cdot R(\tau,x)\, e^{-2i\pi k\cdot x}\, e^{-2i\pi k\cdot v(t-\tau)}\,dv\,dx\,d\tau\\
& = \int_0^t \tilde{G}(k(t-\tau))\cdot \hat{R}(\tau,k)\,d\tau.
\end{align*}

Let us assume that $G$ has a ``high'' gliding analytic regularity $\ov{\lambda}$, and estimate $\sigma$ in regularity
$\lambda t$, with $\lambda<\ov{\lambda}$. Let $\alpha=\alpha(t,\tau)$ satisfy
\[ 0\leq \alpha(t,\tau) \leq (\ov{\lambda}-\lambda)\,(t-\tau);\]
then
\begin{align*}
\|\sigma(t)\|_{\cF^{\lambda t}}
& \leq \sum_{k\neq 0} \int_0^t e^{2\pi \lambda t|k|}\,
|\tilde{G}(k(t-\tau))|\, |\hat{R}(\tau,k)|\,d\tau\\
& \leq \int_0^t \left( \sup_{k\neq 0}\, e^{2\pi [\lambda (t-\tau) +\alpha]\,|k|}\,
|\tilde{G}(k(t-\tau))|\right)\, \left(\sum_k e^{2\pi (\lambda\tau - \alpha)|k|}|\hat{R}(\tau,k)|\right)\,d\tau\\
& \leq \left(\sup_{\eta} e^{2\pi\ov{\lambda}|\eta|} |\tilde{G}(\eta)|\right)\,
\left(\sup_{0\leq \tau\leq t} \|R(\tau,\cdot)\|_{\cF^{\lambda\tau-\alpha}}\right)
\int_0^t e^{-2\pi [(\ov{\lambda}-\lambda)(t-\tau)-\alpha]}\,d\tau,
\end{align*}
where we have used
\[ k\neq 0 \Longrightarrow \ 2\pi (\lambda(t-\tau) + \alpha)|k| \leq
2\pi \ov{\lambda}|k|(t-\tau) - \bigl( (\ov{\lambda}-\lambda)(t-\tau)-\alpha\bigr).\]

Let us choose
\[ \alpha(t,\tau) = \frac{(\ov{\lambda}-\lambda)}{2}\,\min\{ 1 \, ; \,
t-\tau  \} ;\]
then
\[ \int_0^t e^{-2\pi \bigl[(\ov{\lambda}-\lambda)(t-\tau)-\alpha\bigr]}\,d\tau
\leq \int_0^t e^{-\pi (\ov{\lambda}-\lambda)(t-\tau)}\,d\tau \leq \frac1{\pi(\ov{\lambda}-\lambda)}.\]
So in the end
\[ \|\sigma(t)\|_{\cF^{\lambda t}} \leq
\frac{\|G\|_{\cX^{\ov{\lambda}}}}{\pi(\ov{\lambda}-\lambda)}\,
\sup_{0\leq \tau\leq t} \|R(\tau)\|_{\cF^{\lambda\tau - \alpha(t,\tau)}},\]
where $\|G\|_{\cX^{\ov{\lambda}}} = \sup_\eta (e^{2\pi\ov{\lambda}|\eta|}|\tilde{G}(\eta)|)$.
\sm

In the preceding computation there are three important things to
notice, which lie at the heart of Landau damping: \sm

\begin{itemize}
\item The natural index of analytic regularity of $\sigma$ in $x$
increases linearly in time: this is an automatic consequence of the
gliding regularity, already observed in Section \ref{sec:analytic}.
\item A bit $\alpha(t,\tau)$ of analytic regularity of $G$ was
transferred from $G$ to $R$, however not more than a fraction of
$(\ov{\lambda}-\lambda)(t-\tau)$.  We call this the {\bf regularity
  extortion}: if $f$ {\em forces} $\ov{f}$, it satisfies an equation
of the form $\pa_t f + v\cdot\nabla_x f + F[f]\cdot\nabla_v \ov{f} =
S$, then $\ov{f}$ will give away some (gliding) smoothness to $\rho =
\int f\,dv$. 
\item The combination of higher regularity of $G$ and the assumption
$\hat{R}(0)=0$ has been converted into a time decay, so that the
time-integral is bounded, uniformly as $t\to\infty$. Thus there is
{\bf decay by regularity}. 
\end{itemize} \sm

The main goal of this section is to establish quantitative variants of
these effects in some general situations when $G$ is not only a
function of $v$ and $R$ not only a function of $t,x$.  Note that we
shall have to work with regularity indices depending on $t$ and
$\tau$!

Regularity extortion is related to velocity averaging regularity,
well-known in kinetic theory \cite{jabin:ercole}; what is unusual
though is that we are working in analytic regularity, and in large
time, while velocity averaging regularity is mainly a short-time
effect.  In fact we shall study two distinct mechanisms for the
extortion: the first one will be well suited for short times ($t-\tau$
small), and will be crucial later to get rid of small deteriorations
in the functional spaces due to composition; the second one will be
well adapted to large times $(t-\tau\to\infty)$ and will ensure
convergence of the time integrals.

The estimates in this section lead to a serious twist on the popular
view on Landau damping, according to which the waves gives energy to
the particles that it forces; instead, the picture here is that the
wave gains regularity from the background, and regularity is converted
into decay.

For the sake of pedagogy, we shall first establish the basic, simple
bilinear estimate, and then discuss the two mechanisms once at a
time.

\subsection{Basic bilinear estimate}

\begin{Prop}[Basic bilinear estimate in gliding regularity] \label{propbasic}
Let $G=G(\tau,x,v)$, $R=R(\tau,x,v)$,
\[ \beta(\tau,x) = \int_{\R^d} (G\cdot R)\bigl(\tau, x-v(t-\tau),v\bigr)\,dv,\]
\[ \sigma(t,x) = \int_0^t \beta(\tau,x)\,d\tau.\] Then
\begeq\label{betaGR} \|\beta(\tau,\cdot)\|_{\cF^{\lambda t+\mu}} \leq
\|G\|_{\cZ^{\lambda,\mu;1}_\tau}\, \|R\|_{\cZ^{\lambda,\mu}_\tau};
\endeq
and \begeq\label{sigmaGR} \|\sigma(t,\cdot)\|_{\cF^{\lambda t+\mu}}
\leq \int_0^t
\|G\|_{\cZ^{\lambda,\mu;1}_\tau}\,\|R\|_{\cZ^{\lambda,\mu}_\tau}\,d\tau.
\endeq
\end{Prop}

\begin{proof}[Proof of Proposition \ref{propbasic}]
  Obviously \eqref{sigmaGR} follows from \eqref{betaGR}. To prove
  \eqref{betaGR} we apply successively Propositions \ref{propZCF},
  \ref{propTL} and \ref{propalgZ}:
\begin{align*}
\|\beta(\tau,\cdot)\|_{\cF^{\lambda t+\mu}}
& \leq \left\| \int_{\R^d} (G\cdot R)\circ S^0_{\tau-t}\,dv \right\|_{\cF^{\lambda t+\mu}}\\
& \leq \Bigl\| (G\cdot R)\circ S^0_{\tau-t} \Bigr\|_{\cZ^{\lambda,\mu;1}_t}\\
& = \|G\cdot R\|_{\cZ^{\lambda,\mu;1}_\tau} 
\leq \|G\|_{\cZ^{\lambda,\mu;1}_\tau}\, \|R\|_{\cZ^{\lambda,\mu}_\tau}.
\end{align*}
\end{proof}

\subsection{Short-term regularity extortion by time
  cheating} \label{sub:rt}

\begin{Prop}[Short-term regularity extortion] \label{proprt}
Let $G=G(x,v)$, $R=R(x,v)$, and
\[ \beta(x) = \int_{\R^d} (G\cdot R)\, (x-v(t-\tau),v)\,dv.\] Then for
any $\lambda,\mu,t\geq 0$ and any $b>-1$, we have
\begeq\label{beta+rt} \|\beta\|_{\cF^{\lambda t + \mu}} \leq
\|G\|_{\cZ^{\lambda (1+b),\mu;1}_{\tau - \frac{bt}{1+b}}}\,
\|R\|_{\cZ^{\lambda (1+b),\mu}_{\tau - \frac{bt}{1+b}}}.
\endeq
Moreover, if $P_k$ stands for the projection on the $k$th Fourier mode
as in \eqref{Pkf}, one has \begeq\label{beta+rtk} e^{2\pi (\lambda t
  +\mu)|k|}|\hat{\beta}(k)| \leq \sum_{\ell\in\Z^d} \|P_\ell
G\|_{\cZ^{\lambda(1+b),\mu;1}_{\tau-\frac{bt}{1+b}}}
\|P_{k-\ell}R\|_{\cZ^{\lambda(1+b),\mu}_{\tau-\frac{bt}{1+b}}}.
\endeq
\end{Prop}

\begin{Rk} If $R$ only depends on $t,x$, then the norm of $R$ in the
  right-hand side of \eqref{beta+rt} is $\|R\|_{\cF^{\nu}}$ with
  \[ \nu = \lambda (1+b)\,\left|\tau - \frac{bt}{1+b}\right| + \mu =
  (\lambda \tau+\mu) - b(t-\tau),\] as soon as $\tau\geq
  bt/(1+b)$. Thus some regularity has been gained with respect to
  Proposition \ref{propbasic}. Even if $R$ is not a function of $t,x$
  alone, but rather a function of $t,x$ {\em composed} with
    a function depending on all the variables, this gain will be
  preserved through the composition inequality.
\end{Rk}

\begin{proof}[Proof of Proposition \ref{proprt}]
  The proof presented here relies on commutators involving $\nabla_v$,
  $\nabla_x$ and the transport semigroup, all of them classically
  related to hypoelliptic regularity and velocity
  averaging. Separating the different components of $R$ and $G$, we
  may assume that both are scalar-valued.

  Let $S=S^0_{\tau-t}$, so that $R\circ S(x,v) = R(x-v(t-\tau),v)$. By
  direct computation, \begeq\label{tDRS} t\nabla_x (R\circ S) =
  (t\nabla_x R)\circ S = \Bigl[ \bigl( (\tau-b(t-\tau))\nabla_x +
  (1+b)\nabla_v\bigr) R\Bigr]\circ S - (1+b) \nabla_v (R\circ S).
  \endeq
  Let
  \[ D = D_{\tau,t,b} := \bigl(\tau-b(t-\tau)\bigr) \nabla_x +
  (1+b)\nabla_v.\] Then \eqref{tDRS} becomes \begeq\label{commDRS}
  t\nabla_x (R\circ S) = (DR)\circ S - (1+b)\nabla_v (R\circ S).
  \endeq
  Since $\nabla_x$ commutes with $\nabla_v$ and $D$, and with the
  composition by $S$ as well, we deduce from \eqref{commDRS} that
\begin{align*}
t\,\pa_{x_i} \Bigl[ (1+b)^k \nabla_v^k ((D^\ell R)\circ S)\Bigr]
& = (1+b)^k \nabla_v^k \bigl( t\pa_{x_i} (D^\ell R)\circ S\bigr)\\
& = (1+b)^k \nabla_v^k\bigl((D^{\ell+1_i}R)\circ S\bigr)
- (1+b)^k \nabla_v^k \Bigl((1+b) \pa_{v_i}(D^\ell R\circ S)\Bigr)\\
& = \bigl[ (1+b)\nabla_v\bigr]^k \bigl((D^{\ell+1_i}R)\circ S\bigr)
- \bigl[(1+b)\nabla_v\bigr]^{k+1_i}\bigl((D^\ell R)\circ S\bigr).
\end{align*}
So by induction,
\begeq\label{CmnD}
(t\nabla_x)^n (R\circ S)
= \sum_{m\leq n} \Cnk{n}{m} \bigl[-(1+b)\nabla_v\bigr]^m
\bigl((D^{n-m}R)\circ S\bigr).
\endeq

Applying this formula with $R$ replaced by $G\cdot R$ and integrating
in $v$ yields
\begin{align*}
(t\nabla_x)^n \int_{\R^d} (G\cdot R)\circ S^0_{\tau-t}\,dv
& = \int_{\R^d} D^n(G\cdot R)\circ S^0_{\tau-t}\,dv\\
& = \int_{\R^d} D^n(G\cdot R)\,dv.
\end{align*}

It follows by taking Fourier transform that
\begin{align*}
(2i\pi t k)^n \hat{\beta}(k) 
& = \int_{\R^d} \bigl[D^n(G\cdot R)\bigr]^{\hat{ }} \,dv \\
& = \int_{\R^d} \Bigl( (1+b)\nabla_v + 2i\pi \bigl(\tau-b(t-\tau)\bigr) k\Bigr)^n
\, (G\cdot R)^{\hat{ }}(k,v)\,dv,
\end{align*}
whence
\begin{align*}
  \sum_{k,n} e^{2\pi \mu|k|} & \frac{|2\pi \lambda tk|^n}{n!}\,|\hat{\beta}(k)| \\
  & \leq \sum_{k,n} e^{2\pi \mu|k|}
  \frac{\bigl(\lambda(1+b)\bigr)^n}{n!}  \Bigl\| \Bigl[ \nabla_v +
  2i\pi \Bigl(\tau-\frac{bt}{1+b}\Bigr) k\Bigr]^n\,
  (G\cdot R)^{\hat{ }}(k,v)\Bigr\|_{L^1(dv)} \\
  & = \|G\cdot R\|_{\cZ^{\lambda(1+b),\mu;1}_{\tau-\frac{bt}{1+b}}},
\end{align*}
and the conclusion follows by Proposition \ref{propalgZ}.

Inequality \eqref{beta+rtk} is obtained
in a similar way with the help of Proposition \ref{prop424}.  
\end{proof}

Let us conclude this subsection with some comments on Proposition
\ref{proprt}.  When we wish to apply it, what constraints on
$b(t,\tau)$ (assumed to be nonnegative to fix the ideas) does this
presuppose?  First, $b$ should be small, so that $\lambda(1+b)\leq
\ov{\lambda}$ given.  But most importantly, we have estimated $G_\tau$
in a norm ${\cZ_{\tau'}}$ instead of ${\cZ_\tau}$ (this is the time
cheating), where $|\tau'-\tau| = bt/(1+b)$. To compensate for this
discrepancy, we may apply \eqref{moreoverff}, but for this to work
$bt/(1+b)$ should be small, otherwise we would lose a large index of
analyticity in $x$, or at best we would inherit an undesirable
exponentially growing constant.  So all we are allowed is $b(t,\tau) =
O(1/(1+t))$.  This is not enough to get the time-decay which would
lead to Landau damping. Indeed, if $R=R(x)$ with
$\hat{R}(0)=0$, then
\[ \|R\|_{\cZ^{\lambda(1+b),\mu}_{\tau-bt/(1+b)}}
= \|R\|_{\cF^{\lambda \tau + \mu - \lambda\, b (t-\tau)}}
\leq e^{-\lambda b(t-\tau)}\,\|R\|_{\cF^{\lambda \tau+\mu}};\]
so we gain a coefficient $e^{-\lambda b(t-\tau)}$, but then
\[ \int_0^t e^{-\lambda b(t-\tau)}\,d\tau \geq
\int_0^t e^{-\lambda \var \left(\frac{t-\tau}{t}\right)}\,d\tau 
= \left(\frac{1-e^{-\lambda \var}}{\lambda\, \var}\right)t,\]
which of course diverges in large time.

To summarize: Proposition \ref{proprt} is helpful when $(t-\tau) =
O(1)$, or when some extra time-decay is available. This will already be very useful; 
but for long-time estimates we need another, complementary mechanism.

\subsection{Long-term regularity extortion}

To search for the extra decay, let us refine the computation of the beginning of this section.
Assume that $G_\tau = \nabla_v g_\tau$, where $(g_\tau)_{\tau\geq 0}$ solves a transport-like equation, so
$\tilde{G}(\tau,k,\eta) = 2i\pi\eta\,\tilde{g}(\tau,k,\eta)$, and
\[ |\tilde{G}(\tau,k,\eta)| \lesssim 2\pi |\eta|\, e^{-2\pi \ov{\mu}|k|}\,e^{-2\pi \ov{\lambda}|\eta+k\tau|}.\]
Up to slightly increasing $\ov{\lambda}$ and $\ov{\mu}$, we may assume
\begeq\label{compG}
|\tilde{G}(\tau,k,\eta)| \lesssim (1+\tau)\, e^{-2\pi \ov{\mu}|k|}\,e^{-2\pi \ov{\lambda}|\eta+k\tau|}.
\endeq
Let then $\rho(\tau,x) = \int f(\tau,x,v)\,dv$, where also $f$ solves a transport equation, but has a lower
analytic regularity; and $R=\nabla W\ast \rho$. Assuming $|\hat{\nabla W}(k)| = O(|k|^{-\gamma})$
for some $\gamma\geq 0$, we have
\begeq\label{compR}
|\hat{R}(\tau,k)|\lesssim \frac{e^{-2\pi (\lambda \tau+\mu)|k|}\, 1_{k\neq 0}}{1+|k|^\gamma}.
\endeq

Let again
\[ \sigma(t,x) = \int_0^t \int_{\R^d}
 G\bigl( \tau,x-v(t-\tau),v\bigr)\cdot R\bigl(\tau,x-v(t-\tau)\bigr)\,dv\,d\tau.\]
As $t\to +\infty$, $G$ in the integrand of $\sigma$ oscillates wildly in phase space, so it is not clear that it will
help at all. But let us compute:
\begin{align*}
\hat{\sigma}(t,k)
& = \int_0^t \int_{\R^d} \int_{\T^d} G\bigl(\tau,x-v(t-\tau),v\bigr)\cdot R\bigl(\tau,x-v(t-\tau)\bigr)\,e^{-2i\pi k\cdot x}\,dx\,dv\,d\tau\\
& = \int_0^t \int_{\R^d} \int_{\T^d} G(\tau,x,v)\cdot R(\tau,x)\,e^{-2i \pi k\cdot x}\, e^{-2i \pi k\cdot v(t-\tau)}\,dx\,dv\,d\tau\\
& = \int_0^t \int_{\R^d} \hat{G\cdot R}(\tau,k,v)\,e^{-2i\pi k\cdot v(t-\tau)}\,dv\,d\tau\\
& = \int_0^t \int_{\R^d} \sum_\ell \hat{G}(\tau,\ell,v)\cdot\hat{R}(\tau,k-\ell)\,e^{-2i\pi k\cdot v(t-\tau)}\,dv\,d\tau\\
& = \int_0^t \sum_\ell \tilde{G}\bigl(\tau,\ell,k(t-\tau)\bigr)\,\cdot \hat{R}(\tau,k-\ell)\,d\tau.
\end{align*}
At this level, the difference with respect to the beginning of this
section lies in the fact that there is a summation over $\ell\in\Z^d$,
instead of just choosing $\ell=0$. Note that 
$\hat{\sigma}(t,0) = \int_0^t \iint G(\tau,x,v)\cdot
R(\tau,x)\,dx\,dv\,d\tau =0$, because $G$ is a $v$-gradient.

From \eqref{compG} and \eqref{compR} we deduce
\begin{multline*}
\sum_k e^{2\pi (\lambda t+\mu)|k|}|\hat{\sigma}(t,k)|\\
\lesssim \int_0^t (1+\tau)\sum_{\ell\neq k,\ k\neq 0}
e^{2\pi \mu|k|}\, e^{2\pi\lambda t|k|}\, e^{-2\pi \ov{\mu}|\ell|}\,
e^{-2\pi \ov{\lambda}|k(t-\tau)+\ell\tau|}\,
e^{-2\pi \mu|k-\ell|}\,\frac{e^{-2\pi \lambda \tau |k-\ell|}}{1+|k-\ell|^\gamma}.
\end{multline*}
Using the inequalities
\[ e^{-2\pi \mu|k-\ell|}\,e^{2\pi\mu|k|}\,e^{-2\pi\ov{\mu}|\ell|}\leq e^{-2\pi (\ov{\mu}-\mu)|\ell|}\]
and 
\[ e^{-2\pi \lambda \tau |k-\ell|}\,e^{2\pi\lambda t|k|}\,e^{-2\pi\ov{\lambda}|k(t-\tau)+\ell\tau|}\leq
e^{-2\pi(\ov{\lambda}-\lambda)|k(t-\tau)+\ell\tau|},\]
we end up with
\[ \|\sigma(t)\|_{\cF^{\lambda t+\mu}}\lesssim
\sum_{k\neq 0,\ \ell\neq k}
\frac{e^{-2\pi (\ov{\mu}-\mu)|\ell|}}{1+|k-\ell|^\gamma}
\int_0^t e^{-2\pi (\ov{\lambda}-\lambda)|k(t-\tau)+\ell\tau|}\,(1+\tau)\,d\tau.\]
If it were not for the negative exponential, the time-integral would be $O(t^2)$ as $t\to\infty$.
The exponential helps only a bit: its argument vanishes e.g. for $d=1$, $k>0$, $\ell<0$ and $\tau=(k/(k+|\ell|))t$.
Thus we have the essentially optimal bounds
\begeq\label{expbound1}
\int_0^t e^{-2\pi (\ov{\lambda}-\lambda)|k(t-\tau)+\ell\tau|}\,d\tau\leq
\frac1{\pi(\ov{\lambda}-\lambda)\,|k-\ell|}
\endeq
and 
\begeq\label{expbound2}
\int_0^t e^{-2\pi (\ov{\lambda}-\lambda)|k(t-\tau)+\ell\tau|}\,\tau\,d\tau \leq
\frac1{2\pi^2 (\ov{\lambda}-\lambda)^2\, |k-\ell|^2} + \left(\frac1{\pi(\ov{\lambda}-\lambda)}\right)\,
\frac{|k|t}{|k-\ell|}.
\endeq

From this computation we conclude that:
\sm

\bul The higher regularity of $G$ has allowed to reduce the
time-integral thanks to a factor $e^{-\alpha |k(t-\tau)+\ell\tau|}$;
but this factor is not small when $\tau/t$ is equal to $k/(k-\ell)$.
As discussed in the next section, this reflects an important physical
phenomenon called (plasma) echo, which can be assimilated to a resonance.  \sm

\bul If we had (in ``gliding'' norm) $\|G_\tau\|=O(1)$ this would
ensure a uniform bound on the integral, as soon as $\gamma>0$, thanks
to \eqref{expbound1} and
\[ \sum_{k,\ell} \frac{e^{-\alpha|\ell|}}{(1+|k-\ell|)^{1+\gamma}} < +\infty.\]
\sm

\bul But $G_\tau$ is a velocity-gradient, so --- unless of course $G$
depends only on $v$ --- $\|G_\tau\|$ diverges like $O(\tau)$ as
$\tau\to\infty$, which implies a divergence of our bounds in large
time, as can be seen from \eqref{expbound2}. If $\gamma\leq 1$ this
comes with a divergence in the $k$ variable, since in this case
$\sum_{k,\ell}
\frac{e^{-\alpha|\ell|}|k|}{(1+|k-\ell|)^{1+\gamma}}=+\infty$. (The
Coulomb case corresponds to $\gamma=1$, so in this respect it has a
borderline divergence.)  \med

The following estimate adapts this computation to the formalism of
hybrid norms, and at the same time allows a time-cheating similar to
the one in Proposition \ref{proprt}.  Fortunately, we shall only need
to treat the case when $R=R(\tau,x)$; the more general case with
$R=R(\tau,x,v)$ would be much more tricky.

\begin{Thm}[Long-term regularity extortion]
\label{thmcombi'}
Let $G=G(\tau,x,v)$, $R=R(\tau,x)$, and
\[ \sigma(t,x) = \int_0^t \int_{\R^d}
G\bigl(\tau,x-v(t-\tau),v\bigr)\cdot
R\bigl(\tau,x-v(t-\tau)\bigr)\,dv\,d\tau.\] Let
$\lambda,\ov{\lambda},\mu,\ov{\mu}$, $\mu'=\mu'(t,\tau)$, $M\geq 1$ such that
$(1+M)\lambda\geq \ov{\lambda}>\lambda> 0$, $\ov{\mu} \ge \mu'>\mu> 0$,
$\gamma\geq 0$ and $b=b(t,\tau)\geq 0$.
Then
\begin{equation}\label{K=RX}
\|\sigma(t,\cdot) \|_{\dot{\cF}^{\lambda t+\mu}} \leq \int_0^t
K_0^G(t,\tau)\, \|R_\tau\|_{\cF^\nu}\,d\tau + \int_0^t
K_1^G(t,\tau)\,\|R_\tau\|_{\cF^{\nu,\gamma}}\,d\tau,
\end{equation}
where
\begeq\label{nucombi} \nu = \max\left\{ \lambda\tau + \mu' -
  \frac{\lambda}{2}\,b(t-\tau) \, ; \ 0 \right\},
\endeq
\begeq\label{K0G} K_0^G(t,\tau) = e^{-2\pi
  \left(\frac{\ov{\lambda}-\lambda}{2}\right)(t-\tau)}\, \left\| \int
  G(\tau,x,\,\cdot\,)\,dx \right\|_{\cC^{\ov{\lambda}(1+b);1}},
\endeq
\begeq\label{K1G} K_1^G(t,\tau) = \sup_{0\leq\tau\leq t}
\left(\frac{\|G_\tau\|_{\cZ^{\ov{\lambda}(1+b),
        \ov{\mu}}_{\tau-bt/(1+b)}}}{1+\tau}\right)\,
K_1(t,\tau),
\endeq
\begeq\label{K1} K_1(t,\tau) = (1+\tau)\ \sup_{k\neq 0,\ \ell\neq 0}
\left( \frac{e^{-2\pi \left(\frac{\ov{\mu}-\mu}{2}\right)|\ell|}\,
    e^{-2\pi
      \left(\frac{\ov{\lambda}-\lambda}{2M}\right)|k(t-\tau)+\ell\tau|}\,
    e^{-2\pi \bigl[ (\mu'-\mu) + \frac{\lambda\,b}{2}(t-\tau)\bigr]\,
      |k-\ell|}} {1+|k-\ell|^\gamma}\right).
\endeq
\end{Thm}

\begin{Rk} It is essential in \eqref{K=RX} to separate the
  contribution of $\hat{G}(\tau,0,v)$ from the rest. Indeed, if we
  removed the restriction $\ell\neq 0$ in \eqref{K1} the kernel $K_1$
  would be too large to be correctly controlled in large time. What
  makes this separation reasonable is that, although in cases of
  application $G(\tau,x,v)$ is expected to grow like $O(\tau)$ in
  large time, the spatial average $\int G(\tau,x,v)\,dx$ is expected
  to be bounded. Also, we will not need to take advantage of the
  parameter $\gamma$ to handle this term.
\end{Rk}

\begin{proof}[Proof of Theorem \ref{thmcombi'}]
Without loss of generality we may assume that $G$ and $R$ are scalar-valued.
(E.g. choose $\|G\|=\sup_{1\leq i\leq d} \|G^i\|$, $\|R\|=\sum_i \|R^i\|$,
where $G=(G^1,\ldots,G^d)$, $R=(R^1,\ldots,R^d)$; then it suffices to bound
$G^i\,R^i$ for all $i$.) 
First we assume $\hat{G}(\tau,0,v)=0$, and we write as before
\[ \hat{\sigma}(t,k) = \int_0^t \left(
  \sum_{\ell\in\Z^d\setminus\{0\}}\int_{\R^d} \hat{G}(\tau,\ell,v)\,
  \hat{R}(\tau,k-\ell)\, e^{-2i\pi k\cdot
    v(t-\tau)}\,dv\right)\,d\tau,\] \begeq\label{sigma1}
|\hat{\sigma}(t,k)| \leq \int_0^t \sum_{\ell\neq 0} \left|
  \int\hat{G}(\tau,\ell,v)\, e^{-2i\pi k\cdot v(t-\tau)}\,dv\right|\,
|\hat{R}(\tau,k-\ell)|\,d\tau.
\endeq

Next we let $\tau'=\tau-b(t-\tau)$ and write
\begin{multline}\label{sigma2}
  e^{2\pi(\lambda t+\mu)|k|}\leq e^{-2\pi (\ov{\mu} -
    \mu)|\ell|}\,e^{-2\pi \lambda (\tau-\tau')|k-\ell|}\,
  e^{-2\pi (\mu'-\mu)|k-\ell|}\,e^{-2\pi (\ov{\lambda}-\lambda)\,|k(t-\tau')+\ell\tau'|}\\
  e^{2\pi\ov{\mu}|\ell|}\,e^{2\pi(\lambda\tau+\mu')|k-\ell|}\,
  e^{2\pi\ov{\lambda}|k(t-\tau')+\ell\tau'|}.
\end{multline}

Since $0\leq \ov{\lambda}-\lambda\leq M\lambda$, we have
\[ e^{-2\pi (\ov{\lambda}-\lambda)|k(t-\tau')+\ell\tau'|} \leq
e^{-2\pi\left(\frac{\ov{\lambda}-\lambda}{2M}\right)|k(t-\tau)+\ell\tau|}\,
e^{2\pi\frac{\lambda}{2}(\tau-\tau')|k-\ell|};\]
so \eqref{sigma2} implies
\begin{multline} \label{sigma3} e^{2\pi(\lambda t+\mu)|k|} \leq
  e^{-2\pi (\ov{\mu}-\mu)|\ell|}\,e^{-2\pi
    \left(\frac{\ov{\lambda}-\lambda}{2M}\right)
    |k(t-\tau)+\ell\tau|}\, e^{-2\pi \bigl[ (\mu'-\mu) +
    \frac{\lambda}{2}(\tau-\tau')\bigr]\,|k-\ell|}\\
  e^{2\pi \ov{\mu}|\ell|}\,e^{2\pi(\lambda\tau+\mu')|k-\ell|}\,
  \sum_{n\in\N_0^d}
  \frac{\bigl|(2i\pi\ov{\lambda})\,
    \bigl(k(t-\tau')+\ell\tau'\bigr)\bigr|^n}{n!}.
\end{multline}

For each $n\in\N_0^d$,
\begin{align*}
  & \frac{\Bigl|
    (2i\pi\ov{\lambda})\,\bigl(k(t-\tau')+\ell\tau'\bigr)\Bigr|^n}{n!}
  \left|\int \hat{G}(\tau,\ell,v)\,e^{-2i\pi k\cdot v(t-\tau)}\,dv\right|\\
  & = \frac{\ov{\lambda}^n}{n!}  \left| \int\hat{G}(\tau,\ell,v)\,
    \bigl[ 2i\pi (k(t-\tau')+\ell\tau')\bigr]^n\,
    e^{-2i\pi k\cdot v(t-\tau)}\,dv \right|\\
  & = \frac{\ov{\lambda}^n}{n!}  \left(\frac{t-\tau'}{t-\tau}\right)^n
  \left|\int \hat{G}(\tau,\ell,v)\, \Bigl[ (2i\pi) \Bigl(k(t-\tau) +
    \ell\tau' \left(\frac{t-\tau}{t-\tau'}\right)
    \Bigr)\Bigr]^n \, e^{-2i\pi k\cdot v(t-\tau)}\,dv \right|\\
  & = \frac{\ov{\lambda}^n}{n!}  \left(\frac{t-\tau'}{t-\tau}\right)^n
  \left|\int \hat{G}(\tau,\ell,v)\, \Bigl[ -\nabla_v + 2i\pi \ell\tau'
    \left(\frac{t-\tau}{t-\tau'}\right)
    \Bigr)\Bigr]^n \, e^{-2i\pi k\cdot v(t-\tau)}\,dv \right|\\
  & = \frac{\ov{\lambda}^n}{n!}  \left(\frac{t-\tau'}{t-\tau}\right)^n
  \left|\int \Bigl[ \nabla_v + 2i\pi \ell\tau'
    \left(\frac{t-\tau}{t-\tau'}\right)\Bigr]^n \,
    \hat{G}(\tau,\ell,v)\,
    e^{-2i\pi k\cdot v(t-\tau)}\,dv \right|\\
  & \leq \frac{\ov{\lambda}^n}{n!}
  \left(\frac{t-\tau'}{t-\tau}\right)^n \left\|\Bigl( \nabla_v + 2i\pi
    \ell\tau' \left(\frac{t-\tau}{t-\tau'}\right)\Bigr)^n \,
    \hat{G}(\tau,\ell,v)\right\|_{L^1(dv)}\\
  & = \frac{\ov{\lambda}^n (1+b)^n}{n!}  \left\|\Bigl( \nabla_v +
    2i\pi \ell\left(\tau- \frac{bt}{1+b}\right)\Bigr)^n \,
    \hat{G}(\tau,\ell,v)\right\|_{L^1(dv)}.
\end{align*}

Combining this with \eqref{sigma1} and \eqref{sigma3}, summing over
$k$, we deduce
\begin{align*}
  & \bigl\|\sigma(t,\,\cdot\,)\bigr\|_{\dot{\cF}^{\lambda t+\mu}}
  = \sum_{k\neq 0} e^{2\pi(\lambda t+\mu)|k|}\,|\hat{\sigma}(t,k)|\\
  & \leq \int_0^t \sum_{k \ell n} \left( \frac{e^{-2\pi
        (\ov{\mu}-\mu)|\ell|}\,e^{-2\pi
        \left(\frac{\ov{\lambda}-\lambda}{2M}\right)\,
        |k(t-\tau)+\ell\tau|}\, e^{-2\pi \bigl((\mu'-\mu)+
        \frac{\lambda}{2}(\tau-\tau')\bigr)|k-\ell|}}
    {1+|k-\ell|^\gamma}\right) e^{2\pi\ov{\mu}|\ell|}\,e^{2\pi \bigl[
    (\lambda\tau+\mu') - \frac{\lambda}{2} b(t-\tau)\bigr]|k-\ell|}\,
  \\[2mm]
  &\qquad\qquad \frac{\ov{\lambda}^n(1+b)^n}{n!}\,
  |\hat{R}(\tau,k-\ell)|\, \left\|\left(\nabla_v + 2i\pi \ell
      \left(\tau-\frac{bt}{1+b}\right)\right)^n\,
    \hat{G}(\tau,\ell,v)\right\|_{L^1(dv)} \,d\tau,
\end{align*}
and the desired estimate follows readily.  \med

Finally we consider the contribution of $\hat{G}(\tau,0,v) =\int
G(\tau,x,v)\,dx$. This is done in the same way, noting that
\[ \sup_{k\neq 0} \frac{e^{-2\pi (\ov{\lambda}-\lambda)|k|(t-\tau)}}{1+|k|^\gamma}
\leq e^{-2\pi (\ov{\lambda}-\lambda)(t-\tau)}.\]
\end{proof}

To conclude this section we provide a ``mode by mode'' variant of
Theorem \ref{thmcombi'}; this will be useful for very singular
interactions ($\gamma=1$ in Theorem \ref{thmmain}).

\begin{Thm}
\label{thmcombi'k}
Under the same assumptions as Theorem \ref{thmcombi'}, for all
$k\in\Z^d$ we have the estimate
\begin{multline}
  e^{2\pi(\lambda t+\mu)|k|}|\hat{\sigma}(t,k)|\leq
  \int_0^t K_0^G (t,\tau)\, \bigl(e^{2\pi\nu|k|}\, 
  |\hat{R}(\tau,k)|\bigr)\,d\tau \\
  + \int_0^t \sum_{\ell\in\Z^d}
  K_{k,\ell}^G(t,\tau)\,e^{2\pi\nu|k-\ell|}\,(1+|k-\ell|^\gamma)\,
  |\hat{R}(\tau,k-\ell)|\,d\tau,
\end{multline}
where $K_0^G$ is defined by \eqref{K0G}, $\nu$ by \eqref{nucombi}, and
\[ K_{k,\ell}^G(t,\tau) = \sup_{0\leq \tau\leq t}
\left(\frac{\|G\|_{\cZ^{\ov{\lambda}(1+b),\ov{\mu};1}_{\tau-bt/(1+b)}}}{1+\tau}\right)\,
K_{k,\ell}(t,\tau),\]
\[ K_{k,\ell}(t,\tau)
= \frac{(1+\tau)\,e^{-2\pi (\ov{\mu}-\mu)|\ell|}\,
e^{-2\pi \left(\frac{\ov{\lambda}-\lambda}{2M}\right)|k(t-\tau)+\ell\tau|}\,
e^{-2\pi \bigl[ (\mu'-\mu) + \frac{\lambda}{2}\,b(t-\tau)\bigr]\,|k-\ell|}}
{1+|k-\ell|^\gamma}.\]
\end{Thm}

\begin{proof}[Proof of Theorem \ref{thmcombi'k}]
  The proof is similar to the proof of Theorem \ref{thmcombi'}, except
  that $k$ is fixed and we use, for each $\ell$, the crude bound
  \begin{multline*} e^{2\pi\ov{\mu}|\ell|}\ \left\| \left[ \nabla_v +
        2i\pi\ell \left(\tau-\frac{bt}{1+b}\right)\right]^n
      \hat{G}(\tau,\ell,v)\right\|_{L^1(dv)} \\
    \leq \sum_{j\in\Z^d} e^{2\pi\ov{\mu}|j|}\, \left\| \left[ \nabla_v
        + 2i\pi j \left(\tau-\frac{bt}{1+b}\right)\right]^n
      \hat{G}(\tau,j,v)\right\|_{L^1(dv)}.
\end{multline*}
\end{proof}

\section{Control of the time-response}
\label{sec:response}

To motivate this section, let us start from the linearized
equation \eqref{Vllin}, but now assume that $f^0$ depends on $t,x,v$
and that there is an extra source term $S$, decaying in time. Thus the
equation is
\[ \derpar{f}{t} + v\cdot\nabla_x f - (\nabla W\ast \rho)\cdot\nabla_v
f^0 = S,\] 
and the equation for the density $\rho$, as in the proof of
Theorem \ref{thmlineardamping}, is
\begin{multline} \label{rhoagain} 
\rho(t,x) = \int_{\R^d} f_i(x-vt,v)\,dv
+ \int_0^t \int_{\R^d} \nabla_v f^0\bigl(\tau,x-v(t-\tau),v\bigr)\cdot (\nabla W\ast\rho)(\tau,x-v(t-\tau))\,dv\,d\tau\\
+ \int_0^t \int_{\R^d} S(\tau,x-v(t-\tau),v)\,dv\,d\tau.
\end{multline}

Hopefully we may apply Theorem \ref{thmcombi'} to deduce from
\eqref{rhoagain} an integral inequality on $\varphi(t):=
\|\rho(t)\|_{\cF^{\lambda t+\mu}}$, which will look like
\begeq\label{looklike} 
\varphi(t) \leq A + c \, \int_0^t
K(t,\tau)\,\varphi(\tau)\,d\tau,
\endeq
where $A$ is the contribution of the initial datum and the source
term, and $K(t,\tau)$ a kernel looking like, say, \eqref{K1}.

From \eqref{looklike} how do we proceed? Assume for a start that
a smallness condition of the form (a) in Proposition \ref{suffL} is
satisfied. Then the simple and natural way, as in Section
\ref{sec:linear}, would be to write
$$
\varphi(t) \leq A + c \, \left(\int_0^t K(t,\tau)\,d\tau \right) \, 
\left(\sup_{0\leq\tau\leq t} \varphi(\tau)\right),
$$ 
and deduce
\begin{equation}
\label{deducephi} \varphi(t) \leq \frac{A}{\left( 1 - c \, \int_0^t
    K(t,\tau)\,d\tau\right)}
\end{equation}
(assuming of course the denominator to be positive).
However, if $K$ is given by \eqref{K1}, it is easily seen that
$\int_0^t K(t,\tau)\,d\tau\geq \kappa\, t$ as $t\to\infty$, where
$\kappa>0$; then \eqref{deducephi} is useless. In fact
\eqref{looklike} {\em does not} prevent $\varphi$ from going to
$+\infty$ as $t\to\infty$. Nevertheless, its growth may be controlled
under certain assumptions, as we shall see in this section.  Before
embarking on cumbersome calculations, we shall start with a
qualitative discussion.

\subsection{Qualitative discussion}
\label{subqualit}

The kernel $K$ in \eqref{K1} depends on the choice of
$\mu'=\mu(t,\tau)$. How large $\mu'-\mu$ can be depends in turn on the
amount of regularization offered by the convolution with the
interaction $\nabla W$.  We shall distinguish several cases according to
the regularity of the interaction.

\subsubsection{Analytic interaction} \label{analint}

If $\nabla W$ is analytic, there is $\sigma>0$ such that
\[\forall\nu\geq 0,\qquad \|\rho\ast\nabla W\|_{\cF^{\nu+\sigma}} \leq C\, \|\rho\|_{\cF^\nu};\]
then in \eqref{K1} we can afford to choose, say, $\mu'-\mu=\sigma$,
and $\gamma=0$.  Thus, assuming $b=B/(1+t)$ with $B$ small
enough so that $(\mu'-\mu) - \lambda b (t-\tau) \ge \sigma/2$, 
$K$ is bounded by 
\begeq\label{Ka} \ov{K}^{(\alpha)}(t,\tau) = (1+\tau)\,
\sup_{k\neq 0,\ \ell\neq 0}
e^{-\alpha|\ell|}\,e^{-\alpha|k-\ell|}\,e^{-\alpha
  |k(t-\tau)+\ell\tau|},
\endeq
where $\alpha = \frac12 \min \{ \ov{\lambda}-\lambda\, ; \,
\ov{\mu}-\mu \, ; \, \sigma\}$. To fix ideas, let us work in dimension $d=1$.
The goal is to estimate solutions of
\begeq\label{phiKa}
\varphi(t) \leq a + c \int_0^t \ov{K}^{(\alpha)}(t,\tau)\,\varphi(\tau)\,d\tau.
\endeq

Whenever $\tau/t$ is a rational number distinct from $0$ or $1$, there
are $k,\ell\in\Z$ such that $|k(t-\tau)+\ell\tau|=0$, and the size of
$\ov{K}^{(\alpha)}(t,\tau)$ mainly depends on the minimum admissible
values of $k$ and $k-\ell$. Looking at values of $\tau/t$ of the form
$1/(n+1)$ or $n/(n+1)$ suggests the approximation
\begeq\label{approxKa} 
\ov{K}^{(\alpha)} \lesssim (1+\tau) \, \min \ \Bigl\{
e^{-\alpha \left(\frac{\tau}{t-\tau}\right)}\,e^{-2\alpha} \, ; \
e^{-2\alpha\left(\frac{t-\tau}{\tau}\right)}\,e^{-\alpha}\Bigr\}.
\endeq 
But this estimate is terrible: the time-integral of the right-hand side is much
larger than the integral of $\ov{K}^{(\alpha)}$. In fact, the fast
variation and ``wiggling'' behavior of $\ov{K}^{(\alpha)}$ are
essential to get decent estimates.  \bigskip

\begin{figure}[htbp]
\hspace*{-20mm}
\begin{minipage}[t]{.3\linewidth}
  \includegraphics[height=4cm]{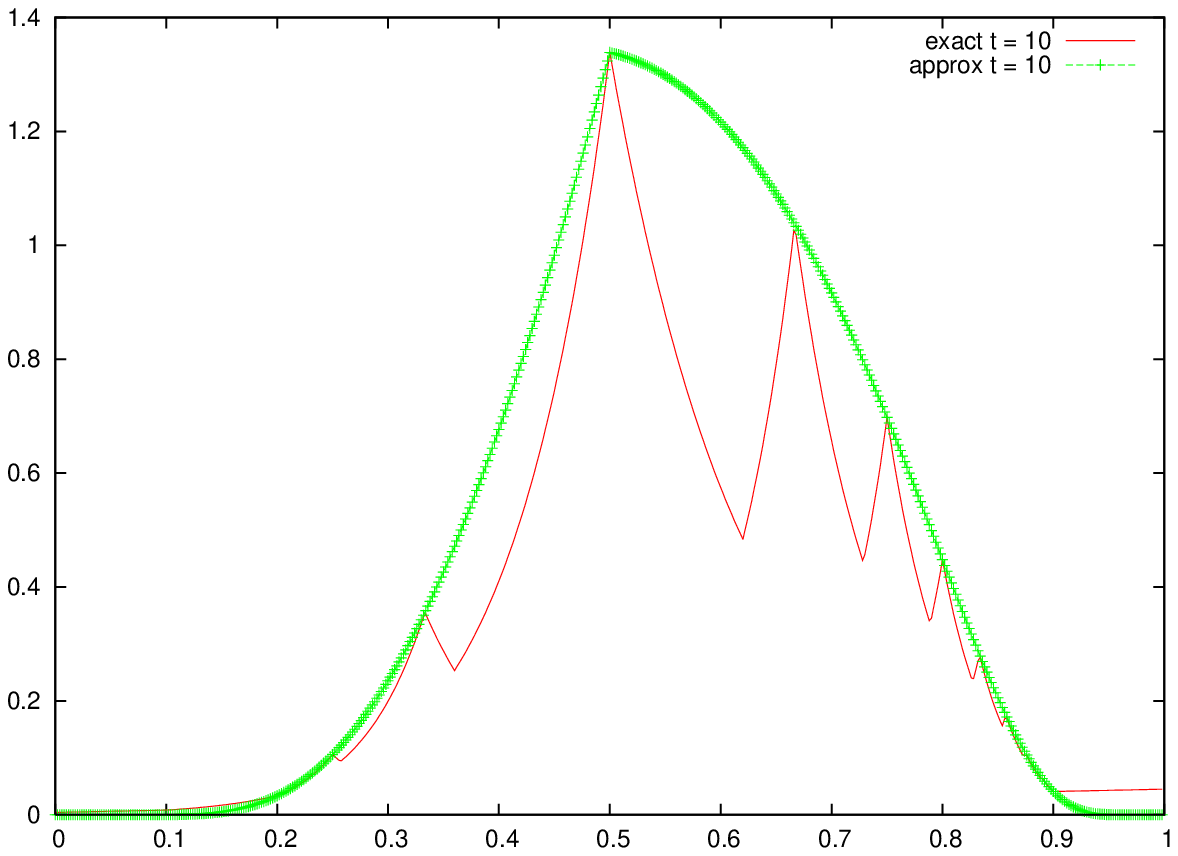}
\end{minipage}\hfill
\begin{minipage}[t]{.3\linewidth}
  \includegraphics[height=4cm]{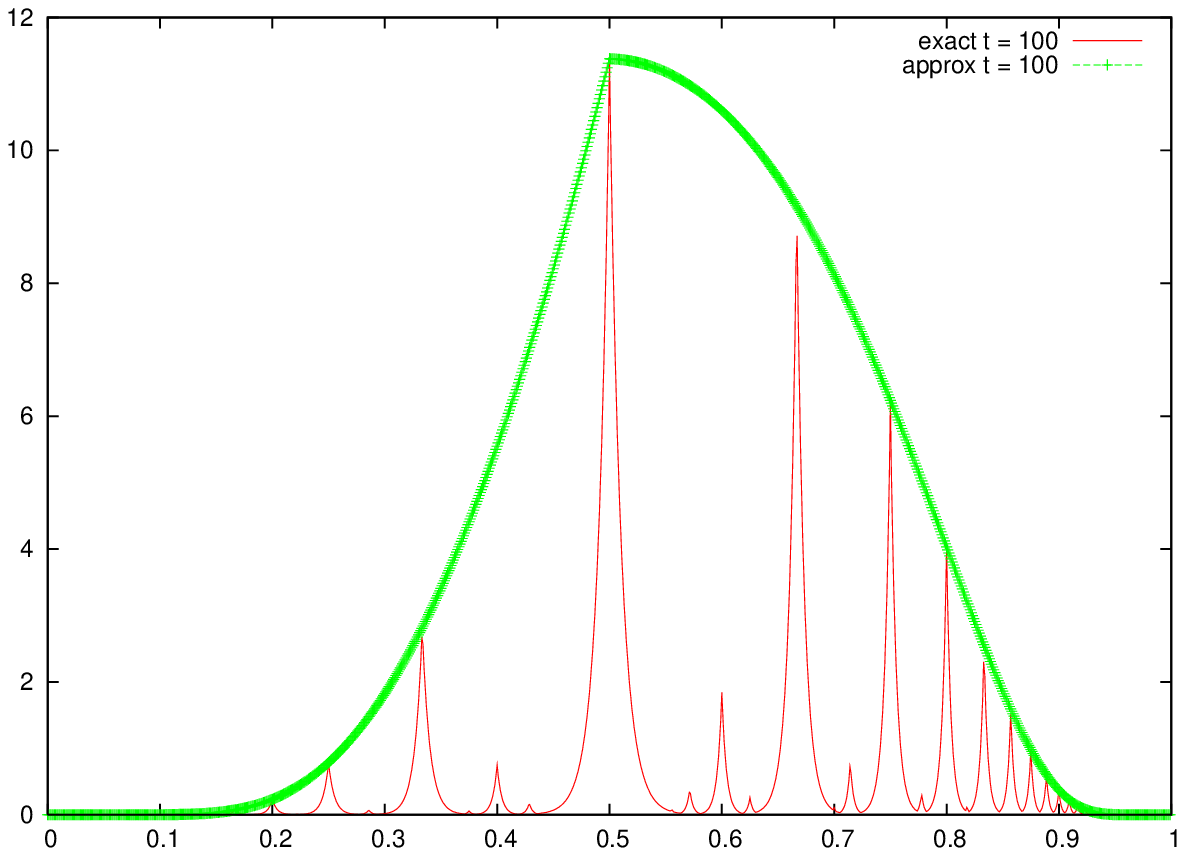}
\end{minipage}\hfill
\begin{minipage}[t]{.3\linewidth}
  \includegraphics[height=4cm]{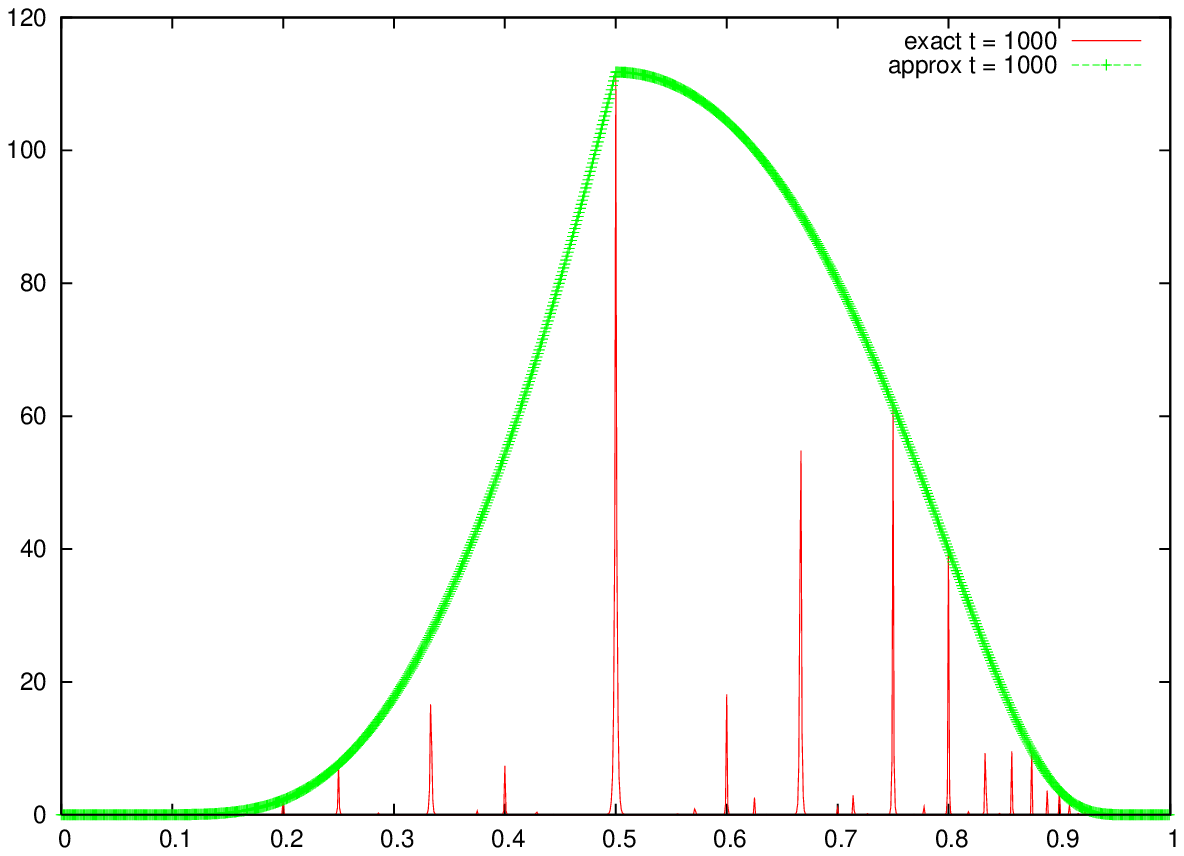}
\end{minipage}\hfill \hspace*{1mm}
\caption{the kernel $\ov{K}^{(\alpha)}(t,\tau)$, together with the approximate upper bound in \eqref{approxKa},
for $\alpha=0.5$ and $t=10$, $t=100$, $t=1000$.}
\label{figK}
\end{figure}

To get a better feeling for $\ov{K}^{(\alpha)}$, let us only retain
the term in $k=1$, $\ell=-1$; this seems reasonable since we have an
exponential decay as $k$ or $\ell$ go to infinity (anyway, throwing
away all other terms can only improve the estimates).  So we look at
$\tilde{K}^{(\alpha)}(t,\tau) = (1+\tau)\,
e^{-3\alpha}\,e^{-\alpha|t-2\tau|}$.  Let us time-rescale by setting
$k_t(\theta) = t\,\tilde{K}^{(\alpha)}(t,t\theta)$ for $\theta\in
[0,1]$ (the $t$ factor because $d\tau = t\, d\theta$); then it is not
hard to see that
\[ 
\frac{k_t}{t} \longrightarrow \frac{e^{-3\alpha}}{2 \, \alpha}\,
\delta_{\frac12}
\]
This suggests the following baby model for \eqref{phiKa}:
\begeq\label{babya} \varphi(t) \leq a + c\, t\,
\varphi\left(\frac{t}{2}\right).
\endeq

The important point in \eqref{babya} is that, although the kernel has
total mass $O(t)$, this mass is located far from the endpoint
$\tau=t$; this is what prevents the fast growth of $\varphi$. Compare
with the inequality $\varphi(t) \leq a + c \, t\, \varphi(t)$, which
implies no restriction at all on $\varphi$.

To be slightly more quantitative, let us look for a power series
$\Phi(t) = \sum_k a_k \, t^k$ achieving equality in
\eqref{babya}. This yields $a_0=a$, $a_{k+1} = c\, a_k \, 2^{-k}$, so
\begeq\label{Phi} \Phi(t) = a \sum_{k=0}^\infty
\frac{c^k\,t^k}{2^{k(k-1)/2}}.
\endeq
The function $\Phi$ exhibits a truly remarkable behavior: it grows
faster than any polynomial, but slower than any fractional exponential
$\exp(c\,t^\nu)$, $\nu \in (0,1)$; essentially it behaves like
$A^{(\log t)^2}$ (as can also be seen directly from \eqref{babya}).
One may conjecture that solutions of \eqref{phiKa} exhibit a similar
kind of growth.  \med

Let us interpret these calculations. Typically, the kernel $K$
controls the time variation of (say) the spatial density $\rho$ which
is due to binary interaction of waves. When two waves of distinct
frequencies interact, the effect over a long time period is most of
the time very small; this is a consequence of the oscillatory nature
of the evolution, and the resulting time-averaging. But at certain
particular times, the interaction becomes strong: this is known in
plasma physics as the {\bf plasma echo}, and can be thought of as a
kind of resonance. Spectacular experiments by Malmberg and
collaborators are based on this effect \cite{echo,echo:expe}. Namely, if one
starts a wave at frequency $\ell$ at time~0, and forces it at time
$\tau$ by a wave of frequency $k-\ell$, a strong response is obtained
at time $t$ and frequency $k$ such that
\begeq\label{dispersion} k(t-\tau) + \ell\tau = 0
\endeq
(which of course is possible only if $k$ and $\ell$ are parallel to each
other, with opposite directions).

In the present nonlinear setting, whatever variation the density
function is subject to, will result in echoes at later times. Even if
each echo in itself will eventually decay, the problem is whether the
accumulation of echoes will trigger an uncontrolled growth
(unstability). As long as the expected growth is eaten by the
time-decay coming from the linear theory, nonlinear Landau damping is expected.
In the present case, the growth of \eqref{Phi} is very slow in regard of the exponential time-decay 
due to the analytic regularity.

\subsubsection{Sobolev interaction} \label{sobolint}

If $\nabla W$ only has Sobolev regularity, we cannot afford in
\eqref{K1} to take $\mu'(t,\tau)$ larger than $\mu + \eta
(t-\tau)/t$ (because the amount of regularity transferred in the
bilinear estimates is only $O((t-\tau)/t)$, recall the discussion at
the end of Subsection \ref{sub:rt}). On the other hand, we have
$\gamma>0$ such that
\[ \forall \nu\geq 0,\qquad \|\nabla W\ast \rho\|_{\cF^{\nu,\gamma}}
\leq C\, \|\rho\|_{\cF^\nu},\] and then we can choose this $\gamma$ in
\eqref{K1}. So, assuming $b=B/(1+t)$ with $B$ small enough
  so that $(\mu'-\mu) - \lambda b (t-\tau) \ge \eta (t-\tau)/(2t)$,
$K$ in \eqref{K1} will be controlled by \begeq\label{hKa}
\nohat{K}^{(\alpha),\gamma}(t,\tau) = (1+\tau) \sup_{k\neq 0,\
  \ell\neq 0} \frac{e^{-\alpha|\ell|}\,
  e^{-\alpha\left(\frac{t-\tau}{t}\right)|k-\ell|}\, e^{-\alpha
    |k(t-\tau)+\ell\tau|}} {1+|k-\ell|^\gamma},
\endeq
where $\alpha = \frac12 \min \{ \ov{\lambda}-\lambda\, ; \,
\ov{\mu}-\mu \, ; \, \eta \}$. The equation we are considering now is
\begeq\label{phiKb} \varphi(t) \leq a + \int_0^t
\nohat{K}^{(\alpha),\gamma}(t,\tau)\,\varphi(\tau)\,d\tau.
\endeq

For, say, $\tau\leq t/2$, we have $\nohat{K}^{(\alpha)}\leq
\ov{K}^{(\alpha/2)}$, and the discussion is similar to that in
\ref{analint}.  But when $\tau$ aproaches $t$, the term
$\exp\bigl(-\alpha (\frac{t-\tau}{t})\,|k-\ell|\bigr)$ hardly helps.
Keeping only $k>0$ and $\ell=-1$ (because of the exponential decay in
$\ell$) leads to consider the kernel
\[ \check{K}^{(\alpha)} (t,\tau) = (1+\tau)\, \sup_{k\neq 0}\
\frac{e^{-\alpha |kt - (k+1)\tau|}}{1+(k+1)^\gamma}.\] 
Once again we
perform a time-rescaling, setting $\check{k}_t(\theta) =
t\,\check{K}^{(\alpha)}(t,t\theta)$, and let $t\to\infty$. In this
limit each exponential $\exp(-\alpha |kt-(k+1)\tau|)$ becomes
localized in a neighborhood of size $O(1/k\,t)$ around
$\theta=k/(k+1)$, and contributes a Dirac mass at $\theta=k/(k+1)$,
with amplitude $2/(\alpha (k+1))$;
  \[ \frac{\check{k}_t}{t} \xrightarrow[t\to\infty]{} \frac{2}{\alpha}
  \sum_k \frac1{1+(k+1)^\gamma}\, \frac{k}{(k+1)^2}\,
  \delta_{1-\frac1{k+1}}.\]

This leads us to the following baby model for \eqref{phiKb}:
\begeq\label{babyb}
\varphi(t) \leq a + c\,t\, \sum_{k\geq 1} \frac1{k^{1+\gamma}}\,\varphi\left( \left(1-\frac1{k}\right)t\right).
\endeq
If we search for $\sum a_n t^n$ achieving equality, this yields
\[ a_0=a, \qquad a_{n+1} = c \left(\sum_{k\geq 1} \frac1{k^{1+\gamma}}\, \left(1-\frac1{k}\right)^n\right)\,a_n.\]
To estimate the behavior of the $\sum_k$ above, we compare it with
\begin{align*}
\int_1^\infty \frac1{t^{1+\gamma}}\, \left(1-\frac1{t}\right)^n\,dt
& = \int_0^t u^{\gamma-1}\, (1-u)^n\,du = B(\gamma,n+1)\qquad \text{(Beta function)}\\
& = \frac{\Gamma(\gamma)\,\Gamma(n+1)}{\Gamma(n+\gamma+1)} = O\left(\frac1{n^\gamma}\right).
\end{align*}
All in all, we may expect $\varphi$ in \eqref{phiKb} to behave
qualitatively like
\[ \Phi(t) = a\, \sum_{n\geq 0} \frac{c^n\,t^n}{(n!)^\gamma}.\] Notice
that $\Phi$ is subexponential for $\gamma>1$ (it grows essentially
like the fractional exponential $\exp(t^{1/\gamma})$) and exponential
for $\gamma=1$. In particular, as soon as $\gamma>1$ we expect nonlinear Landau damping again.

\subsubsection{Coulomb/Newton interaction ($\gamma=1$)} \label{subheurcoul}

When $\gamma=1$, as is the case for Coulomb or Newton interaction, the
previous analysis becomes borderline since we expect \eqref{babyb} to
be compatible with an exponential growth, and the linear decay is also
exponential. To handle this more singular case, we shall work mode by
mode, rather than on just one norm.  Starting again from
\eqref{rhoagain}, we consider, for each $k\in\Z^d$,
\[ \varphi_k(t) = e^{2\pi(\lambda t+\mu)|k|}\,|\hat{\rho}(t,k)|,\] and
hope to get, {\it via} Theorem \ref{thmcombi'k}, an inequality which will
roughly take the form \begeq\label{phikt} \varphi_k(t) \leq A_k +
c\int_0^t \sum_\ell K_{k,\ell}(t,\tau)\,\varphi_{k-\ell}(\tau)\,d\tau.
\endeq
(Note: summing in $k$ would yield an inequality worse than
\eqref{phiKb}.) To fix the ideas, let us work in dimension $d=1$, and
set $k\geq 1$, $\ell = -1$. Reasoning as in subsection
\eqref{sobolint}, we obtain the baby model \begeq\label{baby1}
\varphi_k(t) \leq A_k +
\frac{c\,t}{(k+1)^{1+\gamma}}\,\varphi_{k+1}\left(\frac{kt}{k+1}\right).
\endeq
The gain with respect to \eqref{babyb} is clear: for different values
of $k$, the ``dominant times'' are distinct. From the physical point
of view, we are discovering that, in some sense, {\em echoes occurring
  at distinct frequencies are asymptotically well separated}.

Let us search again for power series solutions: we set
\[ \varphi_k(t) = \sum_{m\geq 0} a_{k,m}\,t^m,\qquad a_{k,0} = A_k.\]
By identification, $a_{k,m} = a_{k+1,m-1} \, c\, (k+1)^{-(1+\gamma)} \,
(k/(k+1))^{m-1}$, and by induction
\[ a_{k,m} = A_{k+m}\, c^m\,\left[
  \frac{k!}{(k+m)!}\right]^{1+\gamma}\, \frac{k^{m-1}\,c^m}{(k+1)
  (k+2)\ldots (k+m)}\simeq A_{k+m}\,
\left[\frac{k!}{(k+m)!}\right]^{\gamma+2}\, k^{m-1}\,c^m.\]
We may expect $A_{k+m}\lesssim A\,e^{-a(k+m)}$; then
\[ a_{k,m} \lesssim A\,(k\,e^{-ak})\,k^m\,c^m\,\frac{e^{-am}}{(m!)^{\gamma+2}},\]
and in particular
\[ \varphi_k(t) \lesssim A\,e^{-ak/2}\,
\sum_m \frac{(c k t)^m}{(m!)^{\gamma+2}}\lesssim A\,e^{(1-\alpha)\,(ckt)^\alpha},
\qquad \alpha = \frac{1}{\gamma+2}.\]
This behaves like a fractional exponential even for $\gamma=1$, and we
can now believe in nonlinear Landau damping for such interactions! 
(The argument above works even for more singular interactions;
but in the proof later the condition $\gamma\geq 1$ will be required for
other reasons, see pp.~\pageref{gg1'} and \pageref{gammageq1}.)

\subsection{Exponential moments of the kernel} \label{subsexp}

Now we start to estimate the kernel $\nohat{K}^{(\alpha),\gamma}$ from
\eqref{hKa}, without any approximation this time. Eventually, instead
of proving that the growth is at most fractional exponential, we shall
compare it with a slow exponential $e^{\var t}$.  For this, the first
step consists in estimating exponential moments of the kernel
$e^{-\var t}\int K(t,\tau)\,e^{\var \tau}\,d\tau$. (To get more precise
estimates, one can study $e^{-\var t^\alpha} \int
K(t,\tau)\,e^{\var\tau^\alpha}\,d\tau$, but such a refinement is not
needed for the proof of Theorem \ref{thmmain}.)

The first step consists in estimating exponential moments.

\begin{Prop}[Exponential moments of the kernel] \label{propexpK} Let
  $\gamma\in [1,\infty)$ be given. For any $\alpha\in (0,1)$, let
  $\nohat{K}^{(\alpha),\gamma}$ be defined by \eqref{hKa}.
Then for any $\gamma<\infty$ there is $\ov{\alpha}=\ov{\alpha}(\gamma)>0$ such
  that if $\alpha\leq \ov{\alpha}$ and $\var \in (0,1)$, then for any
  $t>0$,
  \begin{multline*} e^{-\var t} \int_0^t \nohat{K}^{(\alpha),\gamma}(t,\tau)\,e^{\var \tau}\,d\tau\\
    \leq C \left( \frac{1}{\alpha\,\var^{\gamma}\,t^{\gamma-1}} +
      \frac{\ln \frac1{\alpha}} {\alpha\,\var^\gamma\,t^{\gamma}} +
      \frac{1}{\alpha^2\,\var^{1+\gamma}\,t^{1+\gamma}} + \left(
        \frac1{\alpha^3} \, + \frac{\ln
          \frac1{\alpha}}{\alpha^2\var}\right) \,
      e^{-\frac{\var\,t}{4}} +
      \frac{e^{-\frac{\alpha\,t}2}}{\alpha^3}\right),
\end{multline*}
where $C=C(\gamma)$.
In particular,

\bul If $\gamma>1$ and $\var\leq\alpha$, then $\dps e^{-\var t}
\int_0^t \nohat{K}^{(\alpha),\gamma}(t,\tau)\,e^{\var \tau}\,d\tau
\leq \frac{C(\gamma)}{\alpha^3\,\var^{1+\gamma}\,t^{\gamma-1}}$; \sm

\bul If $\gamma=1$ then $\dps e^{-\var t} \int_0^t
\nohat{K}^{(\alpha),\gamma}(t,\tau)\,e^{\var \tau}\,d\tau \leq
\frac{C}{\alpha^3}\left( \frac1{\var} + \frac1{\var^2\,t}\right)$.
\end{Prop}

\begin{Rk} Much stronger estimates can be obtained if the interaction is analytic;
that is, when $K^{(\alpha),\gamma}$ is replaced by $\ov{K}^{(\alpha)}$ defined in \eqref{Ka}.
A notable point about Proposition \ref{propexpK} is that for $\gamma=1$ we do not have
any time-decay as $t\to\infty$.
\end{Rk}

\begin{proof}[Proof of Proposition \ref{propexpK}] 
  To simplify notation we shall not recall the dependence of
  $\nohat{K}$ on $\gamma$.  We first assume $\gamma<\infty$, and
  consider $\tau\leq t/2$, which is the favorable case. We write
  \[ \nohat{K}^{(\alpha)}(t,\tau)\leq (1+\tau)\, \sup_{k\neq 0}\,
  \sup_{\ell}\,
  e^{-\alpha|\ell|}\,e^{-\frac{\alpha}{2}|k-\ell|}\,e^{-\alpha|k(t-\tau)+\ell\tau|}.\]
  Since we got rid of the condition $\ell\neq 0$, the right-hand side
  is now a nonincreasing function of $d$.  (To see this, pick up a
  nonzero component of $k$, and recall our norm conventions from
  Appendix \ref{app:exp}.) So we assume $d=1$. By symmetry we may also
  assume $k>0$.

Explicit computations yield
\[ \int_0^{t/2} e^{-\alpha |k(t-\tau)+\ell\tau|}\,(1+\tau)\,d\tau
\leq \begin{cases}
  \dps \frac1{\alpha\, (\ell-k)} + \frac1{\alpha^2\,(\ell-k)^2}\qquad \text{if $\ell>k$}\\[4mm]
  \dps e^{-\alpha kt}\, \left(\frac{t}{2} + \frac{t^2}{8}\right)\qquad \text{if $\ell=k$}\\[4mm]
  \dps \frac{e^{-\alpha \left(\frac{k+\ell}{2}\right)t}}{\alpha
    |k-\ell|}\, \left(1+\frac{t}2\right)
  \qquad \text{if $-k\leq\ell<k$}\\[4mm]
  \dps \left(\frac2{\alpha|k-\ell|} + \frac{2 \, kt}{\alpha
      |k-\ell|^2} + \frac1{\alpha^2 |k-\ell|^2}\right) \qquad \text{if
    $\ell<-k$}.
\end{cases}\] 
In all cases,
\begin{multline*}
  \int_0^{t/2} e^{-\alpha |k(t-\tau)+\ell\tau|}\,(1+\tau)\,d\tau \leq
  \left( \frac3{\alpha |k-\ell|} + \frac1{\alpha^2 |k-\ell|^2}
    + \frac{2 \, t}{\alpha|k-\ell|} \right)\,1_{k\neq\ell}\\
  + e^{-\alpha kt}\, \left(\frac{t}{2} +
    \frac{t^2}{8}\right)\,1_{\ell=k}.
\end{multline*}
So
\begin{align*}
  e^{-\var t} & \int_0^{t/2} e^{-\alpha |k(t-\tau)+\ell\tau|}\,(1+\tau)\,e^{\var \tau}\,d\tau\\
  & \leq e^{-\frac{\var t}{2}}\, \left(\frac3{\alpha |k-\ell|} +
    \frac1{\alpha^2 |k-\ell|^2} + \frac{2 \, t}{\alpha|k-\ell|}
  \right)\,1_{k\neq \ell}
  + e^{-\alpha k t}\,\left(\frac{t}2 + \frac{t^2}8\right)\,1_{\ell=k}\\
  & \leq e^{-\frac{\var t}{4}}\, \left( \frac3{\alpha |k-\ell|} +
    \frac1{\alpha^2 |k-\ell|^2} + \frac{8 \, z}{\alpha\var |k-\ell|}
  \right)\,1_{k\neq \ell} + e^{-\frac{t\alpha}{2}}
  \left(\frac{z}{\alpha} + \frac{8 \, z^2}{\alpha^2}\right)\,1_{\ell
    =k},
\end{align*}
where $z=\sup(xe^{-x}) =e^{-1}$. Then
\begin{multline*} 
  e^{\var t} \int_0^{t/2} \nohat{K}^{(\alpha)}(t,\tau)\,e^{\var
    \tau}\,d\tau \\
  \leq e^{-\frac{\var t}{4}} \sum_{k\neq 0} \sum_{\ell\neq k}
  e^{-\alpha |\ell|}\,e^{-\frac{\alpha}{2} |k-\ell|}\,
  \left(\frac{3}{\alpha |k-\ell|} + \frac{1}{\alpha^2 |k-\ell|^2}
    + \frac{8 \, z}{\alpha \var |k-\ell|} \right)\\
  + e^{-\frac{t\alpha}{2}} \sum_\ell e^{-\alpha|\ell|}
  \left(\frac{z}{\alpha} + \frac{8 z^2}{\alpha^2}\right).
\end{multline*}
Using the bounds (for $\alpha \sim 0^+$) 
\[ 
\sum_\ell e^{-\alpha \ell} = O\left( \frac{1}{\alpha}\right),\quad
\sum_\ell \frac{e^{-\alpha \ell}}{\ell} = O \left(\ln
  \frac1{\alpha}\right),\quad \sum_\ell \frac{e^{-\alpha
    \ell}}{\ell^2} = O(1),\] we end up, for $\alpha\leq 1/4$, with a
bound like
\begin{multline*}
  C\,e^{-\frac{\var t}{4}} \, \left( \frac1{\alpha^2}\ln
    \frac1{\alpha} + \frac1{\alpha^3} +
    \frac1{\alpha^2\var}\ln\frac1{\alpha} \right)
  + C\, e^{-\frac{\alpha t}{2}}\, \left(\frac1{\alpha^2}+\frac1{\alpha^3}\right)\\
  \leq C\left[ e^{-\frac{\var t}{4}}\, \left(\frac1{\alpha^3} +
      \frac1{\alpha^2\var}\,\ln\frac1{\alpha}\right) +
    \frac{e^{-\frac{\alpha t}2}}{\alpha^3}\right].
\end{multline*}
(Note that the last term is $O(t^{-3})$, so it is anyway negligible in
front of the other terms if $\gamma\leq 4$; in this case the
restriction $\var\leq\alpha$ can be dispended with.)  \sm

\bul Next we turn to the more delicate contribution of $\tau\geq
t/2$. For this case we write \begeq\label{Kbdt2}
\nohat{K}^{(\alpha)}(t,\tau) \leq (1+\tau) \,\sup_{\ell\neq 0}\,
e^{-\alpha |\ell|}\, \sup_k\, \frac{e^{-\alpha
    |k(t-\tau)+\ell\tau|}}{1+|k-\ell|^\gamma},
\endeq
and the upper bound is a nonincreasing function of $d$, so we assume
$d=1$.  Without loss of generality we restrict the supremum to
$\ell>0$.

The function $x\longmapsto (1+|x-\ell|^\gamma)^{-1}\,e^{-\alpha
  |x(t-\tau)+\ell\tau|}$ is decreasing for $x\geq \ell$, increasing
for $x\leq -\ell\tau/(t-\tau)$; and on the interval
$[-\frac{\ell\tau}{t-\tau},\ell]$ its logarithmic derivative goes from
$$
\left(-\alpha + \frac{\frac{\gamma}{\ell t}}
{1+\left(\frac{(t-\tau)}{\ell t}\right)^\gamma}\right) 
\, (t-\tau) \qquad \mbox{ to } -\alpha(t-\tau).
$$ 
So if $t\geq \gamma/\alpha$ there is a unique maximum at
$x=-\ell\tau/(t-\tau)$, and the supremum in \eqref{Kbdt2} is achieved
for $k$ equal to either the lower integer part, or the upper integer
part of $-\ell \tau/(t-\tau)$. Thus a given integer $k$ occurs in the
supremum only for some times $\tau$ satisfying $k-1 <
-\ell\tau/(t-\tau)<k+1$.  Since only negative values of $k$ occur, let
us change the sign so that $k$ is nonnegative.  The equation
\[ k-1 < \frac{\ell\tau}{t-\tau} < k+1\]
is equivalent to
\[ \left(\frac{k-1}{k+\ell-1}\right)\,t < \tau <
\left(\frac{k+1}{k+\ell+1}\right)\,t.\] Moreover, $\tau>t/2$ implies
$k\geq \ell$. Thus, for $t\geq \gamma/\alpha$ we have
\begeq\label{tba} e^{-\var t} \int_{t/2}^t
\nohat{K}^{(\alpha)}(t,\tau)\,e^{\var \tau}\,d\tau \leq e^{-\var t}
\sum_{\ell\geq 1} e^{-\alpha \ell} \sum_{k\geq \ell}
\int_{\left(\frac{k-1}{k+\ell-1}\right)t}^{\left(\frac{k+1}{k+\ell+1}\right)t}
(1+\tau)\, \frac{e^{-\alpha |k(t-\tau)-\ell\tau|}\,e^{\var
    \tau}}{1+(k+\ell)^\gamma}\,d\tau.
\endeq

For $t\leq \gamma/\alpha$ we have the trivial bound
\[ e^{-\var t} \int_{t/2}^t
\nohat{K}^{(\alpha)}(t,\tau)\,e^{\var\tau}\,d\tau \leq
\frac{\gamma}{2\alpha};\] so in the sequel we shall just focus on the
estimate of \eqref{tba}.

\med

To evaluate the integral in the right-hand side of \eqref{tba}, we
separate according to whether $\tau$ is smaller or larger than
$kt/(k+\ell)$; we use trivial bounds for $e^{\var\tau}$ inside the
integral, and in the end we get the explicit bounds
$$
e^{-\var t} \int_{
  \left(\frac{k-1}{k+\ell-1}\right)t}^{\left(\frac{k}{k+\ell}\right)t}
(1+\tau)\,e^{-\alpha |k(t-\tau)-\ell\tau|}\,e^{\var \tau}\,d\tau \le
e^{-\frac{\var \ell t}{k+\ell}} \, \left[ \frac{1}{\alpha (k + \ell)}
  + \frac{kt}{\alpha (k+ \ell)^2} \right],
$$
$$
e^{-\var t} \int_{
  \left(\frac{k}{k+\ell}\right)t}^{\left(\frac{k+1}{k+\ell+1}\right)t}
(1+\tau)\,e^{-\alpha |k(t-\tau)-\ell\tau|}\,e^{\var \tau}\,d\tau \le
e^{-\frac{\var \ell t}{k+\ell+1}} \, \left[ \frac{1}{\alpha (k +
    \ell)} + \frac{kt}{\alpha (k+ \ell)^2} + \frac{1}{\alpha^2
    (k+\ell)^2} \right].
$$
All in all, there is a numeric constant $C$ such that \eqref{tba} is
bounded above by \begeq\label{tbba} C\ \sum_{\ell\geq 1} e^{-\alpha
  \ell} \sum_{k\geq \ell} \left( \frac{1}{\alpha^2
    (k+\ell)^{2+\gamma}} + \frac{1}{\alpha (k + \ell)^{1+\gamma}} +
  \frac{kt}{\alpha (k+ \ell)^{2+\gamma}} \right)\,
e^{-\frac{\var\,\ell\,t}{k+\ell}},
\endeq
together with an additional similar term where $e^{-\var \ell
  t/(k+\ell)}$ is replaced by $e^{-\var\ell t/(k+\ell+1)}$, and which
will satisfy similar estimates.

We consider separately the three contributions in the right-hand side
of \eqref{tbba}.  The first one is
\[ \frac1{\alpha^2} \sum_{\ell\geq 1} e^{-\alpha \ell} \sum_{k\geq
  \ell} \frac{e^{-\frac{\var\,\ell
      t}{k+\ell}}}{(k+\ell)^{2+\gamma}}.\] To evaluate the behavior of
this sum, we compare it to the two-dimensional integral
\[ I(t) = \frac1{\alpha^2} \int_1^\infty e^{-\alpha x} \int_x^\infty
\frac{e^{-\frac{\var\, xt}{x+y}}}{(x+y)^{2+\gamma}}\,dy\,dx.\] We
change variables $(x,y)\to (x,u)$, where $u(x,y) = \var
xt/(x+y)$. This has Jacobian determinant $(dx\,dy)/(dx\,du) = (\var x
t)/u^2$, and we find
\[ I(t) = \frac1{\alpha^2\,\var^{1+\gamma}\, t^{1+\gamma}}
\int_1^\infty \frac{e^{-\alpha x}}{x^{1+\gamma}}\,dx \int_0^{\var t/2}
e^{-u}\,u^\gamma\,du = O \left(
  \frac1{\alpha^2\,\var^{1+\gamma}\,t^{1+\gamma}}\right).\]

The same computation for the second integral in the right-hand side of
\eqref{tbba} yields
\[ \frac1{\alpha\,\var^{\gamma}\, t^{\gamma}} \int_1^\infty
\frac{e^{-\alpha x}}{x^{\gamma}}\,dx \int_0^{\var t/2}
e^{-u}\,u^{\gamma-1}\,du = O \left(\frac{\ln \frac1{\alpha}}
  {\alpha\,\var^{\gamma}\,t^{\gamma}}\right).\] (The logarithmic
factor arises only for $\gamma=1$.)

The third exponential in the right-hand side of \eqref{tbba} is the
worse.  It yields a contribution \begeq\label{worse} \frac{t}{\alpha}
\sum_{\ell\geq 1} e^{-\alpha \ell} \sum_{k\geq \ell}
\frac{e^{-\frac{\var\, \ell t}{k+\ell}}\,k}{(k+\ell)^{2+\gamma}}.
\endeq
We compare this with the integral
\[ \frac{t}{\alpha} \int_1^\infty e^{-\alpha x} \int_x^\infty
\frac{e^{-\frac{\var\, xt}{x+y}}\,y}{(x+y)^{2+\gamma}}\,dx\,dy,\] and
the same change of variables as before equates this with
\begin{multline*}
\frac1{\alpha\,\var^{\gamma}\, t^{\gamma-1}} \, 
\int_1^\infty \frac{e^{-\alpha x}}{x^{\gamma}}\,dx
\int_0^{\var t/2} e^{-u}\, u^{\gamma-1}\,du\
- \frac1{\alpha\,\var^{1+\gamma}\, t^\gamma} \, 
\int_1^\infty \frac{e^{-\alpha x}}{x^{\gamma}}\,dx
\int_0^{\var t/2} e^{-u}\,u^\gamma\,du \\
= O \left(\frac{\ln \frac1\alpha}{\alpha\,\var^{\gamma}\,t^{\gamma-1}}\right).
\end{multline*} 
(Again the logarithmic factor arises only for $\gamma=1$.)

The proof of Proposition \ref{propexpK} follows by collecting all these bounds
and keeping only the worse one.
\end{proof}

\begin{Rk} It is not easy to catch (say numerically) the behavior of
  \eqref{worse}, because it comes as a superposition of exponentially
  decaying modes; any truncation in $k$ would lead to a radically
  different time-asymptotics.
\end{Rk}

\begin{figure}[htbp]
\begin{minipage}[t]{.4\linewidth}
  \includegraphics[height=5cm]{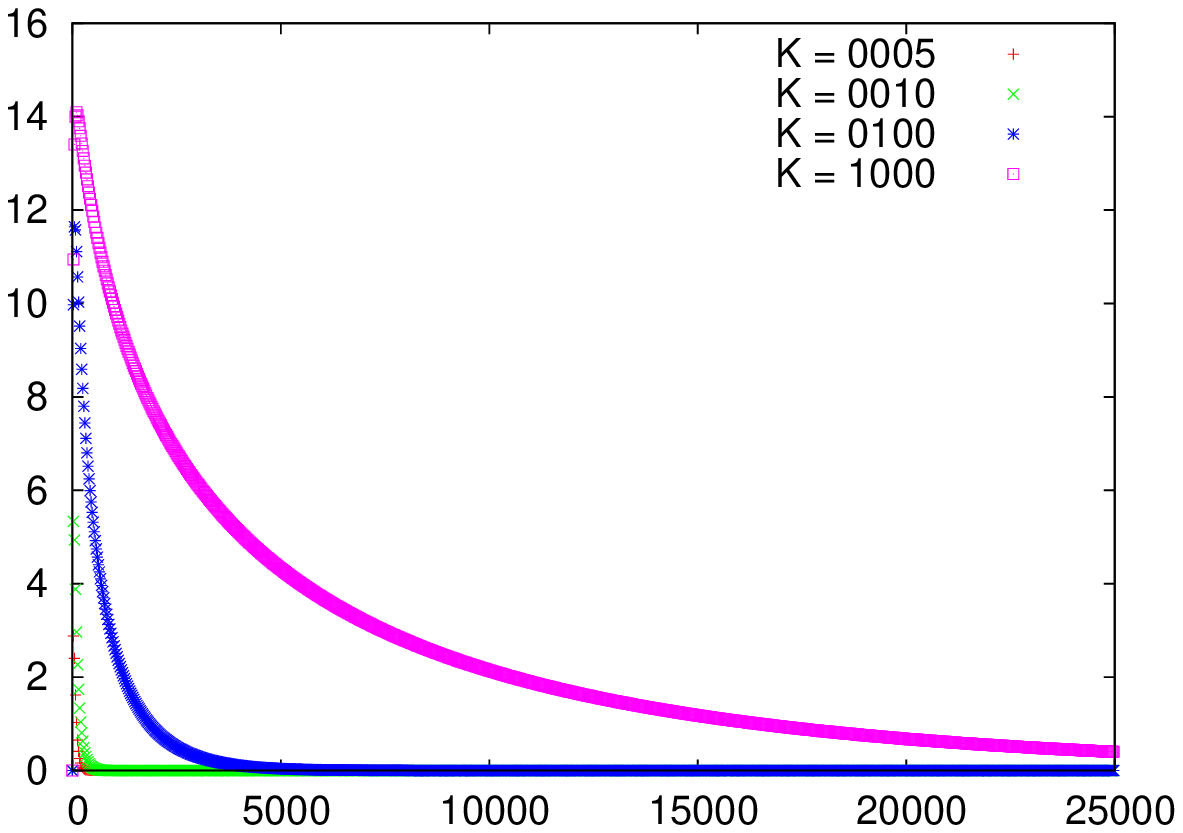}
\end{minipage}\hfill
\begin{minipage}[t]{.4\linewidth}
  \includegraphics[height=5cm]{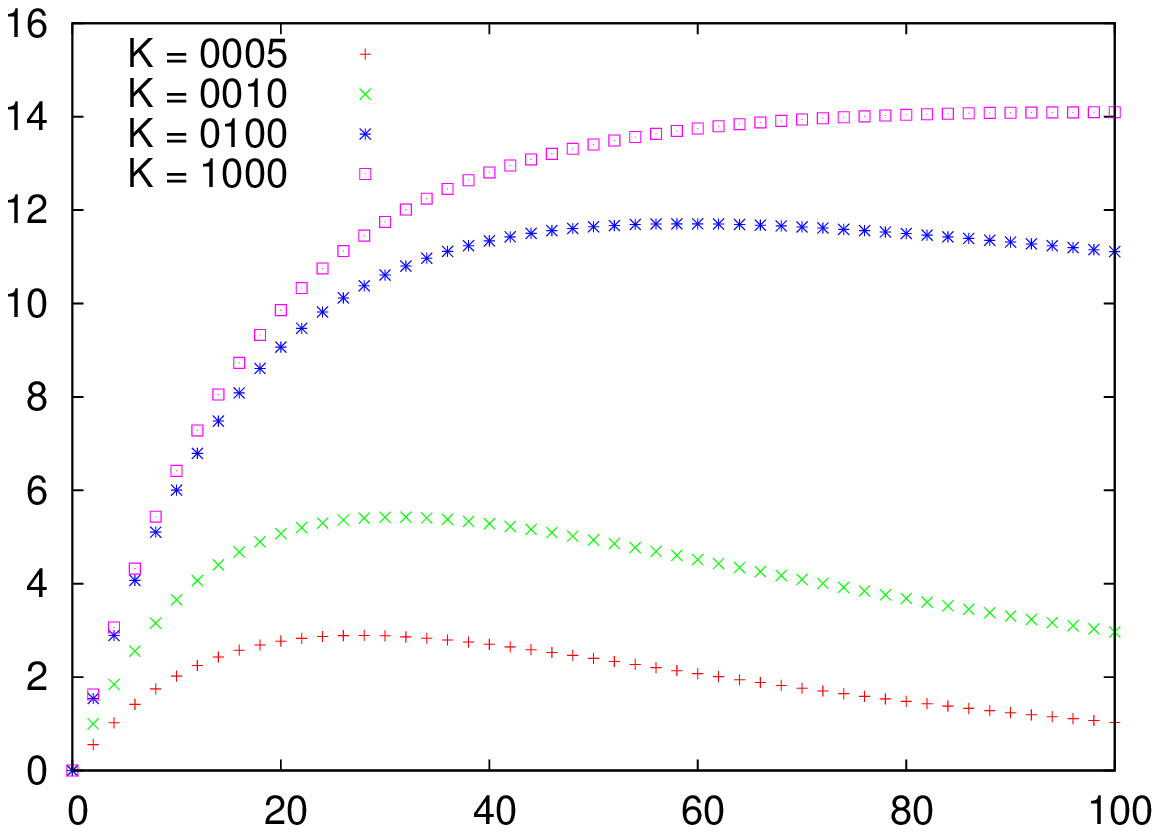}
\end{minipage}\hfill \hspace*{1mm}
\caption{The function \eqref{worse} truncated at $\ell=1$ and $k\leq K$,
for $K=5,10,100,1000$.
The decay is slower and slower, but still
  exponential (picture on the left); however, the maximum value occurs on a much slower time scale
and slowly increases with the truncation parameter (picture on the right, which is a zoom on
shorter times).}
\label{echolointain}
\end{figure}

From Proposition \ref{propexpK} we deduce $L^2$ exponential bounds:

\begin{Cor}[$L^2$ exponential moments of the kernel] \label{cor2expK}
  With the same notation as in Proposition \ref{propexpK},
  
\begeq\label{2expK} e^{-2\var t} \int_0^t
  K^{(\alpha),\gamma}(t,\tau)^2\,e^{2\var\tau}\,d\tau \leq \begin{cases}
    \dps \frac{C(\gamma)}{\alpha^4\,\var^{1+2\gamma}\,t^{2(\gamma-1)}}
 \qquad\qquad \text{if $\gamma>1$}\\[4mm]
    \dps C\, \left(\frac1{\alpha^3\,\var^2} +
      \frac1{\alpha^2\,\var^3\,t}\right)
 \qquad \text{if $\gamma=1$}.
\end{cases}
\endeq
\end{Cor}

\begin{proof}[Proof of Corollary \ref{cor2expK}]
  This follows easily from Proposition \ref{propexpK} and the obvious
  bound
\[ K^{(\alpha),\gamma}(t,\tau)^2 \leq C\,(1+t)\,K^{(2\alpha),2\gamma}(t,\tau).\]
\end{proof} 

\subsection{Dual exponential moments}

\begin{Prop} \label{propdualexpK} With the same notation as in
  Proposition \ref{propexpK}, for any $\gamma\geq 1$ we have
  \begeq\label{supdual} \sup_{\tau\geq 0}\: e^{\var \tau} \,
  \int_\tau^\infty e^{-\var t}\,K^{(\alpha),\gamma}(t,\tau)\,dt \leq
  C(\gamma) \, \left( \frac{1}{\alpha^2\,\var} + \frac{\ln
      \frac1\alpha}{\alpha \, \var^\gamma} \right).
\endeq
\end{Prop}

\begin{Rk} The corresponding computation for the baby model considered 
in Subsection \ref{sobolint} is
\begin{align*}
e^{\var \tau} \, \left(\frac{1+\tau}{\alpha}\right)\,
\sum_{k\geq 1} \frac{e^{-\var \left(\frac{k+1}{k}\right)\tau}}{k^{1+\gamma}}
& \simeq \left(\frac{1+\tau}{\alpha}\right) \int_1^\infty \frac{e^{-\var \tau/x}}{x^{1+\gamma}}\,dx\\
& = \left(\frac{1+\tau}{\tau^\gamma}\right)\, \left(\frac1{\alpha\,\var^\gamma}\right)
\int_0^{\var \tau}e^{-u}\,u^{\gamma-1}\,du.
\end{align*}
So we expect the dependence upon $\var$ in \eqref{supdual} to be sharp
for $\gamma\to 1$.
\end{Rk} 

\begin{proof}[Proof of Proposition \ref{propdualexpK}]

We first reduce to $d=1$, and split the integral as 
\begin{align*}
e^{\var \tau} \, \int_\tau^\infty e^{-\var t}\,K^{(\alpha),\gamma}(t,\tau)\,dt
& = e^{\var \tau} \, \int_{2\tau} ^\infty e^{-\var t}\,K^{(\alpha),\gamma}(t,\tau)\,dt 
+ e^{\var \tau} \int_\tau^{2\tau} e^{-\var t}\,K^{(\alpha),\gamma}(t,\tau)\,dt \\
& =: I_1 + I_2.
\end{align*}

The first term $I_1$ is easy: for $2 \tau \le t \le +\infty$ we have
$$
K^{(\alpha),\gamma}(t,\tau) \le (1+\tau) \, \sum_{k >1, \ \ell \neq 0} 
e^{-\alpha |\ell| - \frac{\alpha}2 |k-\ell|} \le \frac{C\, (1+\tau)}{\alpha^2},
$$
and thus
$$
e^{\var \tau} \, \int_{2\tau} ^\infty e^{-\var t}\,K^{(\alpha),\gamma}(t,\tau)\,dt 
\le \frac{C\,(1+\tau)}{\alpha^2}\,e^{-\var\tau}\leq \frac{C}{\var \, \alpha^{2}}.
$$

We treat the second term $I_2$ as in the proof of Proposition \ref{propexpK}:
\begin{align*} e^{\var\tau} \int_\tau^{2\tau} & K^{(\alpha),\gamma}(t,\tau)\,e^{-\var t}\,dt\\
& \leq e^{\var\tau}\,(1+\tau) \sum_{\ell\geq 1} e^{-\alpha\ell}
\sum_{k\geq \ell} \int_{\left(\frac{k+\ell+1}{k+1}\right)\tau}^{\left(\frac{k+\ell-1}{k-1}\right)\tau}
\frac{e^{-\alpha |k(t-\tau)-\ell\tau|}}{1+(k+\ell)^\gamma}\,e^{-\var t}\,dt\\
& \leq (1+\tau) \sum_{\ell\geq 1} e^{-\alpha \ell}
\sum_{k\geq\ell} \frac{e^{-\var\frac{\ell}{k}\tau}}{k^\gamma}\,\left(\frac2{k\alpha}\right).
\end{align*}
We compare this with
\begin{align*}
\frac{2\,(1+\tau)}{\alpha}
\int_1^\infty & e^{-\alpha x} 
\int_x^\infty \frac{e^{-\var\frac{x}{y}\tau}}{y^{1+\gamma}}\,dy\,dx\\
& = \frac{2}{\alpha\,\var^\gamma}
\left(\frac{1+\tau}{\tau^\gamma}\right)
\int_1^\infty \frac{e^{-\alpha x}}{x^\gamma}
\int_0^{\var\tau} e^{-u}\,u^{\gamma-1}\,du\,dx\\
& \leq \frac{C\, \ln (1/\alpha)}{\alpha\,\var^\gamma},
\end{align*}
where we used the change of variables $u=\var x \tau/y$.  The desired
conclusion follows.  Note that as before the term $\ln (1/\alpha)$
only occurs when $\gamma=1$, and that for $\gamma >1$, one could
improve the estimate above into a time decay of the form
$O(\tau^{-(\gamma-1)})$.
\end{proof}

\subsection{Growth control} \label{subsgrowth}

To state the main result of this section we shall write
$\Z^d_*=\Z^d\setminus\{0\}$; and if a sequence of functions
$\Phi(k,t)$ ($k\in\Z^d_*$, $t\in\R$) is given, then
$\|\Phi(t)\|_\lambda = \sum_k e^{2\pi\lambda |k|}\,|\Phi(k,t)|$. We
shall use $K(s)\,\Phi(t)$ as a shorthand for
$(K(k,s)\,\Phi(k,t))_{k\in\Z^d_*}$, etc.

\begin{Thm}[Growth control {\it via} integral inequalities] 
\label{thmgrowth} 
Let $f^0=f^0(v)$ and $W=W(x)$ satisfy condition {\bf (L)} from Subsection \ref{sub:lineardamping}
with constants $C_0,\lambda_0,\kappa$; in particular
$|\tilde{f}^0(\eta)|\leq C_0\, e^{-2\pi\lambda_0|\eta|}$. Let further
\[C_W = \max \left\{ \sum_{k\in\Z^d_*} |\hat{W}(k)|,\ \
  \sup_{k\in\Z^d_*}\, |k|\,|\hat{W}(k)| \right\}.\] Let $A\geq 0$,
$\mu\geq 0$, $\lambda\in (0,\lambda^*]$ with
$0<\lambda^*<\lambda_0$. Let $(\Phi(k,t))_{k\in\Z^d_*,\ t\geq 0}$ be a
continuous function of $t\geq 0$, valued in $\C^{\Z^d_*}$, such that
\begin{multline} \label{cond74} \forall \, t\geq 0,\qquad
  \left\| \Phi(t) - \int_0^t K^0(t-\tau)\,\Phi(\tau)\,d\tau \right\|_{\lambda t+\mu}\\
  \leq A + \int_0^t \left[ K_0(t,\tau)+ K_1(t,\tau) +
    \frac{c_0}{(1+\tau)^m}\right]\ \|\Phi(\tau)\|_{\lambda\tau
    +\mu}\,d\tau,
\end{multline}
where $c_0\geq 0$, $m>1$ and $K_0(t,\tau)$, $K_1(t,\tau)$ are
nonnegative kernels.  Let $\varphi(t) = \|\Phi(t)\|_{\lambda
  t+\mu}$. Then \sm

(i) Assume $\gamma>1$ and $K_1 = c\, K^{(\alpha),\gamma}$ for
some $c>0$, $\alpha\in (0,\ov{\alpha}(\gamma))$, where
$K^{(\alpha),\gamma}$ is defined by \eqref{hKa}, and
$\ov{\alpha}(\gamma)$ appears in Proposition \ref{propexpK}. Then
there are positive constants $C$ and $\chi$, depending only on
$\gamma,\lambda^*,\lambda_0,\kappa,c_0,C_W,m$, uniform as
$\gamma\to 1$, such that if 
\begeq\label{condchi} \sup_{t\geq 0}
\int_0^t K_0(t,\tau)\,d\tau \leq \chi
\endeq
and
\begeq\label{condc0}
\sup_{t\geq 0} \left( \int_0^t K_0(t,\tau)^2\,d\tau \right)^{1/2}
+ \sup_{\tau \geq 0} \int_\tau^\infty K_0(t,\tau)\,dt \leq 1,
\endeq
then for any $\var \in (0,\alpha)$, 
\begeq\label{phileq} \forall \,
t\geq 0,\qquad \varphi(t) \leq C\, A\, 
\frac{(1+c_0^2)}{\sqrt{\var}}  \, 
e^{C \, c_0}\,\left(1+\frac{c}{\alpha\,\var} \right)\,
e^{CT}\,e^{C\,c \,(1+T^2)}\,e^{\var t},
\endeq
where \begeq\label{condT} T = C\,\max \left\{
  \left(\frac{c^2}{\alpha^5\,\var^{2+\gamma}}\right)^{\frac1{\gamma-1}}\,
  ; \,
  \left(\frac{c}{\alpha^2\,\var^{\gamma+\frac12}}\right)^{\frac1{\gamma-1}};\,
  \left(\frac{c_0^2}{\var}\right)^{\frac1{2m-1}}\right\}.
\endeq 
\sm

(ii) Assume $K_1 = \sum_{1\leq i\leq N} c_i\, K^{(\alpha_i),1}$ for
some $\alpha_i \in (0,\ov{\alpha}(1))$, where $\ov{\alpha}(1)$ appears
in Proposition \ref{propexpK}; then there is a numeric constant
$\Gamma>0$ such that whenever
\[ 1\geq \var \geq \Gamma\ \sum_{i=1} ^N \frac{c_i}{\alpha_i ^3}, \]
one has, with the same notation as in (i), \begeq\label{phileq1}
\forall \, t\geq 0,\qquad \varphi(t) \leq C\,A\, \frac{(1+c_0^2)\,e^{C\,c_0}}
{\sqrt{\var}}\, e^{CT}\,e^{C\, c\,(1+T^2)}\,e^{\var t},
\endeq
where
\[ c= \sum_{i=1} ^N c_i,\qquad
T = C\, \max \left\{ 
  \frac{1}{\var^2} \, \left( \sum_{i=1} ^N \frac{c_i}{\alpha_i^3}\right) \, ;\,
\left(\frac{c_0^2}{\var}\right)^{\frac1{2m-1}}\right\}.\]
\end{Thm}

\begin{Rk} Let apart the term $c_0/(1+\tau)^m$ which will appear as a
  technical correction, there are three different kernels appearing in
  Theorem \ref{thmgrowth}: the kernel $K^0$, which is associated with
  the linearized Landau damping; the kernel $K_1$, describing
  nonlinear echoes (due to interaction between differing Fourier
  modes); and the kernel $K_0$, describing the instantaneous response
  (due to interaction between identical Fourier modes).
\end{Rk}

We shall first prove Theorem \ref{thmgrowth} assuming
\begeq\label{c00}
c_0 =0
\endeq
and 
\begeq\label{K0simplif}
\int_0^\infty \sup_k |K^0(k,t)|\,e^{2\pi\lambda_0|k|t}\,dt 
\leq 1-\kappa, \qquad \kappa \in (0,1),
\endeq
which is a reinforcement of condition {\bf (L)}. Under these assumptions the proof
of Theorem \ref{thmgrowth} is much simpler, and its conclusion can be substantially simplified
too: $\chi$ depends only on $\kappa$; condition \eqref{condc0} on $K_0$ can be dropped;
and the factor $e^{CT}(1+c/(\alpha\,\var^{3/2}))$ in \eqref{phileq} can be omitted.
If $\hat{W}\leq 0$ (as for gravitational interaction) and $\tilde{f}^0\geq 0$ (as for Maxwellian
background), these additional assumptions do not constitute a loss of generality,
since \eqref{K0simplif} becomes essentially equivalent to {\bf (L)}, and
for $c_0$ small enough the term $c_0(1+\tau)^{-m}$ can be incorporated inside $K_0$.

\begin{proof}[Proof of Theorem \ref{thmgrowth} under \eqref{condc0} and \eqref{K0simplif}]
We have
\begeq\label{phiak}
\varphi(t) \leq A + \int_0^t \Bigl(|K^0| (t-\tau)+ K_0(t,\tau) + K_1(t,\tau)\Bigr)\,\varphi(\tau)\,d\tau,
\endeq
where $|K^0(t)| = \sup_k\, |K^0(k,t)|$.
We shall estimate $\varphi$ by a maximum principle argument.
Let $\psi(t) = B\,e^{\var t}$, where $B$ will be chosen later. If $\psi$ satisfies, for some $T\geq 0$,
\begeq\label{systempsi}
\begin{cases} \varphi(t) < \psi(t)\qquad \text{for $0\leq t\leq T$},\\[3mm]
\dps \psi(t)\geq A + \int_0^t \Bigl( |K^0|(t,\tau) + K_0(t,\tau) 
+ K_1(t,\tau)\Bigr) \, \psi(\tau)\,d\tau 
\qquad \text{for $t\geq T$},
\end{cases}
\endeq
then $u(t) := \psi(t) - \varphi(t)$ is positive for $t\leq T$, and satisfies $u(t) \geq \int_0^t K(t,\tau)\,u(\tau)\,d\tau$
for $t\geq T$, with $K=|K^0|+K_0+K_1>0$; this prevents $u$ from vanishing at later times, so $u\geq 0$ and
$\varphi\leq \psi$. Thus it is sufficient to establish \eqref{systempsi}.
\med

\noindent
{\bf Case (i)}: By Proposition \ref{propexpK}, and since $\int (|K^0| +
K_0)\,d\tau \leq 1 - \kappa/2$ (for $\chi \le \kappa/2$),
\begin{multline} \label{proofexpeps} A + \int_0^t \bigl[|K^0|(t,\tau)
  + K_0(t,\tau)\bigr]\,\psi(\tau)\,d\tau
  + c \int_0^t \nohat{K}^{(\alpha),\gamma}(t,\tau)\,\psi(\tau)\,d\tau\\
  \leq A + \left[ \left(1-\frac{\kappa}2\right) +
    \frac{c\,C(\gamma)}{\alpha^3\,\var^{1+\gamma}\,t^{\gamma-1}}\right]\,
  B\,e^{\var  t}.
\end{multline} 
For $t\geq T:= (4 \, c\, C \,
(\alpha^3\,\var^{1+\gamma}\kappa))^{1/(\gamma-1)}$, this is bounded
above by $A+(1-\kappa/4)\,B\,e^{\var t}$, which in turn is bounded by
$B\,e^{\var t}$ as soon as $B\geq 4 \, A/\kappa$.

On the other hand, from the inequality
\[ \varphi(t) \leq A + \left(1-\frac{\kappa}{2}\right)\, \sup_{0\leq
  \tau\leq t}\varphi(\tau) + c\, (1+t) \int_0^t \varphi(\tau)\,d\tau\]
we deduce
\[ \varphi(t) \leq \left(\frac{2 A}{\kappa}\right)\, (1+t) \,
e^{\frac{2c}{\kappa}\left(t+\frac{t^2}{2}\right)}\] 
 In particular, if
\[ \frac{4A}{\kappa}\, e^{c'(T + T^2)} \leq B\] with $c'=c'(c,\kappa)$
large enough, then for $0\leq t\leq T$ we have $\varphi(t) \leq
\psi(t)/2$, and \eqref{systempsi} holds.\med

\noindent
{\bf Case (ii)}: $K_1 = \sum c_i\,K^{(\alpha_i),1}$. We use the same
reasoning, replacing the right-hand side in \eqref{proofexpeps} by
\[ A+ \left[ \left(1-\frac{\kappa}{2}\right) +
  C \, \left( \sum_{i=1} ^N \frac{c_i}{\alpha_i ^3\,\var} 
    + \sum_{i=1} ^N \frac{c_i}{\alpha_i ^3\,\var^2 \, t} \right)
\right] \, B\, e^{\var t}.\] 

To conclude the proof, we may first impose a lower bound on $\var$ to ensure
\begeq\label{lbensures}
C \, \sum_{i=1} ^N \frac{c_i}{\alpha_i ^3\,\var} \le \frac{\kappa}8,
\endeq
and then choose $t$ large enough to guarantee
\begeq\label{tlargeguar}
C \, \sum_{i=1} ^N \frac{c_i}{\alpha_i ^3\,\var^2 \, t} \le
\frac{\kappa}8;
\endeq
this yields (ii). 
\end{proof}

\begin{proof}[Proof of Theorem \ref{thmgrowth} in the general case]
  We only treat (i), since the reasoning for (ii) is rather
  similar; and we only establish the conclusion as an {\it a priori}
  estimate, skipping the continuity/approximation argument needed to
  turn it into a rigorous estimate. Then the proof is done in three
  steps.  \sm

\noindent {\bf Step 1:} {\em Crude pointwise bounds.}
From \eqref{cond74} we have
\begin{align} \label{phileqA0t}
\varphi(t) & = \sum_{k\in\Z^d_*}
|\Phi(k,t)|\,e^{2\pi(\lambda t +\mu)|k|}\\
\nonumber & \leq A + \sum_k \int_0^t \bigl|K^0(k,t-\tau)\bigr|\,e^{2\pi(\lambda t+\mu)|k|}\,
|\Phi(t,\tau)|\,d\tau\\
\nonumber & \qquad + \int_0^t \left[ K_0(t,\tau) + K_1(t,\tau) + \frac{c_0}{(1+\tau)^m}\right]\,
\varphi(\tau)\,d\tau\\
\nonumber & \leq A + \int_0^t \Biggl[ \left(\sup_k\ \bigl|K^0(k,t-\tau)\bigr|\,
e^{2\pi\lambda (t-\tau)|k|}\right) \\
\nonumber & \qquad \qquad\qquad
+K_1(t,\tau) + K_0(t,\tau) + \frac{c_0}{(1+\tau)^m}\Biggr]\,\varphi(\tau)\,d\tau.
\end{align}
We note that for any $k\in\Z^d_*$ and $t\geq 0$,
\begin{align*}
\bigl|K^0(k,t-\tau)\bigr|\,e^{2\pi\lambda |k|(t-\tau)}
& \leq 4\pi^2\, |\hat{W}(k)|\,C_0\,e^{-2\pi (\lambda_0-\lambda)|k|t}\,|k|^2\,t \\
& \leq \frac{C\,C_0}{\lambda_0-\lambda}\,
\left(\sup_{k\neq 0}\ |k|\,|\hat{W}(k)|\right) \leq \frac{C\,C_0\,C_W}{\lambda_0-\lambda},
\end{align*}
where (here as below) $C$ stands for a numeric constant which may change from line to line.
Assuming $\int K_0(t,\tau)\,d\tau \leq 1/2$, we deduce from \eqref{phileqA0t}
\[ \varphi(t) \leq A + \frac12 \, \left(\sup_{0\leq\tau\leq t}\ \varphi(\tau)\right)
+ C \int_0^t \left(\frac{C_0\,C_W}{\lambda_0-\lambda} + c\,(1+t) + \frac{c_0}{(1+\tau)^m}\right)\,
\varphi(\tau)\,d\tau,\]
and by Gronwall's lemma
\begeq\label{Gron}
\varphi(t) \leq 2A\, e^{C\left( \frac{C_0\,C_W}{\lambda_0-\lambda}\,t
+ c (t+ t^2) + c_0\,C_m\right)},
\endeq
where $C_m = \int_0^\infty (1+\tau)^{-m}\,d\tau$. \med

\noindent{\bf Step~2:} {\em $L^2$ bound}. This is the step where the
smallness assumption \eqref{condchi} will be most important. For all
$k\in\Z^d_*$, $t\geq 0$, we define \begeq\label{Psik} \Psi_k(t) =
e^{-\var t}\,\Phi(k,t)\,e^{2\pi(\lambda t+\mu)|k|},
\endeq
\begeq\label{K0k}
\K^0_k(t) = e^{-\var t}\,K^0(k,t)\,e^{2\pi(\lambda t+\mu)|k|},
\endeq
\begin{align}\label{Rkt}
R_k(t) & = e^{-\var t}\,
\left( \Phi(k,t) - \int_0^t K^0(k,t-\tau)\,\Phi(k,\tau)\,d\tau\right)\,e^{2\pi(\lambda t+\mu)|k|}\\
\nonumber & = \bigl(\Psi_k - \Psi_k\ast\K^0_k\bigr)(t),
\end{align}
and we extend all these functions by~$0$ for negative values of $t$.
Taking Fourier transform in the time variable yields $\hat{R}_k =
(1-\hat{\K}^0_k)\,\hat{\Psi}_k$; since condition {\bf (L)} implies
$|1-\hat{\K}^0_k|\geq \kappa$, we deduce $\|\hat{\Psi}_k\|_{L^2}\leq
\kappa^{-1}\,\|\hat{R}_k\|_{L^2}$, {\it i.e.}, \begeq\label{PsikRk}
\|\Psi_k\|_{L^2(dt)} \leq \frac{\|R_k\|_{L^2(dt)}}{\kappa}.
\endeq 
Plugging \eqref{PsikRk} into \eqref{Rkt}, we deduce
\begeq\label{PsiRkleq} \forall \, k\in\Z^d_*,\qquad
\bigl\|\Psi_k-R_k\bigr\|_{L^2(dt)} \leq
\frac{\|\K^0_k\|_{L^1(dt)}}{\kappa}\, \|R_k\|_{L^2(dt)}.
\endeq

Then
\begin{align}\label{phiepsL2}
\bigl\|\varphi(t)\,e^{-\var t}\bigr\|_{L^2(dt)}
& = \left\|\sum_k |\Psi_k|\, \right\|_{L^2(dt)} \\
\nonumber & \leq \left\|\sum_k |R_k|\, \right\|_{L^2(dt)}
+ \sum_k \|R_k-\Psi_k\|_{L^2(dt)}\\
\nonumber & \leq \left\|\sum_k |R_k|\, \right\|_{L^2(dt)}
\ \left(1 + \frac1{\kappa} \sum_{\ell\in\Z^d_*} \|\K^0_\ell\|_{L^1(dt)}\right).
\end{align}
(Note: We bounded $\|R_\ell\|$ by $\|\sum_k |R_k|\|$, which seems very
crude; but the decay of $\K^0_k$ as a function of $k$ will save us.)
Next, we note that
\begin{align*}\|\K^0_k\|_{L^1(dt)}
& \leq 4\pi^2\,|\hat{W}(k)| 
\int_0^\infty C_0\,e^{-2\pi (\lambda_0-\lambda)|k|t}\,
|k|^2\,t\,dt \\
& \leq 4 \pi^2\, |\hat{W}(k)|\,\frac{C_0}{(\lambda_0-\lambda)^2},
\end{align*}
so
\[ \sum_k \|\K^0_k\|_{L^1(dt)} \leq 4\pi^2\, \left(\sum_k
  |\hat{W}(k)|\right)\, \frac{C_0}{(\lambda_0-\lambda)^2}.\]
Plugging this in \eqref{phiepsL2} and using \eqref{cond74} again, we
obtain
\begin{align} \label{estim1L2} \bigl\| & \varphi(t)\,e^{-\var
    t}\bigr\|_{L^2(dt)} \leq \left(1+
    \frac{C\,C_0\,C_W}{\kappa\,(\lambda_0-\lambda)^2}\right)\,
  \left\|\sum_k |R_k|\,\right\|_{L^2(dt)}\\
  \nonumber & \leq \left( 1 +
    \frac{C\,C_0\,C_W}{\kappa\,(\lambda_0-\lambda)^2}\right)\,
  \left\{ \int_0^\infty e^{-2\var t} \left( A + \int_0^t \left[
        K_1+K_0+\frac{c_0}{(1+\tau)^m}\right]\,\varphi(\tau)\,d\tau
    \right)^2\,dt \right\}^{\frac12}.
\end{align}

We separate this (by Minkowski's inequality) into various
contributions which we estimate separately.  First, of course
\begeq\label{oc1} \left(\int_0^\infty e^{-2\var
    t}\,A^2\,dt\right)^{\frac12} = \frac{A}{\sqrt{2\var}}.
\endeq

Next, for any $T\geq 1$, by Step~1 and $\int_0^t
K_1(t,\tau)\,d\tau\leq C c(1+t)/\alpha$,
\begin{align} \label{oc2} \biggl \{ \int_0^T e^{-2 \var t} &
  \left(\int_0^t K_1(t,\tau)\,\varphi(\tau)\,d\tau \right)^2\,
  dt\biggr\}^{\frac12}\\
  & \nonumber \leq \left[\sup_{0\leq t\leq T} \varphi(t)\right]
  \left(\int_0^T e^{-2\var t}\left(\int_0^t K_1(t,\tau)\,d\tau\right)^2\,dt\right)^{\frac12}\\
  & \nonumber \leq
  C\,A\,e^{C\left[\frac{C_0\,C_W}{\lambda_0-\lambda}\,T +
      c\,(T+T^2)\right]}\,
  \frac{c}{\alpha} \left(\int_0^\infty e^{-2\var t} (1+t)^2\,dt\right)^{\frac12}\\
  & \nonumber \leq C\, A\,\frac{c}{\alpha\,\var^{3/2}}\, e^{C
    \left[\frac{C_0\,C_W}{\lambda_0-\lambda}\,T +
      c\,(T+T^2)\right]}.
\end{align}

Invoking Jensen and Fubini, we also have
\begin{align}\label{oc3}
  \biggl\{ \int_T^\infty e^{-2\var t}& 
  \left(\int_0^t K_1(t,\tau)\,\varphi(\tau)\,d\tau\right)^2\, 
  dt\biggr\}^{\frac12}\\
  \nonumber & = \left\{ \int_T^\infty \left(\int_0^t
      K_1(t,\tau)\,e^{-\var (t-\tau)}\,
      e^{-\var \tau}\,\varphi(\tau)\,d\tau\right)^2\,dt \right\}^{\frac12}\\
  \nonumber & \leq \left\{\int_T^\infty \left(\int_0^t
      K_1(t,\tau)\,e^{-\var (t-\tau)}\,d\tau\right) \left(\int_0^t
      K_1(t,\tau)\,e^{-\var (t-\tau)}\,e^{-2\var
        \tau}\varphi(\tau)^2\,d\tau\right) dt \right\}^{\frac12}\\
  \nonumber & \leq\left(\sup_{t\geq T} \int_0^t e^{-\var
      t}\,K_1(t,\tau)\,e^{\var \tau}\,d\tau\right)^{\frac12}
  \left(\int_T^\infty \int_0^t K_1(t,\tau)\,e^{-\var(t-\tau)}\,e^{-2\var \tau}\varphi(\tau)^2\,d\tau\,dt\right)^{\frac12}\\
  \nonumber & = \left(\sup_{t\geq T} \int_0^t e^{-\var
      t}\,K_1(t,\tau)\,e^{\var \tau}\,d\tau\right)^{\frac12}
  \left(\int_0^\infty \int_{\max\{\tau\, ; \, T\}} ^{+\infty}
    K_1(t,\tau)\,e^{-\var (t-\tau)}\,
    e^{-2\var\tau}\,\varphi(\tau)^2\,dt\,d\tau\right)^{\frac12}\\
  \nonumber & \leq \left(\sup_{t\geq T} \int_0^t e^{-\var
      t}\,K_1(t,\tau)\,e^{\var\tau}\,d\tau\right)^{\frac12}
  \left(\sup_{\tau\geq 0} \int_\tau^\infty e^{\var\tau}\,
    K_1(t,\tau)\,e^{-\var t}\,dt\right)^{\frac12} \left(\int_0^\infty
    e^{-2\var\tau}\,\varphi(\tau)^2\,d\tau\right)^{\frac12}.
\end{align}
(Basically we copied the proof of Young's inequality.) Similarly,
\begin{align}\label{oc4}
\biggl\{\int_0^\infty & e^{-2\var t}
\left(\int_0^t K_0(t,\tau)\,\varphi(\tau)\,d\tau\right)^2\,dt\biggr\}^{\frac12}\\
\nonumber & \leq \left(\sup_{t\geq 0} \int_0^t e^{-\var t}\, K_0(t,\tau)\,e^{\var \tau}\,d\tau\right)^{\frac12}
\left(\sup_{\tau\geq 0} \int_\tau^\infty e^{\var\tau}\,K_0(t,\tau)\,e^{-\var t}\,dt\right)^{\frac12}
\left(\int_0^\infty e^{-2\var\tau}\,\varphi(\tau)^2\,d\tau\right)^{\frac12}\\
\nonumber & \leq \left(\sup_{t\geq 0} \int_0^t K_0(t,\tau)\,d\tau\right)^{\frac12}
\left(\sup_{\tau\geq 0} \int_\tau^\infty K_0(t,\tau)\,dt\right)^{\frac12}
\left(\int_0^\infty e^{-2\var \tau}\,\varphi(\tau)^2\,d\tau\right)^{\frac12}.
\end{align}

The last term is also split, this time according to $\tau\leq T$ or $\tau>T$:
\begin{align}\label{oc5}
\biggl\{\int_0^\infty e^{-2\var t} & \, \left(\int_0^T \frac{c_0\,\varphi(\tau)}{(1+\tau)^m}\,d\tau\right)^2
\,dt \biggr\}^{\frac12}\\
\nonumber & \leq c_0\,\left(\sup_{0\leq \tau\leq T} \varphi(\tau)\right)
\left\{\int_0^\infty e^{-2\var t}\left(\int_0^T \frac{d\tau}{(1+\tau)^m}\right)^2\,dt\right\}^{\frac12}\\
\nonumber & \leq c_0\,\frac{C\,A}{\sqrt{\var}}\,
e^{C\left[ \left(\frac{C_0\,C_W}{\lambda_0-\lambda}\right)T + c\, (T+T^2)\right]}\,C_m,
\end{align}
and
\begin{align}\label{oc6}
  \biggl\{\int_0^\infty & e^{-2\var t}
  \left(\int_T^t
    \frac{c_0\,\varphi(\tau)\,d\tau}{(1+\tau)^m}\right)^2\,
   dt\biggr\}^{\frac12}\\
  \nonumber & = c_0 \left\{\int_0^\infty \left(\int_T^t e^{-\var
        (t-\tau)}\,\frac{e^{-\var \tau}\,\varphi(\tau)}
      {(1+\tau)^m}\,d\tau\right)^2\,dt\right\}^{\frac12}\\
  \nonumber & \leq c_0\, \left\{\int_0^\infty \left(\int_T^t
      \frac{e^{-2\var(t-\tau)}}{(1+\tau)^{2m}}\,d\tau\right)
    \left(\int_T^t e^{-2\var \tau}\,\varphi(\tau)^2\,d\tau\right)\,
 dt\right\}^{\frac12}\\
  \nonumber & \leq c_0 \, \left(\int_0^\infty e^{-2\var
      t}\,\varphi(t)^2\,dt\right)^{\frac12}
  \left(\int_0^\infty \int_T^t \frac{e^{-2\var
        (t-\tau)}}{(1+\tau)^{2m}}\,
   d\tau\,dt\right)^{\frac12}\\
  \nonumber & = c_0\, \left(\int_0^\infty e^{-2\var
      t}\,\varphi(t)^2\,dt\right)^{\frac12} \left(\int_T^\infty
    \frac1{(1+\tau)^{2m}}\left(\int_\tau^\infty
      e^{-2\var(t-\tau)}\,dt\right)\,d\tau \right)^{\frac12}\\
  \nonumber & = c_0 \left(\int_0^\infty e^{-2\var
      t}\,\varphi(t)^2\,dt\right)^{\frac12} \left(\int_T^\infty
    \frac{d\tau}{(1+\tau)^{2m}}\right)^{\frac12}
  \left(\int_0^\infty e^{-2\var s}\,ds\right)^{\frac12}\\
  \nonumber & = \frac{C_{2m}^{1/2}\,c_0}{\sqrt{\var}\,T^{m-1/2}}\,
  \left(\int_0^\infty e^{-2\var t}\,\varphi(t)^2\,dt\right)^{\frac12}.
\end{align}

Gathering estimates \eqref{oc1} to \eqref{oc6}, we deduce from
\eqref{estim1L2}
\begin{multline} \label{endstep2} \bigl\|\varphi(t)\,e^{-\var
    t}\bigr\|_{L^2(dt)} \leq \left( 1 +
    \frac{C\,C_0\,C_W}{\kappa\,(\lambda_0-\lambda)^2}\right)\,
  \frac{C\,A}{\sqrt{\var}}\, \left[ 1 + \left(\frac{c}{\alpha\,\var}+
      c_0\,C_m\right)\right]\,
  e^{C\,\left[\frac{C_0\,C_W}{\lambda_0-\lambda}\,T + c\, (T+T^2)\right]}\\
  \ + a\, \bigl\|\varphi(t)\,e^{-\var t}\bigr\|_{L^2(dt)},
\end{multline}
where
\begin{multline*}
  a =
  \left(1+\frac{C\,C_0\,C_W}{\kappa\,(\lambda_0-\lambda)^2}\right)\
  \biggl[ \left(\sup_{t\geq T} \int_0^t e^{-\var
      t}\,K_1(t,\tau)\,e^{\var \tau}\,d\tau\right)^{\frac12}
  \left(\sup_{\tau\geq 0} \int_\tau^\infty e^{\var\tau}\,K_1(t,\tau)\,
 e^{-\var t}\,dt\right)^{\frac12}\\
  + \left(\sup_{t\geq 0}\int_0^t K_0(t,\tau)\,d\tau\right)^{\frac12}
  \left(\sup_{\tau\geq 0}\int_\tau^\infty
    K_0(t,\tau)\,dt\right)^{\frac12} +
  \,\frac{C_{2m}^{1/2}\,c_0}{\sqrt{\var}\,T^{m-1/2}}\biggr].
\end{multline*} 

Using Propositions \ref{propexpK} (case $\gamma >1$) and
\ref{propdualexpK}, as well as assumptions \eqref{condchi} and
\eqref{condc0}, we see that $a\leq 1/2$ for $\chi$ small enough and
$T$ satisfying \eqref{condT}. Then from \eqref{endstep2} follows
\[
\bigl\|\varphi(t)\,e^{-\var t}\bigr\|_{L^2(dt)}
\leq \left( 1 + \frac{C\,C_0\,C_W}{\kappa\,(\lambda_0-\lambda)^2}\right)\,
\frac{C\,A}{\sqrt{\var}}\, 
\left[ 1 + \left(\frac{c}{\alpha\,\var}+ c_0\,C_m\right)\right]\,
e^{C\,\left[\frac{C_0\,C_W}{\lambda_0-\lambda}\,T + c\, (T+T^2)\right]}.
\]
\sm

\noindent{\bf Step 3:} {\em Refined pointwise bounds.}
Let us use \eqref{cond74} a third time, now for $t\geq T$:
\begin{align} \label{3rd}
e^{-\var t}\,\varphi(t) &
\leq A\,e^{-\var t}  + \int_0^t \left(\sup_k\ |K^0(k,t-\tau)|\,e^{2\pi\lambda (t-\tau)|k|}\right)\,
\varphi(\tau)\,e^{-\var\tau}\,d\tau \\
\nonumber  & \qquad\qquad + \int_0^t \left[K_0(t,\tau) + \frac{c_0}{(1+\tau)^m}\right]\,\varphi(\tau)\,e^{-\var \tau}\,d\tau\\
\nonumber & \qquad\qquad + \int_0^t \left(e^{-\var t}\,K_1(t,\tau)\,e^{\var \tau}\right)\,\varphi(\tau)\,e^{-\var \tau}\,d\tau\\
\nonumber & \leq A\,e^{-\var t} 
+ \biggl[ \biggl(\int_0^t \biggl(\sup_{k\in\Z^d_*} |K^0(k,t-\tau)|\,
e^{2\pi\lambda (t-\tau)|k|}\biggr)^2\,d\tau \biggr)^{\frac12}\\
\nonumber & \qquad \qquad\qquad + \left(\int_0^t K_0(t,\tau)^2\,d\tau\right)^{\frac12}
+ \left(\int_0^\infty \frac{c_0^2}{(1+\tau)^{2m}}\,d\tau\right)^{\frac12}\\
\nonumber & \qquad \qquad\qquad + \left(\int_0^t e^{-2\var t}\,K_1(t,\tau)^2\,e^{2\var\tau}\,d\tau\right)^{\frac12}
\biggr]\ \left(\int_0^\infty\varphi(\tau)^2\,e^{-2\var\tau}\,d\tau\right)^{\frac12}.
\end{align}

We note that, for any $k\in\Z^d_*$,
\begin{align*}
\Bigl( |K^0(k,t)|\, e^{2\pi\lambda |k|t}\Bigr)^2
& \leq 16\,\pi^4\,|\hat{W}(k)|^2\, \bigl|\tilde{f}^0(kt)\bigr|^2\,|k|^4\,t^2\,e^{4\pi\lambda |k|t}\\
& \leq C\, C_0^2\,|\hat{W}(k)|^2\,e^{-4\pi(\lambda_0-\lambda)|k|t}\,|k|^4\,t^2\\
& \leq \frac{C\,C_0^2}{(\lambda_0-\lambda)^2}\, |\hat{W}(k)|^2\,
e^{-2\pi (\lambda_0-\lambda)|k|t}\,|k|^2\\
& \leq \frac{C\,C_0^2}{(\lambda_0-\lambda)^2}\, C_W^2\,
e^{-2\pi (\lambda_0-\lambda)|k|t}\\
& \leq \frac{C\,C_0^2}{(\lambda_0-\lambda)^2}\, C_W^2\, e^{-2\pi (\lambda_0-\lambda)t};
\end{align*}
so
\[ \int_0^t \left(\sup_{k\in\Z^d_*}\ \bigl|K^0(k,t-\tau)\bigr|\,
e^{2\pi\lambda (t-\tau)|k|} \right)^2\,d\tau
\leq \frac{C\,C_0^2\,C_W^2}{(\lambda_0-\lambda)^3}.\]
Then the conclusion follows from \eqref{3rd}, Corollary \ref{cor2expK},
conditions \eqref{condT} and \eqref{condc0}, and Step~2.
\end{proof}

\begin{Rk} Theorem \ref{thmgrowth} leads to enormous constants, and it
  is legitimate to ask about their sharpness, say with respect to the
  dependence in $\var$. We expect the constant to be
  roughly of the order of
  \[ \sup_t\ \bigl( e^{(ct)^{1/\gamma}}\, e^{-\var t}\bigr) \simeq
  \exp\left(\var^{-\frac1{\gamma-1}}\right).\] Our bound is roughly
  like $\exp(\var^{-(4+2\gamma)/(\gamma-1)})$; this is worse, but
  displays the expected behavior as an exponential of an inverse power
  of $\var$, with a power that diverges like $O((1-\gamma)^{-1})$ as
  $\gamma\to 1$. 
\end{Rk}

\begin{Rk} Even in the case of an analytic interaction, a similar argument
suggests constants that are at best like $(\ln 1/\var)^{\ln 1/\var}$,
and this grows faster than any inverse power of $1/\var$.
\end{Rk}

 To obtain
sharper results, in Section \ref{sec:coulomb} we shall later ``break
the norm'' and work directly on the Fourier modes of, say, the spatial
density. In this subsection we establish the estimates which will be
used later; the reader who does not particularly care about the case
$\gamma=1$ in Theorem \ref{thmmain} can skip them.

For any $\gamma\geq 1$, $\alpha>0$, $k,\ell\in\Z^d\setminus \{0\} =
\Z^d_*$ and $0\leq\tau\leq t$, we define \begeq\label{Kkl}
K^{(\alpha),\gamma}_{k,\ell}(t,\tau) = \frac{(1+\tau)\,
  e^{-\alpha|\ell|}\,e^{-\alpha\left(\frac{t-\tau}{t}\right)|k-\ell|}\,
  e^{-\alpha|k(t-\tau)+\ell\tau|}}{1+|k-\ell|^\gamma}.
\endeq

We start by exponential moment estimates.

\begin{Prop} \label{propexpkl} Let $\gamma\in [1,\infty)$ be
  given. For any $\alpha\in (0,1)$, $k,\ell\in\Z^d_*$, let
  $K_{k,\ell}^{(\alpha),\gamma}$ be defined by \eqref{Kkl}. Then there
  is $\ov{\alpha}=\ov{\alpha}(\gamma)>0$ such that if $\alpha\leq
  \ov{\alpha}$ and $\var\in (0,\alpha/4)$ then for any $t>0$
  \begeq\label{expkl1} \sup_{k\in\Z^d_*}\ \sum_{\ell\in\Z^d_*}
  e^{-\var t} \, \int_0^t
  K_{k,\ell}^{(\alpha),\gamma}(t,\tau)\,e^{\var\tau}\,d\tau \leq
  \frac{C(d,\gamma)}{\alpha^{1+d}\,\var^{\gamma+1}\,t^\gamma};
  \endeq
  \begeq\label{expkl2} \sup_{k\in\Z^d_*}\ \sum_{\ell\in\Z^d_*}
  e^{-\var t} \left(\int_0^t K_{k,\ell}^{(\alpha),\gamma}(t,\tau)^2\,
    e^{2\var\tau}\,d\tau\right)^{\frac12} \leq
  \frac{C(d,\gamma)}{\alpha^d\,\var^{\gamma+\frac12}\,t^{\gamma-\frac12}};
  \endeq
  \begeq\label{expkl3} \sup_{k\in\Z^d_*}\ \sum_{\ell\in\Z^d_*}
  \sup_{\tau\geq 0} e^{\var\tau} \int_\tau^\infty
  K_{k,\ell}^{(\alpha),\gamma}\,e^{-\var t}\,dt \leq
  \frac{C(d,\gamma)}{\alpha^{2+d}\,\var}.
\endeq
\end{Prop}

\begin{proof}[Proof of Proposition \ref{propexpkl}] We first reduce to
  the case $d=1$.  Monotonicity cannot be used now, but we note that
\[ K_{k,\ell}^{(\alpha),\gamma}(t,\tau)
\leq \sum_{1\leq j\leq d} e^{-\alpha|\ell_1|}\,e^{-\alpha|\ell_2|}\ldots
e^{-\alpha|\ell_{j-1}|}\,K_{k_j,\ell_j}^{(\alpha),\gamma}(t,\tau)\,
e^{-\alpha|\ell_{j+1}|}\ldots e^{-\alpha|\ell_d|},\]
where $K_{k_j,\ell_j}$ stands for a one-dimensional kernel. Thus
\begin{align*}
  \sup_k \sum_\ell \int_0^t & e^{-\var t}\,
  K^{(\alpha),\gamma}_{k,\ell}(t,\tau)\,e^{\var\tau}\,d\tau \leq
  \sup_k \left(\sum_{m\in\Z^d}e^{-\alpha|m|}\right)^{d-1} \sum_{1\leq
    j\leq d}\ \sum_{\ell_j \in\Z} \int_0^t e^{-\var t}\,
  K^{(\alpha),\gamma}_{k_j,\ell_j}(t,\tau)\,
  e^{\var\tau}\,d\tau\\
  & \leq \frac{C(d)}{\alpha^{d-1}}\ \sup_{1\leq j\leq d}\
  \sup_{k_j\in\Z} \sum_{\ell_j\in\Z} \int_0^t e^{-\var t}\,
  K_{k_j,\ell_j}^{(\alpha,\gamma)}(t,\tau)\,e^{\var\tau}\,d\tau.
\end{align*}
In other words, for \eqref{expkl1} we may just consider the
one-dimensional case, provided we allow an extra multiplicative constant
$C(d)/\alpha^{d-1}$. A similar reasoning holds for \eqref{expkl2} and
\eqref{expkl3}. From now on we focus on the case $d=1$.  \sm

Without loss of generality we assume $k>0$, and only treat the worse
case $\ell<0$. (The other case $k,\ell>0$ is simpler and
  yields an exponential decay in time of the form
  $e^{-c\,\min\{\alpha,\eps\}t}$). For simplicity we also write
$K_{k,\ell}=K^{(\alpha),\gamma}_{k,\ell}$. An easy computation yields
\begin{multline*}
  e^{-\var t} \int_0^t K_{k,\ell}(t,\tau)\,e^{\var \tau}\,d\tau \\
  \leq \frac{C\,e^{-\alpha|\ell|}}{1+|k-\ell|^\gamma}\
  \left(\frac1{\alpha |k-\ell|} + \frac{|k|t}{\alpha |k-\ell|^2} +
    \frac1{\alpha^2 |k-\ell|^2}\right)\, e^{-\frac{\var
      |\ell|t}{|k-\ell|}}.
\end{multline*}
Then for any $k\geq 1$, we have (crudely writing $\alpha^2 = O(\alpha)$)
\begin{multline}\label{suml}
  \sum_{\ell \leq -1} \int_0^t e^{-\var t}\,K_{k,\ell}(t,\tau)\,e^{\var\tau}\,d\tau\\
  \leq C \left(\sum_{\ell\geq 1} \frac{e^{-\alpha
        \ell}\,e^{-\frac{\var\ell t}{k+\ell}}}
    {\alpha^2\,(k+\ell)^{1+\gamma}} + \sum_{\ell\geq 1}
    \frac{e^{-\alpha\ell}\,e^{-\frac{\var\ell t}{k+\ell}}} {\alpha\,
      (k+\ell)^{2+\gamma}}\,kt\right).
\end{multline}

For the first sum in the right-hand side of \eqref{suml} we write
\begin{align}\label{firstsum}
  \sum_{\ell\geq 1} \frac{e^{-\alpha\ell}\,e^{-\frac{\var\ell
        t}{k+\ell}}}{(k+\ell)^{1+\gamma}} & \leq\sum_{\ell \ge 1}
  \frac{e^{-\alpha\ell}}{\ell^{1+\gamma}}\, \left[ \left(\frac{\var
        \ell t}{k+\ell}\right)^{1+\gamma}\, e^{-\frac{\var \ell
        t}{k+\ell}}\right]\
  \frac{1}{(\var t)^{1+\gamma}}\\
  & \leq \frac{C(\gamma)}{(\var t)^{1+\gamma}}. \nonumber
\end{align}

For the second sum in the right-hand side of \eqref{suml} we separate
according to $1\leq \ell\leq k$ or $\ell\geq k+1$:
\begin{align}\label{secondsum1}
  \sum_{1\leq \ell\leq k} \frac{e^{-\alpha\ell}\,e^{-\frac{\var\ell
        t}{k+\ell}}}{(k+\ell)^{2+\gamma}}\,kt
  & \leq \sum_{1\leq\ell\leq k} \frac{e^{-\alpha\ell}\, 
    e^{-\frac{\var t}{k+1}}}{(k+1)^{2+\gamma}}\,kt \\
  & \leq \frac{C}{\alpha} \left[ e^{-\frac{\var t}{k+1}}\,
    \left(\frac{\var t}{k+1}\right)^{1+\gamma}\right]
  \left(\frac{k}{k+1}\right)\, \frac{t}{(\var t)^{1+\gamma}} \nonumber \\
  & \leq \frac{C}{\alpha\,\var^{1+\gamma}\,t^\gamma}; \nonumber
\end{align}
\begin{align}\label{secondsum2}
  \sum_{\ell\geq k+1} \frac{e^{-\alpha\ell}\,e^{-\frac{\var\ell
        t}{k+\ell}}}{(k+\ell)^{2+\gamma}}\,kt
  & \leq C \sum_{\ell\geq k+1} \frac{e^{-\alpha\ell}\, 
    e^{-\frac{\var t}{2}}}{k^{2+\gamma}}\,kt \\
  & \leq \frac{C}{\alpha} \frac{e^{-\var t/4}}{\var\,k^{1+\gamma}}
  \leq \frac{C}{\alpha\,\var^{1+\gamma}\,t^\gamma}. \nonumber
\end{align}
The combination of \eqref{suml} \eqref{firstsum}, \eqref{secondsum1}
and \eqref{secondsum2} completes the proof of \eqref{expkl1}. \med

Now we turn to \eqref{expkl2}. The estimates are rather similar, since
\[ K_{k,\ell}(t,\tau)^2 \leq C\, (1+t)\,K_{k,\ell}(t,\tau)\] with
$\gamma \to 2 \gamma$ and $\alpha \to 2 \alpha$.  So \eqref{suml}
should be replaced by
\begin{multline}\label{sum2}
  \sum_\ell e^{-\var t} \left(\int_0^t K_{k,\ell}(t,\tau)^2\, 
    e^{2\var\tau}\,d\tau\right)^{\frac12} \\
  \leq C\left(\sum_{\ell\geq 1}
    \frac{e^{-\alpha\ell}\,e^{-\frac{\var\ell
          t}{k+\ell}}\,(1+t)^{1/2}}
    {\alpha\,(k+\ell)^{\frac12+\gamma}} \ + \sum_{\ell\geq 1}
    \frac{e^{-\alpha\ell}\,e^{-\frac{\var\ell
          t}{k+\ell}}\,(kt)^{1/2}\,(1+t)^{1/2}}
    {\alpha^{1/2}\,(k+\ell)^{1+\gamma}}\right).
\end{multline}
For the first sum we use \eqref{firstsum} with $\gamma$ replaced by
$\gamma-1/2$: for $t\geq 1$, \begeq\label{1sum2} (1+t)^{1/2} \sum_\ell
\frac{e^{-\alpha \ell}\,e^{-\frac{\var\ell t}{k+\ell}}}
{(k+\ell)^{1+(\gamma-1/2)}} \leq \frac{C(\gamma)\,t^{1/2}}{(\var
  t)^{\gamma+1/2}} \leq \frac{C(\gamma)}{\var^{\gamma+1/2}t^\gamma}.
\endeq
For the second sum in the right-hand side of \eqref{sum2} we write
\begin{align*}
  \sum_{1\leq \ell\leq k} \frac{e^{-\alpha\ell}\,e^{-\frac{\var\ell
        t}{k+\ell}}\,k^{1/2}t}{(k+\ell)^{1+\gamma}} & \leq C \sum_\ell
  e^{-\alpha \ell}\ \left[ e^{-\frac{\var t}{k+1}}\, \left(\frac{\var
        t}{k+1}\right)^{\gamma+1/2}
  \right]\, \frac{k^{1/2}}{(1+k)^{1/2}} \, \frac{t}{(\var t)^{\gamma+1/2}}\\
  & \leq \frac{C}{\alpha\,\var^{\gamma+\frac12}\,t^{\gamma-\frac12}}
\end{align*} 
and
\[
\sum_{\ell\geq k+1} \frac{e^{-\alpha\ell}\,e^{-\frac{\var\ell
      t}{k+\ell}}\,k^{1/2}t}{(k+\ell)^{1+\gamma}} \leq C \sum_\ell
\frac{e^{-\alpha\ell}}{\ell^{\gamma+\frac12}}\,e^{-\frac{\var t}2}\,t
\leq C\, e^{-\frac{\var t}{t}} \leq
\frac{C}{(\var\,t)^{\gamma-\frac12}\,\var}.
\]
With this \eqref{expkl2} is readily obtained. 
\med

Finally we consider \eqref{expkl3}. As in Proposition
\ref{propdualexpK} one easily shows that
\[ \sup_k \sum_\ell \sup_\tau e^{\var \tau}\int_{2\tau}^\infty
e^{-\var t}\,K_{k,\ell}(t,\tau)\,d\tau \leq
\frac{C}{\var\,\alpha^2}\,\sum_\ell e^{-\alpha|\ell|}\ \leq
\frac{C}{\var\,\alpha^3}.\] Then one has
\[ e^{\var\tau}\int_\tau^{2\tau} e^{-\alpha
  |k(t-\tau)+\ell\tau|}\,e^{-\var t}\,dt \leq \frac{C}{\alpha^2 k} +
\frac{C}{\alpha \eps k} + \frac{C}{\alpha k}\,e^{-\frac{\var \ell
    \tau}{k}}.\] 
So the problem amounts to estimate
\begin{align*}
\sum_\ell\ \sup_\tau \left[ (1+\tau)\ \frac{e^{-\alpha \ell}\,e^{-\frac{\var \ell \tau}{k}}}
{\alpha k (k+\ell)^\gamma} \right] 
& \leq \sum_\ell e^{-\alpha \ell}\left[\frac1{\alpha}
+ \frac1{\var \ell (k+\ell)^\gamma}\ \left(e^{-\frac{\var\ell\tau}{k}}\,\frac{\var\ell\tau}{k}\right)\right]\\
& \leq C\, \left(\frac1{\alpha^2} + \frac1{\var}\right),
\end{align*}
and the proof is complete.
\end{proof}

We conclude this section with a mode-by-mode analogue of Theorem
\ref{thmgrowth}.

\begin{Thm} \label{thmgrowthk}
Let $f^0=f^0(v)$ and $W=W(x)$ satisfy condition {\bf (L)} from Subsection \ref{sub:lineardamping}
with constants $C_0,\lambda_0,\kappa$; in particular
$|\tilde{f}^0(\eta)|\leq C_0\, e^{-2\pi\lambda_0|\eta|}$. Further let
\[ C_W = \max \left\{ \sum_{k\in\Z^d_*} |\hat{W}(k)|,\ \
\sup_{k\in\Z^d_*}\, |k|\,|\hat{W}(k)| \right\}.\] Let
$(A_k)_{k\in\Z^d_*}$, $\mu\geq 0$, $\lambda\in (0,\lambda^*]$ with
$0<\lambda^*<\lambda_0$. Let $(\Phi(k,t))_{k\in\Z^d_*,\ t\geq 0}$ be a
continuous function of $t\geq 0$, valued in $\C^{\Z^d_*}$, such that
for all $t\geq 0$ and $k\in\Z^d_*$,
\begin{multline} \label{cond74k} e^{2\pi(\lambda t+\mu)|k|}\ \Bigl|
  \Phi(k,t) - \int_0^t K^0(k,t-\tau)\,\Phi(k,\tau)\,d\tau \Bigr|
  \leq A_k + \int_0^t K_0(t,\tau)\,e^{2\pi(\lambda \tau+\mu)|k|}\, 
  |\Phi(k,\tau)|\,d\tau\\
  + \int_0^t \sum_{\ell\in\Z^d_*} \left(
    c\,K^{(\alpha),\gamma}_{k,\ell}(t,\tau) +
    \frac{c_\ell}{(1+\tau)^m}
  \right)\,e^{2\pi(\lambda\tau+\mu)|k-\ell|}\,
  |\Phi(k-\ell,\tau)|\,d\tau,
\end{multline}
where $c>0$, $c_\ell\geq 0$ ($\ell\in\Z^d_*$), $m>1$, $\gamma\geq 1$,
$K_0(t,\tau)$ is a nonnegative kernel, $K_{k,\ell}^{(\alpha),\gamma}$
are defined by \eqref{Kkl}, $\alpha<\ov{\alpha}(\gamma)$ defined in
Proposition \ref{propexpkl}. Then there are positive constants $C$ and
$\chi$, depending only on $\gamma$, $\lambda^*$, $\lambda_0$,
$\kappa$, $\bar c:= \max\{ \sum_\ell c_\ell, \left( \sum_\ell c_\ell ^2
  \right)^{1/2}\}$, $C_W$,
$m$, such that if \begeq\label{condchik} \sup_{t\geq 0} \int_0^t
K_0(t,\tau)\,d\tau \leq \chi
\endeq
and \begeq\label{condc0k} \sup_{t\geq 0} \left( \int_0^t
  K_0(t,\tau)^2\,d\tau \right)^{1/2} + \sup_{\tau \geq 0}
\int_\tau^\infty K_0(t,\tau)\,dt \leq 1,
\endeq
then for any $\var \in (0,\alpha/4)$ and for any $t\geq 0$,
\begeq\label{phileqk} \sup_k\ \bigl(e^{2\pi (\lambda t + \mu)|k|}\,
|\Phi(k,t)| \bigr) \leq C\, \bar A \, \frac{(1+\bar c ^2)}{\sqrt{\var}} \, 
e^{C \, \bar c}\,\left(1+\frac{c}{\alpha^2\,\var}\right)\,
e^{CT}\, e^{C\,\frac{c}{\alpha}\,(1+T^2)}\,e^{\var t},
\endeq
where $\bar A := \left(\sup_k A_k \right)$ and 
\begeq\label{condTk} T = C\,\max
\left\{\left(\frac{c^2}{\alpha^{3+2d}\,\var^{\gamma+2}}
  \right)^{\frac1{\gamma}};
  \,
  \left(\frac{c}{\alpha^d\,\var^{\gamma+\frac12}}
  \right)^{\frac1{\gamma-\frac12}};\,
  \left(\frac{\bar c^2}{\var}\right)^{\frac1{2m-1}}\right\}.
\endeq 
\end{Thm}

\begin{proof}[Proof of Theorem \ref{thmgrowthk}] The proof is quite
  similar to the proof of Theorem \ref{thmgrowth}, so we shall only
  point out the differences. 
  As in the proof of Theorem \ref{thmgrowth} we start by crude
  pointwise bounds obtained by Gronwall inequality; but this time on
  the quantity
  \[ \varphi(t) = \sup_k\, |\Phi(k,t)|\,e^{2\pi(\lambda t+\mu)|k|}.\]
  Since $\sum_\ell K_{k,\ell}(t,\tau) = O((1+\tau)/\alpha)$, we find
  \begeq\label{phisupk} \varphi(t) \leq 2\,\bar A \,e^{C
    \left(\frac{C_0\,C_W}{\lambda_0-\lambda} t +
      \frac{c}{\alpha}(t+t^2) + \bar c \, C_m\right)}.
\endeq

Next we define $\Psi_k$, ${\cal K}^0_k$, $R_k$ as in Step~2 of the
proof of Theorem \ref{thmgrowth}, and we deduce \eqref{PsikRk} and
\eqref{PsiRkleq}. Let \begeq\label{phik} \varphi_k(t) =
e^{2\pi(\lambda t+\mu)|k|}\,|\Phi(k,t)|,
\endeq
then
\begin{align*}
  \bigl\|\varphi_k(t)\,e^{-\var t}\bigr\|_{L^2(dt)}
  & \leq \|R_k\|_{L^2(dt)}\, 
  \left(1+ \frac{\|{\cal K}^0_k\|_{L^1(dt)}}{\kappa}\right) \\
  & \leq \|R_k\|_{L^2(dt)}\, \left(1+
    \frac{C\,C_W\,C_0}{\kappa}\right);
\end{align*}
whence
\begin{multline}\label{phiepsL2k}
  \bigl\|\varphi_k(t)\,e^{-\var t}\bigr\|_{L^2(dt)} \leq \left( 1 +
    \frac{C\,C_0\,C_W}{\kappa\,(\lambda_0-\lambda)^2}\right)\,
  \left\{ \int_0^\infty e^{-2\var t}
    \left( A_k + \int_0^t K_0(t,\tau)\,\varphi_k(\tau)\,d\tau\right.\right.\\
  \left.\left.  + \sum_\ell \int_0^t \left( c\,K_{k,\ell}(t,\tau) +
        \frac{c_\ell}{(1+\tau)^m}\right)\,
      \varphi_{k-\ell}(\tau)\,d\tau\right)^2\,dt\right\}^{\frac12}.
\end{multline}

We separate this into various contributions as in the proof of Theorem
\ref{thmgrowth}.  In particular, using \eqref{phisupk} and $\int_0^t
\sum_\ell K_{k,\ell}\,d\tau = O((1+t)/\alpha^2)$,
we find
\begin{align} \label{oc2k} \biggl \{ \int_0^T e^{-2 \var t} &
  \left(\int_0^t \sum_\ell
    K_{k,\ell}(t,\tau)\,
    \varphi_{k-\ell}(\tau)\,d\tau \right)^2\,dt\biggr\}^{\frac12}\\
  & \nonumber \leq \left[\sup_k\ \sup_{0\leq t\leq T}
    \varphi_k(t)\right] \left(\int_0^T e^{-2\var t}\left(\int_0^t
      \sum_\ell K_{k,\ell}(t,\tau)\,d\tau\right)^2\,dt\right)^{\frac12}\\
  & \nonumber \leq C\, \bar A\,\frac{c}{\alpha^2 \,\var^{3/2}}\, e^{C
    \left[\frac{C_0\,C_W}{\lambda_0-\lambda}\,T +
      \frac{c}{\alpha}\,(T+T^2)\right]}.
\end{align}
Also,
\begin{multline*}
  \biggl\{ \int_T^\infty e^{-2\var t} \left(\int_0^t \sum_\ell
    K_{k,\ell}(t,\tau)\,\varphi_{k-\ell}(\tau)\,
    d\tau\right)^2\,dt\biggr\}^{\frac12}\\
  \leq\left(\sup_{t\geq T} \int_0^t e^{-\var t}\,\sum_\ell
    K_{k,\ell}(t,\tau) \,e^{\var \tau}\,d\tau\right)^{\frac12}
  \left(\int_T^\infty \int_0^t \sum_\ell
    K_{k,\ell}(t,\tau)\,e^{-\var(t-\tau)}\,e^{-2\var
      \tau}\varphi_{k-\ell}(\tau)^2\,d\tau\,dt\right)^{\frac12},
\end{multline*}
and the last term inside parentheses is
\begin{multline*}
  \sum_\ell \int_0^\infty \left(\int_{\max\{\tau;T\}}^\infty
    K_{k,\ell}(t,\tau)\,e^{-\var (t-\tau)}\,dt\right)\,
  e^{-2\var\tau}\,\varphi_{k-\ell}(\tau)^2\,d\tau \\
  \leq \left(\sum_\ell \ \sup_\tau \int_\tau^\infty
    K_{k,\ell}(t,\tau)\,e^{-\var (t-\tau)}\,dt\right) \
  \left[\sup_\ell \int
    e^{-2\var\tau}\,\varphi_\ell(\tau)^2\,d\tau\right].
\end{multline*}

The computation for $K_0$ is the same as in the proof of Theorem
\ref{thmgrowth}, and the terms in $(1+\tau)^{-m}$ are handled in
essentially the same way: simple computations yield
\begin{align}\label{oc5k}
  \biggl\{\int_0^\infty e^{-2\var t} & \, \left(\int_0^T
    \frac{\sum_\ell c_\ell \,
      \varphi_{k-\ell}(\tau)}{(1+\tau)^m}\,d\tau\right)^2 \,dt \biggr\}^{\frac12}\\
  \nonumber & \leq \left(\sup_{0\leq \tau\leq T}\ \sup_\ell
    \varphi_\ell(\tau)\right) \left\{\int_0^\infty e^{-2\var
      t}\left(\int_0^T \frac{(\sum c_\ell)\,
        d\tau}{(1+\tau)^m}\right)^2\,
    dt\right\}^{\frac12}\\
  \nonumber & \leq \bar c \,\frac{C_m\, \bar A}{\sqrt{\var}}\, e^{C\left[
      \left(\frac{C_0\,C_W}{\lambda_0-\lambda}\right)T +
      \frac{c}{\alpha}\, (T+T^2)\right]}
\end{align}
and
\begin{align}\label{oc6k}
  \biggl\{\int_0^\infty & e^{-2\var t}
  \left(\int_T^t \frac{\sum_\ell 
 c_\ell\,\varphi_{k-\ell}(\tau)\,d\tau}{(1+\tau)^m}\right)^2\,dt\biggr\}^{\frac12}\\
  & \leq \left\{ \sup_{t \ge 0, \ \ell} \left(\int_T^t e^{-2\var
        \tau}\,\varphi_\ell(\tau)^2\,d\tau\right)\, \left(\sum_\ell
      c_\ell\right)^2 \left(\int_0^\infty \int_T^t
      \frac{e^{-2\var(t-\tau)}}{(1+\tau)^{2m}}\,d\tau\,dt\right)\right\}^{\frac12}
  \nonumber \\
  & \leq \bar c \, \left(\frac{C_{2m}}{\var\, T^{2m-1}}\right)^{\frac12}\,
   \left(\sup_\ell \int_0 ^{+\infty}
    e^{-2\var\tau}\,\varphi_\ell(\tau)^2\,d\tau\right)^{\frac12}. \nonumber
\end{align}

All in all, we end up with 
\begin{multline} \label{endstep2k} \sup_k\
  \bigl\|\varphi_k(t)\,e^{-\var t}\bigr\|_{L^2(dt)} \leq \left( 1 +
    \frac{C\,C_0\,C_W}{\kappa\,(\lambda_0-\lambda)^2}\right)\,
  \frac{C\,A}{\sqrt{\var}}\, \left[ 1 + \left(\frac{c}{\alpha^2\,\var}+
      \bar c \,C_m\right)\right]\,
  e^{C\,\left[\frac{C_0\,C_W}{\lambda_0-\lambda}\,T + \frac{c}{\alpha}\, (T+T^2)\right]}\\
  \ + a\, \sup_k\ \bigl\|\varphi_k(t)\,e^{-\var t}\bigr\|_{L^2(dt)},
\end{multline}
where
\begin{multline*}
  a =
  \left(1+\frac{C\,C_0\,C_W}{\kappa\,(\lambda_0-\lambda)^2}\right)\
  \Biggl[ \\
  c^2 \, \left(\sup_{t\geq T} \sum_{\ell} \int_0^t e^{-\var
      t}\,K_{k,\ell}(t,\tau)\, e^{\var \tau}\,d\tau\right)^{\frac12}
  \left(\sum_\ell \sup_{\tau\geq 0} \int_\tau^\infty
    e^{\var\tau}\,K_{k,\ell}(t,\tau)\,
    e^{-\var t}\,dt\right)^{\frac12}\\
  + \left(\sup_{t\geq 0}\int_0^t K_0(t,\tau)\,d\tau\right)^{\frac12}
  \left(\sup_{\tau\geq 0}\int_\tau^\infty
    K_0(t,\tau)\,dt\right)^{\frac12} + \,\frac{C_{2m}^{1/2}\,\bar
    c_0}{\sqrt{\var}\,T^{m-1/2}}\Biggr].
\end{multline*}

Applying Proposition \ref{propexpkl}, we see that $a\leq 1/2$ as soon
as $T$ satisfies \eqref{condTk}, and then we deduce from
\eqref{endstep2k} a bound on $\sup_k \|\varphi_k(t)\,e^{-\var
  t}\|_{L^2(dt)}$.

Finally, we conclude as in Step~3 of the proof of Theorem
\ref{thmgrowth}: from \eqref{cond74k},
\begin{align} \label{3rdk} e^{-\var t}\,\varphi_k(t) & \leq
  A_k\,e^{-\var t} + \Biggl[ \biggl(\int_0^t \biggl(\sup_{k\in\Z^d_*}
  |K^0(k,t-\tau)|\,
  e^{2\pi\lambda (t-\tau)|k|}\biggr)^2\,d\tau \biggr)^{\frac12}\\
  \nonumber & \qquad \qquad\qquad + \left(\int_0^t
    K_0(t,\tau)^2\,d\tau\right)^{\frac12}
  + \bar c \, \left(\int_0^\infty 
    \frac{d\tau}{(1+\tau)^{2m}}\right)^{\frac12}\\
  \nonumber & \qquad \qquad\qquad + c\, \sum_\ell \left(\int_0^t e^{-2\var
      t}\,K_{k,\ell}(t,\tau)^2\,e^{2\var\tau}\,d\tau\right)^{\frac12}
  \Biggr]\ \left(\sup_k
    \int_0^\infty\varphi_k(\tau)^2\,e^{-2\var\tau}\,d\tau\right)^{\frac12},
\end{align}
and the conclusion follows by a new application of Proposition
\ref{propexpkl}.
\end{proof}

\section{Approximation schemes}
\label{sec:approx}

Having defined a functional setting (Section \ref{sec:analytic}) and
identified several mathematical/physical mechanisms (Sections
\ref{sec:scattering} to \ref{sec:response}), we are prepared to fight
the Landau damping problem.  For that we need an approximation scheme
solving the nonlinear Vlasov equation. The problem is not to prove the
existence of solutions (this is much easier), but to devise the scheme
in such a way that it leads to relevant estimates for our study.

The first idea which may come to mind is a classical Picard scheme for quasilinear equations:
\begeq\label{picard1}
\pa_t f^{n+1} + v\cdot\nabla_x f^{n+1} + F[f^n]\cdot\nabla_v f^{n+1} = 0.
\endeq
This has two drawbacks: first, $f^{n+1}$ evolves by the characteristics
created by $F[f^n]$, and this will deteriorate the estimates in analytic regularity.
Secondly, there is no hope to get a closed (or approximately closed)
equation on the density associated with $f^{n+1}$. More promising, and more in the
spirit of the linearized approach, would be a scheme like
\begeq\label{picard2}
\pa_t f^{n+1} + v\cdot\nabla_x f^{n+1} + F[f^{n+1}]\cdot\nabla_v f^n =0.
\endeq
(Physically, $f^{n+1}$ forces $f^n$, and the question is whether the reaction
will exhaust $f^{n+1}$ in large time.) But when we write \eqref{picard2} we
are implicitly treating a higher order term ($\nabla_v f$) of the equation
in a perturbative way; so this has no reason to converge.

To circumvent these difficulties, we shall use a {\bf Newton iteration}:
not only will this provide more flexibility in the
regularity indices, but at the same time it will yield an extremely fast
rate of convergence (something like $O(\var^{2^n})$) which will be
most welcome to absorb the large constants coming from Theorem
\ref{thmgrowth} or Theorem \ref{thmgrowthk}.

\subsection{The natural Newton scheme} \label{natural}

Let us adapt the abstract Newton scheme to an abstract evolution equation in the form
\[ \derpar{f}{t} = Q(f),\]
around a stationary solution $f^0$ (so $Q(f^0)=0$). Write the Cauchy problem with initial datum $f_i\simeq f^0$
in the form
\[ \Phi(f) := \Bigl( \pa_t f - Q(f),\ f(0,\,\cdot\,)\Bigr) - (0,f_i).\]
Starting from $f^0$, the Newton iteration consists in solving inductively 
$\Phi(f^{n-1}) + \Phi'(f^{n-1}) \cdot (f^n-f^{n-1}) =0$ for $n\geq 1$. More explicitly, writing $h^n = f^n-f^{n-1}$,
we should solve
\[ \begin{cases} \pa_t h^1 = Q'(f^0)\cdot h^1 \\[2mm]
h^1(0,\,\cdot\,) = f_i - f^0
\end{cases}\]
\[ \forall n\geq 1,\quad
\begin{cases} \pa_t h^{n+1} = Q'(f^n)\cdot h^{n+1} - \bigl[ \pa_t f^n - Q(f^n)\bigr]\\[2mm]
h^{n+1}(0,\,\cdot\,) = 0.
\end{cases}\]
By induction, for $n\geq 1$ this is the same as
\[ \begin{cases} \pa_t h^{n+1} = Q'(f^n)\cdot h^{n+1} + \Bigl[ Q(f^{n-1}+h^n) - Q(f^{n-1}) - Q'(f^{n-1})\cdot h^n\Bigr]\\[2mm]
h^{n+1}(0,\,\cdot\,) = 0.
\end{cases}\]

This is easily applied to the nonlinear Vlasov equation, for which the nonlinearity is quadratic. So we define
the {\bf natural Newton scheme for the nonlinear Vlasov equation} as follows:

\[ f^0 = f^0(v) \quad \text{is given (homogeneous stationary state)}\]
\[ f^n = f^0 + h^1 + \ldots + h^n, \qquad \text{where}\]
\begeq\label{N1}
\begin{cases} \pa_t h^1 + v\cdot\nabla_x h^1 + F[h^1]\cdot\nabla_v f^0 =0 \\[2mm]
h^1(0,\,\cdot\,) = f_i -f^0
\end{cases} \endeq
\begeq\label{N2}
\forall \, n\geq 1,\quad
\begin{cases} \pa_t h^{n+1} + v\cdot\nabla_x h^{n+1} +
  F[f^{n}]\cdot\nabla_v h^{n+1} + F[h^{n+1}]\cdot\nabla_v f^{n} =
  -F[h^{n}]\cdot\nabla_v h^{n}\\[2mm]
  h^{n+1}(0,\,\cdot\,) = 0.
\end{cases} \endeq Here $F[f]$ is the force field created by the
particle distribution $f$, namely \begeq\label{Force} F[f](t,x) = -
\iint_{\T^d\times\R^d} \nabla W(x-y)\, f(t,y,w)\,dy\,dw.\endeq Note
also that all the $\rho^n=\int h^n \, dv$ for $n \ge 1$ have zero
spatial average.

\subsection{Battle plan}

The treatment of \eqref{N1} was performed in Subsection
\ref{sub:revisited}. Now the problem is to handle all equations
appearing in \eqref{N2}. This is much more complicated, because for
$n\geq 1$ the background density $f^{n}$ depends on $t$ and $x$,
instead of just $v$; as a consequence, \sm

(a) Equation \eqref{N2} cannot be considered as a perturbation of free
transport, because of the presence of $\nabla_v h^{n+1}$ in the
left-hand side; \sm

(b) The reaction term $F[h^{n+1}]\cdot\nabla_v f^{n}$ no longer has the
simple product structure (function of $x$)$\times$(function of $v$),
so it becomes harder to get hands on the homogenization phenomenon;
\sm

(c) Because of spatial inhomogeneities, echoes will appear; they are
all the more dangerous that, $\nabla_v f^n$ is unbounded as $t\to\infty$,
{\em even in gliding regularity}.
(It grows like $O(t)$, which is reminiscent of the observation made
by Backus \cite{backus}.)
\sm

The estimates in Sections \ref{sec:scattering} to \ref{sec:response}
have been designed precisely to overcome these problems; however we
still have a few conceptual difficulties to solve before applying
these tools.

Recall the discussion in Subsection \ref{submeasur}: the natural
strategy is to propagate the bound \begeq\label{supf1} \sup_{\tau\geq
  0} \|f_\tau\|_{\cZ^{\lambda,\mu;1}_\tau} < +\infty
\endeq
along the scheme; this estimate contains in particular two crucial
pieces of information:

\bul a control of $\rho_\tau = \int f_\tau\,dv$ in $\cF^{\lambda\tau
  +\mu}$ norm; \sm

\bul a control of $\< f_\tau\> = \int f_\tau\,dx$ in $\cC^{\lambda;1}$
norm.  \sm

So the plan would be to try to get inductively estimates of each $h^n$
in a norm like the one in \eqref{supf1}, in such a way that $h^n$ is
extremely small as $n\to\infty$, and allowing a slight deterioration
of the indices $\lambda,\mu$ as $n\to\infty$. Let us try to see how
this would work: assuming
\[ \forall \, 0 \le k \le n, \quad 
\sup_{\tau\geq 0} \|h_\tau^k\|_{\cZ^{\lambda_k,\mu_k;1}_\tau} \leq
\delta_k,\] we should try to bound $h_\tau^{n+1}$. To ``solve''
\eqref{N2}, we apply the classical method of characteristics: as in
Section \ref{sec:scattering} we define $(X_{\tau,t}^n,V_{\tau,t}^n)$
as the solution of
\[ \begin{cases}
\dps \frac{d}{dt} X_{\tau,t}^n(x,v) = V_{\tau,t}^n(x,v),\qquad
\frac{d}{dt}V_{\tau,t}^n(x,v) = F[f^n]\bigl(t,X^n_{\tau,t}(x,v)\bigr)\\[4mm]
X^n_{\tau,\tau}(x,v) = x, \quad V^n_{\tau,\tau}(x,v) = v.
\end{cases}\]
Then \eqref{N2} is equivalent to
\begeq\label{N2equiv}
\frac{d}{dt} h^{n+1}\Bigl(t,X_{0,t}^n,V_{0,t}^n(x,v)\Bigr)
= \Sigma^{n+1} \Bigl(\tau, X_{0,\tau}^n(x,v),V_{0,\tau}^n(x,v)\Bigr),
\endeq
where
\begeq\label{Sigman} \Sigma^{n+1}(t,x,v) = - F[h^{n+1}]\cdot\nabla_v f^{n}
- F[h^{n}]\cdot\nabla_v h^{n}.
\endeq
Integrating \eqref{N2equiv} in time and recalling that $h^{n+1}(0,\cdot)=0$, we get
\[ h^{n+1}\Bigl( t, X_{0,t}^n(x,v),V_{0,t}^n(x,v)\Bigr)
= \int_0^t \Sigma^{n+1}\Bigl(\tau,X_{0,\tau}^n(x,v),V_{0,\tau}^n(x,v)\Bigr)\,d\tau.\]
Composing with $(X_{t,0}^n,V_{t,0}^n)$ and using \eqref{SSS} yields
\[ h^{n+1}(t,x,v) = \int_0^t \Sigma^{n+1}\Bigl(\tau, X_{t,\tau}^n(x,v), V_{t,\tau}^n(x,v)\Bigr)\,d\tau.\]
We rewrite this using the ``scattering transforms''
\[ \Om_{t,\tau}^n(x,v) = (X_{t,\tau}^n,V_{t,\tau}^n)(x+v(t-\tau),v) = S^n_{t,\tau}\circ S^0_{\tau,t};\]
then we finally obtain
\begin{align}\label{hfinal}
h^{n+1} (t,x,v) & = \int_0^t \bigl(\Sigma_\tau^{n+1}\circ \Om_{t,\tau}^n\bigr) (x-v(t-\tau),v)\,d\tau\\
\nonumber & = 
- \int_0^t \Bigl[ \Bigl( F[h_\tau^{n+1}]\circ \Om_{t,\tau}^n\Bigr)\cdot
\Bigl( \bigl(\nabla_v f_\tau^n\bigr)\circ \Om_{t,\tau}^n\Bigr)\Bigr] (x-v(t-\tau),v)\,d\tau\\
\nonumber & \quad 
- \int_0^t \Bigl[ \Bigl( F[h_\tau^{n}]\circ \Om_{t,\tau}^n\Bigr)\cdot
\Bigl( \bigl(\nabla_v h_\tau^n\bigr)\circ \Om_{t,\tau}^n\Bigr)\Bigr] (x-v(t-\tau),v)\,d\tau.
\end{align}
Since the unknown $h^{n+1}$ appears on both sides of \eqref{hfinal},
we need to get a self-consistent estimate.  For this we have little
choice but to integrate in $v$ and get an integral equation on
$\rho[h^{n+1}] = \int h^n\,dv$, namely
\begin{multline}\label{rho}
\rho[h^{n+1}](t,x) = \int_0^t \int \biggl[ \Bigl( \bigl( \rho[h_\tau^{n+1}]\ast\nabla W\bigr)\circ \Om_{t,\tau}^n\Bigr)\cdot
G_{\tau,t}^n\biggr]\circ S^0_{\tau-t}(x,v)\,dv\,d\tau\ \\
+ (\text{stuff from stage $n$}),
\end{multline}
where $G^n_{\tau,t}=\nabla_v f^n_\tau\circ \Om_{t,\tau}^n$. By
induction hypothesis $G^n_{\tau,t}$ is smooth with regularity indices
roughly equal to $\lambda_n$, $\mu_n$; so if we accept to lose just a
bit more on the regularity we may hope to apply the long-term
regularity extortion and decay estimates from Section
\ref{sec:regdecay}, and then time-response estimates of
  Section \ref{sec:response}, and get the desired damping.

However, we are facing a major problem: composition of $\rho[h_\tau^{n+1}]\ast\nabla W$ by $\Om_{t,\tau}^n$
implies a loss of regularity in the right-hand side with respect to the left-hand side, which is of course
unacceptable if one wants a closed estimate.
The short-term regularity extortion from Section \ref{sec:regdecay} remedies this, but the price to pay is that $G^n$ should
now be estimated at time $\tau'=\tau-bt/(1+b)$ instead of $\tau$, and with index of gliding analytic regularity
roughly equal to $\lambda_n(1+b)$ rather than $\lambda_n$. Now the catch is that the error induced by composition
by $\Om^n$ depends on {\em the whole distribution $f^n$}, not just $h^n$; thus, if the parameter $b$ should
control this error it should stay of order~1 as $n\to\infty$, instead of converging to~0.

So it seems we are sentenced to lose a fixed amount of regularity (or
rather of radius of convergence) in the transition from stage $n$ to
stage $n+1$; this is reminiscent of the ``Nash--Moser syndrom''
\cite{AlGe}.  The strategy introduced by Nash \cite{Na56} to remedy
such a problem (in his case arising in the construction of $C^\infty$
isometric imbeddings) consisted in combining a Newton scheme with
regularization; his method was later developed by Moser \cite{Mo66}
for the $C^\infty$ KAM theorem (see \cite[pp.~19--21]{Mo} for some
interesting historical comments). The Nash--Moser technique is
arguably the most powerful perturbation technique known to this
day. However, despite significant effort, we were unable to set up any
relevant regularization procedure (in gliding regularity, of course)
which could be used in Nash--Moser style, because of three serious
problems:
\sm

\bul The convergence of the Nash--Moser scheme is no longer as fast as
that of the ``raw'' Newton iteration; instead, it is determined by the
regularity of the data, and the resulting rates would be unlikely to
be fast enough to win over the gigantic constants coming from Section
\ref{sec:response}.  \sm

\bul Analytic regularization in the $v$ variable is extremely costly,
especially if we wish to keep a good localization in velocity space,
as the one appearing in Theorem \ref{thminj}(iii), that is exponential
integrability in $v$; then the uncertainty principle basically forces
us to pay $O(e^{C/\var^2})$, where $\var$ is the strength of the
regularization.  \sm

\bul Regularization comes with an increase of amplitude (there is as
usual a trade-off between size and regularity); if we regularize
before composition by $\Om^n$, this will devastate the estimates,
because the analytic regularity of $f\circ g$ depends not only on the
regularity of $f$ and $g$, but also on the amplitude of $g-\Id$.  \sm

Fortunately, it turned out that a ``raw'' Newton scheme could be used; but this required to give up the natural
estimate \eqref{supf1}, and replace it by the {\em pair} of estimates
\begeq\label{pairest}
\begin{cases}
\dps \sup_{\tau\geq 0}\ \|\rho_\tau\|_{\cF^{\lambda\tau + \mu}} < +\infty;\\[4mm]
\dps \sup_{t\geq \tau\geq 0}\ \Bigl\| f_\tau\circ \Om_{t,\tau}\Bigr\|_{\cZ^{\ov{\lambda}(1+b),\ov{\mu};1}_{\tau - \frac{bt}{1+b}}} < +\infty.
\end{cases}
\endeq
Here $b=b(t)$ takes the form ${\rm const.}/(1+t)$, and is kept fixed
all along the scheme; moreover $\lambda,\mu$ will be slightly larger
than $\ov{\lambda}$, $\ov{\mu}$, so that none of the two estimates in
\eqref{pairest} implies the other one. Note carefully that there are
now two times ($t$, $\tau$) explicitly involved, so this is much more
complex than \eqref{supf1}. Let us explain why this strategy is
nonetheless workable.

First, the density $\rho^n=\int f^n\,dv$ determines the
characteristics at stage $n$, and therefore the associated scattering
$\Om^n$. If $\rho_\tau^n$ is bounded in $\cF^{\lambda_n\tau + \mu}$,
then by Theorem \ref{thmscat} we can estimate $\Om^n_{t,\tau}$ in
${\cZ^{\lambda'_n,\mu'_n}_{\tau'}}$, as soon as (essentially)
$\lambda'_n\,\tau' + \mu'_n \leq \lambda_n\,\tau + \mu_n$,
$\lambda'_n<\lambda_n$, and these bounds are uniform in $t$.

Of course, we cannot apply this theorem in the present context,
because $\ov{\lambda}_n (1+b)$ is not bounded above by
$\lambda_n$. However, {\em for large times $t$} we may afford
$\ov{\lambda}_n(1+b(t))<\lambda_n$, while
$\ov{\lambda}_n(1+b)(\tau-bt/(1+b))\leq \lambda_n\tau$ for all times;
this will be sufficient to repeat the arguments in Section
\ref{sec:scattering}, getting uniform estimates {\em in a regularity
  which depends on $t$}.  (The constants are uniform in $t$; but the
index of regularity goes down with $t$.) We can also do this while
preserving the other good properties from Theorem \ref{thmscat},
namely exponential decay in $\tau$, and vanishing near $\tau=t$.

Figure \ref{lambda} below summarizes schematically the way we choose and estimate the
gliding regularity indices.
\med

\begin{figure}[htbp]
\begin{center}
\input{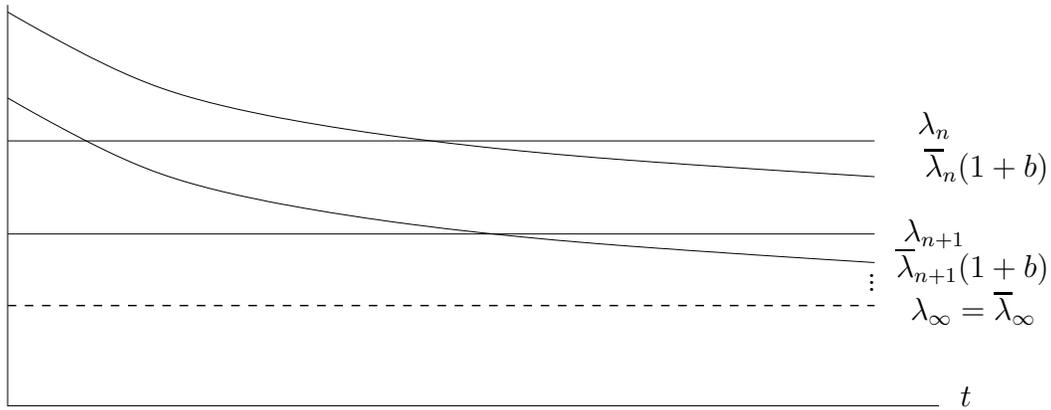}
\caption{Indices of gliding regularity appearing throughout our Newton
  scheme, respectively in the norm of $\rho[h_\tau]$ and in the norm
  of $h_\tau\circ \Om_{t,\tau}$, plotted as functions of $t$}
\label{lambda}
\end{center}
\end{figure}

Besides being uniform in $t$, our bounds need to be uniform in
$n$. For this we shall have to {\em stratify} all our estimates, that
is decompose $\rho[f^n] = \rho[h^1]+\cdots + \rho[h^n]$, and consider
separately the influence of each term in the equations for
characteristics. This can work only if the scheme converges very fast.

Once we have estimates on $\Om^n_{t,\tau}$ in a time-varying
regularity, we can work with the kinetic equation to derive estimates
on $h_\tau^n\circ \Om_{t,\tau}^n$; and then on all
$h_\tau^k\circ\Om^n_{t,\tau}$, also in a norm of time-varying
regularity. We can also estimate their spatial average, in a norm
${\cal C}^{\ov{\lambda}(1+b);1}$; thanks to the exponential
convergence of the scattering transform as $\tau\to\infty$ these
estimates will turn out to be uniform in $\tau$.  

Next, we can use all
this information, in conjunction with Theorem \ref{thmcombi'}, to get
an integral inequality on the norm of $\rho[h^{n+1}_\tau]$ in ${\cal
  F}^{\lambda\tau+\mu}$, where $\lambda$ and $\mu$ are only slightly
smaller than $\lambda_n$ and $\mu_n$.
Then we can go through the response estimates of Section \ref{sec:response},
this gives us an arbitrarily small loss in the
exponential decay rate, at the price of a huge constant which will eventually be
wiped out by the fast convergence of the scheme. So we have an
estimate on $\rho[h^{n+1}]$, and we are in business to continue the
iteration.  (To ensure the propagation of the linear damping
condition, or equivalently of the smallness of $K_0$ in Theorem
\ref{thmgrowth}, throughout the scheme, we shall have to stratify the
estimates once more.)

\section{Local in time iteration}
\label{sec:local}

Before working out the core of the proof of Theorem \ref{thmmain} in
Section \ref{sec:iter}, we shall need a short-time estimate, which
will act as an ``initial regularity layer'' for the Newton scheme.
(This will give us room later to allow the regularity index to depend
on $t$.)  So we run the whole scheme once in this section, and another
time in the next section.

Short-time estimates in the analytic class are not new for the
nonlinear Vlasov equation: see in particular the work of Benachour
\cite{benachour} on Vlasov--Poisson.  His arguments can probably be
adapted for our purpose; also the Cauchy--Kowalevskaya method could
certainly be applied. We shall provide here an alternative method,
based on the analytic function spaces from Section \ref{sec:analytic},
but not needing the apparatus from Sections \ref{sec:scattering} to
\ref{sec:response}. Unlike the more sophisticated estimates which will
be performed in Section \ref{sec:iter}, these ones are ``almost'' Eulerian
(the only characteristics are those of free transport).
The main tool is the

\begin{Lem} \label{lemderivnorm} Let $f$ be an analytic function,
  $\lambda(t) = \lambda - K\,t$, $\mu(t) = \mu - K\,t$; let $T>0$ be
  so small that $\lambda(t), \mu(t)>0$ for $0\leq t\leq T$.
  Then for any $\tau\in [0,T]$ and any $p\geq 1$,
  \begeq\label{leftddtnorm} \left.\frac{d}{dt}^+ \right|_{t=\tau}
  \|f\|_{\cZ^{\lambda(t),\mu(t);p}_\tau} \leq -\frac{K}{1+\tau}\, \|\nabla
  f\|_{\cZ^{\lambda(\tau),\mu(\tau);p}_\tau},
\endeq
where $(d^+/dt)$ stands for the upper right derivative.
\end{Lem}

\begin{Rk} Time-differentiating Lebesgue integrability exponents is
  common practice in certain areas of analysis; see
  e.g. \cite{gross:log:75}. Time-differentiation with respect to
  regularity exponents is less common; however, as pointed out to us
  by Strain, Lemma \ref{lemderivnorm} is strongly reminiscent of a
  method recently used by Chemin \cite{chemin} to derive local
  analytic regularity bounds for the Navier--Stokes equation.  We
  expect that similar ideas can be applied to more general situations
  of Cauchy--Kowalevskaya type, especially for first-order equations,
  and maybe this has already been done.
\end{Rk}

\begin{proof}[Proof of Lemma \ref{lemderivnorm}]
For notational simplicity, let us assume $d=1$.
The left-hand side of \eqref{leftddtnorm} is
\begin{align*}
& \sum_{n,k} e^{2\pi\mu(\tau)|k|}\,2\pi \dot{\mu}(\tau)\,|k|\,
\frac{\lambda^n(\tau)}{n!}\, \Bigl\|\bigl(\nabla_v + 2i\pi k\tau\bigr)^n\,\hat{f}(k,v)\Bigr\|_{L^p(dv)}\\
& \qquad 
+ \sum_{n,k} e^{2\pi \mu(\tau)|k|}\,\dot{\lambda}(\tau)\,\frac{\lambda^{n-1}(\tau)}{(n-1)!}\,
\Bigl\| \bigl(\nabla_v + 2i\pi k\tau\bigr)^n\,\hat{f}(k,v)\Bigr\|_{L^p(dv)}\\
& \leq - K \sum_{n,k} e^{2\pi\mu(\tau)|k|}\,2\pi \,|k|\,
\frac{\lambda^n(\tau)}{n!}\, \Bigl\|\bigl(\nabla_v + 2i\pi k\tau\bigr)^n\,\hat{f}(k,v)\Bigr\|_{L^p(dv)}\\
&\qquad - K \sum_{n,k} e^{2\pi\mu(\tau)|k|}\,
\frac{\lambda^n(\tau)}{n!}\, \Bigl\|\bigl(\nabla_v + 2i\pi k\tau\bigr)^{n+1}\,\hat{f}(k,v)\Bigr\|_{L^p(dv)}\\
& \leq - K \sum_{n,k} e^{2\pi\mu(\tau)|k|}\,\frac{\lambda^n(\tau)}{n!}
\,\Bigl\|\bigl(\nabla_v + 2i\pi k\tau\bigr)^n \hat{\nabla_x f}(k,v)\Bigr\|_{L^p(dv)}\\
&\qquad + \frac{K \tau}{1+\tau} \sum_{n,k} e^{2\pi\mu(\tau)|k|}\,\frac{\lambda^n(\tau)}{n!}
\,\Bigl\|\bigl(\nabla_v + 2i\pi k\tau\bigr)^n \hat{\nabla_x f}(k,v)\Bigr\|_{L^p(dv)}\\
&\qquad - \frac{K}{1+\tau} \sum_{n,k} e^{2\pi\mu(\tau)|k|}\,\frac{\lambda^n(\tau)}{n!}
\,\Bigl\|\bigl(\nabla_v + 2i\pi k\tau\bigr)^n \hat{\nabla_v f}(k,v)\Bigr\|_{L^p(dv)},
\end{align*}
where in the last step we used $\|(\nabla_v + 2i\pi k\tau)h\|\geq (1/(1+\tau))( \|\nabla_v h\|
-\tau \|2i\pi kh\|)$. The conclusion follows.
\end{proof}

Now let us see how to propagate estimates through the Newton scheme
described in Section \ref{sec:iter}.  The first stage of the iteration
($h^1$ in the notation of \eqref{N1}) was considered in Subsection
\ref{sub:revisited}, so we only need to care about higher orders.  For
any $k\geq 1$ we solve $\pa_t h^{k+1} + v\cdot\nabla_x h^{k+1} = \tilde \Sigma^{k+1}$, 
where
$$
\tilde \Sigma^{k+1} 
= - \Bigl(F[h^{k+1}]\cdot\nabla_v f^k + F[f^k]\cdot\nabla_v
h^{k+1} + F[h^k]\cdot\nabla_v h^k\Bigr)
$$
(note the difference with \eqref{N2equiv}--\eqref{Sigman}). Recall that
$f^k=f^0+h^1+\ldots+h^k$.  We define $\lambda_k(t) = \lambda_k - 2\,
K\, t$, $\mu_k(t) = \mu_k - K\,t$, where $(\lambda_k)_{k\in\N}$,
$(\mu_k)_{k\in\N}$ are decreasing sequences of positive numbers.

We assume inductively that at stage $n$ of the iteration, we have
constructed $(\lambda_k)_{k\leq n}$, $(\mu_k)_{k\leq n}$,
$(\delta_k)_{k\leq n}$ such that
\[\forall \, k\leq n,\qquad 
\sup_{0\leq t\leq T}
\bigl\|h^k(t,\,\cdot\,)\bigr\|_{\cZ^{\lambda_k(t),\mu_k(t);1}_t} \leq
\delta_k,\] for some fixed $T>0$. The issue is to construct
$\lambda_{n+1}$, $\mu_{n+1}$ and $\delta_{n+1}$ so that the induction
hypothesis is satisfied at stage $n+1$.

At $t=0$, $h^{n+1}=0$. Then we estimate the time-derivative of
$\|h^{n+1}\|_{\cZ^{\lambda_{n+1}(t),\mu_{n+1}(t);1}_t}$. Let us first
pretend that the regularity indices $\lambda_{n+1}$ and $\mu_{n+1}$ do
not depend on $t$; then $h^{n+1}(t) = \int_0^t \tilde \Sigma^{n+1}\circ
S^0_{-(t-\tau)}\,d\tau$, so by Proposition \ref{propTL}
\begin{align*}
\|h^{n+1}\|_{\cZ^{\lambda_{n+1},\mu_{n+1};1}_t} 
& \leq \int_0^t \bigl\|\tilde \Sigma^{n+1}_\tau\circ S^0_{-(t-\tau)}\bigr\|_{\cZ^{\lambda_{n+1},\mu_{n+1};1}_t}\,d\tau\\
& \leq \int_0^t \bigl\|\tilde \Sigma^{n+1}_\tau\|_{\cZ^{\lambda_{n+1},\mu_{n+1};1}_\tau}\,d\tau,
\end{align*}
and thus
\[ \frac{d^+}{dt} \|h^{n+1}\|_{\cZ^{\lambda_{n+1},\mu_{n+1};1}_t} \leq
\|\tilde \Sigma^{n+1}_t\|_{\cZ^{\lambda_{n+1},\mu_{n+1};1}_t}.\]
Finally, according to Lemma \ref{lemderivnorm}, to this estimate we
should add a negative multiple of the norm of $\nabla h^{n+1}$ 
to take into account the time-dependence of
$\lambda_{n+1}$, $\mu_{n+1}$.

All in all, after application of Proposition \ref{propalgZ}, we get
\begin{align*}
\frac{d^+}{dt}
\bigl\|h^{n+1}(t,\,\cdot\,)\bigr\|_{\cZ^{\lambda_{n+1}(t),\mu_{n+1}(t);1}_t}
& \leq \bigl\|F[h^{n+1}_t]\|_{\cF^{\lambda_{n+1}t+\mu_{n+1}}}\,
\|\nabla_v f_t^n\|_{\cZ^{\lambda_{n+1},\mu_{n+1};1}_t}\\
& \quad + \bigl\|F[f^{n}_t]\|_{\cF^{\lambda_{n+1}t+\mu_{n+1}}}\,
\|\nabla_v h_t^{n+1}\|_{\cZ^{\lambda_{n+1},\mu_{n+1};1}_t}\\
& \quad + \bigl\|F[h^{n}_t]\|_{\cF^{\lambda_{n+1}t+\mu_{n+1}}}\,
\|\nabla_v h_t^n\|_{\cZ^{\lambda_{n+1},\mu_{n+1};1}_t}\\
& \quad -K\, \bigl\|\nabla_x h^{n+1}_t\|_{\cZ^{\lambda_{n+1},\mu_{n+1};1}_t}\,
-K\, \|\nabla_v h_t^{n+1}\|_{\cZ^{\lambda_{n+1},\mu_{n+1};1}_t},
\end{align*}
where $K>0$, $t$ is sufficiently small, and all exponents $\lambda_{n+1}$ and $\mu_{n+1}$ 
in the right-hand side actually depend on $t$.  

From Proposition \ref{propZCF} (iv) we
easily get $\|F[h]\|_{\cF^{\lambda t+\mu}} \leq C\,\|\nabla
h\|_{\cZ^{\lambda,\mu;1}_t}$.  Moreover, by Proposition
\ref{propgrad},
\[ \|\nabla f^n\|_{\cZ^{\lambda_{n+1},\mu_{n+1};1}_t} \leq
\sum_{k\leq n} \|\nabla h^k\|_{\cZ^{\lambda_{k+1},\mu_{k+1};1}_t} \leq
C\,\sum_{k\leq n} \frac{\|h^k\|_{\cZ^{\lambda_{k+1},\mu_{k+1};1}_t}}
{\min \bigl\{\lambda_k-\lambda_{n+1}\, ; \, 
\mu_k-\mu_{n+1}\bigr\}}.\]
We end up with the bound
\begin{multline*}
  \frac{d^+}{dt} \bigl\|h^{n+1}(t,\,\cdot\,)\bigr\|_{\cZ^{\lambda_{n+1}(t),\mu_{n+1}(t);1}_t}\\
  \leq \left[ C \left(\sum_{k\leq n}
      \frac{\delta_k}{\min\bigl\{\lambda_k-\lambda_{n+1}\, ; \, 
        \mu_k-\mu_{n+1}\bigr\}} \right) - K \right] \bigl\|\nabla
  h^{n+1}\bigr\|_{\cZ^{\lambda_{n+1}(t),
      \mu_{n+1}(t);1}_t} \\
  +
  \frac{\delta_n^2}{\min\bigl\{\lambda_n-\lambda_{n+1}\, ; \, 
    \mu_n-\mu_{n+1}\bigr\}}.
\end{multline*}

We conclude that if
\begeq\label{cclif}
\sum_{k\leq n} \frac{\delta_k}{\min\bigl\{\lambda_k-\lambda_{n+1} \, ;
  \, \mu_k-\mu_{n+1}\bigr\}}\leq
\frac{K}{C},
\endeq
then we may choose
\begeq\label{thenccl}
\delta_{n+1} = \frac{\delta_n^2}{\min\bigl\{\lambda_n-\lambda_{n+1} \,
  ; \, \mu_n-\mu_{n+1}\bigr\}}.
\endeq

This is our first encounter with the principle of ``stratification''
of errors, which will be crucial in the next section: to control the
error at stage $n+1$, we use not only the smallness of the error from
stage $n$, but also an information about all previous errors; namely
the fact that the convergence of the size of the error is much faster
than the convergence of the regularity loss. Let us see how this
works.  We choose $\lambda_k-\lambda_{k+1} = \mu_k-\mu_{k+1} =
\Lambda/k^2$, where $\Lambda>0$ is arbitrarily small.  Then for $k\leq
n$, $\lambda_k - \lambda_{n+1} \geq \Lambda/k^2$, and therefore
$\delta_{n+1}\leq \delta_n^2\, n^2/\Lambda$. The problem is to check
\begeq\label{seriescheck} \sum_{n=1}^\infty n^2\,\delta_n < +\infty.
\endeq
Indeed, then we can choose $K$ large enough for \eqref{cclif} to be
satisfied, and then $T$ small enough that, say $\lambda^*-2KT \geq
\lambda^\sharp$, $\mu^*-KT\geq \mu^\sharp$, where $\lambda^\sharp <
\lambda^*$, $\mu^\sharp < \mu^*$ have been fixed in advance.

If $\delta_1=\delta$, the general term in the series of
\eqref{seriescheck} is
\[ n^2\, \frac{\delta^{2^n}}{\Lambda^n}\, (2^2)^{2^{n-1}}\,
(3^2)^{2^{n-2}}\, (4^2)^{2^{n-2}} \ldots ((n-1)^2)^2\, n^2.\] To prove
the convergence for $\delta$ small enough, we assume by induction that
$\delta_n\leq z^{a^n}$, where $a$ is fixed in the interval $(1,2)$
(say $a=1.5$); and we claim that this condition propagates if $z>0$ is
small enough.  Indeed,
\[ \delta_{n+1} \leq \frac{z^{2\,a^n}}{\Lambda}\,n^2 \leq
z^{a^{n+1}}\, \left(\frac{z^{(2-a)a^n}\,n^2}{\Lambda}\right),\] and
this is bounded above by $z^{a^{n+1}}$ if $z$ is so small that
\[ \forall \, n\in\N,\qquad z^{(2-a)\,a^n} \leq \frac{\Lambda}{n^2}.\]

This concludes the iteration argument. Note that the convergence is
still extremely fast --- like $O(z^{a^n})$ for any $a<2$. (Of course,
when $a$ approaches~2, the constants become huge, and the restriction
on the size of the perturbation becomes more and more stringent.)

\begin{Rk} The method used in this section can certainly be applied to
  more general situations of Cauchy--Kowalevskaya type. Actually, as
  pointed out to us by Bony and G\'erard, the use of a regularity
  index which decays linearly in time, combined with a Newton
  iteration, was used by Nirenberg \cite{nirenberg} to prove an
  abstract Cauchy--Kowalevskaya theorem.  Nirenberg uses a
  time-integral formulation, so there is nothing in \cite{nirenberg}
  comparable to Lemma \ref{lemderivnorm}, and the details of the proof
  of convergence differ from ours; but the general strategy is
  similar.  Nirenberg's proof was later simplified by Nishida
  \cite{nishida} with a clever fixed point argument; in the present
  section anyway, our final goal is to provide short-term estimates
  for the successive corrections arising from the Newton scheme.
\end{Rk}

\section{Global in time iteration}
\label{sec:iter}

Now let us implement the scheme described in Section \ref{sec:approx},
with some technical modifications.  If $f$ is a given kinetic
distribution, we write $\rho[f]=\int f\,dv$ and $F[f]=-\nabla W \ast
\rho[f]$.  We let 
\begeq\label{fh} 
f^n = f^0 + h^1 + \ldots + h^n, 
\endeq where the successive corrections $h^k$ are defined by the
natural Newton scheme introduced in Section \ref{sec:approx}.  As in
Section \ref{sec:scattering} we define $\Om_{t,\tau}^k$ as the
scattering from time $t$ to time $\tau$, generated by the force field
$F[f^k]=-\nabla W\ast \rho[f^k]$. (Note that $\Om^0 = \Id$.)  \sm

\subsection{The statement of the induction}
\label{sec:statement-induction}

We shall fix $\ov{p}\in [1,\infty]$ and make the following assumptions:

\bul {\em Regularity of the background:} there are $\lambda>0$ and
$C_0>0$ such that
\[ \forall p \in [1,\ov{p}],\qquad \|f^0\|_{\cC^{\lambda;p}} \leq C_0.\]
\sm

\bul {\em Linear damping condition:} 
The stability condition {\bf (L)}
from Subsection \ref{sub:lineardamping} holds with parameters
$C_0$, $\lambda$, (the same as above) and $\kappa>0$.  \sm

\bul {\em Regularity of the interaction:} There are $\gamma>1$ and
$C_F>0$ such that for any $\nu>0$, 
\begeq\label{hypiterforce}
\bigl\|\nabla W \ast \rho\bigr\|_{\cF^{\nu,\gamma}} \leq C_F\,
\|\rho\|_{\dot{\cF}^\nu}.
\endeq
\sm

\bul {\em Initial layer of regularity} (coming from Section~\ref{sec:local}): 
Having chosen $\lambda^\sharp < \lambda$, $\mu^\sharp < \mu$, we assume
that for all $p\in [1,\ov{p}]$ 
\begeq\label{shorttime} 
\forall \, k \ge 1, \quad \sup_{0\leq t\leq T} \Bigl(\|h_t^k\|_{\cZ^{\lambda^\sharp,\mu^\sharp;p}} +
\|\rho[h_t^k]\|_{\cF^{\mu^\sharp}} \Bigr) 
\leq \zeta_k,
\endeq
where $T$ is {\em some} positive time, and $\zeta_k$ converges to zero
extremely fast: $\zeta_k = O(z_I ^{a_I ^k})$, $z_I\leq C \, \delta
<1$, $1<a_I <2$ ($a_I$ chosen in advance, arbitrarily close to $2$).
\sm

\bul {\em Smallness of the solution of the linearized equation}
(coming from Subsection \ref{sub:revisited}): Given
$\lambda_1<\lambda^\sharp$, $\mu_1<\mu^\sharp$, we assume
\begeq\label{controlN1} \forall \, p \in [1,\ov{p}], \quad
\begin{cases} \dps
\sup_{\tau\geq 0}\: \bigl\|\rho[h_\tau^1]\bigr\|_{\cF^{\lambda_1 \tau +\mu_1}}\leq \delta_1 \\[3mm]
\dps \sup_{t\geq \tau\geq 0}\: \bigl\|h_\tau^1 \bigr\|_{\cZ^{\lambda_1 (1+b),\mu_1;p}_{\tau - \frac{bt}{1+b}}} \leq 
\delta_1,
\end{cases}\endeq
where $\delta_1\leq C\,\delta$.
\med

{\bf Then we prove the following induction}: for any $n \ge 1$,
\begeq\label{inductionglobal}
\forall \, k\in \{1, \dots, n\}, \quad \forall \, p \in [1,\ov{p}], \quad 
\begin{cases} \dps
\sup_{\tau\geq 0}\: \bigl\|\rho[h_\tau^k]\bigr\|_{\cF^{\lambda_k\tau +\mu_k}}\leq \delta_k \\[3mm]
\dps \sup_{t\geq \tau\geq 0}\: \Bigl\|h_\tau^k \circ
\Om_{t,\tau}^{k-1}\Bigr\|_{\cZ^{\lambda_k(1+b),\mu_k;p}_{\tau - \frac{bt}{1+b}}} \leq \delta_k,
\end{cases}\endeq
where 
\sm

\bul $(\delta_k)_{k\in\N}$ is a sequence satisfying $0 < C_F \,
\zeta_k \leq \delta_k$, and $\delta_k = O(z^{a^k})$, $z< z_I$, $1< a<a_I$
($a$ arbitrarily close to $a_I$), \sm

\bul $(\lambda_k,\mu_k)$ are decreasing to
$(\lambda_\infty,\mu_\infty)$, where $(\lambda_\infty,\mu_\infty)$ are
arbitrarily close to $(\lambda_1,\mu_1)$; in particular we impose
\begeq\label{condlambda1} 
\lambda^\sharp - \lambda_\infty \leq \min
\Bigl\{1 \, ;\ \frac{\lambda_\infty}{2} \Bigr\}, \quad \mu^\sharp -
\mu_\infty \leq \min \Bigl\{ 1\, ;\ \frac{\mu_\infty}{2}\Bigr\}.
\endeq

\bul $T$ is some small positive time in \eqref{shorttime}; we impose
\begeq\label{condBlambda}
\lambda^\# \, T \leq \frac{\mu^\sharp - \mu_1}{2}.
\endeq

\bul $\dps b = b(t) = \frac{B}{1+t}$, where $B\in (0,T)$ is a (small) constant.

\subsection{Preparatory remarks}
\label{sec:preparatory-remarks}

As announced in \eqref{inductionglobal}, we shall propagate the following ``primary'' 
controls on the density and distribution:
\begin{equation}
  \label{eq:rhon}
  {\bf (E_\rho ^n)} \qquad \forall \, k\in \{1,\ldots,n\}, 
  \quad  \sup_{\tau\geq 0}\:
  \bigl\|\rho[h_\tau^k]\bigr\|_{\cF^{\lambda_k\tau +\mu_k}}
  \leq \delta_k
\end{equation}
and 
\begin{equation}
  \label{eq:distribn}
  {\bf (E_h ^n)} \qquad \forall \, k\in\{1,\ldots,n\},\quad
  \forall \, p \in [1,\ov{p}], \quad 
  \sup_{t\geq \tau\geq 0}\: \Bigl\|h_\tau^k 
  \circ \Om_{t,\tau}^{k-1}\Bigr\|_{\cZ^{\lambda_k(1+b),\mu_k;p}_{\tau - \frac{bt}{1+b}}} \leq \delta_k.
\end{equation}

Estimate ${\bf (E_\rho^n)}$ obviously implies, {\it via}
\eqref{hypiterforce}, up to a multiplicative constant,
\begin{equation}
  \label{eq:forcen}
  {\bf (\tilde E_\rho ^n)} \qquad \forall \, k\in \{1, \dots, n\}, \quad  \sup_{\tau\geq 0}\:
  \bigl\|F[h_\tau^k]\bigr\|_{\cF^{\lambda_k\tau +\mu_k, \gamma}} 
  \leq \delta_k.
\end{equation}

Before we can go from there to stage $n+1$,
we need an additional set of estimates on the scattering
maps $(\Omega^k)_{k=1,\dots,n}$, which will be used to
\begin{enumerate}
\item update the control on $\Omega^k_{t,\tau} -  \Id$;
\sm

\item establish the needed control along the characteristics for the
  background $(\nabla_v f^{n} _\tau) \circ \Omega^n _{t,\tau}$ (same
  index for the distribution and the scattering);
\sm

\item update some technical controls allowing to exchange (asymptotically) gradient and composition by $\Om^k_{t,\tau}$;
this will be crucial to handle the contribution of the zero mode of the background after composition
by characteristics.
\end{enumerate}

This set of scattering estimates falls into three categories. The first group expresses the closeness of $\Om^k$ to $\Id$:
\begin{equation}
  \label{eq:omegan1}
  {\bf (E_\Omega ^n)} \qquad 
  \forall \, k\in \{1, \dots, n\}, \quad 
  \begin{cases} 
    \dps \sup_{t\geq\tau\geq 0}\:
    \Bigl\|\Om^kX_{t,\tau} - \Id \Bigr\|_{\cZ^{\lambda ^* _k (1+b),(\mu^*_k,\gamma)}
      _{\tau - \frac{bt}{1+b}}} \leq 2\, \RR_2^k (\tau,t), \\[3mm]
    \dps \sup_{t\geq\tau\geq 0}\:
    \Bigl\|\Om^kV_{t,\tau} - \Id \Bigr\|_{\cZ^{\lambda ^* _k (1+b),(\mu^*_k,\gamma)}
      _{\tau - \frac{bt}{1+b}}} \leq \RR_1 ^k (\tau,t),
  \end{cases}
\end{equation}
with $\lambda_k > \lambda_k ^* > \lambda_{k+1}$, $\mu_k > \mu_k ^* > \mu_{k+1}$, and
\begeq\label{R12k}
\begin{cases}\dps \RR_1 ^k(\tau,t) = 
\left(\sum_{j=1}^k \frac{\delta_j\,e^{-2\pi (\lambda_j-\lambda^*_j)\tau}}
{2\pi(\lambda_j-\lambda^*_j)}\right)\, \min 
\left\{ (t-\tau)\, ; \ 1 \right\} \\[3mm]
\dps  \RR_2^k(\tau,t) = \left(\sum_{j=1}^k \frac{\delta_j\, 
e^{-2\pi (\lambda_j-\lambda^*_j)\tau}}
{(2\pi(\lambda_j-\lambda^*_j))^2}\right)\, \min
\left\{ \frac{(t-\tau)^2}{2}\, ; \ 1 \right\}.
\end{cases}
\endeq

The second group of estimates expresses the fact that $\Om^n-\Om^k$ is
very small when $k$ is large:
\begin{equation}
  \label{eq:omegan2}
  {\bf (\tilde E_\Omega ^n)} \quad 
  \forall \, k\in \{0, \dots, n-1\}, \quad 
  \begin{cases} 
     \dps \sup_{t\geq\tau\geq 0}\:
    \Bigl\|\Om^nX_{t,\tau} - \Om^k X_{t,\tau} \Bigr\|_{\cZ^{\lambda ^* _n (1+b),(\mu^*_n,\gamma)}
      _{\tau - \frac{bt}{1+b}}} \leq 2\, \RR_2^{k,n} (\tau,t), \\[4mm]
    \dps \sup_{t\geq\tau\geq 0}\:
    \Bigl\|\Om^n V_{t,\tau} - \Om^k V_{t,\tau} \Bigr\|_{\cZ^{\lambda ^* _n (1+b),(\mu^*_n,\gamma)}
      _{\tau - \frac{bt}{1+b}}} \leq \RR_1 ^{k,n} (\tau,t) + \RR_2 ^{k,n}(\tau,t), \\[4mm]
        \dps \sup_{t\geq\tau\geq 0}\:
    \Bigl\|(\Om^k _{t,\tau})^{-1} \circ \Om^n _{t,\tau} - \Id \Bigr\|_{\cZ^{\lambda ^* _n (1+b),\mu^*_n}
      _{\tau - \frac{bt}{1+b}}} \leq 4\bigl(\RR_1 ^{k,n} (\tau,t) + \RR_2 ^{k,n}(\tau,t)\bigr),
  \end{cases}
\end{equation}
with 
\begeq\label{R12kn} \begin{cases}
\dps \RR_1 ^{k,n} (\tau,t) = \left(\sum_{j=k+1}^n \frac{\delta_j\,e^{-2\pi (\lambda_j-\lambda^*_j)\tau}}
{2\pi(\lambda_j-\lambda^*_j)}\right)\, \min 
\left\{ (t-\tau)\, ; \ 1 \right\}\\[3mm]
\dps \RR_2^{k,n}(\tau,t) = \left(\sum_{j=k+1}^n \frac{\delta_j\,e^{-2\pi
      (\lambda_j-\lambda^*_j)\tau}}
  {(2\pi(\lambda_j-\lambda^*_j))^2}\right)\, \min \left\{
  \frac{(t-\tau)^2}{2}\, ; \ 1 \right\}.
\end{cases}
\endeq
(Choosing $k=0$ brings us back to the previous estimates ${\bf (E_\Omega ^n)}$.)
\sm

The last group of estimates expresses the fact that the differential of the scattering is uniformly close to the identity
(in a way which is more precise than what would follow from the first group of estimates):
\begin{equation}
  \label{eq:omegan3}
  {\bf (E_{\nabla \Omega} ^n)} \qquad 
  \forall \, k=1, \dots, n, \quad 
  \begin{cases} 
     \dps \sup_{t\geq\tau\geq 0}\:
    \Bigl\|\nabla \Om^kX_{t,\tau} - (I,0) \Bigr\|_{\cZ^{\lambda ^* _k (1+b),\mu^*_k}
      _{\tau - \frac{bt}{1+b}}} \leq 2\, \RR_2^k (\tau,t), \\[4mm]
    \dps \sup_{t\geq\tau\geq 0}\:
    \Bigl\|\nabla \Om^kV_{t,\tau} - (0,I) \Bigr\|_{\cZ^{\lambda ^* _k (1+b),\mu^*_k}
      _{\tau - \frac{bt}{1+b}}} \leq \RR_1 ^k (\tau,t) + \RR_2 ^k (\tau,t),
  \end{cases}
\end{equation}
where $\nabla=(\nabla_x,\nabla_v)$, and $I$ is the identity matrix.
\med

An important property of the functions $\RR_1^{k,n} (\tau,t)$, $\RR_2
^{k,n} (\tau,t)$ is their fast decay as $\tau\to\infty$ and as
$k\to\infty$, uniformly in $n\geq k$; this is due to the fast
convergence of the sequence $(\delta_k)_{k\in\N}$. Eventually, if
$r\in\N$ is given, we shall have
\begin{equation}
  \label{eq:estimR}
  \forall \, r \ge 1, \quad \RR_1 ^{k,n} (\tau,t) \le \omega_{k,n} ^{r,1}
  (\tau,t), \quad \RR_2 ^{k,n} (\tau,t) \le \omega_{k,n} ^{r,2}(\tau,t)
\end{equation}
with 
$$
\omega_{k,n} ^{r,1}(\tau,t) := C_\omega ^r \, \left( \sum_{j=k+1} ^n
  \frac{\delta_j}{(2\pi(\lambda_j - \lambda^* _j))^{1+r}} \right) \, 
\frac{\min \left\{ (t-\tau)\, ; 1\right\}}{(1+\tau)^r},
$$
and
$$
\omega_{k,n} ^{r,2}(\tau,t) := C_\omega ^r \, \left( \sum_{j=k+1} ^n
  \frac{\delta_j}{(2\pi(\lambda_j - \lambda^* _j))^{2+r}} \right) \, 
\frac{\min \left\{ (t-\tau)^2/2\, ; 1\right\}}{(1+\tau)^r}
$$
for some absolute constant $C_\omega ^r$ depending only on $r$ (we
also denote $\omega^{r,1} _{0,n} = \omega^{r,1} _n$ and $\omega^{r,2}
_{0,n} = \omega^{r,2} _n$).
\med

From the estimates on the characteristics and ${\bf (E_h^n)}$ will follow the following ``secondary controls''
on the distribution function:
\begin{multline}
  \label{eq:distribk}
  {\bf (\tilde E_h ^n)} \qquad \forall \, k\in \{1, \dots, n\}, \quad
  \forall \, p \in [1,\ov{p}], \\ 
\begin{cases} 
\dps  
  \sup_{t\geq \tau\geq 0}\: 
  \Bigl\|(\nabla_x h_\tau^k)  \circ
  \Om_{t,\tau}^{k-1}\Bigr\|_{\cZ^{\lambda_k(1+b),\mu_k;p}_{\tau - \frac{bt}{1+b}}} \leq \delta_k \\[3mm]
  \dps \sup_{t\geq \tau\geq 0}\: 
  \Bigl\|\nabla_x \bigl(h_\tau^k  \circ
  \Om_{t,\tau}^{k-1}\bigr)\Bigr\|_{\cZ^{\lambda_k(1+b),\mu_k;p}_{\tau - \frac{bt}{1+b}}} \leq \delta_k \\[3mm]
  \dps
  \Bigl\|\bigl((\nabla_v + \tau \nabla_x\bigr) h_\tau^k)  \circ
  \Om_{t,\tau}^{k-1}\Bigr\|_{\cZ^{\lambda_k(1+b),\mu_k;p}_{\tau - \frac{bt}{1+b}}}
  \leq \delta_k \\[3mm]
  \dps
  \Bigl\|(\nabla_v + \tau \nabla_x) \bigl(h_\tau^k  \circ \Om_{t,\tau}^{k-1}\bigr)\Bigr\|
_{\cZ^{\lambda_k(1+b),\mu_k;p}_{\tau - \frac{bt}{1+b}}}
  \leq \delta_k \\[3mm]
\dps \sup_{t\geq \tau\geq 0} \frac1{(1+\tau)^2} \ 
\Bigl\| \bigl(\nabla \nabla h_\tau^k\bigr)\circ \Om^{k-1}_{t,\tau}\Bigr\|_{\cZ^{\lambda_k(1+b),\mu_k;1}_{\tau - \frac{bt}{1+b}}}\leq \delta_k\\[3mm]
\dps \sup_{t\geq \tau\geq 0} 
(1+\tau)^2\ \Bigl\| (\nabla h^k_\tau)\circ\Om^{k-1}_{t,\tau}
- \nabla \bigl( h^k_\tau \circ \Om^{k-1}_{t,\tau}\bigr)\Bigr\|_{\cZ^{\lambda_k(1+b),\mu_k;1}_{\tau-\frac{bt}{1+b}}}
\leq \delta_k.
\end{cases}
\end{multline}
\med

The transition from stage $n$ to stage $n+1$ can be summarized as
follows:
\[ {\bf (\tilde{E}_\rho ^n)} \stackrel{{\bf (A_n)}}{\Longrightarrow} 
\left[{\bf (E_{\Omega} ^n)}+{\bf (\tilde E_{\Omega} ^n)}+{\bf (E_{\nabla \Omega} ^n)}\right]\]
\[ \left[ {\bf (E_\rho ^n)} + {\bf (E_{\Omega} ^n)}+{\bf (\tilde E_{\Omega} ^n)}+{\bf (E_{\nabla \Omega} ^n)}
+ {\bf (E_h ^{n})}+{\bf (\tilde E_h ^{n})} \right]
 \stackrel{{\bf (B_n)}}{\Longrightarrow} \left[ {\bf (E_\rho ^{n+1})}+{\bf (\tilde E_\rho ^{n+1})}+{\bf (E_h ^{n+1})}
+{\bf (\tilde E_h ^{n+1})}\right].
\]

The first implication ${\bf (A_n)}$ is proven by an amplification of
the technique used in Section \ref{sec:scattering}; ultimately, it
relies on repeated application of Picard's fixed point theorem in analytic norms.  The
second implication ${\bf (B_n)}$ is the harder part; it uses the
machinery from Sections \ref{sec:regdecay} and \ref{sec:response},
together with the idea of propagating simultaneously a {\em shifted}
$\cZ$ norm for the kinetic distribution and an $\cF$ norm for the density.

In both implications, the stratification of error estimates will
prevent the blow up of constants.  So we shall decompose the force
field $F^n$ generated by $f^n$ as
\[ F^n = F[f^n] = E^1+\ldots + E^n,\]
where $E^k = F[h^k]=-\nabla W\ast \rho[h^k]$.

The plan of the estimates is as follows.
We shall construct inductively a sequence of constant coefficients
\[ \lambda^\sharp > \lambda_1 > \lambda^*_1> \lambda_2 > \ldots > 
\lambda_{n} > \lambda^*_{n} >\lambda_{n+1} > \dots \]
\[ \mu^\sharp > \mu_1 > \mu^*_1 > \mu_2 >  \ldots > \mu_{n} > \mu^*_{n} >
 \mu_{n+1} > \dots\]
(where $\lambda_n, \mu_n$
  will be fixed in the proof of ${\bf (A_n)}$, and  $\lambda_{n+1},\mu_{n+1}$  
in the proof of ${\bf (B_n)}$)
  converging respectively to $\lambda_\infty$ and $\mu_\infty$; and a
  sequence $(\delta_k)_{k\in\N}$ decreasing very fast to zero. For
  simplicity we shall let
\[ \RR^n(\tau,t) = \RR_1^n(\tau,t) + \RR_2^n(\tau,t),\qquad
\RR^{k,n}(\tau,t) = \RR_1^{k,n}(\tau,t) + \RR_2^{k,n}(\tau,t),\]
and assume $2\pi(\lambda_j-\lambda^*_j)\leq 1$; so
\begeq\label{RR}
\RR^{k,n}(\tau,t) \leq C_\omega ^r \, \left( \sum_{j=k+1} ^n
  \frac{\delta_j}{(2\pi(\lambda_j - \lambda^* _j))^{2+r}} \right) \, 
\frac{\min \{t-\tau\, ; \, 1\}}{(1+\tau)^r},\qquad \RR^{0,n} = \RR^n.
\endeq
It will be sufficient to work with some fixed $r$, large enough
(as we shall see, $r=4$ will do). 

To go from stage $n$ to stage $n+1$, we shall do as follows:
\sm

\noindent
- \underline{Implication ${\bf (A_n)}$ (subsection \ref{subsec:A})}:
\sm

{\bf Step 1.} estimate $\Om^n-\Id$ (the bound should be uniform in $n$);
\sm

{\bf Step 2.} estimate $\Om^n-\Om^k$ ($k\leq n-1$; the
error should be small when $k\to\infty$);
\sm

{\bf Step 3.} estimate $\nabla \Om^n - I$;
\sm

{\bf Step 4.} estimate $(\Omega^k)^{-1} \circ \Omega^n$;
\sm 

\noindent
- \underline{Implication ${\bf (B_n)}$ (subsection \ref{subsec:B})}:
\sm

{\bf Step 5.} estimate $h^k$ and its derivatives along the composition by $\Om^n$;
\sm

{\bf Step 6.} estimate $\rho[h^{n+1}]$, using Sections \ref{sec:regdecay} and \ref{sec:response};
\sm

{\bf Step 7.} estimate $F[h^{n+1}]$ from $\rho[h^{n+1}]$;
\sm

{\bf Step 8.} estimate $h^{n+1}\circ\Om^n$;
\sm

{\bf Step 9.} estimate derivatives of $h^{n+1}$ composed with $\Om^n$;
\sm

{\bf Step 10.} show that for $h^{n+1}$, $\nabla$ and composition by $\Om^n$ asymptotically commute.

\subsection{Estimates on the characteristics} \label{subsec:A}
In this subsection, we assume that estimate
${\bf (E_\rho ^n)}$ is proven, and we establish
${\bf (E_{\Omega} ^n)}+{\bf (\tilde E_{\Omega} ^n)}+{\bf (E_{\nabla \Omega}^n)}$.
Let $\lambda^*_n<\lambda_n$, $\mu^*_n< \mu_n$ to be fixed later on.

\subsubsection{Step 1: Estimate of $\Om^n-\Id$}
This is the first and archetypal estimate.
We shall bound $\Om^nX_{t,\tau} - x$ in the hybrid norm
$\cZ^{\lambda_n^*(1+b),\mu_n^*}_{\tau-\frac{bt}{1+b}}$.
The Sobolev correction $\gamma$ will play no role here in the proofs,
and for simplicity we shall forget it in the computations, just recall it in the
final results. (Use Proposition \ref{prophybrsob} whenever needed.)

Since we expect the characteristics for the force field $F^n$ to be
close to the free transport characteristics, it is natural to write
\begeq\label{XnF} X^n_{t,\tau}(x,v) = x-v(t-\tau) + Z^n_{t,\tau}(x,v),
\endeq
where $Z^n_{t,\tau}$ solves
\begeq\label{eqXnF} \begin{cases}
\dps \frac{\pa^2}{\pa\tau^2} Z^n_{t,\tau}(x,v)
= F^n \Bigl( \tau, x-v(t-\tau) + Z^n_{t,\tau}(x,v)\Bigr)\\[2mm]
\dps Z^n_{t,t}(x,v) = 0,\qquad
\pa_\tau Z^n_{t,\tau}\Bigr|_{t=\tau} (x,v) =0.
\end{cases}
\endeq
(With respect to Section \ref{sec:scattering} we have dropped the parameter $\var$, to take advantage of
the ``stratified'' nature of $F^n$; anyway this parameter was cosmetic.) So if we fix $t>0$,
$(Z^n_{t,\tau})$ is a fixed point of the map
\[ \Psi: (W_{t,\tau})_{0\leq\tau\leq t} \longmapsto (Z_{t,\tau})_{0\leq\tau\leq t}\]
defined by
\begeq\label{EDOstrat} \begin{cases}
\dps \frac{\pa^2}{\pa\tau^2} Z_{t,\tau}
= F^n \Bigl( \tau, x-v(t-\tau) + W_{t,\tau}\Bigr)\\[2mm]
\dps Z_{t,t} = 0,\qquad \pa_\tau Z_{t,\tau}\Big|_{\tau=t} =0. \end{cases}
\endeq
The goal is to estimate $Z^n_{t,\tau} - x$ in the hybrid norm
$\cZ^{\lambda_n^*(1+b),\mu_n^*}_{t-\frac{bt}{1+b}}$.
\med

We first bound $(Z_0^n)_{t,\tau} = \Psi(0)$. Explicitly,
\[ (Z_0^n)_{t,\tau}(x,v) = \int_\tau^t (s-\tau)\, F^n\bigl(s,x-v(t-s)\bigr)\,ds.\]
By Propositions \ref{propZCF} (i) and \ref{propTL},
\begin{align}\label{estZ0n}
\bigl\|(Z_0^n)_{t,\tau}& \bigr\|_{\cZ^{\lambda_n^*(1+b),\mu_n^*}_{t-\frac{bt}{1+b}}} \\ \nonumber
& \leq \int_\tau^t (s-\tau)\, \Bigl\|F^n\bigl(s,\,x-v(t-s)\,\bigr)\Bigr\|_{
\cZ^{\lambda_n^*(1+b),\mu_n^*}_{t-\frac{bt}{1+b}}} \\ \nonumber
& = \int_\tau^t (s-\tau)\,\|F^n(s,\,\cdot\,)\|_{
\cZ^{\lambda_n^*(1+b),\mu_n^*}_{s-\frac{bt}{1+b}}} \\ \nonumber
& = \int_\tau^t (s-\tau)\, \|F^n(s,\,\cdot\,)\|_{\cF^{\nu(s,t)}}\,ds,
\end{align}
where  
\begin{align} \label{nust}
\nu(s,t) = \lambda_n^*\bigl|s-b(t-s)\bigr| + \mu_n^*.
\end{align}

{\bf First case:} If $s\geq bt/(1+b)$, then 
\begeq\label{mumu2}
\nu(s,t)\leq \lambda^*_n s + \mu^*_n \leq \lambda_k\,s + \mu_k - (\lambda_k-\lambda^*_n)s
\qquad (1\leq k\leq n).
\endeq

{\bf Second case:} 
If $s< bt/(1+b)$, then necessarily $s\leq B\leq T$. Taking into account \eqref{condlambda1}, we have
\begin{align} \label{nust2} 
\nu(s,t) 
& = \lambda_n^*\,bt + \mu_n^* - \lambda_n^*(1+b) s \\ \label{nust2bis}
& \leq \lambda_n^* B + \mu^* _{n} - (\lambda_k-\lambda^* _n)s. 
\end{align}
(Of course, the assumption $\lambda^\sharp-\lambda_\infty \le \min\{1,\lambda_\infty /2\}$
implies  $\lambda_k -\lambda_n^* \leq \lambda_n^*$.)
In particular, by \eqref{condBlambda},
\begeq\label{mumu3}
\nu(s,t)\leq \mu^\sharp - (\lambda_k-\lambda^*_n)s 
\qquad (1\leq k\leq n).
\endeq
\sm

We plug these bounds into \eqref{estZ0n}, then use $\hat{E}^k(s,0)=0$ and the bounds
\eqref{eq:forcen} and \eqref{mumu2} (for large times), and \eqref{shorttime} and \eqref{mumu3} (for short times). This yields
\begin{align} \label{Z0nagain} & \|(Z_0^n)_{t,\tau} \|_{
    \cZ^{\lambda_n^*(1+b),\mu_n^*}_{t-\frac{bt}{1+b}}}\\
  \nonumber & \leq \sum_{k=1}^n \Bigg( \int_{\tau\vee\frac{bt}{1+b}}^t
  (s-\tau)\,\|E^k(s,\,\cdot\,)\|_{\cF^{\lambda_k s + \mu_k -
      (\lambda_k -\lambda^*_n)s}}\,ds \\ \nonumber & \hspace{6cm} +
  \int_\tau^{\tau\vee\frac{bt}{1+b}}
  (s-\tau)\,\|E^k(s,\,\cdot\,)\|_{\cF^{\mu^\sharp - (\lambda_k
      -\lambda^*_n)s}}\,ds \Bigg) \\ \nonumber & \leq \sum_{k=1}^n
  \Bigg(\int_{\tau\vee\frac{bt}{1+b}}^t (s-\tau)\,e^{-2\pi
    (\lambda_k-\lambda^*_n)s}\, \|E^k(s,\,\cdot\,)\|_{\cF^{\lambda_k s
      + \mu_k}}\,ds \\ \nonumber & \hspace{6cm} +
  \int_\tau^{\tau\vee\frac{bt}{1+b}} (s-\tau)\,e^{-2\pi
    (\lambda_k-\lambda^*_n)s}\,
  \|E^k(s,\,\cdot\,)\|_{\cF^{\mu^\sharp}}\,ds\Bigg) \\ \nonumber &
  \leq \sum_{k=1}^n \delta_k \int_\tau^t (s-\tau)\,e^{-2\pi
    (\lambda_k-\lambda^* _n)s}\,ds\\ \nonumber & \leq \sum_{k=1}^n
  \delta_k \, e^{-2\pi (\lambda_k - \lambda^* _n)\tau}\, \min \left\{
    \frac{(t-\tau)^2}2\, ; \ \frac1{(2\pi(\lambda_k-\lambda^*_n))^2}
  \right\}\leq \RR_2^n(\tau,t).
\end{align}

Let us define the norm
\[ \bbqn (Z_{t,\tau})_{0\leq \tau\leq t}\bbqn_n := \sup_{0\leq\tau\leq t} 
\frac{ \|Z_{t,\tau}\|_{\cZ^{\lambda_n^*(1+b),\mu_n^*}_{t-\frac{bt}{1+b}}}}
{\RR_2^n(\tau,t)}.\] 
(Note the difference with Section \ref{sec:scattering}: now the regularity exponents depend on time(s).)  Inequality \eqref{Z0nagain} means that $\qn \Psi(0)\qn_n \leq 1$.  We shall check that $\Psi$ is $(1/2)$-Lipschitz on the ball $B(0,2)$ in the norm $\qn\ \cdot\ \qn_n$.  This will be subtle: the uniform bounds on the size of the force field, coming from the preceding steps, will allow to get good decaying exponentials, which in turn will imply uniform error bounds at the present stage.

So let $W, \tilde{W}\in B(0,2)$, and let $Z=\Psi(W)$, $\tilde{Z}=\Psi(\tilde{W})$.
As in Section \ref{sec:scattering}, we write
\begin{multline*}
Z_{t,\tau} - \tilde{Z}_{t,\tau}
= \int_0^1 \int_\tau^t (s-\tau)\, \nabla_x F^n
\Bigl( s, \, x-v(t-s) + \Bigl( \theta\, W_{t,s} + (1-\theta)\,\tilde{W}_{t,s}\Bigr)\Bigr)\\
\cdot (W_{t,s}-\tilde{W}_{t,s})\,ds\,d\theta,
\end{multline*}
and deduce
\[ \bbqn \Bigl( Z_{t,\tau} - \tilde{Z}_{t,\tau}\Bigr)_{0\leq \tau\leq t}\bbqn_n 
\leq A(t)\ \bbqn \Bigl( W_{t,s} - \tilde{W}_{t,s}\Bigr)_{0\leq s\leq t}\bbqn_n, \]
where
\begin{multline*} 
A(t) =  \sup_{0\leq \tau\leq s\leq t}
\frac{\RR_2^n(s,t)}{\RR_2^n(\tau,t)} \times \\ \int_0^1 \int_\tau^t (s-\tau)\, 
\Bigl\| \nabla_x F^n \Bigl( s, x-v(t-s) + \bigr(\theta \, W_{t,s} + (1-\theta)\,\tilde{W}_{t,s}\bigr)\Bigr)
\Bigr\|_{{\cZ^{\lambda_n^*(1+b),\mu_n^*}_{t-\frac{bt}{1+b}}}}\,ds\,d\theta.
\end{multline*}
For $\tau\leq s$ we have $\RR_2^n(s,t) \leq \RR_2^n(\tau,t)$. Also, by Propositions
\ref{propcompos2} (applied with $V=0$, $b=-(t-s)$ and
$\sigma=0$ in that statement) and \ref{propZCF},
\[ A(t) \leq \sup_{0\leq \tau\leq t} \int_\tau^t (s-\tau)\, \|\nabla_x
F^n(s,\,\cdot\,)\|_{\cF^{\nu(s,t) + e(s,t)}}\,ds,\]
where $\nu$ is defined by \eqref{nust} and the ``error'' $e(s,t)$
arising from composition is given by
\begin{align*}
e(s,t) = \sup_{0\leq \theta\leq 1}
\Bigl\|\theta\, W_{t,s} + (1-\theta)\, \tilde{W}_{t,s}\Bigr\|_{
{\cZ^{\lambda_n^*(1+b),\mu_n^*}_{t-\frac{bt}{1+b}}}} \leq 2\,\RR_2^n (s,t).
\end{align*}

Since
$$
\RR_2 ^{n} (s,t) \le \omega_{n} ^{1,2} (s,t) := C_\omega ^1 \, \left( \sum_{k=1} ^n
  \frac{\delta_k}{(2\pi(\lambda_k - \lambda^* _k))^{3}} \right) \, 
\frac{\min \left\{ (t-s)^2/2 \, ; \ 1\right\}}{(1+s)}
$$
we have, for all $0 \leq s\leq t$,
\begeq\label{nunu}
2\, \RR_2^n (s,t) \leq \frac{\lambda_n^*}{2}\, b\,(t-s) \, 1_{s \ge bt/(1+b)}
\, + \frac{\mu^\sharp - \mu^*_n}{2}\, 1_{s\leq bt/(1+b)},
\endeq
as soon as 
\begin{equation}
  \label{eq:cond1-rec}
  {\bf (C_1)} \qquad  \forall \, n\geq 1,\quad
  2 \, C_\omega ^1 \, \left( \sum_{k=1} ^n
    \frac{\delta_k}{(2\pi(\lambda_k - \lambda^* _n))^{3}} \right) \le
  \min \left\{ \frac{\lambda^*_n\,B}6 \, ; \frac{\mu^\sharp - \mu^*_n}{2} \right\}.
\end{equation}
We shall check later in Subsection \ref{subsec:cvg-scheme} the
feasibility of condition ${\bf (C_1)}$ --- as well as a number of other
forthcoming ones.

The extra error term in the exponent is sufficiently small to be
absorbed by what we throw away in \eqref{mumu2} or
\eqref{nust2}-\eqref{nust2bis}-\eqref{mumu3}. So we obtain, as in the
estimate of $Z_0^n$, for any $k\in\{1, \dots, n\}$,
\[ (\nu+e)(s,t)
\begin{cases} \leq \lambda_k\, s + \mu  _k - (\lambda_k -\lambda^*_n)\, s \qquad\qquad
\text{for $s\geq bt/(1+b)$}\\[4mm]
\leq \mu^\sharp - (\lambda_k-\lambda^*_n) \, s
\qquad\qquad\qquad \text{for $s\leq bt/(1+b)$},
\end{cases}\]
and we deduce (using \eqref{eq:forcen} and $\gamma\geq 1$) \label{gg1'}
\begin{align*}
A(t) &\leq \sup_{0\leq\tau\leq t}
\sum_{k=1}^n 
\Bigg(
\int_{\tau\vee\frac{bt}{1+b}}^t (s-\tau)\,\|\nabla_x E^k(s,\,\cdot\,)\|_{\cF^{\lambda_ks
+ \mu_k - (\lambda_k -\lambda^*_n) \, s}}\,ds \\
& \hspace{6cm}
+ \int_\tau^{\tau\vee\frac{bt}{1+b}} (s-\tau)\,\|\nabla_x E^k(s,\,\cdot\,)\|_{\cF^{\mu^\sharp - 
(\lambda_k -\lambda^*_n)s}}\,ds \Bigg)\\
& \leq \sup_{0\leq\tau\leq t} \sum_{k=1}^n \delta_k \, \int_\tau^t (s-\tau)\,e^{-(\lambda_k-\lambda^*_n) \, s}\,ds
 \leq \sup_{0\leq\tau\leq t} \RR_2^n (\tau,t) = \RR_2^n(0,t)\\
& \leq \sum_{k=1}^n \frac{\delta_k}{(2\pi(\lambda_k-\lambda^*_n))^2}.
\end{align*}
If the latter quantity is bounded above by $1/2$, then $\Psi$ is
$(1/2)$-Lipschitz and we may apply the fixed point result from Theorem
\ref{thmfpt}. Therefore, under the condition (whose feasibility will
be checked later)
\begin{equation}
  \label{eq:cond2-rec}
  {\bf (C_2)} \qquad \forall n\geq 1,\ \quad 
  \sum_{k=1}^n \frac{\delta_k}{(2\pi(\lambda_k-\lambda^*_n))^2} \le \frac12
\end{equation}
we deduce
\[ \|Z^n_{t,\tau}\|_{
{\cZ^{\lambda_n^*(1+b),\mu_n^*}_{t-\frac{bt}{1+b}}}}
\leq 2\, \RR_2^n(\tau,t). \]

After that, the estimates on the scattering are obtained exactly as in
Section \ref{sec:scattering}: writing $\Om_{t,\tau}^n=
(\Om^nX_{t,\tau}, \Om^nV_{t,\tau})$, recalling the dependence on
$\gamma$ again, we end up with \begeq\label{estOmstep1unshift}
\begin{cases}
\Bigl\|\Om^nX_{t,\tau} - x \Bigr\|_{\cZ^{\lambda^*_n (1+b),(\mu^*_{n},\gamma)}
_{\tau - \frac{bt}{1+b}}}
\leq 2\, \RR_2^n (\tau,t)\\[2mm]
\Bigl\|\Om^nV_{t,\tau} - v \Bigr\|_{\cZ^{\lambda^* _n(1+b), (\mu^*_{n},\gamma)}
_{\tau - \frac{bt}{1+b}}}
\leq \RR_1^n (\tau,t).
\end{cases}
\endeq

\subsubsection{Step 2: Estimate of $\Om^n-\Om^k$}

In this step our goal is to estimate $\Om^n-\Om^k$ for $1\leq k\leq n-1$. The point is that the error should be small as $k\to\infty$, uniformly in $n$, so we can't just write $\|\Om^n-\Om^k\|\leq \|\Om^n-\Id\| + \|\Om^k-\Id\|$. 
Instead, we start again from the differential equation satisfied by $Z^k$ and $Z^n$:
\begin{align*}
&\frac{\pa^2}{\pa\tau^2}
\bigl(Z_{t,\tau}^n-Z_{t,\tau}^k\bigr)(x,v) \\
& \qquad  = F^n \Bigl(\tau, x-v(t-\tau) + Z^n_{t,\tau}(x,v)\Bigr)
- F^k \Bigl(\tau, x-v(t-\tau) + Z^k_{t,\tau}(x,v)\Bigr)\\
& \qquad = \biggl[ F^n\Bigl(\tau, x-v(t-\tau) + Z^n_{t,\tau}\Bigr)
- F^n \Bigl(\tau,x-v(t-\tau)+Z^k_{t,\tau}\Bigr)\biggr]\\
& \qquad \qquad + (F^n-F^k)\Bigl(\tau, x-v(t-\tau) + Z^k_{t,\tau}\Bigr).
\end{align*}
This, together with the boundary conditions $Z^n_{t,t}-Z^k_{t,t}=0$,
$\pa_\tau (Z^n_{t,\tau}-Z^k_{t,\tau})|_{\tau=t} = 0$, implies
\begin{align*}
& Z^n_{t,\tau} - Z^k_{t,\tau} \\
& \qquad = \int_0^1 \int_\tau^t (s-\tau) \, \nabla_x F^n
\Bigl( s, x-v(t-s) + \bigl(\theta\, Z^k_{t,s} + (1-\theta)\, Z^n_{t,s}\bigr)\Bigr)\cdot
(Z^n_{t,s}-Z^k_{t,s})\,ds\,d\theta\\
&\qquad \quad + \int_\tau^t (s-\tau)\, (F^n-F^k)\,
\Bigl(s, x-v(t-s) + Z^k_{t,s}(x,v)\Bigr)\,ds.
\end{align*}

We fix $t$ and define the norm 
\[ \bbqn (Z_{t,\tau})_{0\leq \tau\leq t}\bbqn_{k,n} := \sup_{0\leq\tau\leq t} \frac{
  \|Z_{t,\tau}\|_{
{\cZ^{\lambda_n^*(1+b),\mu_n^*}_{t-\frac{bt}{1+b}}}}}
{\RR_2^{k,n}(\tau,t)},\] 
where $\RR_2^{k,n}$ is defined in \eqref{R12kn}.  Using the bounds on $Z^n,Z^k$ in $\bqn
\cdot \bqn_n$ (since $\bqn \cdot \bqn_n \leq \bqn \cdot \bqn_k$ by
using the fact that $\RR^k _{2} \le \RR^n _{2}$) and proceeding as
before, we get
\begin{multline} \label{Bqn}
\Bqn \left( Z^n_{t,\tau}-Z^k_{t,\tau} \right)_{0\leq \tau\leq t} \Bqn_{k,n}
\leq \frac12 \, \bbqn \left( Z^n_{t,\tau}-Z^k_{t,\tau} \right)_{0\leq \tau\leq t} \bbqn_{k,n}\\
+ \Bqn \left( \int_\tau^t (s-\tau)\, (F^n-F^k) \Bigl(s,x-v(t-s)+Z^k_{t,s}\Bigr)\,ds \right)_{0\leq \tau\leq t} \Bqn_{k,n}.
\end{multline}

Next we estimate
\begin{align*}
& \Bigl\| (F^n-F^k) \Bigl(s, x-v(t-s) + Z^k_{t,s}\Bigr) \Bigr\|_{
{\cZ^{\lambda_n^*(1+b),\mu_n^*}_{t-\frac{bt}{1+b}}}} \\
& = \Bigl\| (F^n-F^k) (s,X^k_{t,s})\Bigr\|_{
\cZ^{\lambda_n^*(1+b),\mu_n^*}_{t-\frac{bt}{1+b}}}\\
& = \Bigl\| (F^n-F^k) (s,\Om^k_{t,s})\Bigr\|_{
\cZ^{\lambda_n^*(1+b),\mu_n^*}_{s-\frac{bt}{1+b}}}
\leq \Bigl\| (F^n-F^k)(s,\,\cdot\,)\Bigr\|_{\cF^{\nu(s,t)+e(s,t)}},
\end{align*}
where the last inequality follows from Proposition \ref{propcompos2}, $\nu$ is again
given by \eqref{nust}, and
\[e(s,t) = 
\bigl\|\Om^kX_{t,s}-\Id\bigr\|_{
\cZ^{\lambda_n^*(1+b),\mu_n^*}_{s-\frac{bt}{1+b}}}
\leq 2 \, \RR_2^k(s,t) \leq 2 \, \RR_2^n(s,t). \]

The same reasoning as in Step~1 yields, under assumptions ${\bf (C_1)}$-${\bf (C_2)}$,
for $k+1 \leq j \leq n$:
\[ (\nu+e)(s,t)
\begin{cases} \leq \lambda_j\, s + 
  \mu _j  - (\lambda_j -\lambda^*_n)\, s \qquad\qquad
\text{for $s\geq bt/(1+b)$}\\[4mm]
\leq \mu^\sharp - (\lambda_j-\lambda^*_n) \, s
\qquad\qquad\qquad\quad \text{for $s\leq bt/(1+b)$},
\end{cases}\]
and so
\[ \bigl\|F^n_s - F^k_s\bigr\|_{\cF^{\nu+e}} \leq
\sum_{j=k+1}^n \delta_j\,e^{-2\pi (\lambda_j-\lambda^* _n)\, s}.\]

For any $\tau\geq 0$, by integrating in time we find
\begin{align*}
& \Bigl\| \int_\tau^t (s-\tau)\, 
(F^n-F^k) \Bigl(s, x-v(t-s) + Z^k_{t,s}\Bigr) \, ds
\Bigr\|_{\cZ^{\lambda_n^*(1+b),\mu_n^*}_{t-\frac{bt}{1+b}}} \\
& \leq \int_\tau^t (s-\tau)\, \sum_{j=k+1}^n \delta_j\,e^{-2\pi (\lambda_j-\lambda^*_n) \, s} \,ds 
\leq \RR_2 ^{k,n}(\tau,t).
\end{align*}
Therefore
$$
\bbqn \left( \int_\tau^t (s-\tau)\, (F^n-F^k) \Bigl(s,x-v(t-s)+Z^k_{t,s}\Bigr)\,ds \right)_{0\leq \tau\leq t} \bbqn_{k,n} \le 1
$$
and by \eqref{Bqn} 
$$
\bbqn \left( Z^n_{t,\tau}-Z^k_{t,\tau} \right)_{0\leq \tau\leq t} \bbqn_{k,n} \le 2.
$$
Recalling the Sobolev correction, we conclude that
\begin{equation}\label{OmnOmkX}
\bigl\|\Om^n X_{t,\tau} - \Om^k X_{t,\tau}\bigr\|_{
\cZ^{\lambda_n^*(1+b),(\mu_n^*,\gamma)}_{\tau-\frac{bt}{1+b}}}
\leq 2\, \RR_2 ^{k,n}(\tau,t).
\end{equation}

For the velocity component, say $U$, we write
\begin{align*}
&\frac{\pa}{\pa\tau}
\bigl(U_{t,\tau}^n-U_{t,\tau}^k\bigr)(x,v) \\
& \qquad  = F^n \Bigl(\tau, x-v(t-\tau) + Z^n_{t,\tau}(x,v)\Bigr)
- F^k \Bigl(\tau, x-v(t-\tau) + Z^k_{t,\tau}(x,v)\Bigr)\\
& \qquad = \biggl[ F^n\Bigl(\tau, x-v(t-\tau) + Z^n_{t,\tau}\Bigr)
- F^n \Bigl(\tau,x-v(t-\tau)+Z^k_{t,\tau}\Bigr)\biggr]\\
& \qquad \qquad + (F^n-F^k)\Bigl(\tau, x-v(t-\tau) + Z^k_{t,\tau}\Bigr).
\end{align*}
where $Z^n,Z^k$ were estimated above, and the boundary conditions are $U^n_{t,t}-U^k_{t,t}=0$.
Thus
\begin{align*}
& U^n_{t,\tau} - U^k_{t,\tau} \\
& \qquad = \int_0^1 \int_\tau^t \nabla_x F^n
\Bigl( s, x-v(t-s) + \bigl(\theta\, Z^k_{t,s} + (1-\theta)\, Z^n_{t,s}\bigr)\Bigr)\cdot
(Z^n_{t,s}-Z^k_{t,s})\,ds\,d\theta\\
&\qquad \quad + \int_\tau^t (F^n-F^k)\,
\Bigl(s, x-v(t-s) + Z^k_{t,s}(x,v)\Bigr)\,ds,
\end{align*}
and from this one easily derives the similar estimates
\begin{equation*}
\begin{cases}
\bigl\|\Om^n X_{t,\tau} - \Om^k X_{t,\tau}\bigr\|_{\cZ^{\lambda^* _n (1+b),\mu^* _{n}}
_{\tau-\frac{bt}{1+b}}} \leq 2\, \RR_2 ^{k,n}(t,\tau) \\[4mm]
\bigl\|\Om^n V_{t,\tau} - \Om^k V_{t,\tau}\bigr\|_{\cZ^{\lambda^* _n(1+b),\mu^*  _{n}}
_{\tau-\frac{bt}{1+b}}} \leq \RR_1 ^{k,n}(t,\tau) + \RR_2 ^{k,n}(t,\tau).
\end{cases}
\end{equation*}

\subsubsection{Step 3: Estimate of $\nabla \Om^n$}

Now we establish a control on the derivative of the scattering.  Of
course, we could deduce such a control from the bound on $\Om^n-\Id$
and Proposition \ref{prophybrsob}(vi): for instance, if $\lambda^{**}
_n<\lambda^* _n$, $\mu^{**} _n <\mu^* _n$, then
\begeq\label{estOm'iter} \bigl\|\nabla \Omega^n _{t,\tau} -
I\bigr\|_{\cZ^{\lambda^{**}_n(1+b),(\mu^{**} _n,\gamma)} _{\tau -
    \frac{bt}{1+b}}} \leq \frac{C\, \RR_2^n(\tau,t)} {\min\, \bigl\{
  \lambda_n^*-\lambda_n^{**} \, ; \ \mu_n^*-\mu_n^{**}\bigr\}}.
\endeq

But this bound involves very large constants, and is useless in our
argument.  Better estimates can be obtained by using again the
equation \eqref{eqXnF}. Writing
\[(\Om^n_{t,\tau}-\Id)(x,v) = \Bigl(
Z^n_{t,\tau}\bigl(x+v(t-\tau),v\bigr),
\dot{Z}_{t,\tau}^n\bigl(x+v(t-\tau),v\bigr)\Bigr),\] where the dot
stands for $\pa/\pa \tau$, we get by differentiation
\[ \nabla_x \Om^n_{t,\tau} - (I,0) = \Bigl( \nabla_x Z^n_{t,\tau}
\bigl(x+v(t-\tau),v\bigr), \,
\nabla_x\dot{Z}^n_{t,\tau}\bigl(x+v(t-\tau),v\bigr)\Bigr),\]
\[ \nabla_v \Om^n_{t,\tau} - (0,I) = \Bigl( (\nabla_v +
(t-\tau)\nabla_x) Z^n_{t,\tau} \bigl(x+v(t-\tau),v\bigr),\, (\nabla_v
+
(t-\tau)\nabla_x)\dot{Z}^n_{t,\tau}\bigl(x+v(t-\tau),v\bigr)\Bigr).\]

Let us estimate for instance $\nabla_x \Om - (I,0)$, or equivalently
$\nabla_x Z^n_{t,\tau}$. By differentiating \eqref{eqXnF}, we obtain
\[ \frac{\pa^2}{\pa \tau^2} \nabla_x Z^n_{t,\tau}(x,v) = 
\nabla_x F^n\Bigl(\tau, x-v(t-\tau)+Z^n_{t,\tau}(x,v)\Bigr)\cdot (\Id + \nabla_x Z^n_{t,\tau}).\]
So $\nabla_x Z^n_{t,\tau}$ is a fixed point of $\Psi: W\longmapsto Q$, where $W$ and $Q$ are functions of $\tau \in [0,t]$
satisfying
\[ \begin{cases} \dps
 \frac{\pa^2Q }{\pa \tau^2} = \nabla_x F^n \bigl(\tau, x-v(t-\tau)+Z^n_{t,\tau}\bigr)(I+W),\\[3mm]
Q(t) =0,\quad \pa_\tau Q(t) =0.
\end{cases}\] We treat this in the same way as in Steps~1 and~2, and
find on $Q_x$ (the $x$ component of $Q$) the same estimates as we had
previously on the $x$ component of $\Om$.  For the velocity component,
a direct estimate from the integral equation expressing the velocity
in terms of $F$ yields a control by $\RR_1 ^n + \RR_2 ^n$.  Finally
for $\nabla_v\Om$ this is similar, noting that $(\nabla_v
+(t-\tau)\nabla_x)(x-v(t-\tau)) = 0$, the differential equation being
for instance:
\[ \frac{\pa^2}{\pa \tau^2} (\nabla_v + (t-\tau) \nabla_x)
Z^n_{t,\tau}(x,v) = \nabla_x F^n\Bigl(\tau,
x-v(t-\tau)+Z^n_{t,\tau}(x,v)\Bigr)\cdot ((\nabla_v + (t-\tau)
\nabla_x) Z^n_{t,\tau}).\]

In the end we obtain
\begeq\label{smallnablat}
  \begin{cases} 
     \dps \sup_{t\geq\tau\geq 0}\:
    \Bigl\|\nabla \Om^nX_{t,\tau} - (I,0) \Bigr\|_{\cZ^{\lambda ^* _n (1+b),\mu^*_n}
      _{\tau - \frac{bt}{1+b}}} \leq 2\, \RR_2^n (\tau,t), \\[4mm]
    \dps \sup_{t\geq\tau\geq 0}\:
    \Bigl\|\nabla \Om^nV_{t,\tau} - (0,I) \Bigr\|_{\cZ^{\lambda ^* _n (1+b),\mu^*_n}
      _{\tau - \frac{bt}{1+b}}} \leq \RR_1 ^n (\tau,t) + \RR_2 ^n (\tau,t).
  \end{cases}
\endeq

\med



\subsubsection{Step 4: Estimate of $(\Omega^k)^{-1} \circ \Omega^n$} 
\label{sec:step-4:-estimate}

We do this by applying Proposition \ref{propinv} with $F=\Om^k$,
$G=\Om^n$.  (Note: we cannot exchange the roles of $\Om^k$ and $\Om^n$
in this step, because we have a better information on the regularity
of $\Om^k$.) Let $\var=\var(d)$ be the small constant appearing in
Proposition \ref{propinv}.  If \begeq\label{C3new} {\bf (C_3)}\qquad
\forall \, k \geq 1,\qquad 3 \, \RR_2^k(\tau,t) + \RR_1^k(\tau,t)\leq
\var,
\endeq
then $\|\nabla\Om^k_{t,\tau}-I \|_{\cZ^{\lambda_k^*(1+b),\mu_k^*}_{\tau-bt/(1+b)}}\leq \var$; if in addition
\begin{multline}\label{C4new} 
{\bf (C_4)}\qquad
\forall \, k\in\{1,\ldots,n-1\},\quad \forall \, t\geq\tau,\quad\\
2(1+\tau)\, (1+B)\, \bigl(3\, \RR_2^{k,n}+\RR_1^{k,n}\bigr)(\tau,t) \leq
\max \bigl\{\lambda_k^*-\lambda_n^* \, ; \, \mu_k^*-\mu_n^*\bigr\},
\end{multline}
then
\[ \begin{cases}
\dps \lambda_n^*(1+b) + 2 \,
\|\Om^n-\Om^k\|_{\cZ^{\lambda_n^*(1+b),\mu_n^*}_{\tau-\frac{bt}{1+b}}}\leq
\lambda_k^* (1+b)\\[3mm]
\dps \mu_n^* + 2\left(1+\left|\tau - \frac{bt}{1+b}\right|\right) \, 
\|\Om^n-\Om^k\|_{\cZ^{\lambda_n^*(1+b),\mu_n^*}_{\tau-\frac{bt}{1+b}}}\leq \mu_k^*.
\end{cases}\]
(Once again, short times should be
treated separately. Further note that the need for the factor
$(1+\tau)$ in ${\bf (C_4)}$ ultimately comes from the fact that we are
composing also in the $v$ variable, see the coefficient $\sigma$ in
the last norm of \eqref{alphaV}.)  Then Proposition \ref{propinv} (ii)
yields
\begin{align*}
\Bigl\| (\Om^k_{t,\tau})^{-1} \circ \Om^n _{t,\tau} - \Id \Bigr\|_{\cZ^{\lambda_n^*(1+b), \mu_n^*}_{\tau-\frac{bt}{1+b}}}
& \leq 2\, \bigl\|\Om^k_{t,\tau}-\Om^n_{t,\tau}\bigr\|_{\cZ^{\lambda_n^*(1+b), \mu_n^*}_{\tau-\frac{bt}{1+b}}}\\
& \leq 4\, (\RR_1^{k,n}+\RR_2^{k,n})(\tau,t).
\end{align*}

\subsubsection{Partial conclusion}

At this point we have established ${\bf (E_{\Omega} ^n)}+{\bf (\tilde E_{\Omega} ^n)}+{\bf (E_{\nabla \Omega} ^n)}$.

\subsection{Estimates on the density and distribution along characteristics}
In this subsection we establish ${\bf (E_\rho ^{n+1})}+
{\bf (\tilde E_\rho ^{n+1})}+{\bf (E_h ^{n+1})}+{\bf (\tilde E_h ^{n+1})}$.
\label{subsec:B}

\subsubsection{Step 5: Estimate of $h^k\circ\Om^n$ and $(\nabla h^k)\circ\Om^n$ ($k\leq n$)} \label{sub:esthk}

Let $k\in \{1,\ldots,n\}$. Since 
\[ 
h_\tau^k \circ \Om^n_{t,\tau} = \left(h_\tau^k\circ\Om^{k-1}_{t,\tau}\right)
\circ \left( (\Om^{k-1} _{t,\tau})^{-1}\circ\Om^n_{t,\tau}\right),
\] 
the control on $h^k\circ\Om^n$ will follow from the control on
$h^k\circ\Om^{k-1}$ in ${\bf (E_h ^n)}$, together with the control on
$(\Om^{k-1})^{-1} \circ \Omega^n$ in ${\bf (\tilde E_\Omega ^n)}$.  If
\begeq\label{a1+tau} 
(1+\tau)\,
\bigl\|(\Om^{k-1}_{t,\tau})^{-1} \circ \Om^n _{t,\tau} -\Id\bigr\|_{\cZ^{\lambda_n^*(1+b),\mu_n^*}_{\tau-\frac{bt}{1+b}}}
\leq \min \bigl\{(\lambda_k-\lambda_n^*)\, ; \, 
(\mu_k-\mu_n^*)\bigr\},
\endeq
then we can apply Proposition \ref{propcompos2} and get, for any $p\in [1,\ov{p}]$, and $t\geq\tau\geq 0$,
\begeq
\Bigl\|h_\tau^k\circ\Om^n_{t,\tau}\Bigr\|_{\cZ^{\lambda_n^*(1+b),\mu_n^*;p}_{\tau-\frac{bt}{1+b}}}
\leq \Bigl\|h_\tau^k\circ\Om^{k-1}
_{t,\tau}\Bigr\|_{\cZ^{\lambda_{k} (1+b),\mu_k;p}_{\tau-\frac{bt}{1+b}}}
\leq \delta_k.
\endeq

In turn, \eqref{a1+tau} is satisfied if
\begin{multline}\label{C5}
  {\bf (C_5)} \qquad
\forall \, k\in\{1,\ldots,n\}, \quad \forall \, \tau\in [0,t],\\  
4 \, (1+\tau)\ \bigl(\RR_1^{k,n}(\tau,t)+\RR_2^{k,n}(\tau,t)\bigr)
\leq \min \bigl\{ \lambda_k-\lambda_n^* \, ; \,  \mu_k-\mu_n^*\bigr\};
\end{multline}
we shall check later the feasibility of this condition.

Then, by the same argument, we also have
\begin{multline*}
\forall \, k\in \{1, \dots, n\}, \quad
  \forall \, p \in [1,\ov{p}], \\ 
  \sup_{t\geq \tau\geq 0}\: 
  \Bigl\|(\nabla_x h_\tau^k)  \circ
  \Om_{t,\tau}^{n}\Bigr\|_{\cZ^{\lambda_n ^* (1+b),\mu^* _n;p}_{\tau -
      \frac{bt}{1+b}}} + 
  \Bigl\|\bigr((\nabla_v + \tau \nabla_x) h_\tau^k\bigr)  \circ
  \Om_{t,\tau}^{n}\Bigr\|_{\cZ^{\lambda_n ^*(1+b),\mu_n ^*;p}_{\tau -
      \frac{bt}{1+b}}}
  \leq \, \delta_k.
\end{multline*}

\subsubsection{Step 6: estimate on $\rho[h^{n+1}]$} \label{sub:estrho}

This step is the first where we shall use the Vlasov equation.
Starting from \eqref{N2}, we apply the method of characteristics to
get, as in Section \ref{sec:approx}, 
\begeq\label{xi} h^{n+1} \bigl(t,
X^n_{0,t}(x,v),V^n_{0,t}(x,v)\bigr) = \int_0^t \Sigma^{n+1}\bigl(\tau,
X^n_{0,\tau}(x,v), V^n_{0,\tau}(x,v)\bigr)\,d\tau,
\endeq
where
\[ \Sigma^{n+1} = - \Bigl( F[h^{n+1}]\cdot\nabla_vf^n + F[h^n]\cdot\nabla_v h^n\Bigr).\]
We compose this with $(X^n_{t,0},V^n_{t,0})$ and apply \eqref{SSS} to get
\[ h^{n+1}(t,x,v) = \int_0^t \Sigma^{n+1}\bigl(\tau,X^n_{t,\tau}(x,v),V^n_{t,\tau}(x,v)\bigr)\,d\tau,\]
and so, by integration in the $v$ variable,
\begin{align}
\rho[h^{n+1}](t,x)
& = \int_0^t \int_{\R^d}  \Sigma^{n+1}\bigl(\tau,X^n_{t,\tau}(x,v),V^n_{t,\tau}(x,v)\bigr)\,dv\,d\tau\\
& \nonumber
= - \int_0^t \int_{\R^d} (R^{n+1}_{\tau,t}\cdot G^n_{\tau,t})(x - v(t-\tau),v) \,dv\,d\tau \\
& \quad - \int_0^t \int_{\R^d}  (R^n_{\tau,t}\cdot H^n_{\tau,t})(x-v(t-\tau),v) \,dv\,d\tau, \nonumber
\end{align}
where (with a slight inconsistency in the notation)
\begin{equation}\label{RGH}
\begin{cases}
 R^{n+1}_{\tau,t} = F[h^{n+1}]\circ\Om^n_{t,\tau},\quad
 R^n_{\tau,t} = F[h^n]\circ\Om^n_{t,\tau},\\[4mm] 
 G^n_{\tau,t} = (\nabla_v f^n)\circ\Om^n_{t,\tau},\quad
 H^n_{\tau,t} = (\nabla_v h^n)\circ\Om^n_{t,\tau}.
\end{cases}
\end{equation}

Since the free transport semigroup and $\Om^n_{t,\tau}$ are
measure-preserving,
\begin{align*}
\forall \ 0 \le \tau \le t, \quad 
& \int_{\T^d} \int_{\R^d} (R^{n+1}_{\tau,t}\cdot G^{n}_{\tau,t})(x - v(t-\tau),v) \,dv\,dx \\
& = \int \int R^{n+1}_{\tau,t}\cdot G^{n}_{\tau,t} \,dv\,dx \\
& = \int \int F[h^{n+1}]\cdot\nabla_v f^n\,dv\,dx\\
& = \int \int \nabla_v\cdot \bigl( F[h^{n+1}] \, f^n\bigr)\,dv\,dx =0,
\end{align*}
and similarly 
$$
\forall \ 0 \le \tau \le t, \quad 
\int_{\T^d} \int_{\R^d}(R^n_{\tau,t}\cdot H^n_{\tau,t})(x - v(t-\tau),v) \,dv\,dx =0. 
$$
This will allow us to apply the inequalities
from Section \ref{sec:regdecay}.
\sm 

\noindent
{\em Substep a.} Let us first deal with the source term
\begeq\label{sigmann} \sigma^{n,n}(t,x)
:= \int_0^t \int (R^{n}_{\tau,t}\cdot H^n_{\tau,t})(x - v(t-\tau),v) \,dv\,d\tau.
\endeq
By Proposition \ref{proprt},
\begeq\label{estsigman}
\bigl\|\sigma^{n,n} (t,\,\cdot\,)\bigr\|_{\cF^{\lambda^* _nt+\mu^*_n}}
\leq \int_0^t \|R^n_{\tau,t}\|_{\cZ^{\lambda^* _n(1+b),\mu^* _n}_{\tau-\frac{bt}{1+b}}}
\, \|H^n_{\tau,t}\|_{\cZ^{\lambda^* _n(1+b),\mu^* _n;1}_{\tau-\frac{bt}{1+b}}}\,d\tau.
\endeq

On the one hand, we have from Step~5
$$
\|H^n_{\tau,t}\|_{\cZ^{\lambda^* _n(1+b),\mu^* _n;1}_{\tau-\frac{bt}{1+b}}}\leq 2 \, (1+\tau) \, \delta_n.
$$ 

On the other hand, under condition ${\bf (C_1)}$, we may apply Proposition \ref{propcompos2}
(choosing $\sigma=0$ in that proposition) to get
\[ \|R^n_{\tau,t}\|_{\cZ^{\lambda^*_n (1+b),\mu^*_n}_{\tau-\frac{bt}{1+b}}}
\leq \bigl\|F[h^n_\tau]\bigr\|_{\cF^{\nu_n}},\]
where
\begin{align*}
\nu_n (t,\tau) & = \mu^*_n + \lambda_n^*(1+b)\, \left|\tau - \frac{bt}{1+b}\right|
+ \bigl\|\Om^nX_{t,\tau} - \Id\bigr\|_{\cZ^{\lambda^* _n(1+b),\mu^* _n}_{\tau-\frac{bt}{1+b}}}\\
& \leq \mu^* _n + \lambda_n^*(1+b)\,\left|\tau-\frac{bt}{1+b}\right|
+ 2 \, \RR_2 ^n(\tau,t).
\end{align*}

Proceeding as in Step~1 (treating small times separately), we deduce
\begin{multline*} \|R^n_{\tau,t}\|_{\cZ^{\lambda^*_n (1+b),\mu^*_n}_{\tau-\frac{bt}{1+b}}}
\leq  \|F[h^n_\tau]\|_{\cF^{\nu_n}} \leq 
e^{-2 \pi (\lambda_n-\lambda^* _n)\tau} \, \|F[h^n_\tau]\|_{\cF^{\bar \nu_n}}\\
\leq C_F\, e^{-2 \pi (\lambda_n-\lambda^* _n)\tau} \, \|\rho[h^n_\tau]\|_{\cF^{\bar \nu_n}}
\leq C_F\,e^{-2 \pi (\lambda_n-\lambda^* _n)\tau} \, \delta_n,
\end{multline*}
with 
\begin{equation}
\label{def:nu}
\begin{cases}
\bar \nu_n(\tau,t) := \mu^\sharp \quad \mbox { when } 0 \le \tau \le bt/(1+b) \\[4mm]
\bar \nu_n (\tau,t):= \lambda_n \tau+ \mu_n \quad \mbox{ when } \tau \ge bt/(1+b).
\end{cases}
\end{equation}
(We have used the gradient structure of the force  to convert (gliding) regularity into decay.)
Thus
\begin{align} \label{Rntaut}
\int_0 ^t \|R^n_{\tau,t}\|_{\cZ^{\lambda^* _n(1+b),\mu^* _n}_{\tau-\frac{bt}{1+b}}}\,
& \|H^n_{\tau,t}\|_{\cZ^{\lambda^* _n(1+b),\mu^* _n;1}_{\tau-\frac{bt}{1+b}}}\,d\tau\\
& \leq 2 \, C_F \, \delta_n ^2 \, \int_0^t e^{-2 \pi (\lambda_n-\lambda^* _n)\tau}\, (1+\tau) \,d\tau \nonumber \\
&  \leq \frac{2 \, C_F \, \delta_n^2}{(\pi \, (\lambda_n-\lambda^* _n))^2}. \nonumber
\end{align}
(Note: This is the power~$2$ which is responsible for the very fast convergence of the Newton scheme.)
\sm 

\noindent
{\em Substep b.} Now let us handle the term
\begeq\label{sigmann+1} \sigma^{n,n+1}(t,x)
:= \int_0^t \int \bigl(R^{n+1}_{\tau,t}\cdot G^n_{\tau,t}\bigr)(x - v(t-\tau),v) \,dv\,d\tau.
\endeq
This is the focal point of all our analysis, because it is in this
term that the self-consistent nature of the Vlasov equation
appears. In particular, we will make crucial use of the time-cheating
trick to overcome the loss of regularity implied by composition; and
also the other bilinear estimates (regularity extortion) from Section
\ref{sec:regdecay}, as well as the time-response study from Section
\ref{sec:response}.  Particular care should be given to the zero
spatial mode of $G^n$, which is associated with instantaneous response
(no echo). In the linearized equation we did not see this problem
because the contribution of the zero mode was vanishing!  \sm

We start by introducing
\begeq\label{ovGn}
\ov{G}^n_{\tau,t} = \nabla_v f^0 + \sum_{k=1}^n \nabla_v 
\left( h^k_\tau \circ \Om^{k-1}_{t,\tau} \right),
\endeq
and we decompose $\sigma^{n,n+1}$ as
\begeq\label{sigmadecomp}
\sigma^{n,n+1} = \ov{\sigma}^{n,n+1} + {\cal E} + \ov{\cal E},
\endeq
where 
\begeq\label{ovsigmann+1}
\ov{\sigma}^{n,n+1}(t,x) = 
\int_0^t \int F[h^{n+1}_\tau]\cdot \ov{G}^n_{\tau,t}\bigl(x-v(t-\tau),v\bigr)\,dv\,d\tau
\endeq
and the error terms ${\cal E}$ and $\ov{\cal E}$ are defined by
\begeq\label{calE}
{\cal E}(t,x) = \int_0^t \int \Bigl( \Bigl(F[h^{n+1}_\tau]\circ\Om^n_{t,\tau} - F[h^{n+1}_\tau]\Bigr)\cdot G^n\Bigr)
\bigl(\tau, x-v(t-\tau),v\bigr)\,dv\,d\tau,
\endeq
\begeq\label{ovcalE}
\ov{\cal E}(t,x) = \int_0^t \int \Bigl(F[h^{n+1}_\tau]\cdot \bigl(G^n-\ov{G}^n\bigr)\Bigr)
\bigl(\tau, x-v(t-\tau),v\bigr)\,dv\,d\tau.
\endeq

We shall first estimate ${\cal E}$ and $\ov{\cal E}$.
\med

\noindent {\em Control of ${\cal E}$:} This is based on the
time-cheating trick from Section \ref{sec:regdecay}, and the
regularity of the force. By Proposition \ref{proprt},
\begin{multline}
\bigl\|{\cal E}(t,\,\cdot\,)\bigr\|_{\cF^{\lambda_n^*t+\mu_n^*}}
\leq \int_0^t \Bigl\| F[h^{n+1}_\tau]\circ \Om^n_{t,\tau} - F[h^{n+1}_\tau]\Bigr\|
_{\cZ^{\lambda_n^*(1+b),\mu_n^*}_{\tau-\frac{bt}{1+b}}} \, \times \\
\|G^n\|_{\cZ^{\lambda_n^*(1+b),\mu_n^*;1}_{\tau-\frac{bt}{1+b}}}\, d\tau.
\end{multline}

From \eqref{fh} and Step~5,
\begin{align} \label{GnZ}
\|G^n\|_{\cZ^{\lambda_n^*(1+b),\mu_n^*;1}_{\tau-\frac{bt}{1+b}}}
& \leq \bigl\|\nabla_v f^0\circ\Om^n_{t,\tau}\bigr\|_{\cZ^{\lambda_n^*(1+b),\mu_n^*;1}_{\tau-\frac{bt}{1+b}}}
+ \sum_{k=1}^n \bigl\|\nabla_v h^k_\tau \circ\Om^n_{t,\tau}\bigr\|_{\cZ^{\lambda_n^*(1+b),\mu_n^*;1}_{\tau-\frac{bt}{1+b}}}\\
& \leq C'_0 + \left(\sum_{k=1}^n \delta_k\right)\, (1+\tau), \nonumber
\end{align}
where $C'_0$ comes from the contribution of $f^0$.

Next, by Propositions \ref{propalgZ} and \ref{propcompos2} (with $V=0$, $\tau=\sigma$, $b=0$),
\begin{align}\label{FFnO}
& \Bigl\|F[h^{n+1}_\tau]\circ\Om^n_{t,\tau} - F[h^{n+1}_\tau]\Bigr\|_{\cZ^{\lambda_n^*(1+b),\mu_n^*}_{\tau-\frac{bt}{1+b}}}\\
& \leq \left(\int_0^1 \Bigl\|\nabla F[h^{n+1}_\tau]\circ \bigl(\Id + \theta (\Om^n_{t,\tau}-\Id)\bigr)
\Bigr\|_{\cZ^{\lambda_n^*(1+b),\mu_n^*}_{\tau-\frac{bt}{1+b}}}\,d\theta \right) \ \bigl\|\Om^n_{t,\tau}-\Id\bigr\|_{\cZ^{\lambda_n^*(1+b),\mu_n^*}_{\tau-\frac{bt}{1+b}}}\nonumber\\
& \leq \bigl\|\nabla F[h^{n+1}_\tau]\bigr\|_{\cF^{\nu_n}}\ 
\bigl\|\Om^n_{t,\tau}-\Id\bigr\|_{\cZ^{\lambda_n^*(1+b),\mu_n^*}_{\tau-\frac{bt}{1+b}}}, \nonumber
\end{align}

where
\[ \nu_n = \mu_n^* + \lambda_n^* (1+b)\, \left|\tau-\frac{bt}{1+b}\right|
+ \bigl\|\Om^nX_{t,\tau} -x\bigr\|_{\cZ^{\lambda_n^*(1+b),\mu_n^*}_{\tau-\frac{bt}{1+b}}}.\]
Small times are taken care of, as usual, by the initial regularity layer, so we only focus on the case
$\tau \geq bt/(1+b)$; then
\begin{align*}
\nu_n & \leq \bigl(\lambda_n^*\tau + \mu_n^*\bigr) -
\lambda_n^*\,b(t-\tau) + 2 \, \RR^n(\tau,t)\\
& \leq \bigl(\lambda_n^*\tau + \mu_n^*\bigr) - \lambda_n^*\,\frac{B\,(t-\tau)}{1+t} 
+ 4\, C_\omega^1\, \left(\sum_{k=1}^n \frac{\delta_k}{(2\pi(\lambda_k-\lambda_k^*))^3}\right)\,
\frac{\min\{t-\tau\, ; \, 1\}}{1+\tau}.
\end{align*}
To make sure that $\nu_n\leq \lambda_n^*\tau + \mu_n^*$, we assume that
\begin{equation}
  \label{C6}
  {\bf (C_6)} \qquad
4\, C^1_\om \ \sum_{k=1}^n \frac{\delta_k}{(2\pi(\lambda_k-\lambda_k^*))^3}\leq \frac{\lambda_\infty^*\, B}{3},
\end{equation}
and we note that
\[ \frac{\min\{t-\tau\, ; \, 1\}}{1+\tau} \leq 3\,
\left(\frac{t-\tau}{1+t}\right).\] 
(This is easily seen by separating four cases: (a) $t\leq 2$, (b)
$t\geq 2$ and $t-\tau\leq 1$, (c) $t\geq 2$ and $t-\tau\geq 1$ and
$\tau\leq t/2$, (d) $t\geq 2$ and $t-\tau\geq 1$ and $\tau\geq t/2$.)

Then, since $\gamma\geq 1$, we have
\begin{align}\label{gammageq1}
  \bigl\|\nabla F[h^{n+1}_\tau]\bigr\|_{\cF^{\nu_n}}
  & \leq \bigl\|\nabla F[h^{n+1}_\tau]\bigr\|_{\cF^{\lambda_n^*\tau+\mu_n^*}}\\
  & \leq \bigl\|F[h^{n+1}_\tau]\bigr\|_{\cF^{\lambda_n^*\tau+\mu_n^*,\gamma}} \nonumber \\
  & \leq C_F\
  \bigl\|\rho[h^{n+1}_\tau]\bigr\|_{\cF^{\lambda_n^*\tau+\mu_n^*}}. \nonumber
\end{align}
(Note: Applying Proposition \ref{propgrad} instead of the regularity
coming from the interaction would consume more regularity than we can
afford to.)

Plugging this back into \eqref{FFnO}, we get
\begin{align*}
  \Bigl\|F[h^{n+1}_\tau]\circ\Om^n_{t,\tau} & - F[h^{n+1}_\tau]\Bigr\|_{\cZ^{\lambda_n^*(1+b),\mu_n^*}_{\tau-\frac{bt}{1+b}}}\\
  & \leq 2\, \RR^n(\tau,t)\, C_F \, \|\rho[h^{n+1}_\tau]\|_{\cF^{\lambda_n^*\tau+\mu_n^*}}\\
  & \leq 2\, C_\om^3\, C_F\, \left(\sum_{k=1}^n
    \frac{\delta_k}{(2\pi(\lambda_k-\lambda_k^*))^5}\right)\,
  \frac1{(1+\tau)^3}\,
  \bigl\|\rho[h^{n+1}_\tau]\bigr\|_{\cF^{\lambda_n^*\tau+\mu_n^*}}.
\end{align*}
Recalling \eqref{calE} and \eqref{GnZ}, applying Proposition \ref{propalgZ}, we conclude that
\begin{multline}\label{controlcalE}
\bigl\|{\cal E}(t,\,\cdot\,)\bigr\|_{\cF^{\lambda_n^*t+\mu_n^*}}
\leq 2\, C_\om^3\, C_F\, \left(C'_0 + \sum_{k=1}^n \delta_k\right)
\left(\sum_{k=1}^n \frac{\delta_k}{(2\pi(\lambda_k-\lambda_k^*))^5}\right)\\
\int_0^t \bigl\|\rho[h^{n+1}_\tau]\bigr\|_{\cF^{\lambda_n^*\tau+\mu_n^*}}\,\frac{d\tau}{(1+\tau)^2}.
\end{multline}
(We could be a bit more precise; anyway we cannot go further since we do not yet have an estimate on
$\rho[h^{n+1}]$.) 
\med

\noindent {\em b2. Control of $\bar {\cal E}$:}
This will use the control on the derivatives of $h^k$. We start again from Proposition \ref{proprt}:
\begeq\label{ovcalEle}
\bigl\|\ov{\cal E}(t,\,\cdot\,)\bigr\|_{\cF^{\lambda_n^*t+\mu_n^*}}
\leq \int_0^t \bigl\|G^n-\ov{G}^n\bigr\|_{\cZ^{\lambda_n^*(1+b),\mu_n^*;1}_{\tau-\frac{bt}{1+b}}}\
\|F[h^{n+1}_\tau]\|_{\cF^{\beta_n}}\,d\tau,
\endeq
where $\beta_n = \lambda_n^*(1+b)|\tau-bt/(1+b)|+\mu_n^*$. We focus again on the case $\tau\geq bt/(1+b)$, so that
(with crude estimates)
\[ \|F[h^{n+1}_\tau]\|_{\cF^{\beta_n}}\leq
\|F[h^{n+1}_\tau]\|_{\cF^{\lambda_n^*\tau + \mu_n^*}}
\leq C_F\,\|\rho[h^{n+1}_\tau]\|_{\cF^{\lambda_n^*\tau+\mu_n^*}},\]
and the problem is to control $G^n-\ov{G}^n$:
\begin{multline} \label{G-G}
\bigl\|G^n-\ov{G}^n\bigr\|_{\cZ^{\lambda_n^*(1+b),\mu_n^*;1}_{\tau-\frac{bt}{1+b}}}
\leq \Bigl\|(\nabla_v f^0)\circ\Om^n_{t,\tau} - \nabla_v f^0\Bigr\|_{\cZ^{\lambda_n^*(1+b),\mu_n^*;1}_{\tau-\frac{bt}{1+b}}}\\
+ \sum_{k=1}^n \Bigl\|(\nabla_v h^k_\tau)\circ\Om^n_{t,\tau} - (\nabla_v h^k_\tau)\circ\Om^{k-1}_{t,\tau} \Bigr\|
_{\cZ^{\lambda_n^*(1+b),\mu_n^*;1}_{\tau-\frac{bt}{1+b}}}
+ \sum_{k=1}^n \Bigl\|(\nabla_v h^k_\tau)\circ\Om^{k-1}_{t,\tau} - \nabla_v \bigl(h^k_\tau\circ\Om^{k-1}_{t,\tau}\bigr) \Bigr\|
_{\cZ^{\lambda_n^*(1+b),\mu_n^*;1}_{\tau-\frac{bt}{1+b}}}.
\end{multline}

By induction hypothesis ${\bf (\tilde E^n_h)}$, and since the
$\cZ^{\lambda,\mu}_\tau$ norms are increasing as a function of
$\lambda,\mu$,
\[ \sum_{k=1}^n \Bigl\|(\nabla_v h^k_\tau)\circ\Om^{k-1} _{t,\tau} -
\nabla_v \bigl(h^k_\tau\circ\Om^{k-1}_{t,\tau}\bigr) \Bigr\|
_{\cZ^{\lambda_n^*(1+b),\mu_n^*;1}_{\tau-\frac{bt}{1+b}}} \leq
\left(\sum_{k=1}^n \delta_k\right)\ \frac1{(1+\tau)^2}.\] It remains
to treat the first and second terms in the right-hand side of
\eqref{G-G}. This is done by inversion/composition as in Step~5; let us consider
for instance the contribution of $h^k$, $k\geq 1$:
\begin{align*}
\Bigl\|\nabla_v h^k_\tau & \circ\Om^n_{t,\tau}
- \nabla_v h^k_\tau \circ
\Om^{k-1}_{t,\tau}\Bigr\|_{\cZ^{\lambda_n^*(1+b),\mu_n^*;1}_{\tau-\frac{bt}{1+b}}}\\
& \leq \int_0^1 \Bigl\|\nabla \nabla_v h^k_\tau \circ
\Bigl( (1-\theta) \Om^n_{t,\tau} + \theta \Om^{k-1}_{t,\tau}\Bigr)\Bigr\|
_{\cZ^{\lambda_n^*(1+b),\mu_n^*;1}_{\tau-\frac{bt}{1+b}}}\
\Bigl\|\Om^n_{t,\tau} -
\Om^{k-1}_{t,\tau}\Bigr\|_{\cZ^{\lambda_n^*(1+b),\mu_n^*}_{\tau-\frac{bt}{1+b}}}
\, d\theta\\
& \leq 2 \, \Bigl\|\nabla \nabla_v h^k_\tau \circ\Om^{k-1}_{t,\tau}\Bigr\|
_{\cZ^{\lambda_k^*(1+b),\mu_k^*;1}_{\tau-\frac{bt}{1+b}}}
\ 
\Bigl\|\Om^n_{t,\tau} - \Om^{k-1}_{t,\tau}\Bigr\|
_{\cZ^{\lambda_n^*(1+b),\mu_n^*}_{\tau-\frac{bt}{1+b}}}  \\
& \leq 4\, \delta_k\, (1+\tau)^2\, \RR^{k-1,n}(\tau,t) \\
& \leq 4\, C_\om^4 \, \delta_k\ \left(\sum_{j=k}^n
  \frac{\delta_j}{(2\pi(\lambda_j-\lambda_j^*))^6}\right) \
\frac1{(1+\tau)^2},
\end{align*}
where in the but-to-last step we used ${\bf (\tilde E_\Om ^{n})}$,
${\bf (\tilde E_\rho ^{n})}$, Propositions \ref{propalgZ} and
\ref{propinv}, Condition ${\bf (C_5)}$ and the same reasoning as in
Step~5.  

Summing up all contributions and inserting in \eqref{ovcalEle} yields
\begin{multline}\label{controlovcalE}
\bigl\|\ov{\cal E}(t,\,\cdot\,)\bigr\|_{\cF^{\lambda_n^*t+\mu_n^*}}\\
\leq 4\, C_F \left[ C_\om^4\, \left(C'_0 + \sum_{k=1}^n \delta_k\right)
\left(\sum_{j=1}^n
  \frac{\delta_j}{(2\pi(\lambda_j-\lambda_n^*))^6}\right)
+ \sum_{k=1}^n \delta_k\right]
\int_0^t \|\rho[h^{n+1}_\tau]\|_{\cF^{\lambda_n^*\tau+\mu_n^*}}\, \frac{d\tau}{(1+\tau)^2}.
\end{multline}

\sm

\noindent {\em b3. Main contribution:} Now we consider
$\ov{\sigma}^{n,n+1}$, which we decompose as
\[ \ov{\sigma}^{n,n+1}_t = \ov{\sigma}^{n,n+1}_{t,0} + \sum_{k=1}^n
\ov{\sigma}_{t,k}^{n,n+1},\] where
\[ \ov{\sigma}^{n,n+1}_{t,0}(x) = \int_0^t \int F[h^{n+1}]\bigl(\tau,x-v(t-\tau),v\bigr)\cdot\nabla_v f^0(v)\,dv\,d\tau,\]
\[ \ov{\sigma}^{n,n+1}_{t,k}(x) = \int_0^t \int \Bigl(
F[h^{n+1}_\tau]\cdot\nabla_v
\bigl(h^k_\tau\circ\Om^{k-1}_{t,\tau}\bigr) \Bigr)
\bigl(\tau,x-v(t-\tau),v\bigr)\,dv\,d\tau.\] Note that their zero mode
vanishes.  For any $k\geq 1$, we apply Theorem \ref{thmcombi'} (with
$M=1$) to get
\begin{multline*}
  \bigl\|\ov{\sigma}^{n,n+1}_{t,k}\bigr\|_{\cF^{\lambda_n^*t+\mu_n^*}}
  \leq \int_0^t K_1^{n,h^k}(t,\tau)\, 
  \bigl\|F[h^{n+1}_\tau]\bigr\|_{\cF^{\nu'_n,\gamma}}\,d\tau\\
  + \int_0^t K_0^{n,h^k}(t,\tau)\,
  \bigl\|F[h^{n+1}_\tau]\bigr\|_{\cF^{\nu'_n,\gamma}}\,d\tau,
\end{multline*}
where

\noindent \bul $\dps\nu'_n = \lambda_n^* (1+b) \left|\tau-\frac{bt}{1+b}\right| + \mu'_n$ 
\sm

\noindent \bul $\dps K_1^{n,h^k}(t,\tau) =
\sup_{0\leq\tau\leq t} \left( \frac{\Bigl\|\nabla_v \bigl(h^k_\tau \circ\Om^{k-1}_{t,\tau}\bigr)
- \bigl\<\nabla_v \bigl(h^k_\tau \circ\Om^{k-1}_{t,\tau}\bigr)\bigr\> \Bigr\|
_{\cZ^{\lambda_k(1+b),\mu_k}_{\tau-bt/(1+b)}}}{1+\tau}\right)\
K_1^{n,k}$,
\sm

\noindent \bul $\dps K_1^{n,k}(t,\tau) =
(1+\tau)\ \sup_{\ell\neq 0,\ m\neq 0}
e^{-2\pi \left(\frac{\mu_k-\mu_n^*}2\right)\,|m|}\
\left(\frac{e^{-2\pi (\mu'_n-\mu_n^*)|\ell-m|}}{1+|\ell-m|^\gamma}\right) \
e^{-2\pi \left(\frac{\lambda_k-\lambda_n^*}2\right)\,|\ell(t-\tau)+m\tau|}$,
\sm

\noindent \bul $\dps K_0^{n,h^k}(t,\tau) = 
\left(\sup_{0\leq \tau\leq t} \Bigl\|\nabla_v \bigl\<h^k_\tau\circ\Om^{k-1}_{t,\tau}\bigr\>\Bigr\|_{\cC^{\lambda_k(1+b);1}}
\right)\ K_0^{n,k}$,
\sm

\noindent \bul $K_0^{n,k} (t,\tau) = e^{-2\pi \left(\frac{\lambda_k-\lambda_n^*}2\right)(t-\tau)}$.
\med

We assume
\begeq\label{mu'n}
\mu'_n = \mu^*_n + \eta\, \left(\frac{t-\tau}{1+t}\right),\qquad \eta>0\quad \text{small},
\endeq
and check that $\nu'_n\leq \lambda_n^*\tau+\mu_n^*$. Leaving apart the
small-time case, we assume $\tau\geq bt/(1+b)$, so that
\[ \nu'_n = \bigl(\lambda_n^*\tau+\mu_n^*\bigr)
- \frac{B\,\lambda_n^*\,(t-\tau)}{1+t} + \eta\left(\frac{t-\tau}{1+t}\right),\]
which is indeed bounded above by $\lambda_n^*\tau+\mu_n^*$ as soon as
\begeq\label{etaBn}
\eta\leq B\,\lambda_\infty^*.
\endeq
Then, with the notation \eqref{hKa},
\begeq
K_1^{n,k}(t,\tau) \leq K_1^{(\alpha_{n,k}),\gamma}(t,\tau),
\endeq
with
\begeq\label{alphank}
\alpha_{n,k} = 2 \pi \, \min\left\{\frac{\mu_k-\mu_n^*}2 \, ; \,
\frac{\lambda_k-\lambda_n^*}2 \, ; \,\eta\right\}.
\endeq
From the controls on $h^k$ (assumption ${\bf (\tilde E ^n _h)}$) we have
\begin{align*}
  \Bigl\|\nabla_v \bigl(h^k_\tau \circ\Om^{k-1}_{t,\tau}\bigr) -
  \bigl\<\nabla_v \bigl(h^k_\tau \circ\Om^{k-1}_{t,\tau}\bigr)\bigr\>
  \Bigr\|_{\cZ^{\lambda_k(1+b),\mu_k;1}_{\tau-\frac{bt}{1+b}}}
  & \leq \Bigl\|\nabla_v \bigl(h^k_\tau\circ\Om^{k-1}_{t,\tau}\bigr)\Bigr\|_{\cZ^{\lambda_k(1+b),\mu_k;1}_{\tau-\frac{bt}{1+b}}}\\
  & \leq \delta_k\,(1+\tau);
\end{align*}
and
\begin{align*}
\Bigl\| \bigl\<\nabla_v \bigl(h^k_\tau \circ\Om^{k-1}_{t,\tau}\bigr)\bigr\> \Bigr\|_{\cC^{\lambda_k(1+b);1}}
& = \Bigl\| \bigl\< (\nabla_v +\tau\nabla_x)\bigl(h^k_\tau\circ\Om^{k-1}_{t,\tau}\bigr)\bigr\>\Bigr\|
_{\cC^{\lambda_k(1+b);1}} \\
& \leq \Bigl\| (\nabla_v +\tau\nabla_x)\bigl(h^k_\tau\circ\Om^{k-1}_{t,\tau}\bigr)\Bigr\|
_{\cZ^{\lambda_k(1+b);1}_{\tau-\frac{bt}{1+b}}} \\
& \leq \delta_k.
\end{align*}
After controlling $F[h^{n+1}]$ by $\rho[h^{n+1}]$, we end up with
\begin{multline}
\bigl\|\ov{\sigma}^{n,n+1}_{t,k}\bigr\|_{\cF^{\lambda_n^* t + \mu_n^*}}
\leq C_F \int_0^t \left(\sum_{k=1}^n \delta_k\,K_1^{(\alpha_{n,k}),\gamma}(t,\tau)\right)\,
\bigl\|\rho[h^{n+1}_\tau]\bigr\|_{\cF^{\lambda_n^*\tau+\mu_n^*}}\,d\tau\\
+ C_F \int_0^t \left( \sum_{k=1}^n \delta_k\, e^{-2\pi \left(\frac{\lambda_k-\lambda_n^*}2\right)(t-\tau)}\right)\,
\bigl\|\rho[h^{n+1}_\tau]\bigr\|_{\cF^{\lambda_n^*\tau+\mu_n^*}}\,d\tau,
\end{multline}
with $\alpha_{n,k}$ defined by \eqref{alphank}.
\med

\noindent{\em Substep c.}
Gathering all previous controls, we obtain the following integral inequality for $\varphi = \rho[h^{n+1}]$:
\begin{multline}\label{multphi}
\left\| \varphi(t,x) - \int_0^t \int (\nabla W\ast\varphi)(\tau,x-v(t-\tau))\cdot\nabla_v f^0(v)\,dv\,d\tau \right\|
_{\cF^{\lambda_n^*t+\mu_n^*}}\\
\leq A_n + \int_0^t 
\Biggl[K_1^n(t,\tau)+K_0^n(t,\tau) 
+ \frac{c_0^n}{(1+\tau)^2}\Biggr]\, \|\varphi(\tau,\,\cdot\,)\|
_{\cF^{\lambda_n^*\tau+\mu_n^*}},
\end{multline}
where, by \eqref{Rntaut}, \eqref{controlcalE} and \eqref{controlovcalE},
\begeq\label{An}
A_n = \sup_{t\geq 0}\ \bigl\|\sigma^{n,n}(t,\,\cdot\,)\bigr\|_{\cF^{\lambda_n^*t+\mu_n^*}}
\leq \frac{2\,C_F\,\delta_n^2}{(\pi(\lambda_n-\lambda_n^*))^2},
\endeq
\[ K_1^n(t,\tau) = \left( C_F\, \sum_{k=1}^n \delta_k\right)\
K_1^{(\alpha_n),\gamma},\qquad \alpha_n = \alpha_{n,n} =
2\pi\,\min\left\{\frac{\mu_n-\mu_n^*}2 \, ;\, \frac{
    \lambda_n-\lambda_n^*}2\, ; \,\eta\right\},\] 
\[ K_0^n(t,\tau) = C_F\, \sum_{k=1}^n \delta_k\,e^{-2\pi
  \left(\frac{\lambda_k-\lambda_n^*}2\right)(t-\tau)},\]
\[ c_0^n = 3 \, C_F\,C_\om^4\, \left(C'_0 + \sum_{k=1}^n
  \delta_k\right)\, \left(\sum_{k=1}^n
  \frac{\delta_k}{(2\pi(\lambda_k-\lambda_k^*))^6}\right) +
\sum_{k=1}^n\delta_k.\] 
(We are cheating a bit when writing
\eqref{multphi}, because in fact one should take into account small
times separately; but this does not cause any difficulty.)

We easily estimate $K_0^n$:
\[ \int_0^t K_0^n(t,\tau)\,d\tau \leq C_F\,
\sum_{k=1}^n \frac{\delta_k}{\pi(\lambda_k-\lambda_n^*)},\]
\[ \int_\tau^\infty K_0^n(t,\tau)\,dt \leq C_F\,
\sum_{k=1}^n \frac{\delta_k}{\pi(\lambda_k-\lambda_n^*)},\]
\[ \left(\int_0^t K_0^n(t,\tau)^2\,d\tau\right)^{1/2}
\leq C_F \sum_{k=1}^n \frac{\delta_k}
{\sqrt{2\pi(\lambda_k-\lambda_n^*)}}.\]

Let us assume that $\alpha_n$ is smaller than $\ov{\alpha}(\gamma)$
appearing in Theorem \ref{thmgrowth}, and that 
\begeq\label{C7} {\bf
  (C_7)}\qquad 3\, C_F\,C_\om^4\left(C'_0 + \sum_{k=1}^n\delta_k + 1
\right)\, \left(\sum_{k=1}^n
  \frac{\delta_k}{(2\pi(\lambda_k-\lambda_k^*))^6}\right) \leq \frac14,
\endeq
\begeq\label{C8} {\bf (C_8)}\qquad C_F\,
\sum_{k=1}^n
 \frac{\delta_k}{\sqrt{2\pi(\lambda_k-\lambda_k^*)}} \leq \frac12,
\endeq
\begeq\label{C9} {\bf (C_9)}\qquad 
C_F\, \sum_{k=1}^n
  \frac{\delta_k}{\pi(\lambda_k-\lambda_k^*)} \leq 
\max \left\{\frac14\, ; \, \chi\right\},
\endeq
(note that in these conditions we have strenghtened the inequalities
by replacing $\lambda_k - \lambda_n ^*$ by $\lambda_k - \lambda_k ^*$
where $\chi>0$ is also defined by Theorem \ref{thmgrowth}).  Applying
that theorem with $\lambda_0=\lambda$, $\lambda^*=\lambda_1$,
 we deduce that for any $\var\in
(0,\alpha_n)$ and $t\geq 0$, \begeq\label{rhon+1}
\|\rho_t^{n+1}\|_{\cF^{\lambda_n^*t+\mu_n^*}}\leq C \, A_n\, \frac{\left( 1+
  c_0 ^n \right)^2}{\sqrt{\var}} \, e^{C \, c_0 ^n} \, 
 \left(1+\frac{c_n}{\alpha_n\,\var}\right) \, e^{C\,T_{\var,n}} \,
e^{C\,c_n\,(1+T_{\var,n}^2)}\,e^{\var \, t},
\endeq 
where
\[ c_n = 2 \, C_F \, \left(\sum_{k=1}^n\delta_k \right)\]
and
\[ T_{\var,n} = C_\gamma \, \max \left\{ \left(\frac{c_n^2}{\alpha_n^5
      \, \var^{2+\gamma}} \right)^{\frac1{\gamma-1}} \, ; \,
  \left(\frac{c_n}{\alpha_n^2\,\var^{\gamma+\frac12}}\right)^{\frac1{\gamma-1}};\,
  \ \frac{(c_0 ^n)^{2/3}}{\var^{1/3}}\right\}.
\]
Pick up $\lambda_n^\dag < \lambda_n^*$ such that
$2\pi(\lambda_n^*-\lambda_n^\dag)\leq \alpha_n$, 
 and choose $\var = 2\pi (\lambda_n^* - \lambda_n^\dag)$; 
recalling that $\hat{\rho}^{n+1}(t,0)=0$, and that
our conditions imply an upper bound on $c_n$ and $c_0 ^n$,
we deduce the {\em uniform} control
\begin{align}\label{controlrhon+1}
  \|\rho^{n+1}_t\|_{\cF^{\lambda_n^\dag t + \mu_n^*}}
  & \leq e^{-2\pi (\lambda_n^*-\lambda_n^\dag)t}\, 
  \|\rho_t^{n+1}\|_{\cF^{\lambda_n^*t+\mu_n^*}} \\
  & \leq C\,
  A_n\,\left(1+\frac{1}{\alpha_n\,(\lambda_n^*-\lambda_n^\dag)^{3/2}}\right)
  \,e^{C\, T_n^2}, \nonumber
\end{align}
where
\begeq\label{Tn} T_n = C \,\left(\frac{1}{\alpha_n^5\,
    (\lambda_n^*-\lambda_n^\dag)^{2+\gamma}}\right)^{\frac1{\gamma-1}}.
\endeq

\subsubsection{Step 7: estimate on $F[h^{n+1}]$} 

As an immediate consequence of \eqref{hypiterforce} and
\eqref{controlrhon+1}, we have \begeq\label{csqforce} \sup_{t\geq 0}\
\bigl\|F[\rho^{n+1}_t]\bigr\|_{\cF^{\lambda_n^\dag t + \mu_n^*,
    \gamma}} \leq C
A_n\,\left(1+\frac{1}{\alpha_n\,(\lambda_n^*-\lambda_n^\dag)^{3/2}}\right)\,
e^{C \, T_n^2}.
\endeq


\subsubsection{Step 8: estimate of $h^{n+1}\circ\Om^n$}

In this step we shall use again the Vlasov equation. We rewrite
\eqref{xi} as
\[ h^{n+1}\bigl(\tau,X^n_{0,\tau}(x,v),V^n_{0,\tau}(x,v)\bigr) =
\int_0^\tau \Sigma^{n+1}\bigl(s, X^n_{0,s}(x,v),
V^n_{0,s}(x,v)\bigr)\,ds;\] 
but now we compose with
$(X^n_{t,0},V^n_{t,0})$, {\em where $t\geq \tau$ is arbitrary}.  This
gives
\[ h^{n+1}\bigl(\tau, X^n_{t,\tau}(x,v),V^n_{t,\tau}(x,v)\bigr)
= \int_0^\tau \Sigma^{n+1}\bigl(s, X^n_{t,s}(x,v),V^n_{t,s}(x,v)\bigr)\,ds.\]

Then for any $p\in [1,\ov{p}]$ and $\lambda_n^\flat<\lambda_n^\dag$,
using Propositions \ref{propTL} and \ref{propalgZ}, and the notation \eqref{RGH}, we get 
\begin{align*} 
& \bigl\|h^{n+1}_\tau \circ\Om^n_{t,\tau}\bigr\|_{\cZ^{(1+b)\lambda^\flat_n,\mu^* _n;p}_{\tau-\frac{bt}{1+b}}} 
= \Bigl\|h^{n+1}_\tau\circ \bigl(X^n_{t,\tau},V^n_{t,\tau}\bigr)\Bigr\|_{\cZ^{(1+b)\lambda^\flat _n,\mu^*_n;p}
    _{t-\frac{bt}{1+b}}} \\
  & \qquad \leq \int_0^\tau \Bigl\| \Sigma^{n+1}\bigl(s, X^n_{t,s},V^n_{t,s}\bigr)\Bigr\|
  _{\cZ^{(1+b)\lambda^\flat _n,\mu^* _n;p}_{t-\frac{bt}{1+b}}}\,ds
   = \int_0^\tau \bigl\|\Sigma^{n+1}(s,\Om^n_{t,s})\bigr\|_{\cZ^{(1+b)\lambda^\flat _n,\mu^* _n;p}
    _{s-\frac{bt}{1+b}}}\,ds \\
  & \qquad \leq \int_0^\tau \bigl\|R^{n+1}_{s,t}\bigr\|_{\cZ^{(1+b)\lambda^\flat _n,\mu^* _n}_{s-\frac{bt}{1+b}}}
  \|G^n_{s,t}\|_{\cZ^{(1+b)\lambda^\flat _n,\mu^* _n;p}_{s-\frac{bt}{1+b}}}\,ds \\
  & \qquad \qquad \qquad \qquad + \int_0^\tau \|R^n_{s,t}\|_{\cZ^{(1+b)\lambda^\flat _n, \mu^* _n}_{s-\frac{bt}{1+b}}} \|H^n_{s,t}\|_{\cZ^{(1+b)\lambda^\flat _n,\mu^* _n;p}_{s-\frac{bt}{1+b}}}\,ds. 
\end{align*}

Then (proceeding as in Step~6 to check that the exponents lie in the
appropriate range)
\begin{equation*}
\|R^{n+1}_{s,t}\|_{\cZ^{(1+b)\lambda_n^\flat,\mu_n^*}_{s-\frac{bt}{1+b}}}
\le C_F \, e^{-2\pi (\lambda_n^\dag-\lambda_n^\flat) s} \, \| \rho^{n+1}_s \|_{\cF^{\bar \nu_n(s)}}
\end{equation*}
and 
\begin{equation*}
\|R^{n}_{s,t}\|_{\cZ^{(1+b)\lambda^\flat,\mu^*_n}_{s-\frac{bt}{1+b}}}
\le C_F \, e^{-2\pi (\lambda_n^\dag-\lambda_n^\flat) s} \, \| \rho^{n}_s \|_{\cF^{\bar \nu_n(s)}}
\le C_F \, e^{-2\pi (\lambda_n^\dag-\lambda_n^\flat) s} \, \delta_n
\end{equation*}
with 
$$
\begin{cases}
\bar \nu_n (s,t):= \mu^\sharp \qquad\qquad \mbox { when } s \le bt/(1+b) \\[4mm]
\bar \nu_n (s,t):= \lambda_n^\dag \, s + \mu_n ^* \quad \mbox{ when } s \ge bt/(1+b).
\end{cases}
$$ 

On the other hand, from the induction assumption ${\bf (E^n_h)}$-${\bf
  (\tilde{E}^n_h)}$ (and again control of composition {\it via} Proposition
\ref{propcompos2}\ldots),
$$
\|H^n_{s,t}\|_{\cZ^{(1+b)\lambda^\flat _n,\mu^* _n;p}_{s-\frac{bt}{1+b}}} \le 2 \, (1+s) \, \delta_n
$$
and 
$$
\|G^n_{s,t}\|_{\cZ^{(1+b)\lambda^\flat _n,\mu^* _n;p}_{s-\frac{bt}{1+b}}} 
\le 2 \, (1+s) \, \left( \sum_{k=1} ^n \delta_k \right).
$$

We deduce that
$$
y(t,\tau) := \bigl\|h^{n+1}_\tau \circ\Om^n_{t,\tau}\bigr\|_{\cZ^{(1+b)\lambda^\flat_n,\mu^* _n;p}_{\tau-\frac{bt}{1+b}}} 
$$
satisfies
\begin{multline*}
  y(t,\tau) \le 2 \, C_F \, \left( \sum_{k=1} ^n \delta_k \right) \,
  \int_0 ^\tau e^{-2\pi (\lambda_n^\dag-\lambda_n^\flat) s} \, \|
  \rho^{n+1}_s \|_{\cF^{\bar \nu_n(s)}} \, (1+s) \, ds \\ + 2 \, C_F
  \, \delta_n ^2 \, \int_0 ^\tau e^{-2\pi
    (\lambda_n^\dag-\lambda_n^\flat) s} \, (1+s) \, ds;
\end{multline*}
so
\begin{multline}\label{foralltau} 
  \forall \ t\geq \tau \ge 0, \\
  \bigl\|h^{n+1}_\tau
  \circ\Om^n_{t,\tau}\bigr\|_{\cZ^{(1+b)\lambda^\flat_n,\mu^*
      _n;p}_{\tau-\frac{bt}{1+b}}} \le \frac{4 \, C_F \, \max
    \left\{\left( \sum_{k=1} ^n \delta_k \right) \, ; 1
    \right\}}{(2\pi \bigl(\lambda_n^\dag-\lambda_n^\flat)\bigr)^2} \,
  \left( \delta_n ^2 + \sup_{s \ge 0} \| \rho^{n+1}_s \|_{\cF^{\bar
        \nu_n(s)}} \right).
\end{multline}

\subsubsection{Step 9: Crude estimates on the derivatives of $h^{n+1}$}

Again we choose $p\in [1,\ov{p}]$.
From the previous step and Proposition~\ref{propgrad2} we deduce, for any $\lambda_n^\ddag$ such that
$\lambda_n^\ddag<\lambda_n^\flat<\lambda_n^\ddag$, and any
$\mu_n^\ddag<\mu_n^*$,
\begin{multline} \label{derivh}
\Bigl\|\nabla_x \bigl(h^{n+1}_\tau\circ\Om^n_{t,\tau}\bigr)\Bigr\|
_{\cZ^{\lambda_n^\ddag(1+b),\mu_n^\ddag;p}_{\tau-\frac{bt}{1+b}}} 
+ \Bigl\|(\nabla_v+\tau\nabla_x) 
       \bigl(h^{n+1}_\tau\circ\Om^n_{t,\tau}\bigr)\Bigr\|
_{\cZ^{\lambda_n^\ddag(1+b),\mu_n^\ddag;p}_{\tau-\frac{bt}{1+b}}} \\
\leq \frac{C(d)}{\min\,\{\lambda_n^\flat-\lambda_n^\ddag\, ; \,  
\mu_n^*-\mu_n^\ddag\}}\,
\bigl\|h^{n+1}_\tau\circ\Om^n_{t,\tau}\bigr\|_{\cZ^{\lambda_n^\flat(1+b),\mu_n^*;p}_{\tau-\frac{bt}{1+b}}};
\end{multline}
and
\begeq\label{derivvh}
\Bigl\|\nabla \bigl(h^{n+1}_\tau\circ\Om^n_{t,\tau}\bigr)\Bigr\|
_{\cZ^{\lambda_n^\ddag(1+b),\mu_n^\ddag;p}_{\tau-\frac{bt}{1+b}}} 
\leq \frac{C(d)\, (1+\tau)}{\min\,\{\lambda_n^\flat-\lambda_n^\ddag
  \,; \, \mu_n^*-\mu_n^\ddag\}}\,
\bigl\|h^{n+1}_\tau\circ\Om^n_{t,\tau}\bigr\|_{\cZ^{\lambda_n^\flat(1+b),
\mu_n^*;p}_{\tau-\frac{bt}{1+b}}}.
\endeq

Similarly,
\begeq\label{derivxh}
\Bigl\|\nabla \nabla \bigl(h^{n+1}_\tau\circ\Om^n_{t,\tau}\bigr)\Bigr\|
_{\cZ^{\lambda_n^\ddag(1+b),\mu_n^\ddag;p}_{\tau-\frac{bt}{1+b}}} 
\leq \frac{C(d)\,(1+\tau)^2}{\min\,\{\lambda_n^\flat-\lambda_n^\ddag
  \, ; \, \mu_n^* -\mu_n^\ddag\}^2}\,
\bigl\|h^{n+1}_\tau\circ\Om^n_{t,\tau}\bigr\|_{\cZ^{\lambda_n^\flat(1+b),
\mu_n^*;p}_{\tau-\frac{bt}{1+b}}}.
\endeq

\subsubsection{Step 10: Chain-rule and refined estimates on
  derivatives of $h^{n+1}$}

From Step~3 we have 
\begeq\label{nOm-1}
\bigl\|\nabla\Om^n_{t,\tau}\bigr\|_{\cZ^{\lambda_n^* (1+b),\mu_n^*}
_{\tau-\frac{bt}{1+b}}}
+
\bigl\|(\nabla\Om^n_{t,\tau})^{-1}\bigr\|_{\cZ^{\lambda_n^*(1+b),\mu_n^*}
_{\tau-\frac{bt}{1+b}}}
\leq C(d)
\endeq
and ({\it via} Proposition \ref{propgrad2})
\begin{align}\label{nnOmn}
\Bigl\|\nabla \nabla \Om^n_{t,\tau}\Bigr\|_{\cZ^{\lambda_n^\ddag(1+b),\mu_n^\ddag}_{\tau-\frac{bt}{1+b}}}
& \leq \frac{C(d)\,(1+\tau)}{\min\,\{\lambda_n^*-\lambda_n^\ddag \,
  ;\, \mu_n^*-\mu_n^\ddag\}}\,
\bigl\|\nabla\Om^n_{t,\tau}\bigr\|_{\cZ^{\lambda_n^*(1+b),\mu_n^*}_{\tau-\frac{bt}{1+b}}}\\
& \leq \frac{C(d)\, (1+\tau)}{\min\,\{\lambda_n^*-\lambda_n^\ddag\, ;
  \, \mu_n^*-\mu_n^\ddag\}}. \nonumber
\end{align}
Combining these bounds with Step~9, Proposition \ref{propalgZ} and the
identities \begeq\label{chainruleOm}
\begin{cases}
(\nabla h)\circ\Om = (\nabla\Om)^{-1}\, \nabla(h\circ\Om)\\[2mm]
(\nabla\nabla h)\circ\Om = (\nabla\Om)^{-2}\, \nabla\nabla(h\circ\Om) - 
(\nabla\Om)^{-1}\,\nabla^2\Om\, (\nabla\Om)^{-1} (\nabla h\circ\Om),
\end{cases}\endeq
we get
\begin{align} \label{nablahO}
\Bigl\|(\nabla h^{n+1}_\tau)\circ\Om^n_{t,\tau}\Bigr\|_{\cZ^{\lambda_n^\ddag(1+b),\mu_n^\ddag;1}_{\tau-\frac{bt}{1+b}}}
& \leq C(d)\ \Bigl\|\nabla \bigl(h^{n+1}_\tau\circ\Om^n_{t,\tau}\bigr)\Bigr\|
_{\cZ^{\lambda_n^\ddag(1+b),\mu_n^\ddag;1}_{\tau-\frac{bt}{1+b}}}\\
& \leq \frac{C(d)\, (1+\tau)}{\min\{\lambda_n^\flat-\lambda_n^\ddag\,
  ; \,\mu_n^*-\mu_n^\ddag\}} \,
\Bigl\|h^{n+1}_\tau\circ\Om^n_{t,\tau} \Bigr\|
_{\cZ^{\lambda_n^\flat(1+b),\mu_n^*;1}_{\tau-\frac{bt}{1+b}}}
\nonumber
\end{align}
and
\begin{align} \label{Bignabla2}
\Bigl\|& (\nabla^2 h^{n+1}_\tau)\circ\Om^n_{t,\tau}\Bigr\|
_{\cZ^{\lambda_n^\ddag(1+b),\mu_n^\ddag;1}_{\tau-\frac{bt}{1+b}}} \\
& \leq
C(d)\ \left[ \Bigl\|\nabla^2 \bigl(h^{n+1}_\tau \circ\Om^n_{t,\tau}\bigr)\Bigr\|
_{\cZ^{\lambda_n^\ddag(1+b),\mu_n^\ddag;1}_{\tau-\frac{bt}{1+b}}}
+ \bigl\|\nabla^2 \Om^n_{t,\tau}\bigr\|_{\cZ^{\lambda_n^\ddag(1+b),\mu_n^\ddag;1}_{\tau-\frac{bt}{1+b}}}\
\bigl \|(\nabla
h^{n+1}_\tau)\circ\Om^n_{t,\tau}\bigr\|_{\cZ^{\lambda_n^\ddag(1+b),\mu_n^\ddag;1}_{\tau-\frac{bt}{1+b}}} \right]
\nonumber \\
& \leq \frac{C(d)\, (1+\tau)^2}{\min\,\{
  \lambda_n^\flat-\lambda_n^\ddag \, ; \,\mu_n^*-\mu_n^\ddag\}^2}\
\bigl\|h^{n+1}_\tau\circ\Om^n_{t,\tau}\bigr\|_{\cZ^{\lambda_n^\flat(1+b),\mu_n^*;p}_{\tau-\frac{bt}{1+b}}}.
\nonumber
\end{align}

This gives us the bounds $\|(\nabla h^{n+1})\circ\Om^n\| = O(1+\tau)$,
$\|(\nabla^2 h^{n+1})\circ\Om^n\| = O((1+\tau)^2)$, which are optimal if one does not distinguish between the
$x$ and $v$ variables. We shall now refine these estimates. First we write
\[\nabla (h^{n+1}_\tau\circ\Om^n_{t,\tau}) - (\nabla h^{n+1}_\tau)\circ\Om^n_{t,\tau}
= \nabla (\Om^n_{t,\tau} - \Id)\cdot \bigl[(\nabla h^{n+1}_\tau)\circ\Om^n_{t,\tau}\bigr],\]
and we deduce ({\it via} Propositions \ref{propalgZ} and \ref{propgrad2})
\begin{align} \label{nabnab}
\Bigl\| & \nabla\bigl(h^{n+1}_\tau\circ\Om^n_{t,\tau}\bigr)  - (\nabla h^{n+1}_\tau)\circ\Om^n_{t,\tau}\Bigr\|
_{\cZ^{\lambda_n^\ddag(1+b),\mu_n^\ddag;p}_{\tau-\frac{bt}{1+b}}}\\
& \leq \bigl\|\nabla (\Om^n_{t,\tau}-\Id)\bigr\|_{\cZ^{\lambda_n^\ddag(1+b),\mu_n^\ddag}
_{\tau-\frac{bt}{1+b}}}\
\bigl\|(\nabla h^{n+1}_\tau)\circ\Om^n_{t,\tau}\bigr\|_{\cZ^{\lambda_n^\ddag(1+b),\mu_n^\ddag;p}_{\tau-\frac{bt}{1+b}}}
\nonumber \\
& \leq C(d)\
\left(\frac{1+\tau}{\min\,\{\lambda_n^\flat-\lambda_n^\ddag \,; \, \mu_n^*-\mu_n^\ddag\}}\right)^2\
\bigl\|\Om^n_{t,\tau}-\Id\bigr\|_{\cZ^{\lambda_n^\flat(1+b),\mu_n^*}_{\tau-\frac{bt}{1+b}}}\,
\bigl\|h^{n+1}_\tau \circ\Om^n_{t,\tau}\bigr\|_{\cZ^{\lambda_n^\flat(1+b),\mu_n^*;p}_{\tau-\frac{bt}{1+b}}}
\nonumber \\
& \leq \frac{C(d)\,C_\om^4}{\min\,\{\lambda_n^\flat-\lambda_n^\ddag\,;\, \mu_n^*-\mu_n^\ddag\}^2}\,
\left(\sum_{k=1}^n \frac{\delta_k}{(2\pi(\lambda_k-\lambda_k^*))^6}\right)\
(1+\tau)^{-2}\ \bigl\|h^{n+1}_\tau\circ\Om^n_{t,\tau}\bigr\|_{\cZ^{\lambda_n^\flat(1+b),\mu_n^*;p}_{\tau-\frac{bt}{1+b}}}. \nonumber
\end{align}
(Note: $\Om^n-\Id$ brings the time-decay, while $h^{n+1}$ brings the smallness.)
\sm

This shows that $(\nabla h^{n+1})\circ\Om^n\simeq \nabla (h^{n+1}\circ\Om^n)$ as $\tau\to\infty$. In view of Step~9,
this also implies the refined gradient estimates
\begin{multline} \label{refined} 
  \bigl\|(\nabla_x
  h^{n+1}_\tau)\circ\Om^n_{t,\tau}\bigr\|
  _{\cZ^{\lambda_n^\ddag,\mu_n^\ddag;p}_{\tau-\frac{bt}{1+b}}}
  + \Bigl\|\bigl((\nabla_v+\tau\nabla_x)
  h^{n+1}_\tau\bigr)\circ\Om^n_{t,\tau}\Bigr\|
  _{\cZ^{\lambda_n^\ddag,\mu_n^\ddag;p}_{\tau-\frac{bt}{1+b}}}\\
  \leq \ov{C}\,
  \bigl\|h^{n+1}_\tau\circ\Om^n_{t,\tau}\bigr\|
  _{\cZ^{\lambda_n^\flat(1+b),\mu_n^*;p}_{\tau-\frac{bt}{1+b}}},
\end{multline}
with
\[ \ov{C}= C(d)\ \left[ \frac{C_\om^4}{
    \min\,\{\lambda_n^\flat-\lambda_n^\ddag\,;\,
    \mu_n^*-\mu_n^\ddag\}^2}\, \left(\sum_{k=1}^n
    \frac{\delta_k}{(2\pi(\lambda_k-\lambda_k^*))^6}\right) +
  \frac1{\min\,\{\lambda_n^\flat-\lambda_n^\ddag\,; \,
    \mu_n^*-\mu_n^\ddag\}}\right].\]

\subsubsection{Conclusion}
Given $\lambda_{n+1}<\lambda_n^*$, $\mu_{n+1}<\mu_n^*$, we define
\[ \lambda_{n+1}=\lambda_n^\ddag, \qquad \mu_{n+1} = \mu_n^\ddag,\]
and we impose
\[ \lambda_n^*-\lambda_n^\dag = \lambda_n^\dag - \lambda_n^\flat =
\lambda_n^\flat-\lambda_n^\ddag = \frac{\lambda_n^*-\lambda_{n+1}}3,\]
\[ \mu_n^*-\mu_n^\ddag = \mu_n^*-\mu_{n+1}.\] 
Then from
\eqref{controlrhon+1}, \eqref{csqforce}, \eqref{foralltau},
\eqref{derivh}, \eqref{Bignabla2}, \eqref{nabnab} and \eqref{refined}
we see that ${\bf (E_\rho ^{n+1})}$, ${\bf (\tilde E_\rho ^{n+1})}$,
${\bf (E_h ^{n+1})}$, ${\bf (\tilde E_h ^{n+1})}$ have all been
established in the present subsection, with \begeq\label{dn+1}
\delta_{n+1} = \frac{C(d)\,C_F(1+C_F)\,(1+C_\om^4)\,e^{C\,T_n^2}}
{\min\, \{ \lambda_n^*-\lambda_{n+1} \, ; \, \mu_n^*-\mu_{n+1}\}^9}\
\max \left\{ \left( \sum_{k=1} ^n \delta_k \right) \, ; \, 1 \right\}
\ \left(1+ \sum_{k=1}^n
  \frac{\delta_k}{(2\pi(\lambda_k-\lambda_k^*))^6}\right)\,
\delta_n^2.
\endeq

\subsection{Convergence of the scheme} \label{subsec:cvg-scheme}

For any $n\geq 1$, we set \begeq\label{setLambda}
\lambda_n-\lambda_n^* = \lambda_n^*-\lambda_{n+1} =
\mu_n-\mu_n^*=\mu_n^*-\mu_{n+1} = \frac{\Lambda}{n^2},
\endeq
for some $\Lambda>0$. By choosing $\Lambda$ small enough, we can make
sure that the conditions $2\pi(\lambda_k-\lambda_k^*)<1$,
$2\pi(\mu_k-\mu_k^*)<1$ are satisfied for all $k$, as well as the
other smallness assumptions made throughout this section. Moreover, we
have $\lambda_k-\lambda_k^*\geq \Lambda/k^2$, so conditions ${\bf
  (C_1)}$ to ${\bf (C_9)}$ will be satisfied if
\[ \sum_{k=1}^n k^{12}\,\delta_k \leq \Lambda^6\,\omega,\qquad
\sum_{j=k+1}^n j^6\,\delta_j \leq \Lambda^3\,\omega \,
\left(\frac1{k^2}-\frac1{n^2}\right),\] for some small explicit
constant $\omega>0$, depending on the other constants appearing in the
problem. Both
conditions are satisfied if \begeq\label{condkdelta} \sum_{k=1}^\infty
k^{12}\,\delta_k \leq \Lambda^6\,\omega.
\endeq

Then from \eqref{Tn} we have $T_n\leq
C_\gamma\,(n^2/\Lambda)^{\frac{7+\gamma}{\gamma-1}}$,
so the induction relation on $\delta_n$ allows 
\begeq\label{deltan+1} \delta_1\leq C\,\delta,\qquad \delta_{n+1} =
C\,\left(\frac{n^2}{\Lambda}\right)^9\,
e^{C\,(n^2/\Lambda)^{\frac{14+2\gamma}{\gamma-1}}}\,
\delta_n^2.
\endeq 

To establish this relation we also assumed that $\delta_n$ is bounded
below by $C_F\,\zeta_n$, the error coming from the short-time
iteration; but this follows easily by construction, since the
constraints imposed on $\delta_n$ are much worse than those on
$\zeta_n$.  

Having fixed $\Lambda$, we will check that for $\delta$ small enough,
\eqref{deltan+1} implies both the fast convergence of
$(\delta_k)_{k\in\N}$, and the condition \eqref{condkdelta}, which
will justify {\it a posteriori} the derivation of
\eqref{deltan+1}. (An easy induction is enough to turn this into a
rigorous reasoning.)

For this we fix $a\in(1,a_I)$, $0 < z < z_I < 1$, 
and we check by induction
\begeq\label{Delta}
\forall \, n\geq 1,\qquad \delta_n \leq \Delta\,z^{a^n}.
\endeq
If $\Delta$ is given, \eqref{Delta} holds for $n=1$ as soon as
$\delta\leq (\Delta/C)\,z^a$.  Then, to go from stage $n$ to stage
$n+1$, we should check that
\[ \frac{C\,n^{18}}{\Lambda^9}\,
e^{Cn^{\frac{28+4\gamma}{\gamma-1}}/\Lambda^{\frac{14+2\gamma}{\gamma-1}}}\,
\Delta^2\,z^{2\,a^n}\leq \Delta\, z^{a^{n+1}};\] this is true if
\[ \frac1{\Delta} \geq \frac{C}{\Lambda^9}\ \sup_{n\in\N} \left(
  n^{19}\, e^{C\,n^{\frac{28+4\gamma}{\gamma-1}}/
    \Lambda^{\frac{14+2\gamma}{\gamma-1}}}\, z^{(2-a)\,a^n}\right).\]
Since $a<2$, the supremum on the right-hand side is finite, and we
just have to choose $\Delta$ small enough.  Then, reducing $\Delta$
further if necessary, we can ensure \eqref{condkdelta}. This concludes
the proof.

\begin{Rk} This argument almost fully exploits the bi-exponential
  convergence of the Netwon scheme: a convergence like, say,
  $O(e^{-n^{1000}})$, would not be enough to treat values of $\gamma$
  which are close to~1.  In Subsection \ref{sub:itermode} we shall
  present a more cumbersome approach which is less greedy in the
  convergence rate, but still needs convergence like
  $O(e^{-n^\alpha})$ for $\alpha$ large enough.
\end{Rk}

\section{Coulomb/Newton interaction}
\label{sec:coulomb}

In this section we modify the scheme of Section \ref{sec:iter} to
treat the case $\gamma=1$.  We provide two different strategies. The
first one is quite simple and will only come close to treat this case,
since it will hold on (nearly) {\em exponentially large times} in the
inverse of the perturbation size. The second one, somewhat more
involved, will hold up to infinite times. 

\subsection{Estimates on exponentially large times} \label{sub:explarge}

In this subsection we adapt the estimates of Section \ref{sec:iter} to
the case $\gamma=1$, under the additional restriction that $0\leq
t\leq A^{1/(\delta (\log\delta)^2)}$ for some constant $A>1$.  

In the iterative scheme, the only place where we used $\gamma>1$ (and
not just $\gamma\geq 1$) is in Step~6, when it comes to the echo
response {\it via} Theorem \ref{thmgrowth}. Now, in the case
$\gamma=1$, the formula for $K_1^n$ should be
\[ K_1^n(t,\tau) = \sum_{k=1}^n \delta_k\,
K_1^{(\alpha_{n,k}),1}(t,\tau),\] with $\alpha_{n,k} = 2 \pi \,
\min\,\{(\mu_k-\mu_n^*)/2 \, ; \, (\lambda_k-\lambda_n^*)/2\, ; \,
\eta\}$. By Theorem \ref{thmgrowth} (ii) this induces, in addition to
other well-behaved factors, an uncontrolled exponential growth
$O(e^{\epsilon_n t})$, with
  \[ \epsilon_n = \Gamma\, \sum_{k=1}^n
  \frac{\delta_k}{\alpha^3_{n,k}};\] in particular $\epsilon_n$ will
  remain bounded and $O(\delta)$ throughout the scheme.

Let us replace \eqref{setLambda} by
\[ \lambda_n-\lambda_n^* = \lambda_n^*-\lambda_{n+1} =
\mu_n-\mu_n^*=\mu_n^*-\mu_{n+1} = \frac{\Lambda}{n\,(\log(e+n))^2},\]
where $\Lambda>0$ is very small. (This is allowed since the series
$\sum 1/(n(\log(e+n))^2)$ converges --- the power~2 could of course be replaced by
any $r>1$.)  Then during the first stages of
the iteration we can absorb the $O(e^{\epsilon_n t})$ factor by the
loss of regularity if, say,
\[ \epsilon_n \leq \frac{\Lambda}{2\,n\,(\log(e+n))^2}.\] Recalling
that $\epsilon_n=O(\delta)$, this is satisfied as soon as
\begeq\label{Nlogd} n \leq N :=
\frac{K}{\delta\,(\log(1/\delta))^2}, \endeq where $K>0$ is a positive
constant depending on the other parameters of the problem but of
course not on $\delta$.  So during these first stages we get the same
long-time estimates as in Section \ref{sec:iter}.

For $n>N$ we cannot rely on the loss of regularity any longer; at this
stage the error is about
\[ \delta_N \leq C\,\delta^{a^N},\] where $1<a<2$.  To get the bounds
for larger values of $n$, we use impose a restriction on the
time-interval, say $0\leq t\leq \Tm$. Allowing a degradation of the
rate $\delta^{a^n}$ into $\delta^{\und{a}^n}$ with $\und{a}<a$, we see
that the new factor $e^{\epsilon_n \Tm}$ can be eaten up by the scheme
if
\[ e^{\epsilon_n\Tm}\, \delta^{(a-\und{a})\,\und{a}^n}\leq 1,\qquad
\forall \, n\geq N.\]
This is satisfied if
\[ \Tm = O\left( \und{a}^N\,\frac{\log\frac1{\delta}}{\delta}\right).\]

Recalling \eqref{Nlogd}, we see that the latter condition holds true if
\[ \Tm= O \left(A^{\frac1{\delta(\log\delta)^2}}\,
  \frac{\log\frac1{\delta}}{\delta}\right)\] for some well-chosen
constant $A>1$.  Then we can complete the iteration, and end up with a
bound like
\[ \|f_t-f_i\|_{\cZ^{\lambda',\mu'}_t} \leq C\,\delta\qquad \forall \,
t\in [0,\Tm],\] where $C$ is another constant independent of
$\delta$. The conclusion follows easily.

\subsection{Mode-by-mode estimates} \label{sub:itermode}

Now we shall change the estimates of Section \ref{sec:iter} a bit more
in depth to treat arbitrarily large times for $\gamma=1$. The main
idea is to work {\em mode by mode} in the estimate of the spatial
density, instead of looking directly for norm estimates.

Steps~1 to~5 remain the same, and the changes mainly occur in Step~6.

Substep~6(a) is unchanged, but we only retain from that substep
\begeq\label{retainsuba} \forall \, \ell\in\Z^d,\qquad
e^{2\pi(\lambda_n^*t+\mu_n^*)|\ell|}\, \bigl| (\sigma_t^{n,n})^{\hat{
  }}(\ell)\bigr| \leq
\frac{2\,C_F\,\delta_n^2}{(\pi(\lambda_n-\lambda_n^*))^2}.
\endeq

Substep b is more deeply changed. Let $\hat{\mu}_n<\mu_n^*$.  \sm

\bul First, for each $\ell\in\Z^d$, we have, by Propositions
\ref{prop424}, \ref{prop425} and the last part of Proposition
\ref{proprt},
\begin{align*}
  e^{2\pi(\lambda_n^*t+\hat{\mu}_n)|\ell|}\,& \bigl|\hat{\cal E}(t,\ell)\bigr| \\[1mm]
  & \leq \int_0^t \sum_{m\in\Z^d} \Bigl\| P_m \Bigl(
  F[h^{n+1}_\tau]\circ \Om^n_{t,\tau} - F[h^{n+1}_\tau]\Bigr) \Bigr\|
  _{\cZ^{\lambda_n^*(1+b),\hat{\mu}_n}_{\tau-\frac{bt}{1+b}}}\\
  & \qquad\qquad\qquad\qquad\qquad\qquad\qquad
  \bigl\|P_{\ell-m}G^n_{\tau,t}\bigr\|
  _{\cZ^{\lambda_n^*(1+b),\hat{\mu}_n;1}_{\tau-\frac{bt}{1+b}}}\,d\tau\\[1mm]
  & \leq \int_0^t \sum_{m\in\Z^d} \sum_{m' \in \Z^d} \left\| P_{m-m'}
    \int_0^1 \nabla F[h^{n+1}_\tau]\circ \bigl(\Id + \theta
    (\Om^n_{t,\tau}-\Id)\bigr)\,d\theta \right\|
  _{\cZ^{\lambda_n^*(1+b),\hat{\mu}_n}_{\tau-\frac{bt}{1+b}}}\\
  & \qquad\qquad\qquad\qquad \left\| P_{m'} \left( \Omega^n _{t,\tau}
      - \Id \right)
  \right\|_{\cZ^{\lambda_n^*(1+b),\hat{\mu}_n}_{\tau-\frac{bt}{1+b}}}
  \,
  \bigl\|P_{\ell-m}G^n_{\tau,t}\bigr\|
  _{\cZ^{\lambda_n^*(1+b),\hat{\mu}_n;1}_{\tau-\frac{bt}{1+b}}}\,d\tau\\[1mm]
  & \leq \int_0^1 \int_0^t
  \|G^n_{\tau,t}\|_{\cZ^{\lambda_n^*(1+b),\mu_n^*;1}_{\tau-\frac{bt}{1+b}}}
  \, \left\| \Omega^n _{t,\tau}
      - \Id  \right\|_{\cZ^{\lambda_n^*(1+b), \mu^* _n}_{\tau-\frac{bt}{1+b}}}
    \,   \sum_{m,m' \in\Z^d} e^{-2\pi(\mu_n^*-\hat{\mu}_n)|\ell-m|} \\
  & \qquad\qquad  e^{-2\pi(\mu_n^*-\hat{\mu}_n)|m'|}\, 
  \Bigl\|P_{m-m'} \Bigl( \nabla F[h^{n+1}_\tau]\circ
  \bigl(\Id + \theta (\Om^n_{t,\tau}-\Id)\bigr)\Bigr)\Bigr\|
  _{\cZ^{\lambda_n^*(1+b),\hat{\mu}_n}_{\tau-\frac{bt}{1+b}}}\,d\tau\,d\theta\\[1mm]
  & \leq \int_0^t
  \|G^n\|_{\cZ^{\lambda_n^*(1+b),\mu_n^*;1}_{\tau-\frac{bt}{1+b}}} \, 
  \left\| \Omega^n _{t,\tau}
      - \Id
    \right\|_{\cZ^{\lambda_n^*(1+b),\mu^* _n}_{\tau-\frac{bt}{1+b}}}
    \,
  \sum_{m,\,m',\,q\,\in\Z^d} e^{-2\pi (\mu_n^*-\hat{\mu}_n)|\ell-m|}\,
  e^{-2\pi(\mu_n^*-\hat{\mu}_n)|m'|} \\ 
  &\qquad\qquad\qquad\qquad\qquad\qquad
  e^{-2\pi(\mu_n^*-\hat{\mu}_n)|m-m'-q|}\, \Bigl\|P_q
  \bigl( \nabla
  F[h^{n+1}_\tau]\bigr)\Bigr\|_{\cF^{\hat{\nu}_n}}\,d\tau,
\end{align*}
where
\[ \hat{\nu}_n = \hat{\mu}_n + \lambda_n^*(1+b)\left|\tau- \frac{bt}{1+b}\right|
+ \bigl\|\Om^nX_{t,\tau}-x \bigr\|_{\cZ^{\lambda_n^*(1+b),\mu_n^*}_{\tau-\frac{bt}{1+b}}}.\]

For $\alpha\leq 1$ we have
\[ \sum_{m,m'\in\Z^d} e^{-2\pi\alpha |\ell-m|}\,e^{-2\pi\alpha |m'|}
\, e^{-2\pi\alpha|m-m'-q|}\leq
\frac{C(d)}{\alpha^d}\,e^{-2\pi\frac{\alpha}{2}|\ell-q|},\]
we can argue as in Substep 6(b) of Section \ref{sec:iter} to get
\begin{multline}\label{newstepbE}
e^{2\pi(\lambda_n^*t+\mu_n^*)|\ell|}\,
\bigl|\hat{\cal E}(t,\ell)\bigr|
\leq \frac{C}{(\mu_n^*-\hat{\mu}_n)^d}\,
\left(C'_0 + \sum_{k=1}^n\delta_k\right)\
\left(\sum_{k=1}^n \frac{\delta_k}{(2\pi(\lambda_k-\lambda_k^*))^5}\right)\\
\sum_{q\in\Z^d} e^{-\pi (\mu_n^*-\hat{\mu}_n)|\ell-q|}
\int_0^t e^{2\pi(\lambda_n^*\tau+\hat{\mu}_n)|q|}\,
\bigl|\rho[h^{n+1}_\tau]^{\hat{ }}(q)\bigr|\,\frac{d\tau}{(1+\tau)^2}.
\end{multline}
\sm

\bul Next, we use again Proposition \ref{prop424} and simple estimates
to bound $\ov{\cal E}$:
\begin{multline*}
e^{2\pi(\lambda_n^*t+\hat{\mu}_n)|\ell|}\,
\bigl|{\ov{\cal E}\,}^{\hat{ }}(t,\ell)\bigr|\leq
\int_0^t \|G^n-\ov{G}^n\|_{\cZ^{\lambda_n^*(1+b),\mu_n^*;1}_{\tau-\frac{bt}{1+b}}}\\
\sum_{m\in\Z^d} e^{-2\pi(\mu_n^*-\hat{\mu}_n)|m|}\,
\bigl\|P_{\ell-m}\bigl(F[h^{n+1}_\tau]\bigr)\bigr\|_{\cF^{\hat{\beta}_n}}\,d\tau,
\end{multline*}
where $\hat{\beta}_n= \lambda_n^*(1+b)|\tau-bt/(1+b)|+\hat{\mu}_n$.
Reasoning as in Substep 6(b) of Section \ref{sec:iter}, we arrive at
\begin{multline}\label{newstepbEbar}
  e^{2\pi(\lambda_n^*t+\hat{\mu}_n)|\ell|}\, \bigl|{\ov{\cal
      E}\,}^{\hat{ }}(t,\ell)\bigr|\leq C \, \left(C'_0+\sum_{k=1}^n
    \delta_k\right)
  \left(\sum_{j=1}^n \frac{\delta_j}{(2\pi(\lambda_j-\lambda_n^*))^6} 
    + \sum_{k=1}^n\delta_k\right)\\
  \sum_{m\in\Z^d} e^{-2\pi(\mu_n^*-\hat{\mu}_n)|m|} \int_0^t
  e^{2\pi(\lambda_n^*\tau + \hat{\mu}_n)|\ell-m|}
  \bigl|\rho[h^{n+1}_\tau]^{\hat{
    }}(\ell-m)\bigr|\,\frac{d\tau}{(1+\tau)^2}.
\end{multline}
\sm

\bul Then we consider the ``main contribution'' $\ov{\sigma}^{n,n+1}$,
which we decompose as in Section~\ref{sec:iter}: 
\[ \ov{\sigma}^{n,n+1}_t = \ov{\sigma}^{n,n+1}_{t,0} + \sum_{k=1}^n
\ov{\sigma}_{t,k}^{n,n+1},\]
and we write for $k \ge 1$:
\begin{multline*}
  e^{2\pi(\lambda_n^*t+\hat{\mu}_n)}\,
  \bigl|\bigl(\ov{\sigma}^{n,n+1}_{t,k}\bigr)^{\hat{ }}(\ell)\bigr|
  \leq \sum_{m\in\Z^d} \int_0^t K_{\ell,m}^{n,h^k}(t,\tau)\,
  \bigl\|P_{\ell-m}\bigl(F[h^{n+1}_\tau]\bigr)\bigr\|_{\cF^{\nu'_n,\gamma}}\,d\tau\\
  + \int_0^t K_0^{n,h^k}(t,\tau)\, \bigl\|P_\ell
  \bigl(F[h^{n+1}_\tau]\bigr)\bigr\|_{\cF^{\nu'_n,\gamma}}\, d\tau,
\end{multline*}
where
\[ \nu'_n = \lambda_n^*(1+b)\left|\tau-\frac{bt}{1+b}\right| + \mu'_n,\]
\[ K_{\ell,m}^{n,h^k}(t,\tau) = \sup_{0\leq\tau\leq t}\
\sup_{0\leq\tau\leq t} \left( \frac{\Bigl\|\nabla_v \bigl(h^k_\tau
    \circ\Om^{k-1}_{t,\tau}\bigr) - \bigl\<\nabla_v \bigl(h^k_\tau
    \circ\Om^{k-1}_{t,\tau}\bigr)\bigr\> \Bigr\|
    _{\cZ^{\lambda_k(1+b),\mu_k}_{\tau-bt/(1+b)}}}{1+\tau}\right)\
K_{\ell,m}^{n,k},\]
\[ K_{\ell,m}^{n,k}(t,\tau) = (1+\tau)\, e^{-2\pi
  \left(\frac{\mu_k-\hat{\mu}_n}2\right)|m|}\, \left(\frac{e^{-2\pi
      (\mu'_n-\hat{\mu}_n)|\ell-m|}}{1+|\ell-m|^\gamma}\right)\,
e^{-2\pi
  \left(\frac{\lambda_k-\lambda_n^*}{2}\right)\,|\ell(t-\tau)+m\tau|},
\]
and the formula for $K_0^{n,h^k}$ is unchanged with respect to Section
\ref{sec:iter}.

Assuming $\mu'_n = \hat{\mu}_n + \eta\,(t-\tau)/(1+t)$ and reasoning as in 
Substep 6(b) of Section~\ref{sec:iter},
we end up with the following estimate on the ``main term'':
\begin{multline} \label{mainterm'}
e^{2\pi(\lambda_n^*t+\hat{\mu}_n)|\ell|}\, \bigl|\bigl(\ov{\sigma}_{t,k}^{n,n+1}\bigr)^{\hat{ }}
(\ell)\bigr| \\
\leq  C \sum_{m\in\Z^d} \int_0^t \left(\sum_{k=1}^n \delta_k\, K_{\ell,m}^{(\alpha_{n,k}),\gamma}
(t,\tau) \right)\ e^{2\pi(\lambda_n^*\tau+\hat{\mu}_n)|\ell-m|}\,
\bigl|\rho[h^{n+1}_\tau]^{\hat{ }}(\ell-m)\bigr|\,d\tau\\
+ C \sum_{m\in\Z^d} \int_0^t \left(\sum_{k=1}^n \delta_k
e^{-2\pi \left(\frac{\lambda_k-\lambda_n^*}{2}\right)(t-\tau)}\right)\,
e^{2\pi(\lambda_n^*\tau+\hat{\mu}_n)|\ell|}\,
\bigl|\rho[h^{n+1}_\tau]^{\hat{ }}(\ell-m)\bigr|\,d\tau.
\end{multline}
\sm

Then Substep 6(c) becomes, with $\Phi(\ell,t) = \rho[h^{n+1}_\tau]^{\hat{ }}(\ell)$,
\begin{multline*}
  e^{2\pi(\lambda_n^*t+\hat{\mu}_n)|\ell|} \left| \Phi(\ell,t) -
    \int_0^t K^0(\ell,t-\tau)\,\Phi(\ell,\tau)\,d\tau\right|
  \leq \frac{C\,\delta_n^2}{(\lambda_n-\lambda_n^*)^2}\\
  + \sum_{m\in\Z^d} \int_0^t \left[ K^n_{\ell,m}(t,\tau) +
    \frac{c^n_m}{(1+\tau)^2}\right]\,
  e^{2\pi(\lambda_n^*\tau+\hat{\mu}_n)|\ell-m|}\,|\Phi(\ell-m,\tau)|\,d\tau\\
  + \int_0^t
  K_0^n(t,\tau)\,e^{2\pi(\lambda_n^*\tau+\hat{\mu}_n)|\ell|}\,
  |\Phi(\ell,\tau)|\,d\tau,
\end{multline*}
with
\[ K^n_{\ell,m} = C\, \sum_{k=1}^n \delta_k\,
K^{(\hat{\alpha}_n),\gamma}_{\ell,m},\qquad \hat{\alpha}_n = 2 \pi \, \min\
\left\{ \frac{\mu_n-\hat{\mu}_n}{2};\ \frac{\lambda_n-\lambda_n^*}2 ;\
  \eta\right\},\]
\[ K_0^n(t,\tau) = C\, \sum_{k=1}^n \delta_k\,e^{-2\pi
  \left(\frac{\lambda_k-\lambda_n^*}{2}\right) (t-\tau)},\]
\[ c^n_m = \frac{C}{(\mu_n^*-\hat{\mu}_n)^d}\
\left(C'_0 + \sum_{k=1}^n \delta_k\right) \, 
\left(1+ \sum_{k=1}^n \frac{\delta_k}{(2\pi(\lambda_k-\lambda_k^*))^6}\right)\,
e^{-\pi (\mu_n^*-\hat{\mu}_n)|m|}.\]
Note that
\[ \sum_{m\in\Z^d} c_m^n + 
\left( \sum_{m\in\Z^d} \left(c_m^n\right)^2 \right)^{1/2} \leq
\frac{C}{(\mu_n^*-\hat{\mu}_n)^{2d}}
\left(C'_0 + \sum_{k=1}^n \delta_k\right)
\left(1+\sum_{k=1}^n \frac{\delta_k}{(2\pi(\lambda_k-\lambda_k^*))^6}\right).\]

Then we can apply Theorem \ref{thmgrowthk} and deduce (taking already into account, for
the sake of lisibility of the formula, that $\sum \delta_k$ and $\sum
\delta_k / (\lambda_k - \lambda_k ^*)$ are uniformly bounded)
\begin{multline*}
  e^{2\pi(\lambda_n^*t+\hat{\mu}_n)|\ell|}\,
  \bigl|(\rho_t^{n+1})^{\hat{ }} (\ell)\bigr|\leq C\,
  \frac{\delta_n^2}{(\lambda_n-\lambda_n^*)^2 \,
    \hat{\alpha}_n\,\var^{3/2} \, (\mu_n^*-\hat{\mu}_n)^{2d}} \,
  \exp\left(\frac{C\,(1+\hat{T}_{\var,n}^2)}{(\mu_n^*-\hat{\mu}_n)^{2d}}\right)
  \,e^{\var  t},
\end{multline*}
where
\begin{equation*} \hat{T}_{\var,n}  = C\, \max \left\{
    \left(\frac{1}{\alpha_n^{3+2d}\,\var^{\gamma+2}}\right)^{\frac1{\gamma}};\
    \left(\frac{1}{\alpha_n^d\,\var^{\gamma+\frac12}}\right)^{\frac1{\gamma-\frac12}};\
    \left(\frac{\left(\sum_m c^n _m\right)^2}{\var}\right)^{\frac13}\right\}.
\end{equation*} 

If $\lambda_n^\dag < \lambda_n^*$ and $\mu_n^\dag< \hat{\mu}_n$ 
are chosen as before and $\var = 2 \pi (\lambda_n ^* - \lambda^\dag)$, this implies a uniform bound on
\[ \|\rho_t^{n+1}\|_{\cF^{\lambda_n^\dag t + \mu_n^\dag}} \leq
\frac{C}{(\hat{\mu}_n-\mu_n^\dag)^d}\, \sup_{\ell\in\Z^d}\ e^{2\pi
  (\lambda_n^*t+\hat{\mu}_n)|\ell|}\,\bigl|(\rho_t^{n+1})^{\hat{
  }}(\ell)\bigr|\] 
obtained with the formula above with
\[
\hat{T}_{\var,n}=\hat{T}_n = C\, \max \left\{ \frac1{\lambda_n^*-\lambda_n^\dag},\
    \frac1{\lambda_n-\lambda_n^*},\ \frac1{\mu_n-\mu_n^*},\
    \frac1{\mu_n^*-\hat{\mu}_n}\right\}^{\max
    \left\{ \frac{5+\gamma+2d}{\gamma};\, \frac{d + \gamma
        +1/2}{\gamma-1/2}; \, \frac{4d+1}{3}\right\}}.
\]
\med

Then Steps~7 to~10 of the iteration can be repeated with the only
modification that $\mu_n^*$ is replaced by $\mu_n^\dag$.

\sm

The convergence (Subsection \ref{subsec:cvg-scheme}) works just the
same, except that now we need more intermediate regularity indices
$\mu_n$:
\[ \mu_{n+1}=\mu_n^{\ddag} < \mu_n^\dag < \hat{\mu}_n < \mu_n^*;\]
the obvious choice being to let $\mu_n^\dag - \mu_n^\ddag = 
\hat{\mu}_n - \mu_n^\dag = \mu_n^* - \hat{\mu}_n$.

Choosing $\lambda_n-\lambda_{n+1}$ and $\mu_n-\mu_{n+1}$ of the order
of $\Lambda/n^2$, we arrive in the end at the induction
\[ \delta_{n+1} \le C\,\left(\frac{n^2}{\Lambda}\right)^{9+6d}\, e^{C
  \, \left(\frac{n^2}{\Lambda}\right)^{\xi(d,\gamma)}} \,\delta_n^2\]
with 
\begin{equation}\label{eq:defxi}
\xi(d,\gamma) := 2d + 2 \max \left\{ \frac{5+\gamma+2d}{\gamma};\, \frac{d
    + \gamma +1/2}{\gamma-1/2}; \, \frac{4d+1}{3}\right\}.
\end{equation} 

Then the convergence of the scheme (and {\it a posteriori}
justification of all the assumptions) is done exactly as in
Section~\ref{sec:iter}.

\section{Convergence in large time}
\label{sec:convergence}

In this section we prove Theorem \ref{thmmain} as a simple consequence of the
uniform bounds established in Sections \ref{sec:iter} and
\ref{sec:coulomb}.

So let $f^0,L,W$ satisfy the assumptions of Theorem \ref{thmmain}. To
simplify notation we assume $L=1$.

The second part of Assumption \eqref{condanalf0} precisely means that
$f^0\in \cC^{\lambda;1}$. We shall actually assume a slightly more
precise condition, namely that for some $\ov{p}\in [1,\infty]$,
\begeq\label{condanalf0'} \sum_{n\in\N_0^d} \frac{\lambda^n}{n!}\,
\|\nabla_v^n f^0\|_{L^p(\R^d)} \leq C_0 <+\infty,\qquad \forall \, p
\in [1,\ov{p}].
\endeq
(It is sufficient to take $\ov{p}=1$ to get Theorem \ref{thmmain}; but
if this bound is available for some $\ov{p}>1$ then it will be
propagated by the iteration scheme, and will result in
more precise bounds.) Then we pick up $\und{\lambda}\in (0,\lambda)$,
$\und{\mu}\in(0,\mu)$, $\beta>0$, $\beta'\in(0,\beta)$.  By symmetry,
we only consider nonnegative times.

If $f_i$ is an initial datum satisfying the smallness condition
\eqref{fif0}, then by Theorem \ref{thminj}, we have a smallness
estimate on $\|f_i-f^0\|_{\cZ^{\lambda',\mu';p}}$ for all $p\in
[1,\ov{p}]$, $\lambda'<\lambda$, $\mu'<\mu$.  Then, as in Subsection
\ref{sub:revisited} we can estimate the solution $h^1$ to the
linearized equation \begeq\label{linVlencore}
\begin{cases}
\pa_t h^1 + v\cdot\nabla_x h^1 + F[h^1]\cdot\nabla_v f^0 =0\\[3mm]
h^1(0,\,\cdot\,) = f_i-f^0,
\end{cases}
\endeq
and we recover uniform bounds in $\cZ^{\hat{\lambda},\hat{\mu};p}$
spaces, for any $\hat{\lambda}\in (\und{\lambda},\lambda)$,
$\hat{\mu}\in(\und{\mu},\mu)$.  
More precisely, 
\begeq\label{borneh1}
\sup_{t\geq 0} \|\rho[h^1_t]\|_{\cF^{\hat{\lambda} t+\hat{\mu}}} +
\sup_{t\geq 0}
\|h^1(t,\,\cdot\,)\|_{\cZ^{\hat{\lambda},\hat{\mu};p}_t} \leq C\,
\delta,
\endeq
with $C=C(d,\lambda',\hat{\lambda},\mu',\hat{\mu},W,f^0)$ (this is
of course assuming $\var$ in Theorem \ref{thmmain} to be small
enough).  

We now set $\lambda_1 = \lambda'$,
and we run the iterative scheme of Sections
\ref{sec:local}--\ref{sec:iter}--\ref{sec:coulomb} for all $n\geq
2$. If $\var$ is small enough, up to slightly lowering $\lambda_1$, we
may choose all parameters in such a way that
\[ \lambda_k, \lambda_k^* \xrightarrow[k\to\infty]{}
\lambda_\infty>\und{\lambda}, \qquad \mu_k,\mu^*_k
\xrightarrow[k\to\infty]{} \mu_\infty > \und{\mu};\] then we pick up
$B>0$ such that
\[ \mu_\infty - \lambda_\infty (1+B)B \geq \mu'_\infty > \und{\mu},\]
and we let $b(t)=B/(1+t)$.

As a result of the scheme, we have, for all $k\geq 2$,
\begeq\label{k2} \sup_{t\geq \tau \geq 0}\: \bigl\| h^k_\tau \circ
\Om^{k-1}_{t,\tau}\bigr\|_{\cZ^{\lambda_\infty(1+b),\mu_\infty;1}_{\tau
    -\frac{bt}{1+b}}} \leq \delta_k,
\endeq
where $\sum_{k=2}^\infty \delta_k \leq C\,\delta$ and $\Om^k$ is the
scattering associated to the force field generated by
$h^1+\ldots+h^k$. Choosing $t=\tau$ in \eqref{k2} yields
\[ \sup_{t\geq 0}\: \|h_t^k\|_{\cZ^{\lambda_\infty(1+B),\mu_\infty;1}
  _{t-\frac{Bt}{1+B+t}}}\leq \delta_k.\] By Proposition
\ref{propincluZ}, this implies
\[ \sup_{t\geq 0}
\|h_t^k\|_{\cZ^{\lambda_\infty(1+B),\mu_\infty-\lambda_\infty(1+B)B;1}_t}\leq
\delta_k.\] In particular, we have a uniform estimate on $h_t^k$ in
$\cZ^{\lambda_\infty,\mu'_\infty;1}_t$.  Summing up over $k$ yields
for $f=f^0 + \sum_{k\geq 1}h^k$ the estimate \begeq\label{supft}
\sup_{t\geq 0}\:
\bigl\|f(t,\,\cdot\,) - f^0\bigr\|_{\cZ^{\lambda_\infty,\mu'_\infty;1}_t}\leq
C\,\delta.
\endeq

Passing to the limit in the Newton scheme, one shows that $f$ solves
the nonlinear Vlasov equation with initial datum $f_i$. (Once again we
do not check details; to be rigorous one would need to establish
moment estimates, locally in time, before passing to the limit.) This
implies in particular that $f$ stays nonnegative at all times.

Applying Theorem \ref{thminj} again, we deduce from \eqref{supft}
\[ \sup_{t\geq 0}\:
\bigl\|f(t,\,\cdot\,)-f^0\bigr\|_{\cY^{\und{\lambda},\und{\mu}}_t}\leq
C\,\delta;\] or equivalently, with the notation used in Theorem
\ref{thmmain}, \begeq\label{boundff0} \sup_{t\geq 0}\: \left\|
  f(t,x-vt,v) -f^0(v) \right\|_{\und{\lambda},\und{\mu}}\leq
C\,\delta.
\endeq
Moreover, $\rho =\int f\,dv$ satisfies similarly
\[ \sup_{t\geq 0}\: \|\rho(t,\,\cdot\,)\|_{\cF^{\lambda_\infty t +
    \mu_\infty}}\leq C\,\delta.\] It follows that $|\hat{\rho}(t,k)|
\leq C\,\delta\,e^{-2\pi\lambda_\infty|k|t}\,e^{-2\pi\mu_\infty|k|}$,
for any $k\neq 0$.  On the one hand, by Sobolev embedding, we deduce
that for any $r\in\N$,
\[ \|\rho(t,\,\cdot\,)-\<\rho\>\|_{C^r(\T^d)}\leq C_r\, \delta\,
e^{-2\pi \lambda't};\] 
on the other hand, multiplying $\hat{\rho}$ by the
Fourier transform of $\nabla W$, we see that the force $F=F[f]$
satisfies \begeq\label{Faussi} \forall \, t\geq 0,\quad \forall
\, k\in\Z^d,\qquad |\hat{F}(t,k)| \leq
C\,\delta\,e^{-2\pi\lambda'|k|t}\,e^{-2\pi\mu'|k|},\endeq for some
$\lambda'>\und{\lambda}$, $\mu'>\und{\mu}$.  

Now, from \eqref{boundff0} we have, for any $(k,\eta)\in\Z^d\times\R^d$, and any $t\geq 0$,
\begeq\label{fttilde}
\Bigl|\tilde{f}(t,k,\eta+kt) - \tilde{f}^0(\eta)\Bigr|\leq C\,\delta\,e^{-2\pi\mu'|k|}\,e^{-2\pi\lambda'|\eta|};
\endeq
so
\begeq\label{ftildeketa}
|\tilde{f}(t,k,\eta)| \leq \bigl|\tilde{f}^0(\eta+kt)\bigr| + C\,\delta\,e^{-2\pi\mu'|k|}\, e^{-2\pi\lambda'|\eta+kt|}.
\endeq
In particular, for any $k\neq 0$, and any $\eta\in\R^d$,
\begeq\label{tildefketa}
\tilde{f}(t,k,\eta) = O(e^{-2\pi\lambda't}).
\endeq

Thus $f$ is asymptotically close (in the weak topology) to its spatial average $g=\<f\>=\int f\,dx$.
Taking $k=0$ in \eqref{fttilde} shows that, for any $\eta\in\R^d$,
\begeq\label{tildegt} 
|\tilde{g}(t,\eta) - \tilde{f}^0(\eta)|\leq C\,\delta\,e^{-2\pi\lambda'|\eta|}.
\endeq

Also, from the nonlinear Vlasov equation, for any $\eta\in\R^d$ we have
\begin{align*}
\tilde{g}(t,\eta)
& = \tilde{f}_i (0,\eta)
- \int_0^t \int_{\T^d _L} \int_{\R^d} F(\tau,x)\cdot\nabla_v f (\tau,x,v)\,e^{-2i\pi \eta\cdot v}\,dv\,dx\,d\tau\\
& = \tilde{f}_i(0,\eta) - 2i\pi \sum_{\ell\in\Z^d} \int_0^t \hat{F}(\tau,\ell)\cdot\eta\,
\tilde{f}(\tau,-\ell,\eta)\,d\tau.
\end{align*}
Using the bounds \eqref{Faussi} and \eqref{tildefketa}, it is easily shown that the above time-integral converges
exponentially fast as $t\to\infty$, with rate $O(e^{-\lambda'' t})$ for any $\lambda''<\lambda'$, to its limit
\begeq\label{gtilde8} \tilde{g}_\infty(\eta) = 
\tilde{f}_i(0,\eta) - 2i\pi \sum_{\ell\in\Z^d} \int_0^\infty \hat{F}(\tau,\ell)\cdot\eta\,
\tilde{f}(\tau,-\ell,\eta)\,d\tau.
\endeq
By passing to the limit in \eqref{tildegt} we see that
\[ |\tilde{g}_\infty(\eta) - \tilde{f}^0(\eta)|\leq
C\,\delta\,e^{-2\pi\lambda'|\eta|},\] and this concludes the proof of
Theorem \ref{thmmain}.

\section{Non-analytic perturbations} \label{sec:NA}

Although the vast majority of studies of Landau damping assume that
the perturbation is analytic, it is natural to ask whether this
condition can be relaxed. As we noticed in Remark \ref{rklowreg}, this
is the case for the linear problem. As for nonlinear Landau damping,
once the analogy with KAM theory has been identified, it is anybody's
guess that the answer might come from a Moser-type argument. However,
this is not so simple, because the ``loss of convergence'' in our
argument is much more severe than the ``loss of regularity'' which
Moser's scheme allows to overcome. 

For instance, the second-order correction $h^2$ satisfies
\[ \pa_t h^2 + v\cdot\nabla_x h^2 + F[f^1]\cdot\nabla_v h^2 +
F[h^2]\cdot\nabla_v f^1 = - F[h^1]\cdot\nabla_v h^1.\] The action of
$F[f^1]$ is to curve trajectories, which does not help in our
estimates. Discarding this effect and solving by Duhamel's formula and
Fourier transform, we obtain, with $S=-F[h^1]\cdot\nabla_v h^1$,
$\rho^2 = \int h^2\,dv$,
\begin{multline}\label{2ndorder}
  \hat{\rho}^2(t,k) \simeq \int_0^t K^0(t-\tau,k)\,\hat{\rho}^2(\tau,k)\,d\tau\\
  + 2i\pi \int_0^t \sum_\ell (k-\ell) \,
  \hat{W}(k-\ell)\,\hat{\rho}^2(\tau,k-\ell)\, \tilde{\nabla_v
    h^1}\bigl(\tau,\ell,k(t-\tau)\bigr)\,d\tau + \int_0^t
  \tilde{S}\bigl(\tau,k,k(t-\tau)\bigr)\,d\tau.
\end{multline} 
(The term with $K^0$ includes the contribution of $\nabla_v f^0$.)

Our regularity/decay estimates on $h^1$ will never be better than
those on the solution of the free transport equation, {\em i.e.},
$h_i(x-vt,v)$, where $h_i=f_i-f^0$. Let us forget about the effect of
$K^0$ in \eqref{2ndorder}, replace the contribution of $S$ by a
decaying term $A(kt)$. Let us choose $d=1$ and assume
$\hat{h}_i(\ell,\,\cdot\,)=0$ if $\ell\neq \pm 1$. For $k>0$ let us
use the long-time approximation
\[ \tilde{h}_i(-1,k(t-\tau)-\tau)\,1_{[0,t]}(\tau)\,d\tau \simeq
\frac{c}{k+1}\,\delta_{\frac{kt}{k+1}},\qquad c = \int
\tilde{h}_i(-1,s)\,ds = \hat{h}_i(-1,0);\] note that $c\neq 0$ in
general. Plugging all these simplifications in \eqref{2ndorder} and
choosing $\hat{W}(k)=1/|k|^{1+\gamma}$ suggests the {\em a priori}
simpler equation \begeq\label{baby} \varphi(t,k) = A(kt) +
\frac{ckt}{(k+1)^{\gamma+2}}\ \varphi\left(\frac{kt}{k+1}, k+1\right).
\endeq
Replacing $\varphi(t,k)$ by $\varphi(t,k)/A(kt)$ reduces to $A=1$, and
then we can solve this equation by power series as in Subsection
\ref{subheurcoul}, obtaining \begeq\label{varphiA} \varphi(t,k) \simeq
A(kt)\, e^{(ckt)^{\frac1{\gamma+2}}}.
\endeq
With a polynomial deterioration of the rate we could use a
regularization argument, but the fractional exponential is much worse.

However, our bounds are still good enough to establish decay for
Gevrey perturbations.  Let us agree that a function $f=f(x,v)$ lies in
the Gevrey class ${\cal G}^\nu$, $\nu\geq 1$, if $|\tilde{f}(k,\eta)|=
O\bigl(\exp(-c|(k,\eta)|^{1/\nu})\bigr)$ for some $c>0$; in particular
${\cal G}^1$ means analytic.  (An alternative convention would be to
require the $n$th derivative to be $O(n!^\nu)$.)  As we shall explain,
we can still get nonlinear Landau damping if the initial datum $f_i$
lies in ${\cal G}^\nu$ for $\nu$ close enough to~1. Although this is
still quite demanding, it already shows that nonlinear Landau damping
is not tied to analyticity or quasi-analyticity, and holds for a large
class of compactly supported perturbations.

\begin{Thm} \label{thmNA} Let $\lambda >0$. Let $f^0=f^0(v)\geq 0$ be
  an analytic homogeneous profile such that
\[ \sum_{n\in\N_0^d} \frac{\lambda^n}{n!} \, \|\nabla_v^n
f^0\|_{L^1(\R^d)} < +\infty, \] and let $W=W(x)$ satisfy
$|\hat{W}(k)|=O(1/|k|)$, such that Condition {\bf (L)} from Subsection
\ref{sub:lineardamping} holds.  Let $\nu \in (1,1+\theta)$ with
$\theta= 1/\xi(d,\gamma)$, where $\xi$ was defined in
\eqref{eq:defxi}. Let $\beta>0$ and let $\alpha<1/\nu$.  Then
there is $\var>0$ such that if
\[ \delta:=\sup_{k,\eta}\ \Bigl( \bigl|(\tilde{f}_i - \tilde{f}^0)
(k,\eta)\bigr|\, e^{\lambda|\eta|^{1/\nu}}\,
e^{\lambda|k|^{1/\nu}}\Bigr) + \iint
\bigl|(f_i-f^0)(x,v)\bigr|\,e^{\beta|v|}\,dv\,dx \ \leq \var,\] then
as $t\to +\infty$ the solution $f=f(t,x,v)$ of the nonlinear Vlasov
equation on $\T^d\times\R^d$ with interaction potential $W$ and
initial datum $f_i$ satisfies, for all $r\in\N$,
\[ \forall \, (k,\eta),\qquad \Bigl|\tilde{f}(t,k,\eta) -
\tilde{f}_\infty(\eta)\Bigr| = O \bigl(\delta\,
e^{-ct^\alpha}\bigr);\]
\[ \|F(t,\,\cdot\,)\|_{C^r(\T^d)} = O \bigl(\delta\,e^{-c t^\alpha}\bigr)\]
for some $c>0$ and some homogeneous Gevrey profile $f_\infty$,
where $F$ stands for the self-consistent force.
\end{Thm} 

\begin{Rk} In view of \eqref{varphiA}, one may hope that the result
  remains true for $\theta=2$.  Proving this would require much more
  precise estimates, including among other things a qualitative improvement
  of the constants in Theorem \ref{thminj} (recall Remark \ref{rkMiro}).
\end{Rk}

\begin{Rk} One could also relax the analyticity of $f^0$, but there is
  little incentive to do so.
\end{Rk} 

\begin{proof}[Sketch of proof of Theorem \ref{thmNA}]
  We first decompose $h_i=f_i-f^0$, using truncation by a smooth
  partition of unity in Fourier space: 
  \[ h_i = \sum_{n\geq 0} \cF^{-1} \left( \tilde h_i\,\chi_n\right)
  \equiv \sum_{n\geq 0} h_i^n,\] 
where $\cF$ is the Fourier transform. Each bump function $\chi_n$ should be localized around the domain
(in Fourier space)
\[ D_n = \Bigl\{ n^K \leq |(k,\eta)| \leq (n+1)^K \Bigr\},\] for some
exponent $K>1$; but at the same time $\cF^{-1}(\chi_n)$ should be
exponentially decreasing in $v$.  To achieve this, we let
\[ \chi_n = 1_{D_n}\ast \gamma,\qquad \gamma(\eta) = e^{-\pi
  |\eta|^2}.\] Then $\cF^{-1}(\chi_n) = \cF^{-1}(1_{D_n})\,\gamma$ has
Gaussian decay, independently of $n$; so there is a uniform bound on
$\iint |h_i^n(x,v)|\,e^{\beta|v|}\,dv\,dx$, for some $\beta>0$.  

On the other hand, if $(k,\eta)\in D_n$ and $(k',\eta')\notin
(D_{n-1}\cup D_n\cup D_{n+1})$, then $|k-k'|+|\eta-\eta'|\geq
c\,n^{K-1}$ for some $c>0$; from this one obtains, after simple
computations,
\[ \bigl|\chi_n(k,\eta)\bigr| \leq 1_{(n-1)^K\leq |(k,\eta)|\leq (n+2)^K}\,
+ C\,e^{-c\,n^{2(K-1)}}\,e^{-c\,(|k|^2+|\eta|^2)}. \]
So (with constants $C$ and $c$ changing from line to line)
\begin{align*}
  \bigl|\tilde{h}_i^n(k,\eta)\bigr| & \leq C\,
  e^{-\lambda|k|^{\frac1{\nu}}}\,e^{-\lambda|\eta|^{\frac1{\nu}}}\,
  1_{(n-1)^K \leq |(k,\eta)|\leq (n+2)^K}\ + C\,e^{-c\,n^{2(K-1)}}\,e^{-c(|k|+|\eta|)} \\
  & \leq C\, \max \Bigl\{ e^{-\frac{\lambda}2\,(n-1)^{\frac{K}{\nu}}},\,
  e^{-c\,n^{2(K-1)}}\Bigr\}\, e^{-\ov{\lambda}_n (|k|+|\eta|)},
\end{align*}
where
\[ \bar \lambda_n \sim
\frac{\lambda}{2}\,(n+2)^{-\bigl(1-\frac1{\nu}\bigr)K}.\]

If $K\geq 2$ then $2(K-1)>K/\nu$; so
$\|h_i^n\|_{\cY^{\ov{\lambda}_n,\ov{\lambda}_n}}\leq C\,
e^{-\frac{\lambda}2\,n^{K/\nu}}$. Then we may apply Theorem
\ref{thminj} to get a bound on $h_i^n$ in the space
$\cZ^{\hat{\lambda}_n,\hat{\lambda}_n;1}$ with
$\hat{\lambda}_n=\ov{\lambda}_n/2$, at the price of a constant
$\exp(C\,(n+2)^{(1-1/\nu)K})$. Assuming $K\nu > (1-1/\nu)K$, {\em
  i.e.}, $\nu<2$, we end up with \begeq\label{endupGevrey}
\|h_i^n\|_{\cZ^{\hat{\lambda}_n,\hat{\lambda}_n};1} = O \bigl( e^{-c
  n^{K/\nu}}\bigr),\qquad \hat{\lambda}_n = \frac{\ov{\lambda}_n}{2}.
\endeq

Then we run the iteration scheme of Section \ref{sec:approx} with the
following modifications: (1) instead of $h^n(0,\,\cdot\,)=0$, choose
$h^n(0,\,\cdot\,)=h_i^n$, and (2) choose regularity indices $\lambda_n
\sim \hat \lambda_n$ which go to zero as $n$ goes to infinity. This
generates an additional error term of size
$O\bigl(\exp(-c\,n^{\frac{K}{\nu}})\bigr)$, and imposes that
$\lambda_n-\lambda_{n+1}$ be of order
$n^{-\bigl[\bigl(1-\frac1{\nu}\bigr)K+1\bigr]}$. When we apply the
bilinear estimates from Section \ref{sec:regdecay}, we can take
$\ov{\lambda}-\lambda$ to be of the same order; so $\alpha=\alpha_n$
and $\var=\var_n$ can be chosen proportional to
$n^{-\bigl[\bigl(1-\frac1{\nu}\bigr)K+1\bigr]}$. 
Then the large constants coming from the time-response will be, as in
Section \ref{sec:coulomb}, of order $n^q\,e^{cn^r}$, with $q\in\N$ and
$r=[(1-1/\nu)K+1]\xi$, and the scheme will still converge like
$O(e^{-cn^s})$ for any $s<K/\nu$, provided that $K/\nu>r$, {\em i.e.},
\[ (\nu-1) + \frac{\nu}{K} < \frac1{\xi}.\] 

The rest of the argument is similar to what we did in Sections
\ref{sec:iter} to \ref{sec:convergence}.  In the end the decay rate
of any nonzero mode of the spatial density $\rho$ is controlled by
\begin{multline*} \sum_n e^{-c n^s}\,e^{-\lambda_n t} \leq
  \left(\sum_n e^{-c n^s}\right) \ \sup_n \ \Bigl[ \exp(-cn^s)\, \exp
  \bigl(-c\, n^{-\bigl(1-\frac1{\nu}\bigr)}\, t\bigr)\Bigr] \\ \leq
  C\,\exp(-c\,t^{s/K}),
\end{multline*}
and the result follows since $s/K$ is arbitrarily close to $1/\nu$.
\end{proof} 

\begin{Rk} An alternative approach to Gevrey regularity consists in
  rewriting the whole proof with the help of Gevrey norms such as
  \[ \|f\|_{\cC^{\lambda}_\nu} = \sum_{n\in\N}
  \frac{\lambda^n\,\|f^{(n)}\|_\infty}{n!^{\nu}}, \qquad
  \|f\|_{\cF^{\lambda}_\nu} = \sum_{k\in\Z} e^{2\pi\lambda
    |k|^{1/\nu}}\, |\hat{f}(k)|,\] which satisfy the algebra property
  for any $\nu\geq 1$. Then one can hybridize these norms, rewrite the
  time-response in this setting, estimate fractional exponential
  moments of the kernel, etc. 
\end{Rk}

\begin{Rk} \label{rkCr} In a more general $C^r$ context, we do not
  know whether decay holds for the nonlinear Vlasov--Poisson equation.
  Speculations about this issue can be found in \cite{linzeng:preprint} where it is
  shown that (unlike in the linearized case) one needs more than one derivative
  on the perturbation. As a first step in this direction, we mention that our methods imply
  a bound like $O(\delta/(1+t)^{r-\ov{r}})$ for times $t=O(1/\delta)$, where $\ov{r}$
  is a constant and $r>\ov{r}$, as soon as the initial perturbation has norm $\delta$ 
  in a functional
  space $\cW^r$ involving $r$ derivatives in a certain sense. 
  The reason why this is nontrivial is that the natural
  time scale for nonlinear effects in the Vlasov--Poisson equation is not $O(1/\delta)$,
  but $O(1/\sqrt{\delta})$, as predicted by O'Neil \cite{oneil} and very
  well checked in numerical simulations~\cite{manfredi:landau:97}.\footnote{ Passing from
    $O(1/\sqrt{\delta})$ to $O(1/\delta)$ is arguably an
    infinite-dimensional counterpart of Laplace's averaging principle,
    which yields stability for certain Hamiltonian systems over time
    intervals $O(1/\delta^2)$ rather than $O(1/\delta)$.} Let us sketch the argument
  in a few lines. Assume that
    (for some positive constants $c,C$)
  \begeq\label{hin} h_i = \sum_n h_i^n,\qquad
  \| h_i ^n \|_{\cZ^{\lambda_n,\lambda_n;1}} \leq
  \frac{C^n}{2^{rn}},\qquad \lambda_n = \frac{c\,n}{2^n}.
  \endeq
  Then we may run the Newton scheme again choosing $\alpha_n\sim c\, n/2^n$,
  $c_n = O(\delta\,2^{-(r-r_1)n})$ 
  and $\var_n =c'\,\delta$. Over a time-interval
  of length $O(1/\delta)$, Theorem \ref{thmgrowth}(ii) only yields a
  multiplicative constant $O(e^{c\,\delta\,t}/\alpha_n^9) = O(2^{10n})$.  
  In the end, after Sobolev injection again, we recover
  a time-decay on the force $F$ like
\[ \delta \sum_n 2^{nr_2}\,2^{-nr}\,e^{-\lambda_n t}\leq
C\,\delta\,\sup_n \Bigl(2^{-n(r-r_3)}\,e^{-\lambda_n t}\Bigr) \leq \frac{C\,\delta}{(1+t)^{r-r_4}},\]
  as desired. Equation \eqref{hin} means that $h_i$ is of size $O(\delta)$ in a functional
  space $\cW^r$ whose definition is close to the Littlewood--Paley characterization
  of a Sobolev space with $r$ derivatives. In fact, if the conjecture formulated in Remark 
  \ref{rkMiro} holds true, then it can be shown that $\cW^r$ contains all functions 
  in the Sobolev space $W^{r+r_0,2}$ satisfying an adequate moment condition, 
  for some constant $r_0$.  
\end{Rk}

\section{Expansions and counterexamples}
\label{sec:cex}

A most important consequence of the proof of Theorem \ref{thmmain} is
that the asymptotic behavior of the solution of the nonlinear Vlasov
equation can in principle be determined at arbitrary precision as the
size of the perturbation goes to~0. Indeed, if we define
$g^k_\infty(v)$ as the large-time limit of $h^k$ (say in positive
time), then $\|g^k\|=O(\delta_k)$, so $f^0+g^1_\infty + \ldots +
g^n_\infty$ converges very fast to $f_\infty$. In other words, to
investigate the properties of the time-asymptotics of the system, we
may freely exchange the limits $t\to\infty$ and $\delta\to 0$, perform
expansions, etc. This
at once puts on rigorous grounds many asymptotic expansions used by
various authors --- who so far implicitly postulated the possibility
of this exchange.

With this in mind, let us estimate the first corrections to the
linearized theory, in the regime of a very small perturbation and
small interaction strength (which can be achieved by a proper
scaling of physical quantities).
We shall work in dimension $d=1$ and in a periodic box of length $L=1$. 

\subsection{Simple excitation}
For a start, let us consider the case where the perturbation affects only the first spatial frequency.
We let\sm

\bul $f^0(v) = e^{-\pi\,v^2}$: the homogeneous (Maxwellian)
distribution; \sm

\bul $\var \, \rho_i(x) = \var\,\cos(2\pi x)$: the initial space
density perturbation; \sm

\bul $\var \, \rho_i(x)\,\theta(v)$: the initial perturbation of the
distribution function; we denote by $\varphi$ the Fourier transform of
$\theta$; \sm

\bul $\alpha\, W$: the interaction potential, with $W(-x)=W(x)$. We do
not specify its form, but it should satisfy the assumptions in Theorem
\ref{thmmain}.  \sm

We work in the asymptotic regime $\var\to 0$, $\alpha\to 0$.
The parameter $\var$ measures the size of the perturbation, while $\alpha$ measures
the strength of the interaction; after dimensional change, if $W$ is an inverse power,
$\alpha$ can be thought of as an inverse power of the ratio (Debye length)/(perturbation wavelength).
We will not write norms explicitly, but all our computations can be made in
the norms introduced in Section~\ref{sec:analytic}, with small losses
in the regularity indices --- as we have done in all this paper.

The first-order correction $h^1=O(\var)$ to $f^0$ is provided by the
solution of the linearized equation \eqref{Vllin}, here taking the
form
\[ \pa_t h^1 + v\cdot\nabla_x h^1 + F[h^1]\cdot\nabla_v f^0 =0,\]
with initial datum $h^1(0,\,\cdot\,)= h_i:= f_i - f^0$. As in Section \ref{sec:linear} we get a closed equation
for the associated density $\rho[h^1]$:
\[ \hat{\rho}[h^1](t,k) = \tilde{h}_i(k,kt) - 4\pi^2\,\alpha\,\hat{W}(k) \int_0^t \hat{\rho}[h^1](\tau,k)\,
e^{-\pi (k(t-\tau))^2}\,(t-\tau)\,k^2\,d\tau.\]
It follows that $\hat{\rho}[h^1](t,k)=0$ for $k\neq \pm 1$, so the behavior of $\hat{\rho}[h^1]$ is entirely
determined by $u_1(t) = \hat{\rho}[h^1](t,1)$ and $u_{-1}(t) = \hat{\rho}[h^1](t,-1)$, which satisfy
\begin{align}\label{u1}
u_1(t) & = \frac{\var}2\,\varphi(t) - 4\pi^2\alpha\, \hat{W}(1) \int_0^t u_1(\tau)\,e^{-\pi(t-\tau)^2}\,(t-\tau)\,d\tau\\
& = \frac{\var}2\, \Bigl [ \varphi(t) + O(\alpha)\Bigr]. \nonumber
\end{align}
(This equation can be solved explicitly \cite[eq.~6]{bouchet}, but we only need the expansion.)
Similarly,
\begin{align}\label{u-1}
u_{-1}(t) & = \frac{\var}2\, \varphi(-t)  - 4\pi^2\alpha\, \hat{W}(1) \int_0^t u_{-1}(\tau)\,
e^{-\pi(t-\tau)^2}\,(t-\tau)\,d\tau\\
& = \frac{\var}2\, \Bigl [ \varphi(-t) + O(\alpha)\Bigr]. \nonumber
\end{align}

The corresponding force, in Fourier transform, is given by $\hat{F}^1(t,1) =
-2i\pi\alpha \, \hat{W}(1)\, u_1(t)$ and $\hat{F}^1(t,-1) = 2i\pi \alpha
\, \hat{W}(1)\,u_{-1}(t)$. 

From this we also deduce the Fourier transform of $h^1$ itself:
\begeq\label{Fh1} \tilde{h}^1(t,k,\eta) = 
\tilde{h}_i(k,\eta+kt) - 4\pi^2\alpha \, \hat{W}(k)
\int_0^t \hat{\rho}[h^1](\tau,k)\,e^{-\pi (\eta+k(t-\tau))^2}\,\bigl(\eta+k(t-\tau)\bigr)\cdot k\,d\tau;
\endeq
this is 0 if $k\neq \pm 1$, while
\begin{align}\label{h1t} \tilde{h}^1(t,1,\eta) 
& = \frac{\var}{2}\varphi(\eta+t)
- 4\pi^2\alpha\, \hat{W}(1) \int_0^t u_1(\tau)\,e^{-\pi (\eta+(t-\tau))^2}\,\bigl(\eta+(t-\tau)\bigr)\,d\tau\\
&  = \frac{\var}2 \Bigl[ \varphi(\eta+t) + O(\alpha)\Bigr], \nonumber
\end{align}
\begin{align}\label{h-1t} \tilde{h}^1(t,-1,\eta)
& = \frac{\var}{2}\varphi(\eta-t)
+ 4\pi^2\alpha\, \hat{W}(1) \int_0^t u_{-1}(\tau)\,e^{-\pi (\eta-(t-\tau))^2}\,\bigl(\eta-(t-\tau)\bigr)\,d\tau\\
& = \frac{\var}2 \Bigl[ \varphi(\eta-t) + O(\alpha)\Bigr]. \nonumber
\end{align}

To get the next order correction, we solve, as in Section \ref{sec:iter},
\[ \pa_t h^2 + v\cdot\nabla_x h^2 + F[h^1]\cdot\nabla_v h^2 + F[h^2]\cdot(\nabla_v f^0 + \nabla_v h^1)
= - F[h^1]\cdot\nabla_v h^1,\]
with zero initial datum.
Since $h^2= O(\var^2)$, we may neglect the terms $F[h^1]\cdot\nabla_v h^2$ and $F[h^2]\cdot\nabla_v h^1$
which are both $O(\alpha\,\var^3)$. 
So it is sufficient to solve
\begeq\label{path'2} 
\pa_t h'_2 + v\cdot\nabla_x h'_2 + F[h'_2]\cdot\nabla_v f^0 = - F[h^1]\cdot\nabla_v h^1
\endeq
with vanishing initial datum.
As $t\to\infty$, we know that the solution $h'_2(t,x,v)$ is asymptotically close to its spatial average
$\<h'_2\> = \int h'_2\,dx$. Taking the integral over $\T^d$ in \eqref{path'2} yields
\[ \pa_t \<h'_2\> = - \bigl\< F[h^1]\cdot\nabla_v h^1\bigr\>.\]

Since $h^1$ converges to $\<h_i\>$, the deviation of $f$ to $\<f_i\>$ is given, at order $\var^2$, by
\begin{align*}
g(v) & = -\int_0^{+\infty} \bigl\< F[h_1]\cdot\nabla_v h^1\bigr\>(t,v)\,dt \\
& = -\int_0^{+\infty} \sum_{k\in\Z} \hat{F}[h^1](t,-k)\cdot\nabla_v \hat{h}^1(t,k,v)\,dt.
\end{align*}
Applying the Fourier transform and using
\eqref{u1}-\eqref{u-1}-\eqref{h1t}-\eqref{h-1t}, we deduce
\begin{align*}
\tilde{g}(\eta) & = - \int_0^{+\infty} \sum_{k\in\Z} \hat{F}[h^1](t,-k)\cdot\tilde{\nabla_v h}^1(t,k,\eta)\,dt\\
& = - \int_0^{+\infty} \hat{F}[h^1](t,-1)\,(2i\pi\eta)\,\tilde{h}^1(t,1,\eta)\,dt\\
& \quad - \int_0^{+\infty} \hat{F}[h^1](t,1)\,(2i\pi\eta)\,\tilde{h}^1(t,-1,\eta)\,dt\\
& = \pi^2\var^2\alpha\, \hat{W}(1)\,\eta
\left( \int_0^{+\infty} \varphi(-t)\, \varphi(\eta+t)\,dt
- \int_0^{+\infty} \varphi(t)\,\varphi(\eta-t)\,dt + O(\alpha) \right)\\
& = - \pi^2\var^2\alpha\, \hat{W}(1)\,\eta
\left( \int_{-\infty}^{+\infty} \varphi(t)\, \varphi(\eta-t)\,\sign (t)\,dt + O(\alpha)\right).
\end{align*}

Summarizing:
\begeq\label{sumcex}
\begin{cases}
\dps \lim_{t\to\infty} \tilde{f}(t,k,\eta) = 0 \quad \text{if $k\neq 0$} \\[3mm]
\dps \lim_{t\to\infty} \tilde{f}(t,0,\eta) =
\tilde{f}_i(t,0,\eta) 
- \var^2\,\alpha\, \bigl(\pi^2\,\hat{W}(1)\bigr)\,\eta
\left( \int_{-\infty}^{+\infty} \varphi(t)\, \varphi(\eta-t)\,\sign (t)\,dt + O(\alpha)\right).
\end{cases}
\endeq

Since $\varphi$ is an arbitrary analytic profile, this simple
calculation already shows that the asymptotic profile is not necessarily the spatial mean of the
initial datum.

Assuming $\var\ll \alpha$, higher order expansions in $\alpha$ can be obtained by bootstrap
on the equations \eqref{u1}-\eqref{u-1}-\eqref{h1t}-\eqref{h-1t}: for instance,
\begin{align*} & \lim_{t\to\infty} \tilde{f}(t,0,\eta) = \
\tilde{f}_i(0,\eta) - \var^2\,\alpha\,\bigl(\pi^2\,\hat{W}(1)\bigr)\,\eta\,
\int_{-\infty}^{+\infty} \varphi(t)\, \varphi(\eta-t)\,\sign (t)\,dt  \\
& \quad 
- \var^2\,\alpha^2\, \bigl(2\pi^2\hat{W}(1)\bigr)^2\,\eta \,
\biggl\{ \int_0^\infty \int_0^t 
\Bigl( \varphi(\eta+t)\,\varphi(-\tau) - \varphi(\eta-t)\,\varphi(\tau)\Bigr)\,e^{-\pi(t-\tau)^2}\,
(t-\tau) \\
& \qqquad  + \varphi(\tau)\, \varphi(-t)\, e^{-\pi (\eta+(t-\tau))^2}\, \bigl(\eta+(t-\tau)\bigr)\\
& \qqquad + \varphi(-\tau)\,\varphi(t)\, e^{-\pi (\eta-(t-\tau))^2}\, \bigl(\eta-(t-\tau)\bigr) \biggr\}\,d\tau
\ \ + O (\var^2\alpha^3).
\end{align*} 

What about the limit in negative time? Reversing time is equivalent to
changing $f(t,x,v)$ into $f(t,x,-v)$ and letting time go forward. So
we define $S(v):=-v$, $T(\varphi)(\eta) :=
\var^2\,\alpha\,\pi^2\,\hat{W}(1)\,\eta \int_{-\infty}^{+\infty}
\varphi(t) \varphi(\eta-t)\,\sign(t)\,dt$; then $T(\varphi\circ S) =
T(\varphi)\circ S$, which means that the solutions constructed above
are always {\em homoclinic} at order $O(\var^2\alpha)$. The same is
true for the more precise expansions at order $O(\var^2\alpha^2)$, and
in fact it can be checked that the whole distribution $f^2$ is
homoclinic; in other words, $f$ is homoclinic up to possible
corrections of order $O(\var^4)$.  To exhibit heteroclinic deviations,
we shall consider more general perturbations.

\subsection{General perturbation}

Let us now consider a ``general'' initial datum $f_i(x,v)$ close to
$f^0(v)$, and expand the solution $f$.  We write $\var \, \varphi_k(\eta) =
(f_i-f^0)^{\tilde{\quad }}(k,\eta)$ and $\rho^m = \rho[h^m]$. The
interaction potential is assumed to be of the form $\alpha\,W$ with
$\alpha\ll 1$ and $W(x)=W(-x)$. The first equations of the Newton
scheme are \begeq\label{rho1gen} \hat{\rho}^1(t,k) =
\var\,\varphi_k(kt) - 4\pi^2\alpha\,\hat{W}(k) \int_0^t
\hat{\rho}^1(\tau,k)\,\tilde{f}^0\bigl(k(t-\tau)\bigr)\,|k|^2\,(t-\tau)\,d\tau,
\endeq
\begeq\label{h1gen}
\tilde{h}^1(t,k,\eta) = \var\,\varphi_k(\eta + kt)
- 4 \pi^2\alpha\,\hat{W}(k) \int_0^t \hat{\rho}^1(\tau,k)\,\tilde{f}^0\bigl(\eta+k(t-\tau)\bigr)\,
k\cdot \bigl(\eta+k(t-\tau)\bigr)\,d\tau,
\endeq
\begin{align}\label{h2gen}
\tilde{h}^2(t,k,\eta)
& = -4\pi^2\alpha\, \hat{W}(k)
\int_0^t \hat{\rho}^2 (\tau,k)\, \tilde{f}^0\bigl(\eta+k(t-\tau)\bigr)\,
k\cdot \bigl(\eta+k(t-\tau)\bigr)\,d\tau \\
& \nonumber - 4 \pi^2 \alpha \int_0^t \sum_\ell \hat{W}(\ell)\,\hat{\rho}^1(\tau,\ell)\,
\tilde{h}^1\bigl(\tau,k-\ell, \eta+k(t-\tau)\bigr)\,
\ell\cdot \bigl(\eta+k(t-\tau)\bigr)\,d\tau \\
& \nonumber - 4 \pi^2 \alpha \int_0^t \sum_\ell \hat{W}(\ell)\,\hat{\rho}^2(\tau,\ell)\,\tilde{h}^1
\bigl(\tau,k-\ell,\eta+k(t-\tau)\bigr)\,
\ell\cdot \bigl(\eta+k(t-\tau)\bigr)\,d\tau \\
& \nonumber - 4\pi^2\alpha \int_0^t \sum_\ell \hat{W}(\ell)\, \hat{\rho}^1(\tau,\ell)\,\tilde{h}^2
\bigl(\tau,k-\ell,\eta+k(t-\tau)\bigr)\,
\ell\cdot \bigl(\eta+k(t-\tau)\bigr)\,d\tau,
\end{align}
\begin{align}\label{r2gen}
\hat{\rho}^2(t,k)
& = -4\pi^2\alpha\, \hat{W}(k)
\int_0^t \hat{\rho}^2 (\tau,k)\, \tilde{f}^0\bigl(k(t-\tau)\bigr)\,
|k|^2\,(t-\tau)\bigr)\,d\tau \\
& \nonumber - 4 \pi^2 \alpha \int_0^t \sum_\ell \hat{W}(\ell)\,\hat{\rho}^1(\tau,\ell)\,
\tilde{h}^1\bigl(\tau,k-\ell, k(t-\tau)\bigr)\,
\ell\cdot k\, (t-\tau)\,d\tau \\
& \nonumber - 4 \pi^2 \alpha \int_0^t \sum_\ell \hat{W}(\ell)\,\hat{\rho}^2(\tau,\ell)\,\tilde{h}^1
\bigl(\tau,k-\ell,k(t-\tau)\bigr)\,
\ell\cdot k\, (t-\tau)\,d\tau \\
& \nonumber - 4\pi^2\alpha \int_0^t \sum_\ell \hat{W}(\ell)\, \hat{\rho}^1(\tau,\ell)\,\tilde{h}^2
\bigl(\tau,k-\ell,k(t-\tau)\bigr)\,
\ell\cdot k\, (t-\tau)\,d\tau.
\end{align}
Here $k$ and $\ell$ run over $\Z^d$.

From \eqref{rho1gen}--\eqref{h1gen} we see that $\rho^1$ and $h^1$ depend {\em linearly} on $\var$, and
\begeq\label{hatrho11} 
\hat{\rho}^1(t,k) = \var\,\bigl[\varphi_k(kt) + O(\alpha)\bigr],\qquad
\tilde{h}^1(t,k,\eta) = \var\, \bigl[\varphi_k(\eta+kt) + O(\alpha)\bigr].
\endeq

Then from \eqref{h2gen}--\eqref{r2gen}, $\rho^2$ and $h^2$ are $O(\var^2\,\alpha)$; so by plugging
\eqref{hatrho11} in these equations we obtain
\begeq\label{22}
\hat{\rho}^2(t,k) = -4\pi^2\var^2\,\alpha \int_0^t \sum_\ell \hat{W}(\ell)\,
\varphi_\ell(\ell\tau)\, \varphi_{k-\ell}(kt-\ell\tau)\,\ell\cdot k\, (t-\tau)\,d\tau\
+ O(\var^2\,\alpha^2) + O(\var^3\,\alpha),
\endeq
\begeq\label{32}
\tilde{h}^2(t,k,\eta) = -4\pi^2\var^2\,\alpha
\int_0^t \sum_\ell \hat{W}(\ell)\,\varphi_\ell(\ell\tau)\,\varphi_{k-\ell}(\eta+kt-\ell\tau)\,
\ell\cdot \bigl(\eta+k(t-\tau)\bigr)\,d\tau\ + O(\var^2\,\alpha^2) + O(\var^3\,\alpha).
\endeq
We plug these bounds again in the right-hand side of \eqref{h2gen} to find
\begeq\label{tildeh20}
\tilde{h}^2(t,0,\eta) = (\II)_\var(t,\eta) + (\III)_\var(t,\eta) + O(\var^3\,\alpha^3),
\endeq
where
\[ (\II)_\var(t,\eta) = -4\pi^2\,\alpha\int_0^t \sum_\ell (\ell\cdot\eta)\,
\hat{W}(\ell)\,\hat{\rho}^1(\tau,\ell)\,\tilde{h}^1(\tau,-\ell,\eta)\,d\tau\]
is quadratic in $\var$, and $(\III)_\var(t,\eta)$ is a third-order correction:
\begin{align}\label{III}
(\III)_\var(t,\eta)
= & 16 \pi^4\,\var^3\,\alpha^2
\sum_{m,\ell\in\Z^d} \hat{W}(\ell)\,\hat{W}(m)\\
& \int_0^t \int_0^\tau \varphi_m(ms)\,
\Bigl\{\varphi_{\ell-m}(\ell\tau-ms)\,\varphi_{-\ell}(\eta-\ell\tau)\,
(\ell\cdot m)\,(\tau-s) \nonumber \\
& \qquad\qquad\qquad + \varphi_\ell (\ell\tau)\,\varphi_{-\ell-m}(\eta-\ell\tau-ms)\,m\cdot (\eta-\ell(\tau-s))\Bigr\}\,
(\ell\cdot \eta)\,ds\,d\tau. \nonumber
\end{align}

If $\tilde{f}^0$ is even, changing $\varphi_k$ into $\varphi_k(-\,\cdot\,)$ and $\eta$ into $-\eta$
amounts to change $k$ into $-k$ at the level of \eqref{rho1gen}--\eqref{h1gen}; but then
$(\II)_\var$ is invariant under this operation. We conclude that $f$ is always homoclinic at
second order in $\var$, and we consider the influence of the third-order term \eqref{III}.
Let
\[ C[\varphi](\eta) := \lim_{t\to\infty} (\III)_\var(t,\eta).\]
After some relabelling, we find
\begin{multline}\label{limIII}
C[\varphi](\eta) = 16 \pi^4\,\var^3\,\alpha^2
\sum_{k,\ell\in\Z^d}
\hat{W}(k)\,\hat{W}(\ell)\\
\int_0^\infty \int_0^t \varphi_\ell(\ell\tau)\ \Bigl\{
\varphi_{k-\ell}(kt-\ell\tau)\,\varphi_{-k}(\eta-kt)\, (k\cdot\ell)\,(t-\tau)\\
+ \varphi_k(kt) \, \varphi_{-k-\ell}(\eta-kt-\ell\tau)\,\ell\cdot (\eta-k(t-\tau))
\Bigr\}\, (k\cdot\eta)\,d\tau\,dt.
\end{multline}
Now assume that $\varphi_{-k}=\sigma\,\varphi_k$ with $\sigma=\pm 1$.
($\sigma=1$ means that the perturbation is even in $x$; $\sigma=-1$ that it is odd.)
Using the symmetry $(k,\ell)\leftrightarrow (-k,-\ell)$ one can check that
\[ C[\varphi\circ S]\circ S = \sigma\,C[\varphi],\] where
$S(z)=-z$. In particular, if the perturbation is odd in $x$, then the
third-order correction imposes a heteroclinic behavior for the
solution, as soon as $C[\varphi]\neq 0$. 

To construct an example where $C[\varphi]\neq 0$, we set $d=1$, $f^0$
$=$ Gaussian, $f_i-f^0 = \sin(2\pi x)\,\theta_1(v) + \sin(4\pi x)\,\theta_2(v)$,
$\varphi_1 = -\varphi_{-1} = \tilde{\theta}_1/2$,
$\varphi_2 = -\varphi_{-2} = \tilde{\theta}_2/2$.
The six pairs $(k,\ell)$ contributing to
\eqref{limIII} are $(-1,1)$, $(1,-1)$, $(1,2)$, $(2,1)$, $(-1,-2)$, $(-2,-1)$,
By playing on the respective sizes of $\hat{W}(1)$ and
$\hat{W}(2)$ (which amounts in fact to changing the size of the box),
it is sufficient to consider the terms with coefficient
$\hat{W}(1)^2$, {\it i.e.}, the pairs $(-1,1)$ and $(1,-1)$.  Then the
corresponding bit of $C[\varphi](\eta)$ is
\begin{align*} 
- 16 \pi^4\,\var^3\,\alpha^2\,\hat{W}(1)^2\,
\eta \int_0^\infty \int_0^t 
\Bigl[ & \varphi_1(\tau)\,\varphi_1(\eta+t)\,\varphi_2(-t+\tau)\,(t-\tau) \\
& + \varphi_1(\tau) \,\varphi_1(t)\,\varphi_2(\eta+t-\tau)\,(\eta+t-\tau) \\
& +\varphi_1(-\tau)\,\varphi_1(\eta-t)\,\varphi_2(t+\tau)\,(t-\tau)\\
& +\varphi_1(-\tau)\,\varphi_1(t)\,\varphi_2(\eta-t+\tau)\,(t-\tau-\eta)\Bigr]\,d\tau\,dt.
\end{align*}
If we let $\varphi_1$ and $\varphi_2$ vary in such a way that they become
positive and almost concentrated on $\R_+$, the only remaining term is
the one in $\varphi_1(\tau)\,\varphi_2(\eta+t-\tau)\,\varphi_1(t)$,
and its contribution is negative for $\eta>0$. 
So, at least for certain values of $W(1)$ and $W(2)$ there is a choice of analytic
functions $\varphi_1$ and $\varphi_2$, such that $C[\varphi]\neq 0$.
This demonstrates the existence of heteroclinic trajectories.
\sm

To summarize: At first order in $\var$, the convergence is to the
spatial average; at second order there is a homoclinic correction; at
third order, if at least three modes with zero sum are excited, there
is possibility of heteroclinic behavior.

\begin{Rk} \label{rkFB}
As pointed out to us by Bouchet, the existence of heteroclinic trajectories implies
that the asymptotic behavior cannot be predicted on the basis of the invariants of the equation
and the interaction; indeed, the latter do not distinguish between the forward and backward solutions.
\end{Rk}

\section{Beyond Landau damping}
\label{sec:beyond}

We conclude this paper with some general comments about the physical
implications of Landau damping.

Remark \ref{rkFB} show in particular that
there is no ``universal'' large-time behavior of the solution of the
nonlinear Vlasov equation in terms of just, say, conservation laws and
the initial datum; the {\em dynamics} also have to enter
explicitly. One can also interpret this as a lack of ergodicity: the
nonlinearity is not sufficient to make the system explore the space of
all ``possible'' distributions and to choose the most favorable one,
whatever this means. Failure of ergodicity for a system of finitely
many particles was already known to occur, in relation to the KAM
theorem; this is mentioned e.g. in \cite[p.~257]{MP:Euler:book} for
the vortex system. There it is hoped that such behavior disappears as
the dimension goes to infinity; but now we see that it also exists
even in the infinite-dimensional setting of the Vlasov equation.

At first, this seems to be bad news for the statistical theory of the Vlasov
equation, pioneered by Lynden-Bell \cite{LB:violent} and explored by
various authors \cite{CSR,miller,robert,THLB:Hfunctions,WZS:mnras}, since even the
sophisticated variants of this theory try to predict the likely final
states in terms of just the characteristics of the initial data.  In
this sense, our results provide support for an objection raised by
Isichenko \cite[p.~2372]{isichenko:nonlinearlandau} against the
statistical theory.

However, looking more closely at our proofs and results, proponents of
the statistical theory will have a lot to rejoice about.

To start with, our results are the first to rigorously establish that the nonlinear Vlasov
equation does enjoy some asymptotic ``stabilization'' property in
large time, without the help of any extra diffusion or ensemble
averaging.

Next, the whole analysis is perturbative: each stable spatially
homogeneous distribution will have its small ``basin of damping'', and
it may be that some distributions are ``much more stable'' than
others, say in the sense of having a larger basin.

Even more importantly, in Section \ref{sec:response} we have crucially
used the smoothness to overcome the potentially destabilizing
nonlinear effects. So any theory based on nonsmooth functions might not be constrained
by Landau damping.  This certainly applies to a statistical theory,
for which smooth functions should be a zero-probability set.

Finally, to overcome the nonlinearity, we had to cope with huge constants
(even qualitatively larger than those appearing in classical KAM theory).
If one believes in the explanatory virtues of proofs, 
these large constants might be the indication that Landau damping is a
thin effect, which might be neglected when it comes to predict the
``final'' state in a ``turbulent'' situation.

Further work needs to be done to understand whether these
considerations apply equally to the electrostatic and gravitational
cases, or whether the electrostatic case is favored in these respects;
and what happens in ``low'' regularity.

Although the underlying mathematical and physical mechanisms differ,
nonlinear Landau damping (as defined by Theorem \ref{thmmain}) may
arguably be to the theory of Vlasov equation what the KAM theorem is
to the theory of Hamiltonian systems. Like the KAM theorem, it might
be conceptually important in theory and practice, and still be severely
limited.\footnote{It is a well-known scientific paradox that the KAM
  theorem was at the same time tremendously influential in the science
  of the twentieth century, and so restrictive that its assumptions
  are essentially never satisfied in practice.}

Beyond the range of application of KAM theory lies the softer, more
robust {\em weak KAM theory} developed by Fathi \cite{fathi} in
relation to Aubry--Mather theory. By a nice coincidence, a Vlasov
version of the weak KAM theory has just been developed by Gangbo and
Tudorascu \cite{GT}, although with no relation to Landau
damping. Making the connection is just one of the many developments
which may be explored in the future.

\appendix

\section*{Appendix}

\def\thesection{A} In this appendix we gather some elementary tools,
our conventions, and some reminders about calculus.
We write $\N_0 = \{0,1,2,\ldots\}$.

\subsection{Calculus in dimension $d$} \label{app:exp}

If $n\in\N_0^d$ we define
\[ n! = n_1!\ldots n_d!\]
and 
\[ \Cnk{n}{m} = \Cnk{n_1}{m_1}\ldots \Cnk{n_d}{m_d}.\]

If $z\in\C^d$ and $n\in\Z^d$, we let 
\[ \|z\| = |z_1|+\ldots + |z_d|;\qquad
z^n = z_1^{n_1}\ldots z_n^{n_d}\in\C;\qquad
|z|^n = |z^n|.\]

In particular, if $z\in\C^d$ we have
\[ e^{\|z\|} = e^{|z_1|+\ldots +|z_d|}
= \sum_{n\in\N_0^d} \frac{\|z\|^n}{n!}.\]
We may write $e^{|z|}$ instead of $e^{\|z\|}$.

\subsection{Multi-dimensional differential calculus} \label{appdc}

The Leibniz formula for functions $f,g:\R\to\R$ is
\[ (fg)^{(n)} = \sum_{m\leq n} \Cnk{n}{m} f^{(m)} g^{(n-m)},\]
where of course $f^{(n)}=d^nf/dx^n$. The expression of derivatives of composed functions
is given by the Fa\`a di Bruno formula:
\[ (f\circ G)^{(n)} = \sum_{\sum j m_j = n}
\frac{n!}{m_1!\ldots m_n!}\,
\Bigl(f^{(m_1+\ldots+m_n)}\circ G \Bigr)
\prod_{j=1}^n \left( \frac{G^{(j)}}{j!}\right)^{m_j}.\]
\med

These formulas remain valid in several dimensions, provided that one defines,
for a multi-index $n=(n_1,\ldots,n_d)$,
\[ f^{(n)} = \frac{\pa^{n_1}}{\pa x_1^{n_1}}\ldots \frac{\pa^{n_d}}{\pa x_d^{n_d}} f.\]

They also remain true if $(\pa_1,\ldots,\pa_d)$ is replaced by a $d$-tuple of 
{\em commuting} derivation operators.

As a consequence, we shall establish the following Leibniz-type formula for
operators that are combinations of gradients and multiplications.

\begin{Lem} \label{lemlei}
Let $f$ and $g$ be functions of $v\in\R^d$, and $a,b\in\C^d$. Then for any $n\in\N^d$,
\[ (\nabla_v + (a+b))^n (fg) = 
\sum_{m\leq n} \Cnk{n}{m} (\nabla_v + a)^m f\, (\nabla_v + b)^{n-m} g.\]
\end{Lem}

\begin{proof}
The right-hand side is equal to
\[ \sum_{m,q,r} \Cnk{n}{m} \Cnk{m}{q}\Cnk{n-m}{r}
\nabla_v^q f\, \nabla_v^r g\, a^{m-q}\, b^{n-m-r}.\]
After changing indices $p=q+r$, $s=m-q$, this becomes
\begin{align*} 
\sum_{s,p,r} \Cnk{n}{p} \Cnk{p}{r} \Cnk{n-p}{s} 
\nabla_v^rg\, \nabla_v^{p-r}f\, a^s \, b^{n-p-s}
& = \sum_p \Cnk{n}{p} \nabla_v^p (fg)\, (a+b)^{n-p}\\
& = (\nabla_v + (a+b))^n (fg).
\end{align*}
\end{proof}

\subsection{Fourier transform} \label{appFourier}

If $f$ is a function $\R^d\to\R$, we define
\begeq\label{tildef}
\tilde{f}(\eta) = \int_{\R^d} e^{-2i\pi \eta\cdot v}\,f(v)\,dv;
\endeq
then we have the usual formulas
\[ f(v) = \int_{\R^d} \tilde{f}(\eta)\, e^{2i\pi \eta\cdot v}\,d\eta;\qquad
\tilde{\nabla f}(\eta) = 2i\pi \eta\, \tilde{f}(\eta).\]

Let $\T^d_L = \R^d/(L \Z^d)$. If $f$ is a function $\T^d_L\to\R$, we define
\begeq\label{hatf} \hat{f}^{(L)}(k) = \int_{\T^d_L} e^{-2i\pi \frac{k}{L}\cdot x}\, f(x)\,dx;
\endeq
then we have
\[ f(x) = \frac1{L^d} \sum_{k\in\Z^d} \hat{f}^{(L)}(k)\, e^{2i\pi \frac{k}{L}x};\qquad
\hat{\nabla f}^{(L)}(k) = 2i\pi \frac{k}{L}\, \hat{f}^{(L)}(k).\]

If $f$ is a function $\T^d_L \times \R^d \to\R$, we define 
\begeq\label{tf}
\tilde{f}^{(L)}(k,\eta) = \int_{\T^d_L} \int_{\R^d} e^{-2i\pi
  \frac{k}{L}\cdot x}\, e^{-2i\pi \eta\cdot v}\,f(x,v)\,dx\,dv;
\endeq
so that the reconstruction formula reads
\[f(x,v) = \frac1{L^d}\sum_{k\in\Z^d} \int_{\R^d} \tilde{f}^{(L)}(k,\eta)\, 
e^{2i\pi \frac{k}{L}\cdot x} e^{2i\pi \eta\cdot v}\,dv.\]

When $L=1$ we do not specify it: so we just write
\[ \hat{f} = \hat{f}^{(1)};\qquad \tilde{f} = \tilde{f}^{(1)}.\]
(There is no risk of confusion since in that case, \eqref{tf} and \eqref{tildef} coincide.)

\subsection{Fixed point theorem} \label{app:FP}

The following theorem is one of the many variants of the Picard fixed point theorem.
We write $B(0,R)$ for the ball of center~0 and radius~$R$.

\begin{Thm}[Fixed point theorem] \label{thmfpt}
Let $E$ be a Banach space, $F:E\to E$, and $R=2\|F(0)\|$. If $F$ is $(1/2)$-Lipschitz
$B(0,R)\to E$, then it has a unique fixed point in $B(0,R)$.
\end{Thm}

\begin{proof}
Uniqueness is obvious. To prove existence, run the classical Picard iterative scheme
initialized at 0: $x_0=0$, $x_1=F(0)$, $x_2=F(F(0))$, etc. It is clear that
$(x_n)$ is a Cauchy sequence and
$\|x_n\|\leq \|F(0)\|(1+\ldots + 1/2^n)\leq 2 \|F(0)\|$, so $x_n$ converges
in $B(0,R)$ to a fixed point of $F$.
\end{proof}

\bibliographystyle{acm}
\bibliography{./landau}


\signcm \signcv

\end{document}